\documentclass[reqno,twoside]{amsart}
\usepackage{etex}
\usepackage{amscd}
\usepackage{amsfonts}
\usepackage{amsmath}
\usepackage{amssymb}
\usepackage{amsthm}
\usepackage{amsaddr}
\usepackage{fancyhdr}
\usepackage{latexsym}
\usepackage[colorlinks=true, pdfstartview=FitV, linkcolor=blue, citecolor=blue, urlcolor=blue]{hyperref}
\usepackage{enumitem}      
\usepackage{mathtools}            
\usepackage{indentfirst} 
\usepackage{color}
\usepackage{caption}          
\usepackage[normalem]{ulem}
\usepackage{upgreek}
\usepackage{dutchcal}
\usepackage{mathrsfs}
\newcommand{\nn}{\nonumber}

\renewcommand{\b}{\textcolor{blue}}

\usepackage{subcaption}
\let\Re\relax
\let\Im\relax

\DeclareMathOperator{\Re}{Re}
\DeclareMathOperator{\Im}{Im}
\DeclareMathOperator{\supp}{supp}

\newcommand{\N}{\mathbb{N}}
\newcommand{\Z}{\mathbb{Z}}

\newcommand{\R}{\mathbb{R}}
\newcommand{\C}{\mathbb{C}}

\newcommand{\eq}[1]{\begin{align}#1\end{align}}
\newcommand{\eqs}[1]{\begin{align*}#1\end{align*}}

\newcommand{\p}[1]{\left( #1 \right)}
\newcommand{\set}[1]{\left\{ #1 \right\}}
\newcommand{\abs}[1]{\left\lvert #1 \right\rvert}
\newcommand{\norm}[1]{\left\lVert #1 \right\rVert}
\newcommand{\ceil}[1]{\left\lceil #1 \right\rceil}
\newcommand{\floor}[1]{\left\lfloor #1 \right\rfloor}

\newcommand{\comment}[1]{}

\newcommand{\alignlabel}[1]{\stepcounter{equation} \tag{\theequation} \label{#1}}
\newcommand{\cle}{\lesssim}

\let\EPSILON\epsilon
\let\VAREPSILON\varepsilon

\renewcommand{\epsilon}{\VAREPSILON}
\renewcommand{\varepsilon}{\EPSILON}

\renewcommand{\subset}{\subseteq}

\renewcommand{\L}{\mathcal{L}}
\renewcommand{\b}[1]{\left[ #1 \right]}

\renewcommand{\hat}{\widehat}
\renewcommand{\tilde}{\widetilde}

\definecolor{red}{RGB}{220,0,0}
\definecolor{green}{RGB}{0,180,0}
\usepackage{thmtools}
\declaretheoremstyle[headfont=\kpfonts]{normalhead}

\newtheorem{theorem}{Theorem}[section]

\newtheorem{lemma}{Lemma}[section]
\newtheorem{corollary}{Corollary}[section]

\theoremstyle{definition}

\newtheorem{remark}{Remark}[section]

\allowdisplaybreaks
\numberwithin{equation}{section}
\numberwithin{figure}{section}
\usepackage{geometry}
\usepackage{tikz}
\usetikzlibrary{arrows.meta}
\usetikzlibrary{decorations.markings}
\tikzset{->-/.style={decoration={
  markings,
  mark=at position #1 with {\arrow{>}}},postaction={decorate}}}
  \tikzset{middlearrow/.style={
        decoration={markings,
            mark= at position 0.55 with {\arrow{#1}} ,
        },
        postaction={decorate}
    }
}
\let\OLDthebibliography\thebibliography
\renewcommand\thebibliography[1]{
  \OLDthebibliography{#1}
  \setlength{\parskip}{0pt}
  \setlength{\itemsep}{0pt plus 0.3ex}
}
\AtBeginDocument{
   \def\MR#1{}
}

\makeatletter
  \renewcommand\subsection{\@startsection{subsection}{2}%
  \z@{-0.8\linespacing\@plus-0.7\linespacing}{0.7\linespacing}%
  {\bfseries}}
\makeatother

\makeatletter
\@namedef{subjclassname@2020}{%
  \textup{2020} Mathematics Subject Classification}
\makeatother

\begin{document}

\title{Well-posedness of the higher-order nonlinear Schr\"odinger equation on a finite interval}

\author{Chris Mayo$^\dagger$, Dionyssios Mantzavinos$^\dagger$,  Türker Özsarı$^{*}$}

\address{\normalfont $^{\dagger}$Department of Mathematics, University of Kansas
\\
\normalfont $^{*}$Department of Mathematics, Bilkent University
} \email{\!cmayo@ku.edu, mantzavinos@ku.edu \textnormal{(corresponding author)}, turker.ozsari@bilkent.edu.tr}

\thanks{\textit{Acknowledgements.} DM's and CM's research was partially supported by the U.S. National Science Foundation (NSF-DMS 2206270 and NSF-DMS 2509146) and the Simons Foundation (SFI-MPS-TSM-00013970). 
TÖ's research was supported by BAGEP 2020 Young Scientist Award.
}

\date{December 19, 2024. Revised: November 11, 2025.}

\begin{abstract}
We establish the local Hadamard well-posedness of a certain third-order nonlinear Schr\"odinger equation with a multi-term linear part and a general power nonlinearity  known as the higher-order nonlinear Schr\"odinger equation, formulated on a finite interval with a combination of nonzero Dirichlet and Neumann boundary conditions. Specifically, for initial and boundary data in suitable Sobolev spaces that are related to one another through the time regularity induced by the equation, we prove the existence of a unique solution as well as the continuous dependence of that solution on the data. The precise choice of solution space depends on the value of the Sobolev exponent and is dictated both by the linear estimates associated with the forced linear counterpart of the nonlinear initial-boundary value problem and, in the low-regularity setting below the Sobolev algebra property threshold, by certain nonlinear estimates that control the Sobolev norm of the power nonlinearity. In particular, as usual in Schr\"odinger-type equations, in the case of low regularity it is necessary to derive Strichartz estimates in suitable Lebesgue/Bessel potential spaces. The proof of well-posedness is based on a contraction mapping argument combined with the aforementioned linear estimates, which are established by employing the explicit solution formula for the forced linear problem derived via the unified transform of Fokas. Due to the nature of the finite interval problem, this formula involves contour integrals in the complex Fourier plane with corresponding integrands that contain differences of exponentials in their denominators, thus requiring  delicate handling through appropriate contour deformations. It is worth noting that, in addition to the various linear and nonlinear results obtained for the finite interval problem, novel time regularity results are established here also for the relevant half-line problem.
\end{abstract}

\vspace*{-0.5cm}

\maketitle
\markboth
{C. Mayo, D. Mantzavinos \& T. Özsarı}
{Well-posedness of the higher-order nonlinear Schr\"odinger equation on a finite interval}

\section{Introduction}

We consider the higher-order nonlinear Schr\"odinger (HNLS) equation with a power nonlinearity, formulated on the finite interval $(0, \ell)$ with nonzero boundary conditions, namely
\begin{equation}
\begin{aligned}\label{hnls-ibvp}
&i u_t + i \beta u_{xxx} + \alpha u_{xx} + i \delta u_x = \kappa |u|^{\lambda - 1} u, \quad 0 < x < \ell, \ 0 < t < T,
\\
&u(x, 0) = u_0 (x) \in H^s(0, \ell),\\
&u(0, t) = g_0(t) \in H^{\frac{s+1}{3}}(0, T), \quad  u(\ell, t) = h_0(t) \in H^{\frac{s+1}{3}}(0, T), \quad u_x (\ell, t) = h_1(t) \in H^{\frac{s}{3}}(0, T),
\end{aligned}
\end{equation}
where $\beta > 0$, $\alpha, \delta \in \mathbb R$, $\kappa \in \mathbb C$, $\lambda>1$, $\ell>0$, and $T>0$ is an appropriate lifespan to be determined. 
In the above initial-boundary value problem, the initial data is taken from the 
 $L^2$-based Sobolev space $H^s(0, \ell)$, which is defined as the restriction on the finite interval $(0, \ell)$ of the usual Fourier-based Sobolev space on the whole line 
\eq{
H^s(\R) := \left\{\phi \in L^2(\R): \p{1+k^2}^{\frac s2} \mathcal F\{\phi\} \in L^2(\R)\right\}, \quad s\geq 0,
}
where $\mathcal F\{\cdot\}$ denotes the Fourier transform. The equivalent characterization of $H^s(0, \ell)$ as the space $W^{s, 2}(0, \ell)$ of functions in $L^2(0, \ell)$ whose first $s$ derivatives belong in $L^2(0, \ell)$ will also prove useful for our purposes.
The Sobolev spaces $H^{\frac{s+1}{3}}(0, T)$ and $H^{\frac{s}{3}}(0, T)$ for the three pieces of boundary data are defined similarly. It should be noted that the two Sobolev exponents $\frac{s+1}{3}$ and $\frac{s}{3}$ for the boundary data are determined in terms of the Sobolev exponent $s$ for the initial data through the study of both the spatial and the temporal regularity of the solution of problem~\eqref{hnls-ibvp} (see Theorem \ref{thm:fi-se} and Section \ref{dec-s} for more details). In other words, the (time) regularity of the boundary data is fully dictated by the (space) regularity of the initial data. 

The HNLS equation is a higher-order analogue of the renowned NLS equation $iu_t + u_{xx} = \kappa |u|^{\lambda-1} u$, which is a ubiquitous model in mathematical physics with applications ranging from nonlinear optics to water waves to plasmas to Bose-Einstein condensates. In the context of nonlinear optics, the HNLS equation has been derived as an improved approximation (in comparison to NLS) to the three-dimensional Maxwell equations, as it additionally involves a third-order dispersive term that serves as a necessary correction for modeling pulses in the femtosecond regime \cite{k1985,kh1987}. In its original form, the HNLS equation appeared with a \textit{cubic} power nonlinearity ($\lambda=3$) as well as additional cubic nonlinearities involving derivatives. Here, we consider the case of a \textit{general} power nonlinearity ($\lambda>1$) but without the terms involving the derivative cubic nonlinearities, so that the HNLS equation in \eqref{hnls-ibvp} is a direct higher-dispersion analogue of the NLS equation. Furthermore, it is interesting to note that the third-order dispersion of the HNLS equation induces a time regularity (through the relevant time estimates established in Section  \ref{dec-s}) of the same type with the Korteweg-de Vries (KdV) equation. In this light and from a purely mathematical viewpoint, the HNLS equation can be seen as an interesting hybrid between the NLS and KdV equations ---  the two most celebrated nonlinear dispersive equations in one spatial dimension.

The well-posedness of the initial value problem (also known as the Cauchy problem) for the HNLS equation on the whole line has been studied extensively, e.g. see the works \cite{l1997,s1997,tak2000,cl2003,c2004,c2006,f2023}. On the other hand, following the general trend in the well-posedness theory of nonlinear dispersive equations, the analysis of \textit{nonhomogeneous initial-boundary value problems} for HNLS is much more limited. An important challenge associated with such problems has to do with the presence of a boundary in the spatial domain, which in turn introduces the need for prescribing appropriate boundary conditions. In one space dimension, such domains are the half-line $(0, \infty)$ and the finite interval $(0, \ell)$. While in the case of the initial value problem one can prove local well-posedness by first using the Fourier transform on the whole line to solve the forced linear counterpart of the nonlinear equation and then obtaining the solution to the nonlinear problem as a fixed point (in a suitable function space) through a contraction mapping argument, the situation for initial-boundary value problems is less straightforward. Now, the Fourier transform method is not available and, therefore, an important challenge arises right away concerning the solution of the forced linear problem which is needed for defining the iteration map. Importantly, even when this obstacle is overcome, working outside the standard Fourier transform framework means that several tools from harmonic analysis that play a key role in the derivation of linear estimates for the initial value problem (which are then used in order to establish the contraction) are either no longer available or must be adapted in nontrivial ways. 

A general method for proving the (local) well-posedness of initial-boundary value problems for nonlinear dispersive equations has been developed systematically over the last decade or so. This method, which first appeared in the context of the NLS and KdV equations on the half-line \cite{fhm2017,fhm2016}, relies on the solution of the forced linear problem via the unified transform, also known as the Fokas method \cite{f1997,f2008}. The unified transform was introduced by Fokas in 1997 for solving initial-boundary value problems of (i)~linear and (ii)~integrable nonlinear evolution equations. In the former setting, the unified transform provides the direct analogue of the Fourier transform in domains with a boundary. In the latter setting, it provides the counterpart of the inverse scattering transform used for solving the initial value problem of integrable equations (including NLS and KdV) via the spectral analysis of their Lax pair. As most physical nonlinear dispersive models --- including HNLS --- are not  integrable and hence do not possess a Lax pair, it is not possible to employ the inverse scattering transform (for the initial value problem) and the unified transform (for initial-boundary value problems) in order to study these equations directly at the nonlinear level. As noted above, in the case of the initial value problem, one uses the Fourier transform/contraction mapping approach that leads to well-posedness through a fixed point theorem. In the case of initial-boundary value problems, an analogue of this approach is the method of \cite{fhm2017,fhm2016}, where the Fourier transform is replaced by the unified transform.

In the present work, the unified transform is combined with the basic ideas of \cite{fhm2017,fhm2016} as well as with new techniques needed specifically due to the finite interval setting (as opposed to the half-line) and the multi-term nature of the HNLS equation (as opposed to the simpler linear parts of NLS and KdV) in order to establish the Hadamard well-posedness (namely, the existence and uniqueness of solution, as well as the continuity of the data-to-solution map) of the HNLS finite interval problem~\eqref{hnls-ibvp}. More precisely, in the high-regularity setting of $s>\frac 12$ (``smooth'' data), our result reads as follows:
\begin{theorem}[High-regularity Hadamard well-posedness]\label{high-wp-t}
Suppose $\frac 12<s\leq 2$, $s\neq \frac 32$, and $\lambda>1$, where if $\lambda \notin 2\N+1$ then the following conditions are satisfied:
\eq{\label{l-cond}
\begin{aligned}
&\text{if $s\in\N$, then $\lambda\geq s+1$ if $\lambda\in 2\N$; $\floor{\lambda} \geq s$ if $\lambda \notin \N$},
\\
&\text{if $s\notin\N$, then $\lambda > s+1$ if $\lambda\in 2\N$; $\floor{\lambda} \geq \floor s + 1$ if $\lambda \notin \N$}.
\end{aligned}
}
Furthermore, let $T>0$ be such that
\eq{\label{T-cond-high}
&|\kappa| \max\{c_s, c(s, \lambda)\} \b{2 c(s, T)}^{\lambda} \sqrt T 
\Big(
\norm{u_0}_{H^s(0, \ell)}  
+
\norm{g_0}_{H^{\frac{s+1}{3}}(0, T)}
+
\norm{h_0}_{H^{\frac{s+1}{3}}(0, T)}
+
\norm{h_1}_{H^{\frac{s}{3}}(0, T)}\Big)^{\lambda-1} < 1,
}
where $c(s, T)=\max\big\{c_1(s, T), c_2(s, T), c_2(s, T) \sqrt T\big\}$ with  $c_1(s, T)$ and $c_2(s, T)$ being the constants in the Sobolev estimates  \eqref{sob-est} and \eqref{smooth-est}, $c(s,\lambda)$ is the constant in Lemma \ref{lemma:H-pow-diff}, and $c_s$ is the constant of the algebra property inequality in $H^s(0, \ell)$.
Then, under the compatibility conditions
\begin{equation}\label{comp-cond}
g_0(0) = u_0(0), \quad h_0(0) = u_0(\ell), \quad s> \frac 12, 
\qquad
h_1 (0) =  u_0' (\ell), \quad s > \frac{3}{2},
\end{equation}
the initial-boundary value problem \eqref{hnls-ibvp} for the HNLS equation on a finite interval is locally well-posed in the sense of Hadamard. More specifically, \eqref{hnls-ibvp} possesses a unique solution 
$u \in C_t([0, T]; H_x^s(0, \ell)) \cap L_t^2((0, T); H_x^{s+1}(0, \ell))$
which satisfies the estimate
\eq{
\begin{split}
&\quad
\max\left\{\norm{u}_{C_t([0, T]; H_x^s(0, \ell))}, \norm{u}_{L_t^2((0, T); H_x^{s+1}(0, \ell))}\right\} 
\\
&\leq 2 c(s, T) 
\Big(
\norm{u_0}_{H^s(0, \ell)}  
+
\norm{g_0}_{H^{\frac{s+1}{3}}(0, T)}
+
\norm{h_0}_{H^{\frac{s+1}{3}}(0, T)}
+
\norm{h_1}_{H^{\frac{s}{3}}(0, T)}\Big).
\end{split}
}
Furthermore, the data-to-solution map is locally Lipschitz continuous. In addition, uniqueness holds in the whole of $C_t([0, T]; H_x^s(0, \ell))$.
\end{theorem}

In the setting of low regularity $s<\frac 12$ (``rough'' data), the well-posedness of the finite interval problem \eqref{hnls-ibvp} is more challenging, as it requires the derivation of suitable Strichartz estimates. The precise statement of our result is the following:
\begin{theorem}[Low-regularity Hadamard well-posedness]\label{low-wp-t}
Suppose $0\leq s < \frac 12$ and $2\leq \lambda \leq \frac{7-2s}{1-2s}$, and let   $T>0$ be such that
\eq{\label{T-cond-low}
&|\kappa| c(s, \lambda) \b{2c(s, \lambda, T)}^\lambda T^{\frac{7-\lambda+2s(\lambda-1)}{6}} 
\Big(
\norm{u_0}_{H^s(0, \ell)}  
+
\norm{g_0}_{H^{\frac{s+1}{3}}(0, T)}
+
\norm{h_0}_{H^{\frac{s+1}{3}}(0, T)}
+
\norm{h_1}_{H^{\frac{s}{3}}(0, T)}\Big)^{\lambda-1} < 1,
}
where $c(s, \lambda, T) := \max\big\{c_2(s, T), c_3(s, 2, T), c_3(s, \frac{2\lambda}{1+2(\lambda-1)s}, T)\big\}$ with the three constants involved coming from the Sobolev estimate \eqref{smooth-est} and the Strichartz estimate \eqref{strich-est}, and $c(s,\lambda)$ is the constant in Lemma \ref{qnls-non-l}.
Then, the initial-boundary value problem~\eqref{hnls-ibvp} for the HNLS equation on a finite interval is locally well-posed in the sense of Hadamard. More specifically, \eqref{hnls-ibvp} possesses a unique solution 
$$
u \in C_t([0, T]; H_x^s(0, \ell)) \cap L_t^2((0, T); H_x^{s+1}(0, \ell)) \cap L_t^{\frac{6\lambda}{(1-2s)(\lambda-1)}}((0, T); H_x^{s, \frac{2\lambda}{1+2(\lambda-1)s}}(0, \ell))
$$
which satisfies the estimate
\eq{
\begin{aligned}
&\quad
\max\Big\{\norm{u}_{C_t([0, T]; H_x^s(0, \ell))}, \norm{u}_{L_t^2((0, T); H_x^{s+1}(0, \ell))}, \norm{u}_{L_t^{\frac{6\lambda}{(1-2s)(\lambda-1)}}((0, T); H_x^{s, \frac{2\lambda}{1+2(\lambda-1)s}}(0, \ell))}\Big\}  
\\
&\leq 2 c(s, \lambda, T)
\Big(
\norm{u_0}_{H^s(0, \ell)}  
+
\norm{g_0}_{H^{\frac{s+1}{3}}(0, T)}
+
\norm{h_0}_{H^{\frac{s+1}{3}}(0, T)}
+
\norm{h_1}_{H^{\frac{s}{3}}(0, T)}\Big).
\end{aligned}
}
Furthermore, the data-to-solution map is locally Lipschitz continuous. In addition, uniqueness holds in the whole of $C_t([0, T]; H_x^s(0, \ell)) \cap L_t^{\frac{6\lambda}{(1-2s)(\lambda-1)}}((0, T); H_x^{s, \frac{2\lambda}{1+2(\lambda-1)s}}(0, \ell))$.
\end{theorem}

The proofs of Theorems \ref{high-wp-t} and \ref{low-wp-t} combine a contraction mapping argument (in the relevant solution spaces) with appropriate linear estimates for the forced linear counterpart of the nonlinear problem~\eqref{hnls-ibvp}, namely for the problem
\begin{equation}
\begin{aligned}\label{hls-ibvp}
&i u_t + i \beta u_{xxx} + \alpha u_{xx} + i \delta u_x = f(x, t), \quad 0 < x < \ell, \  0 < t < T,
\\
&u(x, 0) = u_0 (x),\\
&u(0, t) = g_0(t), \quad  u(\ell, t) = h_0(t), \quad u_x (\ell, t) = h_1(t),
\end{aligned}
\end{equation}
where $f(x, t)$ is a given forcing. In particular, we establish the following crucial linear estimates:
\begin{theorem}[Linear estimates]\label{lin-est-t}
For initial data $u_0 \in H^s(0, \ell)$, Dirichlet boundary data $g_0, h_0 \in H^{\frac{s+1}{3}}(0, T)$, Neumann boundary data $h_1 \in H^{\frac{s}{3}}(0, T)$, and forcing $f \in L_t^2((0, T); H_x^s(0, \ell))$ when $s>\frac 12$ and $f \in L_t^1((0, T); H_x^s(0, \ell))$ when $s<\frac 12$, the solution to the initial-boundary value problem \eqref{hls-ibvp} for the linear higher-order Schr\"odinger equation on a finite interval satisfies the Sobolev estimate
\eq{
\begin{split}
\norm{u}_{L_t^\infty((0, T); H_x^s(0, \ell))}
\leq
c_1(s, T) 
\Big(
&\norm{u_0}_{H^s(0, \ell)}  
+
\norm{g_0}_{H^{\frac{s+1}{3}}(0, T)}
+
\norm{h_0}_{H^{\frac{s+1}{3}}(0, T)}
+
\norm{h_1}_{H^{\frac{s}{3}}(0, T)}
\\
&
+ 
\norm{f}_{L_t^2((0, T); H_x^s(0, \ell))}
\Big),
\quad
\frac 12 < s \leq 2, \ s\neq \frac 32, 
\end{split}
\label{sob-est}
}
the additional smoothing estimate
\eq{
\begin{split}
\norm{u}_{L_t^2((0, T); H_x^{s+1}(0, \ell))}
\leq
c_2(s, T) 
\bigg(
&\norm{u_0}_{H^s(0, \ell)}  
+
\norm{g_0}_{H^{\frac{s+1}{3}}(0, T)}
+
\norm{h_0}_{H^{\frac{s+1}{3}}(0, T)}
+
\norm{h_1}_{H^{\frac{s}{3}}(0, T)}
\\
&
+ 
\begin{cases}
\norm{f}_{L_t^2((0, T); H_x^s(0, \ell))}, &\frac 12 < s \leq 2, \ s\neq \frac 32
\\
\norm{f}_{L_t^1((0, T); H_x^s(0, \ell))}, &0 \leq s < \frac 12
\end{cases}
\bigg),
\end{split}
\label{smooth-est}
}
and the family of Strichartz-type estimates (which contains the analogue of estimate \eqref{sob-est} as a special case)
\begin{equation}
\begin{aligned}
\norm{u}_{L_t^q((0, T); H_x^{s, p}(0, \ell))} 
\leq 
c_3(s, p, T) \Big(
&\norm{u_0}_{H^s (0, \ell)} + \norm{g_0}_{H^{\frac{s + 1}{3}} (0, T)} 
+ \norm{h_0}_{H^{\frac{s + 1}{3}} (0, T)} + \norm{h_1}_{H^{\frac{s}{3}} (0, T)}
\\
&
+
\norm{f}_{L_t^1((0, T); H_x^s (0,\ell))}
\Big), \quad 0 \leq s < \frac 12, 
\label{strich-est}
\end{aligned}
\end{equation}
where the positive constants $c_1(s, T)$, $c_2(s, T)$ and $c_3(s, p, T)$ remain bounded as $T\to 0^+$ and $(q, p)$ is any pair satisfy the admissibility condition
\begin{equation}\label{adm-pair}
q, p \ge 2, \quad \frac{3}{q} + \frac{1}{p} = \frac{1}{2}.
\end{equation} 
\end{theorem}

The Sobolev estimate \eqref{sob-est} provides the basis for the proof of the high-regularity Theorem \ref{high-wp-t}, while the Strichartz estimate \eqref{strich-est} used both for $(q, p) = (\infty, 2)$ (in which case it corresponds to the analogue of the Sobolev estimate \eqref{sob-est} in the low-regularity setting) and for $(q, p) = \big(\frac{6\lambda}{(1-2s)(\lambda-1)}, \frac{2\lambda}{1+2(\lambda-1)s}\big)$ plays an instrumental role in the proof of the low-regularity Theorem \ref{low-wp-t}. The proof of Theorem \ref{lin-est-t} relies on the following novel solution formula for the forced linear problem \eqref{hls-ibvp}, which is derived in Appendix \ref{app-s} via the unified transform:
\eq{\label{hls-sol}
\begin{aligned}
u(x, t) &=\frac{1}{2 \pi} \int_{-\infty}^{\infty} e^{i k x + i \omega t} \Big[\hat u_0 (k) - i \int_0^t e^{-i \omega t'} \hat f(k, t') dt'\Big] dk
\\
&\quad - \frac{1}{2 \pi} \int_{\partial \tilde D_0} \frac{e^{i k x + i \omega t}}{\Delta(k)}  \b{\p{\nu_- - k} e^{-i \nu_+ \ell} - \p{\nu_+ - k} e^{-i \nu_- \ell}} \Big[\hat u_0 (k) - i \int_0^T e^{-i \omega t'} \hat f(k, t') dt'\Big] dk
\\
&\quad + \frac{1}{2 \pi} \int_{\partial \tilde D_+ \cup \partial \tilde D_-} \frac{e^{-i k (\ell - x) + i \omega t}}{\Delta(k)}  \p{\nu_+ - \nu_-} \Big[\hat u_0 (k) - i \int_0^T e^{-i \omega t'} \hat f(k, t') dt'\Big] dk
\\
&\quad + \frac{1}{2 \pi} \int_{\partial \tilde D_0 \cup \partial \tilde D_+ \cup \partial \tilde D_-} \frac{e^{-i k \p{\ell - x} + i \omega t}}{\Delta(k)}  \bigg\{ \p{\nu_- - k} \Big[\hat u_0 (\nu_+) - i \int_0^T e^{-i \omega t'} \hat f(\nu_+, t') dt'\Big] 
\\
&\hspace{6cm}
- \p{\nu_+ - k} \Big[\hat u_0 (\nu_-) - i \int_0^T e^{-i \omega t'} \hat f(\nu_-, t') dt'\Big]
\\
&\hspace{6cm}
 - \p{\nu_+ - \nu_-} \omega' \, \tilde g_0 (\omega, T) 
 - \p{\nu_- e^{-i \nu_+ \ell} - \nu_+ e^{-i \nu_- \ell}} \omega' \, \tilde h_0 (\omega, T)
\\
&\hspace{6cm}
- i \p{e^{-i \nu_+ \ell} - e^{-i \nu_- \ell}} \omega' \, \tilde h_1 (\omega, T) 
 \bigg\} \, dk.
\end{aligned}
}
In the above formula, $\omega$ and $\Delta$ are given by \eqref{omega} and \eqref{Delta}, $\hat u_0, \hat f$ denote the finite interval Fourier transforms of $u_0, f$ defined by \eqref{fi-ft-def}, $\tilde g_0, \tilde h_0, \tilde h_1$ are certain time transforms defined by \eqref{tilde-transform}, 
$\nu_\pm$ are given by \eqref{nupm}, and the contours of integration $\partial \tilde D_0, \partial \tilde D_\pm$ are the positively oriented boundaries of the regions $\tilde D_0, \tilde D_\pm$ defined by \eqref{dntil-def} and depicted in Figure \ref{fig:Dbranches}. 

It should be noted that the well-posedness of the HNLS equation with a Dirichlet boundary condition on the half-line was established in the recent work \cite{amo2024}. However, the finite interval problem \eqref{hnls-ibvp} considered here involves several new and important challenges, which can be summarized as follows:
\begin{enumerate}[label=(\roman*), leftmargin=8mm, itemsep=1mm, topsep=1mm]
\item The unified transform solution formula \eqref{hls-sol} for the forced linear problem \eqref{hls-ibvp} is significantly more complicated than the one for the corresponding half-line problem. In particular, formula \eqref{hls-sol} involves a certain combination of exponentials in the denominator of the relevant integrands through the quantity $\Delta(k)$, which requires special care while proving the linear estimates of Theorem \ref{lin-est-t} (e.g. see Lemma \ref{lemma:Deltabound}).
\item Formula \eqref{hls-sol} involves integrals over three different complex contours (as opposed to just one in the case of the half-line) --- a byproduct of the two additional boundary conditions $u(\ell, t) = h_0(t)$ and $u_x(\ell, t) = h_1(t)$ that were not present in the half-line problem. The simultaneous presence of all three contours in the linear solution formula imposes a different choice of branch cut for the multivalued functions $\nu_\pm$, which in turn requires a different approach in the derivation of the relevant estimates.
\item The boundary datum $h_1(t)$ is of Neumann type, while the half-line problem considered in \cite{amo2024} involved a Dirichlet datum. In particular, the new type of datum introduces the need for new estimates at the stage of the linear decomposition of Section \ref{dec-s}.
\item Indeed, the proof of Theorem \ref{lin-est-t} motivates the derivation of linear time estimates both for the finite interval problem and, importantly, for the half-line problem. Such time estimates were not necessary in \cite{amo2024} and so they constitute novel results even at the level of the half-line.
\end{enumerate}

As noted earlier, the literature on the well-posedness of nonlinear dispersive initial-boundary value problems is much more limited than the one on their initial value problem counterparts. Nevertheless, over the course of the past two decades or so, important works have appeared on the rigorous analysis of initial-boundary value problems, starting from those on the well-posedness of the KdV equation on the half-line by Bona, Sun and Zhang~\cite{bsz2002} as well as Colliander and Kenig \cite{ck2002} (the latter work being on the generalized KdV equation) and continuing with the works of Holmer on both the NLS and the KdV half-line problems \cite{h2005,h2006} and, more recently, Cavalcante and collaborators \cite{c2017,cc2020,ck2020,ccg2020}. In the first of these works, the forced linear problem is solved via a temporal Laplace transform, while in the rest of them it is handled through a clever decomposition into appropriate initial value problems that relies on the construction of a certain boundary forcing operator. The temporal Laplace transform approach of \cite{bsz2002} has been employed in several other works; indicatively, we mention \cite{bsz2006,bsz2008,kai2013,ozs2015,bo2016,et2016,bsz2018}. The approach of~\cite{fhm2017,fhm2016}, which relies on the the unified transform (as the analogue of the Fourier transform in domains with a boundary) and is the one used in the present work, has also been further developed in recent years, e.g. see~\cite{hmy2019-kdv,oy2019,bfo2020,hm2020,hm2022,ko2022,hy2022b,mo2024}. Other works on the well-posedness of initial-boundary value problems for the classical NLS equation also include \cite{cb1991,sb2001}.

Concerning the HNLS equation in the initial-boundary value problem setting, we note that, in addition to the local well-posedness established in \cite{amo2024} for the half-line problem, Faminskii has recently obtained results on global solutions on the half-line \cite{f2024} as well as on the well-posedness of inverse problems with integral overdetermination on a bounded interval \cite{fm2023}. We also note that the analysis of the HNLS equation on a finite interval in the case of $\beta<0$ can be carried out in an \textit{entirely analogous way} as the case of $\beta>0$ presented in this work, the only difference being the prescription of two boundary data at $x=0$ and one at $x=\ell$ (as opposed to the one datum at $x=0$ and two data at $x=\ell$ in problem \eqref{hnls-ibvp}). On the other hand, the case of $\beta<0$ on the half-line is more interesting, as it requires the prescription of \textit{two} boundary data at $x=0$ (as opposed to the \textit{single} boundary condition present in \cite{amo2024}) and will be considered in the upcoming work \cite{amo2025}.

Through the present work on the HNLS equation, we develop a general approach for handling nonlinear dispersive models on a finite interval with a \textit{multi-term} linear part that directly reflects the relevance of these models in applications. 
In this regard, we note that, besides the half-line, \cite{h2006} considers the KdV equation also on a finite interval but \textit{without} the linear term $u_x$ that appears in the original physical model. Although this simplification is indeed possible on the whole line thanks to a Galilean transformation, in the case of domains with a boundary, like the half-line and the finite interval, the term $u_x$  cannot be removed and, as noted in \cite{ck2002}, the relevant function spaces must be modified appropriately. Specifically for KdV on a finite interval, an alternative analysis via the temporal Laplace transform approach of \cite{bsz2002} was given in \cite{bsz2003}. On the other hand, here we advance the spatial Fourier transform/unified transform approach of \cite{fhm2017} for the HNLS equation.

Finally, there are also numerical results \cite{ccs2019} as well as works on the controllability of HNLS~\cite{cpv2005,bbv2007,chen2018,boy2021,oy2022} --- a direction of research that directly involves the well-posedness of initial-boundary value problems.
In fact, the representation formulas for solutions of initial-boundary value problems obtained through the unified transform play an important role from control theoretical perspectives.  For instance, boundary or interior controllability problems can be recast or characterized as integral equations that involve given data (i.e. the initial state and a target state) and also sought-after control(s), e.g. one or more nonhomogeneous boundary inputs or a locally supported interior source function, respectively. Then, an analysis on the solvability of these integral equations can determine whether a given evolution is controllable or not.  For instance, \cite{ko2020} used unified transform formulas to revisit the classical lack of null controllability problem for the heat equation on the half line \cite{mz2001}  via boundary controls and established failure of the controllability feature in a quite elementary fashion.   Using unified transform formulas, the authors also extended their result of lack of controllability  to partial differential equations of different nature such as the Schr\"odinger and biharmonic Schr\"odinger equations in \cite{ok2023}, a question that had not been answered with the classical tools of control theory previously.  
It is well known that, in the framework of a bounded domain (e.g., a finite interval), the heat equation has the feature of (null) controllability via boundary controls in contrast with the case of an unbounded domain.  Although, this is a well-known result, its applicability to real-life problems was limited because most approaches in the literature are abstract, leading to existence of a control input without providing an explicit formula for a physically reasonable control in terms of the given initial and target states. It was shown recently in \cite{kod2024} that such a control in the form of a nonhomogeneous boundary input can be constructed explicitly by utilizing the representation formula established through the unified transform. 
The unified transform formulas can also be used for establishing Hadamard well-posedness for linear and nonlinear feedback control problems.  For instance, for a boundary feedback problem, one can simply replace the boundary terms in the unified transform formula with the given (possibly nonlinear) feedback terms and take the right-hand side of the resulting formula as the definition of a solution operator whose fixed point becomes the sought-after local solution, e.g. see the preprint \cite{mok2024}.
\\[2mm]
\noindent
\textbf{Structure}. The crucial Sobolev and Strichartz linear estimates of Theorem \ref{lin-est-t} are established in Section \ref{reduced-s} for the so-called reduced initial-boundary value problem, namely for problem \eqref{hls-ibvp} in the special case of zero initial data, zero forcing, zero Dirichlet data at $x=0$, and boundary data at $x=\ell$ that are supported inside a compact set. Once this reduced problem is estimated, the full version of Theorem \ref{lin-est-t} is established for the original forced linear problem \eqref{hls-ibvp} in Section \ref{dec-s} through a delicate linear decomposition. This part of the analysis is quite involved and motivates the derivation of novel time estimates for the higher-order Schr\"odinger equation on the half-line. The nonlinear analysis leading to the proof of Theorem \ref{high-wp-t} in the high-regularity setting is presented in the first part of Section \ref{wp-s} via the combination of the Sobolev estimate \eqref{sob-est} with the algebra property in $H_x^s(0,\ell)$ and a contraction mapping argument. The second part of that section contains the proof of the low-regularity Theorem~\ref{low-wp-t}, which relies on the Strichartz estimates of Theorem \ref{lin-est-t} and suitable nonlinear estimates that fix the pair of exponents $(q, p)$ in \eqref{strich-est} to the values appearing in the solution space $Y_T^{s,\lambda}$. Finally, the unified transform solution formula~\eqref{hls-sol} for the forced linear problem \eqref{hls-ibvp} is derived in Appendix \ref{app-s}.

\section{Linear estimates on a finite interval} \label{reduced-s}

We begin our analysis from the most essential part of the forced linear finite interval problem \eqref{hls-ibvp}, namely the following \textit{reduced} initial-boundary value problem:
\eq{\label{HNLSreduced}
\begin{aligned}
&i v_t + i \beta v_{xxx} + \alpha v_{xx} + i \delta v_x = 0, \quad 0 < x < \ell, \ 0 < t < T',
\\
&v(x, 0) = 0,\\
&v(0, t) = 0, \quad  v(\ell, t) = \psi_0(t), \quad v_x (\ell, t) = \psi_1(t),
\end{aligned}
}
where the boundary data $\psi_0, \psi_1$ are globally defined on $\R$ but only supported in the compact set $[0, T']$, i.e.
\eq{\label{supp-cond}
\text{supp}(\psi_0) \subset [0, T'], \quad \text{supp}(\psi_1) \subset [0, T'].
}
The analysis of the reduced problem \eqref{HNLSreduced} will allow us to determine the role played by the boundary conditions at $x = \ell$, which were not present in the half-line problem of \cite{amo2024}. Notably, in addition to the Dirichlet condition, we now also have a Neumann condition. Once we have estimated the solution to the reduced interval problem~\eqref{HNLSreduced}, we will be able to deduce linear estimates for the full interval problem~\eqref{hls-ibvp} by a careful use of the linear superposition principle (see Section \ref{dec-s}). Importantly, as noted in the introduction, in that process it will become necessary to establish new time regularity results for the half-line problem (which were not obtained in~\cite{amo2024}). 

In the case of the reduced problem \eqref{HNLSreduced}, the unified transform solution formula~\eqref{hls-sol} becomes
\eq{
\begin{split}
v(x, t) &= -\frac{i}{2 \pi} \int_{\partial \tilde D} \frac{e^{-i k \p{\ell - x} + i \omega t}}{\Delta(k)} \p{e^{-i \nu_+ \ell} - e^{-i \nu_- \ell}} \omega' \mathcal F\{\psi_1\} (\omega) dk\\
&\quad - \frac{1}{2 \pi} \int_{\partial \tilde D} \frac{e^{-i k \p{\ell - x} + i \omega t}}{\Delta(k)} \p{\nu_- e^{-i \nu_+ \ell} - \nu_+ e^{-i \nu_- \ell}} \omega' \mathcal F\{\psi_0\} (\omega) dk,
\end{split}\label{reducedsol}}
where $\mathcal F\{\cdot\}$ denotes the regular Fourier transform on the whole line, the dispersion relation $\omega$ is given by
\eq{\omega := \beta k^3 - \alpha k^2 - \delta k, \label{omega}}
the denominator $\Delta$ involved in the two integrands is
\eq{\label{Delta}
\Delta(k) := \p{\nu_+ - \nu_-} e^{-i k \ell} + \p{\nu_- - k} e^{-i \nu_+ \ell} + \p{k - \nu_+} e^{-i \nu_- \ell}
}
with the quantities $\nu_\pm$ defined by
\eqs{
\nu_\pm :=  -\frac{1}{2} \p{k - \frac{\alpha}{\beta}} \pm \frac{\sqrt 3}{2} i \b{\p{k - \frac{\alpha}{3 \beta}}^2 - \frac{4}{9 \beta^2} \p{\alpha^2 + 3 \beta \delta}}^{\frac 12}, \alignlabel{nupm}}
and the complex contour of integration $\partial \tilde D$ is defined as follows. Letting $D := D_0 \cup D_+ \cup D_-$, where the individual regions are defined by
\begin{equation}\label{dn-def}
\begin{aligned}
D_0 &:= \set{k \in \C : \Im(\omega) < 0,~ \Im(k) > 0},\\
D_+ &:= \set{k \in \C : \Im(\omega) < 0,~ \Im(k) < 0,~ \Re \p{k - \frac{\alpha}{3 \beta}} > 0},\\
D_- &:= \set{k \in \C : \Im(\omega) < 0,~ \Im(k) < 0,~ \Re \p{k - \frac{\alpha}{3 \beta}} < 0},
\end{aligned}
\end{equation}
we define the ``punctured'' regions
\begin{equation}\label{dntil-def}
\tilde D_n := D_n \setminus \left\{k\in\mathbb C: \left|k-\frac{\alpha}{3\beta}\right| \leq R_\Delta\right\}, \quad n \in \set{0, +, -},
\end{equation}
where the radius $R_\Delta$ (motivated by Lemma \ref{lemma:Deltabound})  is given by
\begin{equation}\label{rd-def}
R_\Delta := \max \set{\frac{2 \sqrt 2}{\sqrt 3 \beta} \sqrt{\abs{\alpha^2 + 3 \beta \delta}}, \frac{9}{\ell}},
\end{equation}
and then let $\partial \tilde D$ be the positively oriented boundary of the union
\begin{equation}\label{dtil-def}
\tilde D := \tilde D_0 \cup \tilde D_+ \cup \tilde D_-
=
D \setminus \left\{k\in\mathbb C: \left|k-\frac{\alpha}{3\beta}\right| \leq R_\Delta\right\},
\end{equation}
as shown in Figure \ref{fig:Dbranches}.

\begin{remark}\label{ft-til-r}
The presence of the regular Fourier transform of the data $\psi_0, \psi_1$  in formula \eqref{HNLSreduced}, instead of the time transform \eqref{tilde-transform} that normally appears in the general formula \eqref{hls-sol}, is possible thanks to the support condition~\eqref{supp-cond} and will turn out useful in the derivation of estimates given below.
\end{remark}

We will use the unified transform formula \eqref{reducedsol} in order to derive estimates for the solution $v(x, t)$ of the reduced interval problem \eqref{HNLSreduced} in both Sobolev and Strichartz-type spaces. We begin with the Sobolev estimate and then proceed to the Strichartz estimate, which is necessary in the low-regularity setting.

\subsection{Sobolev estimate}
Requiring that the solution of problem \eqref{HNLSreduced} belongs to the Sobolev space $H^s(0, \ell)$ as a function of $x$ and for each $t\in [0, T']$, we are led to the following result.
\begin{theorem}[Sobolev estimate]\label{thm:fi-se}
Let $s\geq 0$ and $\psi_0 \in H^{\frac{s+1}{3}}(\mathbb R)$, $\psi_1 \in H^{\frac{s}{3}}(\mathbb R)$ satisfy  the support condition~\eqref{supp-cond}. Then, the solution to the reduced finite interval problem \eqref{HNLSreduced}, as given by formula \eqref{reducedsol}, admits the estimate
\eq{\label{fi-se}
\norm{v(t)}_{H_x^s (0, \ell)} \le c_s  \max\left\{1, \sqrt{T'} e^{M_\Delta T'} \right\} \Big(\norm{\psi_0}_{H^{\frac{s + 1}{3}} (\R)} + \norm{\psi_1}_{H^{\frac{s}{3}} (\R)}\Big), 
\quad
t\in [0, T'],
}
where $c_s>0$ is a constant depending only on $s, \alpha, \beta, \delta$, and the constant $M_\Delta>0$ is given by \eqref{md-def}.
\end{theorem}

\begin{remark}
The fact that $\psi_0 \in H^{\frac{s+1}{3}}(\mathbb R)$ must satisfy the support condition \eqref{supp-cond} imposes certain trace conditions on $\psi_0$ at the points $t=0, T'$. Specifically, according to Theorems 11.4 and 11.5 of \cite{lm1972}, it must be that $\partial^j \psi_0(0) = \partial^j \psi_0(T') = 0$ for all integers $0\leq j < \frac{s+1}{3}-\frac 12$. Similarly,   $\psi_1$ must satisfy $\partial^j \psi_1(0) = \partial^j \psi_1(T') = 0$ for all integers $0\leq j < \frac{s}{3}-\frac 12$. These conditions will be verified in Section \ref{dec-s} for the specific forms of $\psi_0, \psi_1$ that arise when employing Theorem \ref{thm:fi-se} in the decomposition of the full interval problem \eqref{hls-ibvp}. 
\end{remark}

\begin{proof}
Let $t\in [0, T']$. For any $s \in \N_0 := \N \cup \{0\}$, by the physical space characterization of the $H^s(0, \ell)$ norm (namely, the fact that, thanks to Plancherel's theorem, the spaces $H^s(0, \ell)$ and $W^{s, 2}(0, \ell)$ are equivalent for $s\geq 0$), 
\eq{\label{hs-l2-def}
\norm{v(t)}_{H_x^s(0, \ell)}
=
\sum_{j=0}^s \norm{\partial_x^j v(t)}_{L_x^2(0, \ell)},
\quad
s\in\N_0.
}

\begin{figure}[bt!]
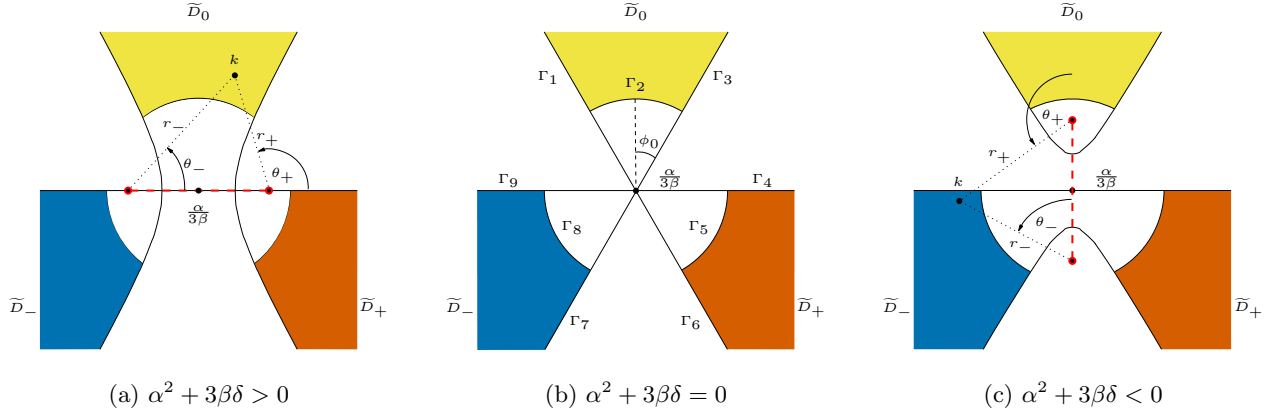

\centering
\begin{subfigure}{.33\textwidth}
\centering
%
%
{\pgfkeys{/pgf/fpu/.try=false}%
\ifx\XFigwidth\undefined\dimen1=0pt\else\dimen1\XFigwidth\fi
\divide\dimen1 by 6087
\ifx\XFigheight\undefined\dimen3=0pt\else\dimen3\XFigheight\fi
\divide\dimen3 by 5550
\ifdim\dimen1=0pt\ifdim\dimen3=0pt\dimen1=1657sp\dimen3\dimen1
  \else\dimen1\dimen3\fi\else\ifdim\dimen3=0pt\dimen3\dimen1\fi\fi
\tikzpicture[x=+\dimen1, y=+\dimen3]
{\ifx\XFigu\undefined\catcode`\@11
\def\temp{\alloc@1\dimen\dimendef\insc@unt}\temp\XFigu\catcode`\@12\fi}
\XFigu1657sp
\ifdim\XFigu<0pt\XFigu-\XFigu\fi
\pgfdeclarearrow{
  name = xfiga1,
  parameters = {
    \the\pgfarrowlinewidth \the\pgfarrowlength \the\pgfarrowwidth\ifpgfarrowopen o\fi},
  defaults = {
	  line width=+7.5\XFigu, length=+120\XFigu, width=+60\XFigu},
  setup code = {
    \dimen7 2.1\pgfarrowlength\pgfmathveclen{\the\dimen7}{\the\pgfarrowwidth}
    \dimen7 2\pgfarrowwidth\pgfmathdivide{\pgfmathresult}{\the\dimen7}
    \dimen7 \pgfmathresult\pgfarrowlinewidth
    \pgfarrowssettipend{+\dimen7}
    \pgfarrowssetbackend{+-\pgfarrowlength}
    \dimen9 -\pgfarrowlength\advance\dimen9 by-0.45\pgfarrowlinewidth
    \pgfarrowssetlineend{+\dimen9}
    \dimen9 -\pgfarrowlength\advance\dimen9 by-0.5\pgfarrowlinewidth
    \pgfarrowssetvisualbackend{+\dimen9}
    \pgfarrowshullpoint{+\dimen7}{+0pt}
    \pgfarrowsupperhullpoint{+-\pgfarrowlength}{+0.5\pgfarrowwidth}
    \pgfarrowssavethe\pgfarrowlinewidth
    \pgfarrowssavethe\pgfarrowlength
    \pgfarrowssavethe\pgfarrowwidth
  },
  drawing code = {\pgfsetdash{}{+0pt}
    \ifdim\pgfarrowlinewidth=\pgflinewidth\else\pgfsetlinewidth{+\pgfarrowlinewidth}\fi
    \pgfpathmoveto{\pgfqpoint{-\pgfarrowlength}{-0.5\pgfarrowwidth}}
    \pgfpathlineto{\pgfqpoint{0pt}{0pt}}
    \pgfpathlineto{\pgfqpoint{-\pgfarrowlength}{0.5\pgfarrowwidth}}
    \pgfpathclose
    \ifpgfarrowopen\pgfusepathqstroke\else\pgfsetfillcolor{.}
	\ifdim\pgfarrowlinewidth>0pt\pgfusepathqfillstroke\else\pgfusepathqfill\fi\fi
  }
}
\pgfdeclarearrow{
  name = xfiga5,
  parameters = {
    \the\pgfarrowlinewidth \the\pgfarrowlength\ifpgfarrowopen o\fi},
  defaults = {
	  line width=+7.5\XFigu, length=+120\XFigu},
  setup code = {
    \dimen7 0.5\pgfarrowlinewidth
    \pgfarrowssettipend{+\dimen7}
    \pgfarrowssetbackend{+-\pgfarrowlength}
    \dimen9 -\pgfarrowlength\advance\dimen9 by0.45\pgfarrowlinewidth
    \pgfarrowssetlineend{+\dimen9}
    \dimen9 -\pgfarrowlength\advance\dimen9 by-0.5\pgfarrowlinewidth
    \pgfarrowssetvisualbackend{+\dimen9}
    \pgfarrowshullpoint{+\dimen7}{+0pt}
    \pgfarrowsupperhullpoint{+-0.25\pgfarrowlength}{+0.35\pgfarrowwidth}\pgfarrowsupperhullpoint{+-0.5\pgfarrowlength}{+0.5\pgfarrowwidth}\pgfarrowsupperhullpoint{+-0.75\pgfarrowlength}{+0.35\pgfarrowwidth}
    \pgfarrowshullpoint{+\dimen9}{+0pt}
    \pgfarrowssavethe\pgfarrowlinewidth
    \pgfarrowssavethe\pgfarrowlength
  },
  drawing code = {\pgfsetdash{}{+0pt}
    \ifdim\pgfarrowlinewidth=\pgflinewidth\else\pgfsetlinewidth{+\pgfarrowlinewidth}\fi
    \dimen3 0.5\pgfarrowlength
    \pgfpathcircle{\pgfqpoint{-\dimen3}{0pt}}{+\dimen3}
    \ifpgfarrowopen\pgfusepathqstroke\else\pgfsetfillcolor{.}
	\ifdim\pgfarrowlinewidth>0pt\pgfusepathqfillstroke\else\pgfusepathqfill\fi\fi
  }
}
\definecolor{xfigc32}{rgb}{0.941,0.894,0.259}
\definecolor{xfigc33}{rgb}{0.000,0.447,0.698}
\definecolor{xfigc34}{rgb}{0.835,0.369,0.000}
\clip(1684,-7297) rectangle (7771,-1747);
\tikzset{inner sep=+0pt, outer sep=+0pt}
\pgfsetbeveljoin
\pgfsetfillcolor{xfigc32}
\fill (3255,-2362)--(3256,-2363)--(3257,-2366)--(3260,-2371)--(3264,-2379)--(3270,-2390)
  --(3277,-2403)--(3286,-2419)--(3296,-2438)--(3307,-2458)--(3318,-2480)--(3331,-2504)
  --(3344,-2529)--(3358,-2555)--(3373,-2583)--(3388,-2612)--(3405,-2643)--(3422,-2677)
  --(3441,-2713)--(3461,-2751)--(3482,-2792)--(3504,-2835)--(3526,-2877)--(3546,-2917)
  --(3564,-2953)--(3580,-2983)--(3592,-3007)--(3602,-3027)--(3610,-3041)--(3615,-3052)
  --(3619,-3060)--(3622,-3066)--(3625,-3071)--(3628,-3076)--(3631,-3082)--(3635,-3090)
  --(3640,-3101)--(3647,-3115)--(3657,-3135)--(3669,-3159)--(3684,-3189)--(3701,-3225)
  --(3720,-3265)--(3740,-3307)--(3760,-3350)--(3779,-3390)--(3796,-3426)--(3810,-3457)
  --(3822,-3483)--(3832,-3505)--(3840,-3522)--(3846,-3535)--(3851,-3546)--(3855,-3555)
  --(3858,-3563)--(3862,-3571)--(3865,-3580)--(3870,-3590)--(3875,-3602)--(3882,-3617)
  --(3890,-3636)--(3900,-3659)--(3911,-3685)--(3924,-3715)--(3938,-3747)--(3952,-3780)
  --(3972,-3828)--(3989,-3868)--(4002,-3900)--(4012,-3926)--(4020,-3947)--(4026,-3965)
  --(4031,-3979)--(4035,-3991)--(4038,-4000)--(4040,-4007)--(4041,-4012)--(4042,-4015)
  --(4042,-4016)--(4044,-4016)--(4048,-4016)--(4055,-4016)--(4067,-4016)--(4083,-4016)
  --(4105,-4016)--(4133,-4016)--(4168,-4016)--(4209,-4016)--(4256,-4016)--(4309,-4016)
  --(4368,-4016)--(4432,-4016)--(4501,-4016)--(4573,-4016)--(4647,-4016)--(4723,-4016)
  --(4798,-4016)--(4872,-4016)--(4944,-4016)--(5013,-4016)--(5077,-4016)--(5136,-4016)
  --(5189,-4016)--(5236,-4016)--(5277,-4016)--(5312,-4016)--(5340,-4016)--(5362,-4016)
  --(5378,-4016)--(5390,-4016)--(5397,-4016)--(5401,-4016)--(5403,-4016)--(5403,-4015)
  --(5404,-4012)--(5405,-4007)--(5407,-4000)--(5410,-3991)--(5414,-3979)--(5419,-3965)
  --(5425,-3947)--(5433,-3926)--(5443,-3900)--(5456,-3868)--(5473,-3828)--(5493,-3780)
  --(5507,-3747)--(5521,-3715)--(5534,-3685)--(5545,-3659)--(5555,-3636)--(5563,-3617)
  --(5570,-3602)--(5575,-3590)--(5580,-3580)--(5583,-3571)--(5587,-3563)--(5590,-3555)
  --(5594,-3546)--(5599,-3535)--(5605,-3522)--(5613,-3505)--(5623,-3483)--(5635,-3457)
  --(5649,-3426)--(5666,-3390)--(5685,-3350)--(5705,-3307)--(5725,-3265)--(5744,-3225)
  --(5761,-3189)--(5776,-3159)--(5788,-3135)--(5798,-3115)--(5805,-3101)--(5810,-3090)
  --(5814,-3082)--(5817,-3076)--(5820,-3071)--(5823,-3066)--(5826,-3060)--(5830,-3052)
  --(5835,-3041)--(5843,-3027)--(5853,-3007)--(5865,-2983)--(5881,-2953)--(5899,-2917)
  --(5919,-2877)--(5941,-2835)--(5963,-2792)--(5984,-2751)--(6004,-2713)--(6023,-2677)
  --(6040,-2643)--(6057,-2612)--(6072,-2583)--(6087,-2555)--(6101,-2529)--(6114,-2504)
  --(6127,-2480)--(6138,-2458)--(6149,-2438)--(6159,-2419)--(6168,-2403)--(6175,-2390)
  --(6181,-2379)--(6185,-2371)--(6188,-2366)--(6189,-2363)--(6190,-2362);
\pgfsetfillcolor{xfigc33}
\fill (2362,-7088)--(2362,-7087)--(2362,-7084)--(2362,-7078)--(2362,-7069)--(2362,-7056)
  --(2362,-7039)--(2362,-7017)--(2362,-6989)--(2362,-6955)--(2362,-6915)--(2362,-6869)
  --(2362,-6817)--(2362,-6759)--(2362,-6694)--(2362,-6624)--(2362,-6548)--(2362,-6467)
  --(2362,-6381)--(2362,-6291)--(2362,-6198)--(2362,-6102)--(2362,-6004)--(2362,-5906)
  --(2362,-5808)--(2362,-5710)--(2362,-5614)--(2362,-5521)--(2362,-5431)--(2362,-5345)
  --(2362,-5264)--(2362,-5188)--(2362,-5118)--(2362,-5053)--(2362,-4995)--(2362,-4943)
  --(2362,-4897)--(2362,-4857)--(2362,-4823)--(2362,-4795)--(2362,-4773)--(2362,-4756)
  --(2362,-4743)--(2362,-4734)--(2362,-4728)--(2362,-4725)--(2362,-4724)--(2363,-4724)
  --(2367,-4724)--(2373,-4724)--(2383,-4724)--(2398,-4724)--(2417,-4724)--(2442,-4724)
  --(2472,-4724)--(2509,-4724)--(2552,-4724)--(2601,-4724)--(2657,-4724)--(2718,-4724)
  --(2785,-4724)--(2858,-4724)--(2934,-4724)--(3015,-4724)--(3098,-4724)--(3183,-4724)
  --(3270,-4724)--(3356,-4724)--(3441,-4724)--(3524,-4724)--(3605,-4724)--(3681,-4724)
  --(3754,-4724)--(3821,-4724)--(3882,-4724)--(3938,-4724)--(3987,-4724)--(4030,-4724)
  --(4067,-4724)--(4097,-4724)--(4122,-4724)--(4141,-4724)--(4156,-4724)--(4166,-4724)
  --(4172,-4724)--(4176,-4724)--(4177,-4724)--(4177,-4726)--(4176,-4729)--(4175,-4735)
  --(4173,-4745)--(4171,-4759)--(4168,-4776)--(4164,-4797)--(4160,-4822)--(4155,-4850)
  --(4150,-4880)--(4145,-4912)--(4139,-4945)--(4133,-4980)--(4126,-5015)--(4120,-5051)
  --(4113,-5088)--(4106,-5125)--(4099,-5162)--(4091,-5200)--(4083,-5238)--(4075,-5277)
  --(4067,-5316)--(4059,-5356)--(4050,-5395)--(4041,-5432)--(4026,-5493)--(4014,-5539)
  --(4005,-5568)--(4001,-5582)--(4000,-5585)--(4001,-5579)--(4002,-5570)--(4004,-5560)
  --(4004,-5553)--(4002,-5553)--(3997,-5563)--(3986,-5585)--(3971,-5621)--(3951,-5668)
  --(3937,-5701)--(3923,-5733)--(3910,-5763)--(3899,-5789)--(3889,-5812)--(3881,-5831)
  --(3874,-5846)--(3869,-5858)--(3864,-5868)--(3861,-5877)--(3857,-5885)--(3854,-5893)
  --(3850,-5902)--(3845,-5913)--(3839,-5926)--(3831,-5943)--(3821,-5965)--(3809,-5991)
  --(3795,-6022)--(3778,-6058)--(3759,-6098)--(3739,-6141)--(3719,-6183)--(3700,-6223)
  --(3683,-6259)--(3668,-6289)--(3656,-6313)--(3646,-6333)--(3639,-6347)--(3634,-6358)
  --(3630,-6366)--(3627,-6372)--(3624,-6377)--(3621,-6382)--(3618,-6388)--(3614,-6396)
  --(3609,-6407)--(3601,-6421)--(3591,-6441)--(3579,-6465)--(3563,-6495)--(3545,-6531)
  --(3525,-6571)--(3503,-6613)--(3481,-6656)--(3460,-6697)--(3440,-6735)--(3421,-6771)
  --(3404,-6805)--(3387,-6836)--(3372,-6865)--(3357,-6893)--(3343,-6919)--(3330,-6944)
  --(3317,-6968)--(3306,-6990)--(3295,-7010)--(3285,-7029)--(3276,-7045)--(3269,-7058)
  --(3263,-7069)--(3259,-7077)--(3256,-7082)--(3255,-7085)--(3254,-7086);
\pgfsetfillcolor{xfigc34}
\fill (7086,-7088)--(7086,-7087)--(7086,-7084)--(7086,-7078)--(7086,-7069)--(7086,-7056)
  --(7086,-7039)--(7086,-7017)--(7086,-6989)--(7086,-6955)--(7086,-6915)--(7086,-6869)
  --(7086,-6817)--(7086,-6759)--(7086,-6694)--(7086,-6624)--(7086,-6548)--(7086,-6467)
  --(7086,-6381)--(7086,-6291)--(7086,-6198)--(7086,-6102)--(7086,-6004)--(7087,-5906)
  --(7087,-5808)--(7087,-5710)--(7087,-5614)--(7087,-5521)--(7087,-5431)--(7087,-5345)
  --(7087,-5264)--(7087,-5188)--(7087,-5118)--(7087,-5053)--(7087,-4995)--(7087,-4943)
  --(7087,-4897)--(7087,-4857)--(7087,-4823)--(7087,-4795)--(7087,-4773)--(7087,-4756)
  --(7087,-4743)--(7087,-4734)--(7087,-4728)--(7087,-4725)--(7087,-4724)--(7086,-4724)
  --(7082,-4724)--(7076,-4724)--(7066,-4724)--(7051,-4724)--(7032,-4724)--(7007,-4724)
  --(6977,-4724)--(6940,-4724)--(6897,-4724)--(6848,-4724)--(6792,-4724)--(6731,-4724)
  --(6664,-4724)--(6591,-4724)--(6515,-4724)--(6434,-4724)--(6351,-4724)--(6266,-4724)
  --(6179,-4724)--(6093,-4724)--(6008,-4724)--(5925,-4724)--(5844,-4724)--(5768,-4724)
  --(5695,-4724)--(5628,-4724)--(5567,-4724)--(5511,-4724)--(5462,-4724)--(5419,-4724)
  --(5382,-4724)--(5352,-4724)--(5327,-4724)--(5308,-4724)--(5293,-4724)--(5283,-4724)
  --(5277,-4724)--(5273,-4724)--(5272,-4724)--(5272,-4726)--(5273,-4729)--(5274,-4735)
  --(5276,-4745)--(5278,-4759)--(5281,-4776)--(5285,-4797)--(5289,-4822)--(5294,-4850)
  --(5299,-4880)--(5304,-4912)--(5310,-4945)--(5316,-4980)--(5323,-5015)--(5329,-5051)
  --(5336,-5088)--(5343,-5125)--(5350,-5162)--(5358,-5200)--(5366,-5238)--(5374,-5277)
  --(5382,-5316)--(5390,-5356)--(5399,-5395)--(5408,-5432)--(5423,-5493)--(5435,-5539)
  --(5444,-5568)--(5448,-5582)--(5449,-5585)--(5448,-5579)--(5447,-5570)--(5445,-5560)
  --(5445,-5553)--(5447,-5553)--(5452,-5563)--(5463,-5585)--(5478,-5621)--(5498,-5668)
  --(5512,-5701)--(5526,-5733)--(5539,-5763)--(5550,-5789)--(5560,-5812)--(5568,-5831)
  --(5575,-5846)--(5580,-5858)--(5585,-5868)--(5588,-5877)--(5592,-5885)--(5595,-5893)
  --(5599,-5902)--(5604,-5913)--(5610,-5926)--(5618,-5943)--(5628,-5965)--(5640,-5991)
  --(5654,-6022)--(5671,-6058)--(5690,-6098)--(5710,-6141)--(5730,-6183)--(5749,-6223)
  --(5766,-6259)--(5781,-6289)--(5793,-6313)--(5803,-6333)--(5810,-6347)--(5815,-6358)
  --(5819,-6366)--(5822,-6372)--(5825,-6377)--(5828,-6382)--(5831,-6388)--(5835,-6396)
  --(5840,-6407)--(5848,-6421)--(5858,-6441)--(5870,-6465)--(5886,-6495)--(5904,-6531)
  --(5924,-6571)--(5946,-6613)--(5968,-6656)--(5989,-6697)--(6009,-6735)--(6028,-6771)
  --(6045,-6805)--(6062,-6836)--(6077,-6865)--(6092,-6893)--(6106,-6919)--(6119,-6944)
  --(6132,-6968)--(6143,-6990)--(6154,-7010)--(6164,-7029)--(6173,-7045)--(6180,-7058)
  --(6186,-7069)--(6190,-7077)--(6193,-7082)--(6194,-7085)--(6195,-7086);
\pgfsetlinewidth{+7.5\XFigu}
\pgfsetdash{}{+0pt}
\pgfsetstrokecolor{black}
\pgfsetfillcolor{white}
\filldraw (5545,-3638) arc[start angle=+51.35, end angle=+128.65, radius=+1317.9];
\filldraw (6089,-4714) arc[start angle=+-0.50, end angle=+-51.49, radius=+1412.7];
\filldraw (3361,-4714) arc[start angle=+-179.50, end angle=+-128.51, radius=+1412.7];
\fill  (4724,-4724) circle [radius=+1362];
\draw (2362,-4724)--(7087,-4724);
\pgfsetlinewidth{+30\XFigu}
\pgfsetstrokecolor{red}
\pgfsetdash{{+120\XFigu}{+120\XFigu}}{++0pt}
\pgfsetarrows{[line width=30\XFigu, length=90\XFigu]}
\pgfsetarrows{xfiga5-xfiga5}
\draw (3615,-4724)--(5832,-4724);
\pgfsetarrows{-}
\pgfsetlinewidth{+7.5\XFigu}
\pgfsetdash{}{+0pt}
\pgfsetstrokecolor{black}
\draw (3255,-7087)--(3256,-7086)--(3257,-7083)--(3260,-7078)--(3264,-7070)--(3270,-7059)
  --(3277,-7046)--(3286,-7030)--(3296,-7011)--(3307,-6991)--(3318,-6969)--(3331,-6945)
  --(3344,-6920)--(3358,-6894)--(3373,-6866)--(3388,-6837)--(3405,-6806)--(3422,-6772)
  --(3441,-6736)--(3461,-6698)--(3482,-6657)--(3504,-6614)--(3526,-6572)--(3546,-6532)
  --(3564,-6496)--(3580,-6466)--(3592,-6442)--(3602,-6422)--(3610,-6408)--(3615,-6397)
  --(3619,-6389)--(3622,-6383)--(3625,-6378)--(3628,-6373)--(3631,-6367)--(3635,-6359)
  --(3640,-6348)--(3647,-6334)--(3657,-6314)--(3669,-6290)--(3684,-6260)--(3701,-6224)
  --(3720,-6184)--(3740,-6142)--(3760,-6100)--(3779,-6060)--(3795,-6024)--(3809,-5994)
  --(3821,-5969)--(3830,-5950)--(3837,-5935)--(3843,-5924)--(3847,-5916)--(3850,-5910)
  --(3852,-5905)--(3855,-5901)--(3858,-5895)--(3861,-5887)--(3866,-5876)--(3873,-5861)
  --(3881,-5842)--(3891,-5817)--(3904,-5787)--(3919,-5751)--(3935,-5711)--(3952,-5669)
  --(3969,-5627)--(3984,-5587)--(3998,-5551)--(4009,-5521)--(4019,-5497)--(4026,-5477)
  --(4032,-5463)--(4037,-5452)--(4040,-5444)--(4043,-5438)--(4045,-5433)--(4048,-5428)
  --(4050,-5422)--(4053,-5414)--(4057,-5403)--(4062,-5389)--(4068,-5369)--(4075,-5345)
  --(4083,-5315)--(4093,-5279)--(4103,-5239)--(4114,-5197)--(4125,-5150)--(4134,-5107)
  --(4142,-5070)--(4148,-5039)--(4153,-5014)--(4157,-4996)--(4160,-4982)--(4162,-4973)
  --(4164,-4966)--(4165,-4960)--(4166,-4955)--(4167,-4948)--(4168,-4939)--(4169,-4925)
  --(4171,-4907)--(4173,-4882)--(4174,-4851)--(4176,-4814)--(4177,-4771)--(4178,-4724)
  --(4177,-4677)--(4176,-4634)--(4174,-4597)--(4173,-4566)--(4171,-4541)--(4169,-4523)
  --(4168,-4510)--(4167,-4500)--(4166,-4493)--(4165,-4488)--(4164,-4482)--(4162,-4476)
  --(4160,-4466)--(4157,-4453)--(4153,-4435)--(4148,-4410)--(4142,-4379)--(4134,-4342)
  --(4125,-4299)--(4114,-4252)--(4103,-4210)--(4093,-4170)--(4083,-4134)--(4075,-4104)
  --(4068,-4080)--(4062,-4060)--(4057,-4046)--(4053,-4035)--(4050,-4027)--(4048,-4021)
  --(4045,-4016)--(4043,-4011)--(4040,-4005)--(4037,-3997)--(4032,-3986)--(4026,-3972)
  --(4019,-3952)--(4009,-3928)--(3998,-3898)--(3984,-3862)--(3969,-3822)--(3952,-3780)
  --(3935,-3738)--(3919,-3698)--(3904,-3662)--(3891,-3632)--(3881,-3607)--(3873,-3588)
  --(3866,-3573)--(3861,-3562)--(3858,-3554)--(3855,-3548)--(3852,-3543)--(3850,-3539)
  --(3847,-3533)--(3843,-3525)--(3837,-3514)--(3830,-3499)--(3821,-3480)--(3809,-3455)
  --(3795,-3425)--(3779,-3389)--(3760,-3349)--(3740,-3307)--(3720,-3265)--(3701,-3225)
  --(3684,-3189)--(3669,-3159)--(3657,-3135)--(3647,-3115)--(3640,-3101)--(3635,-3090)
  --(3631,-3082)--(3628,-3076)--(3625,-3071)--(3622,-3066)--(3619,-3060)--(3615,-3052)
  --(3610,-3041)--(3602,-3027)--(3592,-3007)--(3580,-2983)--(3564,-2953)--(3546,-2917)
  --(3526,-2877)--(3504,-2835)--(3482,-2792)--(3461,-2751)--(3441,-2713)--(3422,-2677)
  --(3405,-2643)--(3388,-2612)--(3373,-2583)--(3358,-2555)--(3344,-2529)--(3331,-2504)
  --(3318,-2480)--(3307,-2458)--(3296,-2438)--(3286,-2419)--(3277,-2403)--(3270,-2390)
  --(3264,-2379)--(3260,-2371)--(3257,-2366)--(3256,-2363)--(3255,-2362);
\draw (6193,-7087)--(6192,-7086)--(6191,-7083)--(6188,-7078)--(6184,-7070)--(6178,-7059)
  --(6171,-7046)--(6162,-7030)--(6152,-7011)--(6141,-6991)--(6130,-6969)--(6117,-6945)
  --(6104,-6920)--(6090,-6894)--(6075,-6866)--(6060,-6837)--(6043,-6806)--(6026,-6772)
  --(6007,-6736)--(5987,-6698)--(5966,-6657)--(5944,-6614)--(5922,-6572)--(5902,-6532)
  --(5884,-6496)--(5868,-6466)--(5856,-6442)--(5846,-6422)--(5838,-6408)--(5833,-6397)
  --(5829,-6389)--(5826,-6383)--(5823,-6378)--(5820,-6373)--(5817,-6367)--(5813,-6359)
  --(5808,-6348)--(5801,-6334)--(5791,-6314)--(5779,-6290)--(5764,-6260)--(5747,-6224)
  --(5728,-6184)--(5708,-6142)--(5688,-6100)--(5669,-6060)--(5653,-6024)--(5639,-5994)
  --(5627,-5969)--(5618,-5950)--(5611,-5935)--(5605,-5924)--(5601,-5916)--(5598,-5910)
  --(5596,-5905)--(5593,-5901)--(5590,-5895)--(5587,-5887)--(5582,-5876)--(5575,-5861)
  --(5567,-5842)--(5557,-5817)--(5544,-5787)--(5529,-5751)--(5513,-5711)--(5496,-5669)
  --(5479,-5627)--(5464,-5587)--(5450,-5551)--(5439,-5521)--(5429,-5497)--(5422,-5477)
  --(5416,-5463)--(5411,-5452)--(5408,-5444)--(5405,-5438)--(5403,-5433)--(5400,-5428)
  --(5398,-5422)--(5395,-5414)--(5391,-5403)--(5386,-5389)--(5380,-5369)--(5373,-5345)
  --(5365,-5315)--(5355,-5279)--(5345,-5239)--(5334,-5197)--(5323,-5150)--(5314,-5107)
  --(5306,-5070)--(5300,-5039)--(5295,-5014)--(5291,-4996)--(5288,-4982)--(5286,-4973)
  --(5284,-4966)--(5283,-4960)--(5282,-4955)--(5281,-4948)--(5280,-4939)--(5279,-4925)
  --(5277,-4907)--(5275,-4882)--(5274,-4851)--(5272,-4814)--(5271,-4771)--(5270,-4724)
  --(5271,-4677)--(5272,-4634)--(5274,-4597)--(5275,-4566)--(5277,-4541)--(5279,-4523)
  --(5280,-4510)--(5281,-4500)--(5282,-4493)--(5283,-4488)--(5284,-4482)--(5286,-4476)
  --(5288,-4466)--(5291,-4453)--(5295,-4435)--(5300,-4410)--(5306,-4379)--(5314,-4342)
  --(5323,-4299)--(5334,-4252)--(5345,-4210)--(5355,-4170)--(5365,-4134)--(5373,-4104)
  --(5380,-4080)--(5386,-4060)--(5391,-4046)--(5395,-4035)--(5398,-4027)--(5400,-4021)
  --(5403,-4016)--(5405,-4011)--(5408,-4005)--(5411,-3997)--(5416,-3986)--(5422,-3972)
  --(5429,-3952)--(5439,-3928)--(5450,-3898)--(5464,-3862)--(5479,-3822)--(5496,-3780)
  --(5513,-3738)--(5529,-3698)--(5544,-3662)--(5557,-3632)--(5567,-3607)--(5575,-3588)
  --(5582,-3573)--(5587,-3562)--(5590,-3554)--(5593,-3548)--(5596,-3543)--(5598,-3539)
  --(5601,-3533)--(5605,-3525)--(5611,-3514)--(5618,-3499)--(5627,-3480)--(5639,-3455)
  --(5653,-3425)--(5669,-3389)--(5688,-3349)--(5708,-3307)--(5728,-3265)--(5747,-3225)
  --(5764,-3189)--(5779,-3159)--(5791,-3135)--(5801,-3115)--(5808,-3101)--(5813,-3090)
  --(5817,-3082)--(5820,-3076)--(5823,-3071)--(5826,-3066)--(5829,-3060)--(5833,-3052)
  --(5838,-3041)--(5846,-3027)--(5856,-3007)--(5868,-2983)--(5884,-2953)--(5902,-2917)
  --(5922,-2877)--(5944,-2835)--(5966,-2792)--(5987,-2751)--(6007,-2713)--(6026,-2677)
  --(6043,-2643)--(6060,-2612)--(6075,-2583)--(6090,-2555)--(6104,-2529)--(6117,-2504)
  --(6130,-2480)--(6141,-2458)--(6152,-2438)--(6162,-2419)--(6171,-2403)--(6178,-2390)
  --(6184,-2379)--(6188,-2371)--(6191,-2366)--(6192,-2363)--(6193,-2362);
\pgfsetfillcolor{black}
\pgftext[base,at=\pgfqpointxy{4725}{-2127}] {\fontsize{5}{6}\usefont{T1}{ptm}{m}{n}$\tilde D_0$}
\pgftext[base,right,at=\pgfqpointxy{7558}{-6518}] {\fontsize{5}{6}\usefont{T1}{ptm}{m}{n}$\tilde D_+$}
\pgftext[base,left,at=\pgfqpointxy{1897}{-6529}] {\fontsize{5}{6}\usefont{T1}{ptm}{m}{n}$\tilde D_-$}
\pgfsetdash{{+15\XFigu}{+60\XFigu}}{+15\XFigu}
\draw (3678,-4720)--(5263,-3000);
\draw (5780,-4720)--(5263,-3008);
\pgfsetdash{}{+0pt}
\filldraw  (4724,-4724) circle [radius=+37];
\pgfsetarrows{[line width=7.5\XFigu, width=45\XFigu]}
\pgfsetarrowsend{xfiga1}
\draw (4517,-4720) arc[start angle=+-2.40, end angle=+48.05, radius=+819.7];
\draw (6356,-4703) arc[start angle=+-2.3, end angle=+109.2, radius=+567.7];
\filldraw  (5264,-3009) circle [radius=+37];
\pgftext[base,at=\pgfqpointxy{4722}{-5098}] {\fontsize{5}{6}\usefont{T1}{ptm}{m}{n}$\frac{\alpha}{3 \beta}$}
\pgftext[base,at=\pgfqpointxy{5254}{-2841}] {\fontsize{5}{6}\usefont{T1}{ptm}{m}{n}$k$}
\pgftext[base,left,at=\pgfqpointxy{4492}{-4392}] {\fontsize{5}{6}\usefont{T1}{ptm}{m}{n}$\theta_-$}
\pgftext[base,at=\pgfqpointxy{5983}{-4519}] {\fontsize{5}{6}\usefont{T1}{ptm}{m}{n}$\theta_+$}
\pgftext[base,right,at=\pgfqpointxy{4551}{-3774}] {\fontsize{5}{6}\usefont{T1}{ptm}{m}{n}$r_-$}
\pgftext[base,left,at=\pgfqpointxy{5576}{-3960}] {\fontsize{5}{6}\usefont{T1}{ptm}{m}{n}$r_+$}
\endtikzpicture}%
\caption{$\alpha^2 + 3 \beta \delta > 0$}
\end{subfigure}%
\begin{subfigure}{.33\textwidth}
\centering
%
%
{\pgfkeys{/pgf/fpu/.try=false}%
\ifx\XFigwidth\undefined\dimen1=0pt\else\dimen1\XFigwidth\fi
\divide\dimen1 by 6087
\ifx\XFigheight\undefined\dimen3=0pt\else\dimen3\XFigheight\fi
\divide\dimen3 by 5550
\ifdim\dimen1=0pt\ifdim\dimen3=0pt\dimen1=1657sp\dimen3\dimen1
  \else\dimen1\dimen3\fi\else\ifdim\dimen3=0pt\dimen3\dimen1\fi\fi
\tikzpicture[x=+\dimen1, y=+\dimen3]
{\ifx\XFigu\undefined\catcode`\@11
\def\temp{\alloc@1\dimen\dimendef\insc@unt}\temp\XFigu\catcode`\@12\fi}
\XFigu1657sp
\ifdim\XFigu<0pt\XFigu-\XFigu\fi
\definecolor{xfigc32}{rgb}{0.941,0.894,0.259}
\definecolor{xfigc33}{rgb}{0.000,0.447,0.698}
\definecolor{xfigc34}{rgb}{0.835,0.369,0.000}
\clip(1684,-7297) rectangle (7771,-1747);
\tikzset{inner sep=+0pt, outer sep=+0pt}
\pgfsetbeveljoin
\pgfsetfillcolor{xfigc32}
\fill (3361,-2362)--(3362,-2363)--(3363,-2365)--(3366,-2370)--(3370,-2377)--(3376,-2387)
  --(3384,-2401)--(3394,-2419)--(3407,-2441)--(3423,-2469)--(3441,-2501)--(3463,-2538)
  --(3487,-2581)--(3515,-2629)--(3546,-2682)--(3580,-2740)--(3616,-2804)--(3656,-2873)
  --(3698,-2946)--(3743,-3023)--(3789,-3104)--(3838,-3188)--(3888,-3275)--(3939,-3363)
  --(3992,-3454)--(4044,-3545)--(4096,-3635)--(4149,-3726)--(4200,-3814)--(4250,-3901)
  --(4299,-3985)--(4345,-4066)--(4390,-4143)--(4432,-4216)--(4472,-4285)--(4508,-4349)
  --(4542,-4407)--(4573,-4460)--(4601,-4508)--(4625,-4551)--(4647,-4588)--(4665,-4620)
  --(4681,-4648)--(4694,-4670)--(4704,-4688)--(4712,-4702)--(4718,-4712)--(4722,-4719)
  --(4725,-4724)--(4726,-4726)--(4727,-4727)--(4728,-4726)--(4729,-4724)--(4732,-4719)
  --(4736,-4712)--(4742,-4702)--(4750,-4688)--(4760,-4670)--(4773,-4648)--(4789,-4620)
  --(4807,-4588)--(4829,-4551)--(4854,-4508)--(4882,-4460)--(4913,-4407)--(4947,-4349)
  --(4983,-4285)--(5023,-4216)--(5066,-4143)--(5110,-4066)--(5157,-3985)--(5206,-3901)
  --(5256,-3814)--(5308,-3726)--(5360,-3635)--(5413,-3544)--(5466,-3454)--(5518,-3363)
  --(5570,-3275)--(5620,-3188)--(5669,-3104)--(5716,-3023)--(5760,-2946)--(5803,-2873)
  --(5843,-2804)--(5879,-2740)--(5913,-2682)--(5944,-2629)--(5972,-2581)--(5997,-2538)
  --(6019,-2501)--(6037,-2469)--(6053,-2441)--(6066,-2419)--(6076,-2401)--(6084,-2387)
  --(6090,-2377)--(6094,-2370)--(6097,-2365)--(6098,-2363)--(6099,-2362);
\pgfsetfillcolor{xfigc34}
\fill (7086,-7088)--(7086,-7087)--(7086,-7084)--(7086,-7078)--(7086,-7069)--(7086,-7056)
  --(7086,-7039)--(7086,-7017)--(7086,-6989)--(7086,-6955)--(7086,-6915)--(7086,-6869)
  --(7086,-6817)--(7086,-6759)--(7086,-6694)--(7086,-6624)--(7086,-6548)--(7086,-6467)
  --(7086,-6381)--(7086,-6291)--(7086,-6198)--(7086,-6102)--(7086,-6004)--(7087,-5906)
  --(7087,-5808)--(7087,-5710)--(7087,-5614)--(7087,-5521)--(7087,-5431)--(7087,-5345)
  --(7087,-5264)--(7087,-5188)--(7087,-5118)--(7087,-5053)--(7087,-4995)--(7087,-4943)
  --(7087,-4897)--(7087,-4857)--(7087,-4823)--(7087,-4795)--(7087,-4773)--(7087,-4756)
  --(7087,-4743)--(7087,-4734)--(7087,-4728)--(7087,-4725)--(7087,-4724)--(7086,-4724)
  --(7083,-4724)--(7077,-4724)--(7068,-4724)--(7055,-4724)--(7038,-4724)--(7016,-4724)
  --(6988,-4724)--(6954,-4724)--(6915,-4724)--(6869,-4724)--(6817,-4724)--(6758,-4724)
  --(6694,-4724)--(6624,-4725)--(6548,-4725)--(6467,-4725)--(6381,-4725)--(6291,-4725)
  --(6198,-4725)--(6103,-4725)--(6005,-4725)--(5907,-4726)--(5809,-4726)--(5711,-4726)
  --(5616,-4726)--(5523,-4726)--(5433,-4726)--(5347,-4726)--(5266,-4726)--(5190,-4726)
  --(5120,-4727)--(5056,-4727)--(4997,-4727)--(4945,-4727)--(4899,-4727)--(4860,-4727)
  --(4826,-4727)--(4798,-4727)--(4776,-4727)--(4759,-4727)--(4746,-4727)--(4737,-4727)
  --(4731,-4727)--(4728,-4727)--(4727,-4727)--(4728,-4728)--(4729,-4730)--(4731,-4735)
  --(4736,-4742)--(4741,-4752)--(4749,-4766)--(4760,-4784)--(4773,-4806)--(4788,-4833)
  --(4807,-4865)--(4828,-4903)--(4853,-4945)--(4880,-4993)--(4911,-5046)--(4945,-5105)
  --(4981,-5168)--(5021,-5236)--(5063,-5309)--(5107,-5386)--(5154,-5467)--(5202,-5551)
  --(5252,-5638)--(5303,-5726)--(5355,-5816)--(5408,-5907)--(5460,-5998)--(5512,-6088)
  --(5563,-6176)--(5613,-6263)--(5661,-6347)--(5708,-6428)--(5752,-6505)--(5794,-6578)
  --(5834,-6646)--(5870,-6709)--(5904,-6768)--(5935,-6821)--(5962,-6869)--(5987,-6911)
  --(6008,-6949)--(6027,-6981)--(6042,-7008)--(6055,-7030)--(6066,-7048)--(6074,-7062)
  --(6079,-7072)--(6084,-7079)--(6086,-7084)--(6087,-7086)--(6088,-7087);
\pgfsetfillcolor{xfigc33}
\fill (2365,-7097)--(2365,-7096)--(2365,-7093)--(2365,-7087)--(2365,-7078)--(2365,-7065)
  --(2365,-7048)--(2365,-7025)--(2365,-6997)--(2365,-6964)--(2365,-6924)--(2365,-6878)
  --(2365,-6825)--(2365,-6767)--(2365,-6702)--(2364,-6631)--(2364,-6555)--(2364,-6473)
  --(2364,-6387)--(2364,-6297)--(2364,-6203)--(2364,-6107)--(2364,-6009)--(2363,-5910)
  --(2363,-5812)--(2363,-5714)--(2363,-5618)--(2363,-5524)--(2363,-5434)--(2363,-5348)
  --(2363,-5266)--(2363,-5190)--(2362,-5119)--(2362,-5054)--(2362,-4996)--(2362,-4943)
  --(2362,-4897)--(2362,-4857)--(2362,-4824)--(2362,-4796)--(2362,-4773)--(2362,-4756)
  --(2362,-4743)--(2362,-4734)--(2362,-4728)--(2362,-4725)--(2362,-4724)--(2363,-4724)
  --(2366,-4724)--(2372,-4724)--(2381,-4724)--(2394,-4724)--(2411,-4724)--(2433,-4724)
  --(2461,-4724)--(2495,-4724)--(2535,-4724)--(2581,-4724)--(2633,-4724)--(2691,-4724)
  --(2756,-4724)--(2826,-4725)--(2902,-4725)--(2984,-4725)--(3070,-4725)--(3160,-4725)
  --(3253,-4725)--(3349,-4725)--(3446,-4725)--(3545,-4726)--(3643,-4726)--(3740,-4726)
  --(3836,-4726)--(3929,-4726)--(4019,-4726)--(4105,-4726)--(4187,-4726)--(4263,-4726)
  --(4333,-4727)--(4398,-4727)--(4456,-4727)--(4508,-4727)--(4554,-4727)--(4594,-4727)
  --(4628,-4727)--(4656,-4727)--(4678,-4727)--(4695,-4727)--(4708,-4727)--(4717,-4727)
  --(4723,-4727)--(4726,-4727)--(4727,-4727)--(4726,-4728)--(4725,-4730)--(4723,-4735)
  --(4718,-4742)--(4713,-4752)--(4705,-4766)--(4694,-4784)--(4682,-4807)--(4666,-4834)
  --(4648,-4866)--(4626,-4903)--(4602,-4946)--(4575,-4994)--(4544,-5047)--(4511,-5106)
  --(4474,-5170)--(4435,-5238)--(4393,-5312)--(4349,-5389)--(4303,-5470)--(4255,-5554)
  --(4205,-5641)--(4154,-5730)--(4103,-5821)--(4051,-5912)--(3999,-6003)--(3948,-6093)
  --(3897,-6182)--(3847,-6269)--(3799,-6353)--(3753,-6434)--(3709,-6511)--(3667,-6585)
  --(3628,-6653)--(3591,-6717)--(3558,-6776)--(3527,-6829)--(3500,-6877)--(3476,-6920)
  --(3454,-6957)--(3436,-6989)--(3420,-7016)--(3408,-7039)--(3397,-7057)--(3389,-7071)
  --(3384,-7081)--(3379,-7088)--(3377,-7093)--(3376,-7095)--(3375,-7096);
\pgfsetfillcolor{white}
\fill  (4724,-4724) circle [radius=+1362];
\fill  (4730,-4724) circle [radius=+1362];
\pgfsetlinewidth{+7.5\XFigu}
\pgfsetdash{}{+0pt}
\pgfsetstrokecolor{black}
\filldraw (3363,-4724) arc[start angle=+179.86, end angle=+240.21, radius=+1359.3];
\filldraw (6091,-4724) arc[start angle=+0.14, end angle=+-60.21, radius=+1359.3];
\filldraw (5413,-3545) arc[start angle=+59.82, end angle=+120.10, radius=+1365.3];
\draw (2362,-4724)--(7087,-4724);
\draw (3372,-7087)--(6099,-2362);
\draw (6088,-7087)--(3361,-2362);
\pgfsetfillcolor{black}
\pgftext[base,at=\pgfqpointxy{4725}{-2127}] {\fontsize{5}{6}\usefont{T1}{ptm}{m}{n}$\tilde D_0$}
\pgftext[base,right,at=\pgfqpointxy{7558}{-6518}] {\fontsize{5}{6}\usefont{T1}{ptm}{m}{n}$\tilde D_+$}
\pgftext[base,left,at=\pgfqpointxy{1897}{-6529}] {\fontsize{5}{6}\usefont{T1}{ptm}{m}{n}$\tilde D_-$}
\filldraw  (4727,-4727) circle [radius=+37];
\pgftext[base,left,at=\pgfqpointxy{5034}{-4570}] {\fontsize{5}{6}\usefont{T1}{ptm}{m}{n}$\frac{\alpha}{3 \beta}$}
\pgfsetdash{{+60\XFigu}{+60\XFigu}}{++0pt}
\draw (4714,-3360)--(4727,-4727);
\pgfsetdash{}{+0pt}
\draw (4718,-4153) arc[start angle=+98.5, end angle=+49.0, radius=+366.5];
\pgftext[base,at=\pgfqpointxy{3425}{-3071}] {\fontsize{5}{6}\usefont{T1}{ptm}{m}{n}$\Gamma_1$}
\pgftext[base,at=\pgfqpointxy{6024}{-3071}] {\fontsize{5}{6}\usefont{T1}{ptm}{m}{n}$\Gamma_3$}
\pgftext[base,at=\pgfqpointxy{4724}{-3189}] {\fontsize{5}{6}\usefont{T1}{ptm}{m}{n}$\Gamma_2$}
\pgftext[base,at=\pgfqpointxy{6614}{-4606}] {\fontsize{5}{6}\usefont{T1}{ptm}{m}{n}$\Gamma_4$}
\pgftext[base,at=\pgfqpointxy{5669}{-5315}] {\fontsize{5}{6}\usefont{T1}{ptm}{m}{n}$\Gamma_5$}
\pgftext[base,at=\pgfqpointxy{2835}{-4606}] {\fontsize{5}{6}\usefont{T1}{ptm}{m}{n}$\Gamma_9$}
\pgftext[base,at=\pgfqpointxy{3780}{-5315}] {\fontsize{5}{6}\usefont{T1}{ptm}{m}{n}$\Gamma_8$}
\pgftext[base,at=\pgfqpointxy{5551}{-6732}] {\fontsize{5}{6}\usefont{T1}{ptm}{m}{n}$\Gamma_6$}
\pgftext[base,at=\pgfqpointxy{3898}{-6732}] {\fontsize{5}{6}\usefont{T1}{ptm}{m}{n}$\Gamma_7$}
\pgftext[base,at=\pgfqpointxy{4927}{-4030}] {\fontsize{5}{6}\usefont{T1}{ptm}{m}{n}$\phi_0$}
\endtikzpicture}%
\caption{$\alpha^2 + 3 \beta \delta = 0$}
\end{subfigure}%
\begin{subfigure}{.33\textwidth}
\centering
%
%
{\pgfkeys{/pgf/fpu/.try=false}%
\ifx\XFigwidth\undefined\dimen1=0pt\else\dimen1\XFigwidth\fi
\divide\dimen1 by 6104
\ifx\XFigheight\undefined\dimen3=0pt\else\dimen3\XFigheight\fi
\divide\dimen3 by 5548
\ifdim\dimen1=0pt\ifdim\dimen3=0pt\dimen1=1657sp\dimen3\dimen1
  \else\dimen1\dimen3\fi\else\ifdim\dimen3=0pt\dimen3\dimen1\fi\fi
\tikzpicture[x=+\dimen1, y=+\dimen3]
{\ifx\XFigu\undefined\catcode`\@11
\def\temp{\alloc@1\dimen\dimendef\insc@unt}\temp\XFigu\catcode`\@12\fi}
\XFigu1657sp
\ifdim\XFigu<0pt\XFigu-\XFigu\fi
\pgfdeclarearrow{
  name = xfiga1,
  parameters = {
    \the\pgfarrowlinewidth \the\pgfarrowlength \the\pgfarrowwidth\ifpgfarrowopen o\fi},
  defaults = {
	  line width=+7.5\XFigu, length=+120\XFigu, width=+60\XFigu},
  setup code = {
    \dimen7 2.1\pgfarrowlength\pgfmathveclen{\the\dimen7}{\the\pgfarrowwidth}
    \dimen7 2\pgfarrowwidth\pgfmathdivide{\pgfmathresult}{\the\dimen7}
    \dimen7 \pgfmathresult\pgfarrowlinewidth
    \pgfarrowssettipend{+\dimen7}
    \pgfarrowssetbackend{+-\pgfarrowlength}
    \dimen9 -\pgfarrowlength\advance\dimen9 by-0.45\pgfarrowlinewidth
    \pgfarrowssetlineend{+\dimen9}
    \dimen9 -\pgfarrowlength\advance\dimen9 by-0.5\pgfarrowlinewidth
    \pgfarrowssetvisualbackend{+\dimen9}
    \pgfarrowshullpoint{+\dimen7}{+0pt}
    \pgfarrowsupperhullpoint{+-\pgfarrowlength}{+0.5\pgfarrowwidth}
    \pgfarrowssavethe\pgfarrowlinewidth
    \pgfarrowssavethe\pgfarrowlength
    \pgfarrowssavethe\pgfarrowwidth
  },
  drawing code = {\pgfsetdash{}{+0pt}
    \ifdim\pgfarrowlinewidth=\pgflinewidth\else\pgfsetlinewidth{+\pgfarrowlinewidth}\fi
    \pgfpathmoveto{\pgfqpoint{-\pgfarrowlength}{-0.5\pgfarrowwidth}}
    \pgfpathlineto{\pgfqpoint{0pt}{0pt}}
    \pgfpathlineto{\pgfqpoint{-\pgfarrowlength}{0.5\pgfarrowwidth}}
    \pgfpathclose
    \ifpgfarrowopen\pgfusepathqstroke\else\pgfsetfillcolor{.}
	\ifdim\pgfarrowlinewidth>0pt\pgfusepathqfillstroke\else\pgfusepathqfill\fi\fi
  }
}
\pgfdeclarearrow{
  name = xfiga5,
  parameters = {
    \the\pgfarrowlinewidth \the\pgfarrowlength\ifpgfarrowopen o\fi},
  defaults = {
	  line width=+7.5\XFigu, length=+120\XFigu},
  setup code = {
    \dimen7 0.5\pgfarrowlinewidth
    \pgfarrowssettipend{+\dimen7}
    \pgfarrowssetbackend{+-\pgfarrowlength}
    \dimen9 -\pgfarrowlength\advance\dimen9 by0.45\pgfarrowlinewidth
    \pgfarrowssetlineend{+\dimen9}
    \dimen9 -\pgfarrowlength\advance\dimen9 by-0.5\pgfarrowlinewidth
    \pgfarrowssetvisualbackend{+\dimen9}
    \pgfarrowshullpoint{+\dimen7}{+0pt}
    \pgfarrowsupperhullpoint{+-0.25\pgfarrowlength}{+0.35\pgfarrowwidth}\pgfarrowsupperhullpoint{+-0.5\pgfarrowlength}{+0.5\pgfarrowwidth}\pgfarrowsupperhullpoint{+-0.75\pgfarrowlength}{+0.35\pgfarrowwidth}
    \pgfarrowshullpoint{+\dimen9}{+0pt}
    \pgfarrowssavethe\pgfarrowlinewidth
    \pgfarrowssavethe\pgfarrowlength
  },
  drawing code = {\pgfsetdash{}{+0pt}
    \ifdim\pgfarrowlinewidth=\pgflinewidth\else\pgfsetlinewidth{+\pgfarrowlinewidth}\fi
    \dimen3 0.5\pgfarrowlength
    \pgfpathcircle{\pgfqpoint{-\dimen3}{0pt}}{+\dimen3}
    \ifpgfarrowopen\pgfusepathqstroke\else\pgfsetfillcolor{.}
	\ifdim\pgfarrowlinewidth>0pt\pgfusepathqfillstroke\else\pgfusepathqfill\fi\fi
  }
}
\definecolor{xfigc32}{rgb}{0.941,0.894,0.259}
\definecolor{xfigc33}{rgb}{0.000,0.447,0.698}
\definecolor{xfigc34}{rgb}{0.835,0.369,0.000}
\clip(1684,-7298) rectangle (7788,-1750);
\tikzset{inner sep=+0pt, outer sep=+0pt}
\pgfsetbeveljoin
\pgfsetfillcolor{xfigc32}
\fill (3397,-2362)--(3398,-2363)--(3399,-2365)--(3401,-2369)--(3405,-2375)--(3409,-2384)
  --(3416,-2395)--(3424,-2408)--(3433,-2424)--(3444,-2443)--(3455,-2463)--(3468,-2485)
  --(3482,-2509)--(3497,-2534)--(3513,-2561)--(3530,-2589)--(3547,-2619)--(3566,-2651)
  --(3587,-2684)--(3608,-2721)--(3632,-2759)--(3658,-2801)--(3685,-2846)--(3715,-2894)
  --(3747,-2946)--(3780,-2999)--(3809,-3046)--(3838,-3092)--(3866,-3136)--(3891,-3176)
  --(3915,-3213)--(3936,-3247)--(3955,-3276)--(3971,-3302)--(3985,-3325)--(3998,-3344)
  --(4009,-3362)--(4018,-3376)--(4026,-3390)--(4034,-3402)--(4042,-3414)--(4049,-3425)
  --(4057,-3437)--(4065,-3450)--(4074,-3464)--(4084,-3480)--(4096,-3497)--(4109,-3518)
  --(4124,-3541)--(4141,-3567)--(4160,-3595)--(4181,-3627)--(4203,-3661)--(4227,-3696)
  --(4251,-3732)--(4275,-3768)--(4312,-3822)--(4345,-3869)--(4372,-3907)--(4395,-3938)
  --(4413,-3963)--(4429,-3983)--(4442,-3998)--(4452,-4010)--(4461,-4020)--(4468,-4027)
  --(4474,-4032)--(4479,-4036)--(4483,-4039)--(4485,-4040)--(4487,-4041)--(4488,-4042)
  --(4491,-4042)--(4499,-4042)--(4512,-4042)--(4532,-4042)--(4559,-4042)--(4593,-4042)
  --(4633,-4042)--(4678,-4042)--(4725,-4042)--(4771,-4042)--(4816,-4042)--(4856,-4042)
  --(4890,-4042)--(4917,-4042)--(4937,-4042)--(4950,-4042)--(4958,-4042)--(4961,-4042)
  --(4962,-4042)--(4963,-4041)--(4965,-4039)--(4969,-4036)--(4974,-4031)--(4981,-4025)
  --(4988,-4018)--(4998,-4008)--(5009,-3996)--(5021,-3982)--(5036,-3964)--(5053,-3942)
  --(5074,-3915)--(5098,-3883)--(5126,-3843)--(5159,-3796)--(5197,-3741)--(5223,-3703)
  --(5249,-3664)--(5275,-3624)--(5301,-3585)--(5326,-3546)--(5350,-3508)--(5374,-3471)
  --(5398,-3434)--(5421,-3398)--(5443,-3363)--(5465,-3328)--(5487,-3293)--(5508,-3259)
  --(5529,-3226)--(5549,-3194)--(5568,-3163)--(5586,-3134)--(5603,-3107)--(5618,-3083)
  --(5631,-3061)--(5642,-3043)--(5651,-3028)--(5658,-3017)--(5663,-3008)--(5666,-3003)
  --(5668,-3000)--(5669,-2999)--(5670,-2997)--(5673,-2992)--(5678,-2984)--(5686,-2970)
  --(5697,-2952)--(5712,-2927)--(5730,-2897)--(5751,-2861)--(5776,-2821)--(5802,-2777)
  --(5831,-2729)--(5860,-2680)--(5889,-2632)--(5918,-2584)--(5944,-2540)--(5969,-2500)
  --(5990,-2464)--(6008,-2434)--(6023,-2409)--(6034,-2391)--(6042,-2377)--(6047,-2369)
  --(6050,-2364)--(6051,-2362);
\pgfsetfillcolor{xfigc33}
\fill (2362,-7087)--(2362,-7086)--(2362,-7083)--(2362,-7077)--(2362,-7068)--(2362,-7055)
  --(2362,-7038)--(2362,-7016)--(2362,-6988)--(2362,-6954)--(2362,-6914)--(2362,-6869)
  --(2362,-6816)--(2362,-6758)--(2362,-6693)--(2362,-6623)--(2362,-6547)--(2362,-6466)
  --(2362,-6380)--(2362,-6290)--(2362,-6197)--(2362,-6101)--(2362,-6004)--(2362,-5905)
  --(2362,-5807)--(2362,-5710)--(2362,-5614)--(2362,-5521)--(2362,-5431)--(2362,-5345)
  --(2362,-5264)--(2362,-5188)--(2362,-5118)--(2362,-5053)--(2362,-4995)--(2362,-4942)
  --(2362,-4897)--(2362,-4857)--(2362,-4823)--(2362,-4795)--(2362,-4773)--(2362,-4756)
  --(2362,-4743)--(2362,-4734)--(2362,-4728)--(2362,-4725)--(2362,-4724)--(2364,-4724)
  --(2368,-4724)--(2375,-4724)--(2387,-4724)--(2404,-4724)--(2426,-4724)--(2455,-4724)
  --(2489,-4724)--(2530,-4724)--(2577,-4724)--(2630,-4724)--(2687,-4724)--(2750,-4724)
  --(2815,-4724)--(2883,-4724)--(2953,-4724)--(3022,-4724)--(3090,-4724)--(3155,-4724)
  --(3218,-4724)--(3275,-4724)--(3328,-4724)--(3375,-4724)--(3416,-4724)--(3450,-4724)
  --(3479,-4724)--(3501,-4724)--(3518,-4724)--(3530,-4724)--(3537,-4724)--(3541,-4724)
  --(3543,-4724)--(3544,-4725)--(3548,-4727)--(3554,-4732)--(3563,-4739)--(3577,-4748)
  --(3594,-4761)--(3617,-4778)--(3645,-4798)--(3677,-4821)--(3715,-4848)--(3757,-4879)
  --(3803,-4912)--(3853,-4948)--(3906,-4986)--(3960,-5026)--(4016,-5066)--(4071,-5106)
  --(4125,-5146)--(4178,-5184)--(4228,-5220)--(4274,-5253)--(4316,-5284)--(4354,-5311)
  --(4386,-5334)--(4414,-5354)--(4437,-5371)--(4454,-5384)--(4468,-5393)--(4477,-5400)
  --(4483,-5405)--(4487,-5407)--(4488,-5408)--(4487,-5408)--(4486,-5409)--(4484,-5411)
  --(4480,-5414)--(4475,-5419)--(4468,-5425)--(4461,-5432)--(4451,-5442)--(4440,-5454)
  --(4428,-5468)--(4413,-5486)--(4396,-5508)--(4375,-5535)--(4351,-5567)--(4323,-5607)
  --(4290,-5654)--(4252,-5709)--(4226,-5747)--(4201,-5785)--(4176,-5823)--(4153,-5858)
  --(4132,-5890)--(4113,-5918)--(4097,-5944)--(4083,-5966)--(4070,-5985)--(4060,-6002)
  --(4051,-6016)--(4043,-6029)--(4035,-6041)--(4028,-6052)--(4021,-6063)--(4014,-6076)
  --(4005,-6089)--(3996,-6105)--(3985,-6123)--(3972,-6144)--(3956,-6168)--(3938,-6197)
  --(3918,-6230)--(3894,-6267)--(3869,-6309)--(3840,-6354)--(3811,-6402)--(3780,-6451)
  --(3748,-6503)--(3716,-6554)--(3687,-6602)--(3660,-6646)--(3635,-6688)--(3611,-6726)
  --(3590,-6762)--(3569,-6796)--(3550,-6828)--(3532,-6858)--(3516,-6886)--(3500,-6913)
  --(3484,-6939)--(3470,-6963)--(3457,-6985)--(3445,-7006)--(3434,-7024)--(3424,-7041)
  --(3416,-7055)--(3410,-7066)--(3405,-7075)--(3401,-7081)--(3399,-7085)--(3398,-7087)
  --(3397,-7088);
\pgfsetfillcolor{xfigc34}
\fill (5906,-4720)--(5908,-4720)--(5912,-4720)--(5919,-4720)--(5931,-4720)--(5948,-4720)
  --(5970,-4720)--(5999,-4720)--(6033,-4721)--(6074,-4721)--(6121,-4721)--(6174,-4721)
  --(6231,-4721)--(6294,-4722)--(6359,-4722)--(6427,-4722)--(6497,-4723)--(6566,-4723)
  --(6634,-4723)--(6699,-4723)--(6762,-4724)--(6819,-4724)--(6872,-4724)--(6919,-4724)
  --(6960,-4724)--(6994,-4725)--(7023,-4725)--(7045,-4725)--(7062,-4725)--(7074,-4725)
  --(7081,-4725)--(7085,-4725)--(7087,-4725)--(7087,-4726)--(7087,-4729)--(7087,-4735)
  --(7087,-4744)--(7087,-4757)--(7087,-4774)--(7087,-4796)--(7087,-4824)--(7087,-4858)
  --(7087,-4898)--(7087,-4943)--(7087,-4996)--(7087,-5054)--(7087,-5119)--(7087,-5189)
  --(7087,-5265)--(7087,-5346)--(7087,-5432)--(7087,-5522)--(7087,-5615)--(7087,-5711)
  --(7087,-5808)--(7087,-5907)--(7087,-6005)--(7087,-6102)--(7087,-6198)--(7087,-6291)
  --(7087,-6381)--(7087,-6467)--(7087,-6548)--(7087,-6624)--(7087,-6694)--(7087,-6759)
  --(7087,-6817)--(7087,-6870)--(7087,-6915)--(7087,-6955)--(7087,-6989)--(7087,-7017)
  --(7087,-7039)--(7087,-7056)--(7087,-7069)--(7087,-7078)--(7087,-7084)--(7087,-7087)
  --(7087,-7088)--(7085,-7088)--(7081,-7088)--(7073,-7088)--(7061,-7088)--(7043,-7088)
  --(7020,-7088)--(6991,-7088)--(6956,-7088)--(6914,-7088)--(6867,-7088)--(6815,-7088)
  --(6758,-7088)--(6697,-7088)--(6634,-7088)--(6569,-7088)--(6504,-7088)--(6441,-7088)
  --(6380,-7088)--(6323,-7088)--(6271,-7088)--(6224,-7088)--(6182,-7088)--(6147,-7088)
  --(6118,-7088)--(6095,-7088)--(6077,-7088)--(6065,-7088)--(6057,-7088)--(6053,-7088)
  --(6051,-7088)--(6050,-7087)--(6049,-7085)--(6047,-7081)--(6043,-7075)--(6038,-7066)
  --(6032,-7055)--(6024,-7041)--(6014,-7024)--(6003,-7006)--(5991,-6985)--(5978,-6963)
  --(5964,-6939)--(5949,-6913)--(5933,-6886)--(5916,-6858)--(5898,-6828)--(5879,-6796)
  --(5859,-6762)--(5837,-6726)--(5814,-6688)--(5789,-6646)--(5762,-6602)--(5732,-6554)
  --(5701,-6503)--(5669,-6451)--(5638,-6402)--(5609,-6354)--(5581,-6309)--(5555,-6267)
  --(5531,-6230)--(5511,-6197)--(5493,-6168)--(5478,-6144)--(5464,-6123)--(5453,-6105)
  --(5444,-6089)--(5435,-6076)--(5428,-6063)--(5421,-6052)--(5414,-6041)--(5406,-6029)
  --(5398,-6016)--(5389,-6002)--(5379,-5985)--(5367,-5966)--(5352,-5944)--(5336,-5918)
  --(5317,-5890)--(5296,-5858)--(5273,-5823)--(5248,-5785)--(5223,-5747)--(5197,-5709)
  --(5159,-5654)--(5126,-5607)--(5098,-5567)--(5074,-5535)--(5053,-5508)--(5036,-5486)
  --(5021,-5468)--(5009,-5454)--(4998,-5442)--(4988,-5432)--(4981,-5425)--(4974,-5419)
  --(4969,-5414)--(4965,-5411)--(4963,-5409)--(4962,-5408)--(4961,-5408);
\pgfsetfillcolor{white}
\fill  (4727,-4726) circle [radius=+1362];
\pgfsetlinewidth{+7.5\XFigu}
\pgfsetstrokecolor{black}
\pgfsetdash{}{+0pt}
\draw (2362,-4724)--(7087,-4724);
\draw (3397,-2362)--(3398,-2363)--(3399,-2365)--(3401,-2369)--(3405,-2375)--(3410,-2384)
  --(3416,-2395)--(3424,-2409)--(3434,-2426)--(3445,-2444)--(3457,-2465)--(3470,-2487)
  --(3484,-2511)--(3500,-2537)--(3516,-2564)--(3532,-2592)--(3550,-2622)--(3569,-2654)
  --(3590,-2688)--(3611,-2724)--(3635,-2762)--(3660,-2804)--(3687,-2848)--(3716,-2896)
  --(3748,-2947)--(3780,-2999)--(3811,-3048)--(3840,-3096)--(3869,-3141)--(3894,-3183)
  --(3918,-3220)--(3938,-3253)--(3956,-3282)--(3972,-3306)--(3985,-3327)--(3996,-3345)
  --(4005,-3361)--(4014,-3374)--(4021,-3387)--(4028,-3398)--(4035,-3409)--(4043,-3421)
  --(4051,-3434)--(4060,-3448)--(4070,-3465)--(4083,-3484)--(4097,-3506)--(4113,-3532)
  --(4132,-3560)--(4153,-3592)--(4176,-3627)--(4201,-3665)--(4226,-3703)--(4252,-3741)
  --(4289,-3796)--(4321,-3842)--(4347,-3878)--(4365,-3905)--(4378,-3923)--(4385,-3934)
  --(4389,-3940)--(4390,-3942)--(4389,-3941)--(4389,-3942)--(4392,-3944)--(4397,-3950)
  --(4406,-3960)--(4420,-3975)--(4439,-3994)--(4462,-4017)--(4488,-4042)--(4514,-4065)
  --(4537,-4085)--(4557,-4102)--(4573,-4115)--(4585,-4125)--(4593,-4133)--(4599,-4139)
  --(4603,-4143)--(4606,-4147)--(4609,-4150)--(4613,-4153)--(4619,-4157)--(4627,-4161)
  --(4639,-4165)--(4655,-4170)--(4675,-4174)--(4698,-4178)--(4724,-4179)--(4750,-4178)
  --(4773,-4174)--(4793,-4170)--(4809,-4165)--(4821,-4161)--(4830,-4157)--(4835,-4153)
  --(4839,-4150)--(4843,-4147)--(4846,-4143)--(4850,-4139)--(4855,-4133)--(4864,-4125)
  --(4876,-4115)--(4892,-4102)--(4912,-4085)--(4935,-4065)--(4961,-4042)--(4987,-4017)
  --(5010,-3994)--(5029,-3975)--(5043,-3960)--(5052,-3950)--(5057,-3944)--(5060,-3942)
  --(5060,-3941)--(5059,-3942)--(5060,-3940)--(5064,-3934)--(5071,-3923)--(5084,-3905)
  --(5102,-3878)--(5128,-3842)--(5160,-3796)--(5197,-3741)--(5223,-3703)--(5248,-3665)
  --(5273,-3627)--(5296,-3592)--(5317,-3560)--(5336,-3532)--(5352,-3506)--(5367,-3484)
  --(5379,-3465)--(5389,-3448)--(5398,-3434)--(5406,-3421)--(5414,-3409)--(5421,-3398)
  --(5428,-3387)--(5435,-3374)--(5444,-3361)--(5453,-3345)--(5464,-3327)--(5478,-3306)
  --(5493,-3282)--(5511,-3253)--(5531,-3220)--(5555,-3183)--(5581,-3141)--(5609,-3096)
  --(5638,-3048)--(5669,-2999)--(5701,-2947)--(5732,-2896)--(5762,-2848)--(5789,-2804)
  --(5814,-2762)--(5837,-2724)--(5859,-2688)--(5879,-2654)--(5898,-2622)--(5916,-2592)
  --(5933,-2564)--(5949,-2537)--(5964,-2511)--(5978,-2487)--(5991,-2465)--(6003,-2444)
  --(6014,-2426)--(6024,-2409)--(6032,-2395)--(6038,-2384)--(6043,-2375)--(6047,-2369)
  --(6049,-2365)--(6050,-2363)--(6051,-2362);
\draw (3397,-7088)--(3398,-7087)--(3399,-7085)--(3401,-7081)--(3405,-7075)--(3410,-7066)
  --(3416,-7055)--(3424,-7041)--(3434,-7024)--(3445,-7006)--(3457,-6985)--(3470,-6963)
  --(3484,-6939)--(3500,-6913)--(3516,-6886)--(3532,-6858)--(3550,-6828)--(3569,-6796)
  --(3590,-6762)--(3611,-6726)--(3635,-6688)--(3660,-6646)--(3687,-6602)--(3716,-6554)
  --(3748,-6503)--(3780,-6451)--(3811,-6402)--(3840,-6354)--(3869,-6309)--(3894,-6267)
  --(3918,-6230)--(3938,-6197)--(3956,-6168)--(3972,-6144)--(3985,-6123)--(3996,-6105)
  --(4005,-6089)--(4014,-6076)--(4021,-6063)--(4028,-6052)--(4035,-6041)--(4043,-6029)
  --(4051,-6016)--(4060,-6002)--(4070,-5985)--(4083,-5966)--(4097,-5944)--(4113,-5918)
  --(4132,-5890)--(4153,-5858)--(4176,-5823)--(4201,-5785)--(4226,-5747)--(4252,-5709)
  --(4289,-5654)--(4321,-5608)--(4347,-5572)--(4365,-5545)--(4378,-5527)--(4385,-5516)
  --(4389,-5510)--(4390,-5508)--(4389,-5509)--(4389,-5508)--(4392,-5506)--(4397,-5500)
  --(4406,-5490)--(4420,-5475)--(4439,-5456)--(4462,-5433)--(4488,-5408)--(4514,-5385)
  --(4537,-5365)--(4557,-5348)--(4573,-5335)--(4585,-5325)--(4593,-5317)--(4599,-5311)
  --(4603,-5307)--(4606,-5303)--(4609,-5300)--(4613,-5297)--(4619,-5293)--(4627,-5289)
  --(4639,-5285)--(4655,-5280)--(4675,-5276)--(4698,-5272)--(4724,-5271)--(4750,-5272)
  --(4773,-5276)--(4793,-5280)--(4809,-5285)--(4821,-5289)--(4830,-5293)--(4835,-5297)
  --(4839,-5300)--(4843,-5303)--(4846,-5307)--(4850,-5311)--(4855,-5317)--(4864,-5325)
  --(4876,-5335)--(4892,-5348)--(4912,-5365)--(4935,-5385)--(4961,-5408)--(4987,-5433)
  --(5010,-5456)--(5029,-5475)--(5043,-5490)--(5052,-5500)--(5057,-5506)--(5060,-5508)
  --(5060,-5509)--(5059,-5508)--(5060,-5510)--(5064,-5516)--(5071,-5527)--(5084,-5545)
  --(5102,-5572)--(5128,-5608)--(5160,-5654)--(5197,-5709)--(5223,-5747)--(5248,-5785)
  --(5273,-5823)--(5296,-5858)--(5317,-5890)--(5336,-5918)--(5352,-5944)--(5367,-5966)
  --(5379,-5985)--(5389,-6002)--(5398,-6016)--(5406,-6029)--(5414,-6041)--(5421,-6052)
  --(5428,-6063)--(5435,-6076)--(5444,-6089)--(5453,-6105)--(5464,-6123)--(5478,-6144)
  --(5493,-6168)--(5511,-6197)--(5531,-6230)--(5555,-6267)--(5581,-6309)--(5609,-6354)
  --(5638,-6402)--(5669,-6451)--(5701,-6503)--(5732,-6554)--(5762,-6602)--(5789,-6646)
  --(5814,-6688)--(5837,-6726)--(5859,-6762)--(5879,-6796)--(5898,-6828)--(5916,-6858)
  --(5933,-6886)--(5949,-6913)--(5964,-6939)--(5978,-6963)--(5991,-6985)--(6003,-7006)
  --(6014,-7024)--(6024,-7041)--(6032,-7055)--(6038,-7066)--(6043,-7075)--(6047,-7081)
  --(6049,-7085)--(6050,-7087)--(6051,-7088);
\pgfsetfillcolor{black}
\pgftext[base,at=\pgfqpointxy{4725}{-2127}] {\fontsize{5}{6}\usefont{T1}{ptm}{m}{n}$\tilde D_0$}
\pgftext[base,left,at=\pgfqpointxy{1897}{-6529}] {\fontsize{5}{6}\usefont{T1}{ptm}{m}{n}$\tilde D_-$}
\draw (4100,-3509) arc[start angle=+116.67, end angle=+63.33, radius=+1392.3];
\pgftext[base,right,at=\pgfqpointxy{7575}{-6525}] {\fontsize{5}{6}\usefont{T1}{ptm}{m}{n}$\tilde D_+$}
\draw (6089,-4721) arc[start angle=+-0.33, end angle=+-62.21, radius=+1378.9];
\draw (3363,-4721) arc[start angle=+-179.67, end angle=+-117.79, radius=+1378.9];
\pgfsetarrows{[line width=7.5\XFigu, width=45\XFigu, length=90\XFigu]}
\pgfsetarrowsend{xfiga1}
\draw (4712,-4856) arc[start angle=+90.58, end angle=+151.85, radius=+914.4];
\filldraw  (3042,-4881) circle [radius=+37];
\pgfsetdash{{+15\XFigu}{+60\XFigu}}{+15\XFigu}
\pgfsetarrowsend{}
\draw (4720,-3669)--(3042,-4881);
\draw (4720,-5771)--(3042,-4881);
\pgfsetdash{}{+0pt}
\filldraw  (4724,-4724) circle [radius=+37];
\pgfsetarrowsend{xfiga1}
\draw (4720,-2992) arc[start angle=+89.50, end angle=+215.52, radius=+669.8];
\pgftext[base,at=\pgfqpointxy{4356}{-5257}] {\fontsize{5}{6}\usefont{T1}{ptm}{m}{n}$\theta_-$}
\pgftext[base,at=\pgfqpointxy{3025}{-4655}] {\fontsize{5}{6}\usefont{T1}{ptm}{m}{n}$k$}
\pgftext[base,left,at=\pgfqpointxy{5034}{-4570}] {\fontsize{5}{6}\usefont{T1}{ptm}{m}{n}$\frac{\alpha}{3 \beta}$}
\pgftext[base,right,at=\pgfqpointxy{3822}{-4214}] {\fontsize{5}{6}\usefont{T1}{ptm}{m}{n}$r_+$}
\pgftext[base,right,at=\pgfqpointxy{4144}{-5562}] {\fontsize{5}{6}\usefont{T1}{ptm}{m}{n}$r_-$}
\pgftext[base,at=\pgfqpointxy{4441}{-3655}] {\fontsize{5}{6}\usefont{T1}{ptm}{m}{n}$\theta_+$}
\pgfsetlinewidth{+30\XFigu}
\pgfsetstrokecolor{red}
\pgfsetdash{{+120\XFigu}{+120\XFigu}}{++0pt}
\pgfsetarrows{[line width=30\XFigu]}
\pgfsetarrows{xfiga5-xfiga5}
\draw (4724,-3617)--(4724,-5832);
\endtikzpicture}%
\caption{$\alpha^2 + 3 \beta \delta < 0$}
\end{subfigure}
\caption{The open set $\tilde D$ is defined by \eqref{dtil-def} through the curve $\Im(\omega) = 0$ (which, in addition to the real axis, takes the form of a pair of intersecting lines in the case of $\alpha^2 + 3 \beta \delta = 0$ and a hyperbola otherwise), as well as the circle of radius $R_\Delta$ (given by \eqref{rd-def}) centered at $\frac{\alpha}{3 \beta}$. The positively oriented boundary of $\tilde D$ consists of the nine distinct curves $\Gamma_m$, $m = 1, 2, \ldots, 9$ (labeled only in the second figure for clarity). Additionally, $\tilde D$ is comprised of three connected subsets, $\tilde D_0$, $\tilde D_\pm$, where the subscript refers to the symmetry of $\omega$ (out of $\nu_0=k$ and $\nu_\pm$ given by \eqref{nupm}) that has positive imaginary part within that subset. Also, in all three cases, the angle $\phi_0$ is equal to half the measure of the arc $\Gamma_2$. Finally, when they exist, the branch points of the square root in the expression \eqref{nupm} for $\nu_\pm$ are associated with a branch cut taken so that it is entirely contained within the circle $\abs{k-\frac{\alpha}{3\beta}} = R_\Delta$ and hence lies outside $\tilde D$.}
\label{fig:Dbranches}
\end{figure}

For $m = 1, \ldots, 9$, let $v_m (x, t)$ be the portion of $v(x, t)$ which is integrated over the contour $\Gamma_m$ (see Figure \ref{fig:Dbranches}). Then, the solution \eqref{reducedsol} can be written as
\eq{\label{v-dec}
v(x, t) = \sum_{m=1}^9 v_m(x, t)
}
and, fixing $s \in \N_0$, for each $j \in \{0, 1, \ldots, s\}$ we have
\eq{\label{vmj-dec}
\partial_x^j v_m(x, t) = -i I_{j, \Gamma_m}^1(x, t) - I_{j, \Gamma_m}^0(x, t)
}
where
\eqs{
I_{j, \Gamma_m}^1(x, t) &:= \frac{1}{2\pi} \int_{\Gamma_m} \frac{(i k)^j e^{-i k \p{\ell - x} + i \omega t}}{\Delta(k)} \p{e^{-i \nu_+ \ell} - e^{-i \nu_- \ell}} \omega' \mathcal{F} \set{\psi_1} (\omega) dk, 
\alignlabel{I1def}\\
I_{j, \Gamma_m}^0(x, t) &:= \frac{1}{2\pi} \int_{\Gamma_m} \frac{(i k)^j e^{-i k \p{\ell - x} + i \omega t}}{\Delta(k)} \p{\nu_- e^{-i \nu_+ \ell} - \nu_+ e^{-i \nu_- \ell}}  \omega' \mathcal{F} \set{\psi_0} (\omega) dk.\alignlabel{I0def}}
Hence, it suffices to estimate $I_{j, \Gamma_m}^1$ and $I_{j, \Gamma_m}^0$ for $m=1, 2, 9$, since the remaining contours can be handled similarly thanks to symmetry. Each of the contours $\Gamma_1, \Gamma_2, \Gamma_9$  requires a different parametrization. Recalling the definition~\eqref{dntil-def} and computing 
\eq{\label{reomega}
\Im(\omega) = \beta \Im(k) \left\{3 \b{\Re(k) - \frac{\alpha}{3 \beta}}^2 - \b{\Im(k)}^2 - \frac{\alpha^2 + 3 \beta \delta}{3 \beta^2}\right\},
}
we parametrize $\Gamma_1, \Gamma_2, \Gamma_9$ respectively as follows:
\begin{subequations}\label{gamma-param}
\begin{align}
\gamma_1 (r) &= \frac{\alpha}{3 \beta} - \frac{\sqrt{3 \beta^2 r^2 + \alpha^2 + 3 \beta \delta}}{3 \beta} + i r, \quad r \ge R_\Delta \cos \phi_0 =: r_0>0, \label{gamma1}\\
\gamma_2 (\theta) &= \frac{\alpha}{3 \beta} + R_\Delta e^{i \theta}, \quad \frac{\pi}{2} - \phi_0 \le \theta \le \frac{\pi}{2} + \phi_0, \label{gamma2}\\
\gamma_9 (r) &= r, \quad r \le \frac{\alpha}{3 \beta} - R_\Delta, \label{gamma9}
\end{align}
\end{subequations}
where, as shown in Figure \ref{fig:Dbranches}, $\phi_0$ is the measure of the angle between the upward vertical and the ray emanating from $\frac{\alpha}{3 \beta}$ and passing through the intersection of $\Gamma_2$ and $\Gamma_3$. Note that the square root in $\gamma_1$ is always real.
\\[2mm]
\noindent
\textbf{Estimation of $v_1$.} This term consists of the integrals $I_{j, \Gamma_1}^1$ and $I_{j, \Gamma_1}^0$, which are similar to each other and hence can be estimated in the same way. We begin with $I_{j, \Gamma_1}^1$. 
By the parametrization \eqref{gamma1} of $\Gamma_1$,
\eq{I_{j, \Gamma_1}^1(x, t) = \frac{1}{2\pi} \int_\infty^{r_0} \frac{(i \gamma_1)^j e^{i \Re(\gamma_1) x - r x + i \omega t}}{e^{i \gamma_1 \ell} \Delta(\gamma_1)} \p{e^{-i \nu_+ \ell} - e^{-i \nu_- \ell}} \omega'(\gamma_1) \mathcal{F} \set{\psi_1} (\omega) \gamma_1'(r) dr.}
Note that $\omega$ is strictly real on $\Gamma_1$ by the definition \eqref{dntil-def}. Therefore, when taking the $L^2 (0, \ell)$ norm of the above expression, we can use the triangle inequality to eliminate the exponential factor containing both $i \Re(\gamma_1)$ and $i \omega t$. This leads to the estimate
\eq{\norm{I_{j, \Gamma_1}^1 (t)}_{L_x^2 (0, \ell)} \cle \norm{\int_{r_0}^\infty \frac{\abs{\gamma_1}^j e^{-r x}}{\abs{e^{i \gamma_1 \ell} \Delta(\gamma_1)}} \abs{\p{e^{-i \nu_+ \ell} - e^{-i \nu_- \ell}} \omega'(\gamma_1) \mathcal{F} \set{\psi_1} (\omega) \gamma_1'(r)} dr}_{L_x^2 (0, \ell)}. \label{I1G1est0}}
At this point, we invoke Lemma \ref{lemma:Deltabound} to note that the denominator within the integral is bounded below by a multiple of $\big|\gamma_1 - \frac{\alpha}{3 \beta}\big|$, as well as Lemma \ref{lemma:charDbar} to note that $\Im(\nu_\pm) \leq 0$ for $k\in \Gamma_1 \subseteq \overline D_0$ and hence the difference of exponentials in \eqref{I1G1est0} has a magnitude bounded above by 2. Thus, we find
\eq{\norm{I_{j, \Gamma_1}^1 (t)}_{L_x^2 (0, \ell)} 
&\cle 
\norm{\int_{r_0}^\infty e^{-r x} \frac{\abs{\gamma_1}^j}{\big|\gamma_1 - \frac{\alpha}{3 \beta}\big|} \abs{\mathcal{F} \set{\psi_1} (\omega)} \abs{\omega'(\gamma_1) \gamma_1'(r)} dr}_{L_x^2 (0, \ell)}
\nn\\
&\leq \norm{\L \set{\frac{\abs{\gamma_1}^j}{\big|\gamma_1 - \frac{\alpha}{3 \beta}\big|} \abs{\mathcal{F} \set{\psi_1} (\omega)} \abs{\omega'(\gamma_1) \gamma_1'(r)} \chi_{[r_0, \infty)}}}_{L_x^2 (0, \infty)},}
where $\L \set{\phi}(x) := \int_0^\infty e^{-rx} \phi (r) dr$, $x>0$, is the usual Laplace transform and $ \chi_{[r_0, \infty)}$ is the characteristic function of the interval $[r_0,\infty)$. At this point, we use the fact that the Laplace transform is bounded in $L^2(0, \infty)$:
\begin{lemma}[Hardy \cite{h1929}] \label{lemma:hardy}
The Laplace transform is bounded from $L^2(0, \infty)$ to $L^2(0, \infty)$ with
\[\norm{\mathcal L \set \phi}_{L^2 (0, \infty)} \le \sqrt \pi \norm{\phi}_{L^2 (0, \infty)}.\]
\end{lemma}

Lemma \ref{lemma:hardy}, whose proof can be found in \cite{fhm2017}, implies
\eqs{
\norm{I_{j, \Gamma_1}^1 (t)}_{L_x^2 (0, \ell)}^2 
\cle 
\int_{r_0}^\infty \frac{\abs{\gamma_1}^{2 j}}{\big|\gamma_1 - \frac{\alpha}{3 \beta}\big|^2} \abs{\mathcal{F} \set{\psi_1} (\omega)}^2 \abs{\frac{d\omega}{dr}}^2 dr,}
where $\omega'(\gamma_1)$ and $\gamma_1'(r)$ have been combined into $\frac{d\omega}{dr}$ via the chain rule.
At this point, we make the a change of variable from $r$ to $\omega$, as one can see by direct calculation that
\eqs{
\frac{d\omega}{dr} &= \frac{d\omega}{dk} \gamma_1'
= \frac{2 r}{\sqrt{3 \beta^2 r^2 + \alpha^2 + 3 \beta \delta}} \p{\frac{\beta^2 r^2}{3 \beta^2 r^2 + \alpha^2 + 3 \beta \delta} + 1} >0. \alignlabel{dtaudr}}
Hence, letting $\omega_0 := \omega(r_0)$, we have
\eq{\norm{I_{j, \Gamma_1}^1 (t)}_{L_x^2 (0, \ell)}^2 \cle \int_{\omega_0}^\infty \frac{\abs{\gamma_1}^{2 j}}{\big|\gamma_1 - \frac{\alpha}{3 \beta}\big|^2} \abs{\mathcal{F} \set{\psi_1} (\omega)}^2 \frac{d\omega}{dr} d\omega. \label{I1G1est1}}
For large $\gamma_1$, and hence for large $\omega$, $\big|\gamma_1 - \frac{\alpha}{3 \beta}\big| \simeq \abs{\gamma_1}$. More precisely, there is an $\omega_1 > \omega_0$ such that $\big|\gamma_1 - \frac{\alpha}{3 \beta}\big| \ge \frac{1}{2} \abs{\gamma_1}$. In view of this observation, we consider the cases of large $\omega$ and small $\omega$ separately to obtain
\eqs{
\norm{I_{j, \Gamma_1}^1 (t)}_{L_x^2 (0, \ell)}^2 &\cle \int_{\omega_0}^{\omega_1} \frac{\abs{\gamma_1}^{2 j}}{\big|\gamma_1 - \frac{\alpha}{3 \beta}\big|^2} \abs{\mathcal{F} \set{\psi_1} (\omega)}^2 \frac{d\omega}{dr} d\omega + \int_{\omega_1}^\infty \frac{\abs{\gamma_1}^{2 j}}{\big|\gamma_1 - \frac{\alpha}{3 \beta}\big|^2} \abs{\mathcal{F} \set{\psi_1} (\omega)}^2 \frac{d\omega}{dr} d\omega\\
&\cle \max_{\omega \in [\omega_0, \omega_1]} \Bigg(\frac{\abs{\gamma_1}^{2 j}}{\big|\gamma_1 - \frac{\alpha}{3 \beta}\big|^2} \frac{d\omega}{dr}\Bigg)  \norm{\mathcal{F} \set{\psi_1}}_{L_\omega^2 (\R)}^2 + \int_{\omega_1}^\infty \abs{\gamma_1}^{2 (j - 1)} \abs{\mathcal{F} \set{\psi_1} (\omega)}^2 \frac{d\omega}{dr} d\omega.}
For the first term, we use Plancherel's theorem. For the second term, we  employ Lemma 2 of \cite{amo2024} to infer 
$$
\abs{\gamma_1}^{2 (j-1)} \frac{d\omega}{dr} \le c \left(1 + \omega^2\right)^{\frac{j}{3}}
$$
for some $c > 0$ depending only on $\alpha, \beta, \delta$. 
Hence, we deduce the estimate 
\eqs{
\norm{I_{j, \Gamma_1}^1 (t)}_{L_x^2 (0, \ell)}^2 &\cle \norm{\psi_1}_{L^2 (\R)}^2 + \int_{\omega_1}^\infty \p{1 + \omega^2}^{\frac{j}{3}} \abs{\mathcal{F} \set{\psi_1} (\omega)}^2 d\omega
\cle \norm{\psi_1}_{H^{\frac{j}{3}} (\R)}^2, \quad j \in \N_0. \alignlabel{I1G1est2}}

Moving on to the term $I_{j, \Gamma_1}^0$, we note that it involves the extra factors of $\nu_-$ and $\nu_+$ within the sum of exponentials. This has the effect of introducing a factor of $\p{\abs{\nu_-} + \abs{\nu_+}}^2$ in the right-hand side of the analogue of \eqref{I1G1est1} for $I_{j, \Gamma_1}^0$, resulting in
\eq{\norm{I_{j, \Gamma_1}^0 (t)}_{L_x^2 (0, \ell)}^2 \cle \int_{\omega_0}^\infty \frac{\abs{\gamma_1}^{2 j}}{\big|\gamma_1 - \frac{\alpha}{3 \beta}\big|^2} \p{\abs{\nu_-} + \abs{\nu_+}}^2 \abs{\mathcal{F} \set{\psi_0} (\omega)}^2 \frac{d\omega}{dr} d\omega. \alignlabel{I0G1est1}}
From here, we observe that
\begin{equation}
\abs{\nu_\pm} 
\le 
\abs k \p{\frac{1}{2} + \frac{\abs \alpha}{2 \beta \abs k} + \frac{\sqrt 3}{2} \sqrt{1 + \frac{2 \abs \alpha}{3 \beta \abs k} + \frac{\abs{\alpha^2 + 4 \beta \delta}}{3 \beta^2 \abs{k}^2}}\,} \label{nupmbound}
\end{equation}
and so, for $\abs k$ large enough, $\abs{\nu_\pm}$ is bounded by $2 |k|$. Similarly to the estimation of $I_{j, \Gamma_1}^1$, we choose $\omega_1$ large enough so that $\big|\gamma_1 - \frac{\alpha}{3 \beta}\big| \ge \frac{1}{2} \abs{\gamma_1}$ but now, we also ensure that $\omega_1$ is large enough so that $\abs{\nu_\pm} \le 2 \abs{\gamma_1}$ for $\omega \ge \omega_1$. Then,~\eqref{I0G1est1} becomes
\eq{\norm{I_{j, \Gamma_1}^0 (t)}_{L_x^2 (0, \ell)}^2 \cle \int_{\omega_0}^{\omega_1} \frac{\abs{\gamma_1}^{2 j}}{\big|\gamma_1 - \frac{\alpha}{3 \beta}\big|^2} \p{\abs{\nu_+} + \abs{\nu_-}}^2 \abs{\mathcal{F} \set{\psi_0} (\omega)}^2 \frac{d\omega}{dr} d\omega + \int_{\omega_1}^\infty \abs{\gamma_1}^{2 j} \abs{\mathcal{F} \set{\psi_0} (\omega)}^2 \frac{d\omega}{dr} d\omega.}
After this adjustment, we follow the same argument as for  $I_{j, \Gamma_1}^1$ to conclude that
\eq{\norm{I_{j, \Gamma_1}^0 (t)}_{L_x^2 (0, \ell)} \cle \norm{\psi_0}_{H^{\frac{j + 1}{3}} (\R)}. \label{I0G1est2}}

Overall, for each $j\in \set{0, 1, \ldots, s}$, we have the estimate
\eq{\norm{\partial_x^j v_1 (t)}_{L_x^2 (0, \ell)} \cle \norm{\psi_0}_{H^{\frac{j + 1}{3}} (\R)} + \norm{\psi_1}_{H^{\frac{j}{3}} (\R)}}
and, therefore, in view of the definition \eqref{hs-l2-def} of the Sobolev norm, 
\eq{\label{v1-hs-est}
\norm{v_1 (t)}_{H_x^s (0, \ell)} \cle \norm{\psi_0}_{H^{\frac{s + 1}{3}} (\R)} + \norm{\psi_1}_{H^{\frac{s}{3}} (\R)}, \quad s \in \N_0, \ t\in [0, T'].}
\vskip 2mm
\noindent
\textbf{Estimation of $v_2$.} 
The second contour to consider is the circular arc $\Gamma_2$. Contrary to $\Gamma_1$, $\omega$ is not purely real on $\Gamma_2$ and so the method used for $v_1$ no longer applies. Instead, we will exploit the fact that $\Gamma_2$ is finite. For this, in view of Remark \ref{ft-til-r} and the parametrization \eqref{gamma2}, we express $I_{j, \Gamma_2}^1$ in the form
\eqs{
I_{j, \Gamma_2}^1(x, t) 
&= \frac{1}{2\pi} \int_{\frac{\pi}{2} + \phi_0}^{\frac{\pi}{2} - \phi_0} \frac{(i \gamma_2)^j e^{i \gamma_2 x + i \omega t}}{e^{i \gamma_2 \ell} \Delta(\gamma_2)} \p{e^{-i \nu_+ \ell} - e^{-i \nu_- \ell}} \omega'(\gamma_2) \tilde \psi_1 (\omega, T') \cdot i R_\Delta e^{i \theta} d\theta.}
Taking the $L^2(0,\ell)$ norm and applying the triangle inequality yields
$$
\norm{I_{j, \Gamma_2}^1 (t)}_{L_x^2 (0, \ell)} \cle \norm{\int_{\frac{\pi}{2} - \phi_0}^{\frac{\pi}{2} + \phi_0} \frac{\abs{\gamma_2}^j e^{-R_\Delta \sin(\theta) x - \Im(\omega) t}}{\abs{e^{i \gamma_2 \ell} \Delta(\gamma_2)}} \p{e^{\Im \p{\nu_+} \ell} + e^{\Im \p{\nu_-} \ell}} \abs{\omega'(\gamma_2)} \abs{\tilde \psi_1 (\omega, T')} R_\Delta d\theta}_{L_x^2 (0, \ell)}.
$$
We can reuse some previous ideas to simplify. Lemma \ref{lemma:Deltabound} addresses the denominator and Lemma \ref{lemma:charDbar} allows us to boost the sum of exponentials to 2. Furthermore, because $\Gamma_2$ is part of $\partial \tilde D_0$  (see Figure \ref{fig:Dbranches}), we have $\sin(\theta) > 0$. Thus, we can eliminate the dependence on $x$ and obtain the bound
\eqs{
\norm{I_{j, \Gamma_2}^1 (t)}_{L_x^2 (0, \ell)} 
\cle 
\sqrt \ell \int_{\frac{\pi}{2} - \phi_0}^{\frac{\pi}{2} + \phi_0} \abs{\gamma_2}^j e^{-\Im(\omega) t} \abs{\omega'(\gamma_2)} \abs{\tilde \psi_1 (\omega, T')} d\theta.}
The derivative can be absorbed as a constant since 
\eqs{
\abs{\omega'(\gamma_2)} = \abs{3 \beta \gamma_2^2 - 2 \alpha \gamma_2 - \delta} \le 3 \beta \p{\frac{\abs \alpha}{3 \beta} + R_\Delta}^2 + 2 \abs \alpha \p{\frac{\abs \alpha}{3 \beta} + R_\Delta} + \abs \delta.\alignlabel{domegadkbound}}
Therefore, noting also that $\abs{\gamma_2} 
\le \frac{\abs \alpha}{3 \beta} + R_\Delta$ and recalling the definition \eqref{tilde-transform}, we find
\eqs{
\norm{I_{j, \Gamma_2}^1 (t)}_{L_x^2 (0, \ell)} \cle 
\int_{\frac{\pi}{2} - \phi_0}^{\frac{\pi}{2} + \phi_0} \int_0^{T'} e^{\Im(\omega) \p{t' - t}} \abs{\psi_1 (t')} dt' d\theta.}
Since $\Im(\omega) \le 0$ on $\Gamma_2$, for any fixed $T'\geq t$ we have $e^{\Im(\omega) (t'-t)} \leq e^{-\Im(\omega) t} \leq e^{|\omega| T'} \leq e^{M_\Delta T'}$ 
after observing that, along $\Gamma_2$, 
\eq{\label{md-def}
|\omega| \leq M_\Delta := \beta\p{\abs{\frac{\alpha}{3\beta}}+R_\Delta}^3+|\alpha| \p{\abs{\frac{\alpha}{3\beta}}+R_\Delta}^2 + |\delta| \p{\abs{\frac{\alpha}{3\beta}}+R_\Delta}.
}
Hence, 
\eqs{
\norm{I_{j, \Gamma_2}^1 (t)}_{L_x^2 (0, \ell)} &\cle e^{M_\Delta T'}  \int_{\frac{\pi}{2} - \phi_0}^{\frac{\pi}{2} + \phi_0} \int_0^{T'} \abs{\psi_1 (t')} dt' d\theta \simeq e^{M_\Delta T'}  \int_0^{T'} \abs{\psi_1 (t')} dt'}
and by the Cauchy-Schwarz inequality 
\eq{
\norm{I_{j, \Gamma_2}^1 (t)}_{L_x^2 (0, \ell)} 
\cle \sqrt{T'} e^{M_\Delta T'} \norm{\psi_1}_{L^2 (0, T')} 
\label{I1G2est1}}
for all $j \in \set{0, 1, \ldots, s}$ and  each $t \in [0, T']$. 
The integral $I_{j, \Gamma_2}^0 (t)$ can be addressed in much the same way to establish the bound 
\eq{\norm{I_{j, \Gamma_2}^0 (t)}_{L_x^2 (0, \ell)} \cle \sqrt{T'} e^{M_\Delta T'} \norm{\psi_0}_{L^2(0, T')}.\label{I0G2est1}}
Therefore, we conclude that
\eqs{
\norm{\partial_x^j v_2 (t)}_{L_x^2 (0, \ell)} &\cle \sqrt{T'} e^{M_\Delta T'} \p{\norm{\psi_0}_{L^2(0, T')} + \norm{\psi_1}_{L^2(0, T')}}
}
and, in turn,
\eq{\label{v2-hs-est}
\norm{v_2 (t)}_{H_x^s (0, \ell)} \cle \sqrt{T'} e^{M_\Delta T'}\Big(\norm{\psi_0}_{H^{\frac{s + 1}{3}} (\R)} + \norm{\psi_1}_{H^{\frac{s}{3}} (\R)}\Big), \quad s\in\N_0, \ t\in [0, T'].}
\vskip 2mm
\noindent
\textbf{Estimation of $v_9$.} 
Because $\Gamma_9$ is a real contour (contrary to $\Gamma_1$ which is complex), a different approach than the one used for the estimation of $v_1$ will be followed. Specifically, instead of the  $L^2$ characterization of the $H^s(0, \ell)$ norm  \eqref{hs-l2-def}, we will exploit the fact that the expression for  $v_9$ actually makes sense for all $x\in\mathbb R$ and not just for $x\in (0, \ell)$ in order to estimate the $H^s(\R)$ norm of this term via the usual Fourier transform characterization which is available for $L^2$-based Sobolev spaces on the infinite line, namely
\begin{equation}
\norm{v_9(t)}_{H_x^s(\R)}^2 = \int_{-\infty}^\infty \left(1+k^2\right)^s \left|\mathcal F_x\{v_9\}(k, t)\right|^2 dk.
\end{equation}
From \eqref{vmj-dec}, we have $v_9 = -i I_{0, \Gamma_9}^1 - I_{0, \Gamma_9}^0$. We begin with the former term, which contains
\eqs{
I_{0, \Gamma_9}^1 (x, t)
&= \frac{1}{2 \pi} \int_{-\infty}^{\frac{\alpha}{3 \beta} - R_\Delta} e^{i r x} \frac{-e^{-i r \ell + i \omega t}}{e^{i \nu_- \ell} \Delta(r)} \p{e^{i \p{\nu_- - \nu_+} \ell} - 1} \omega'(r) \mathcal{F} \set{\psi_1} (\omega) dr. \alignlabel{I0G91def}}
Observe that this is an inverse Fourier transform and hence, by injectivity,
\eq{\mathcal{F} \set{I_{0, \Gamma_9}^1} (r, t) = \frac{-e^{-i r \ell + i \omega t}}{e^{i \nu_- \ell} \Delta(r)} \p{e^{i \p{\nu_- - \nu_+} \ell} - 1} \omega'(r) \mathcal{F} \set{\psi_1} (\omega) \chi_{(-\infty, \frac{\alpha}{3 \beta} - R_\Delta]}. \label{I0G91hat}}
Taking the magnitude of this quantity and using previous techniques to simplify (i.e. that the numerator has unit magnitude, the magnitude of the denominator is bounded below by a multiple of $\big|r - \frac{\alpha}{3 \beta}\big|$, and the difference of exponentials in \eqref{I0G91hat} can be bounded by 2), we are left with
\eq{\abs{\mathcal{F} \set{I_{0, \Gamma_9}^1} (r, t)} \cle \frac{1}{\big|r - \frac{\alpha}{3 \beta}\big|} \abs{\omega'(r)} \abs{\mathcal{F} \set{\psi_1} (\omega)} \chi_{(-\infty, \frac{\alpha}{3 \beta} - R_\Delta]}.}
Now, a direct calculation yields
$\omega'(r) = 
3 \beta \big(r - \frac{\alpha}{3 \beta}\big)^2 - \frac{\alpha^2 + 3 \beta \delta}{3 \beta}$.
Since $\Gamma_9$ lies beyond the branch points, i.e. $\big|r - \frac{\alpha}{3 \beta}\big| > \frac{2}{3 \beta} \sqrt{\abs{\alpha^2 + 3 \beta \delta}}$, it follows that $\omega'(r)$ must be positive. Hence,
\eqs{
\abs{\mathcal{F} \set{I_{0, \Gamma_9}^1}(r, t)} &\cle 
\sqrt{3 \beta - \frac{\alpha^2 + 3 \beta \delta}{3 \beta \big(r - \frac{\alpha}{3 \beta}\big)^2}} 
\sqrt{\omega'(r)} \abs{\mathcal{F} \set{\psi_1} (\omega)} \chi_{(-\infty, \frac{\alpha}{3 \beta} - R_\Delta]}.}
If $\alpha^2 + 3 \beta \delta \ge 0$, the first square root may be increased to $\sqrt{3 \beta}$. Otherwise, 
$
\big|r - \frac{\alpha}{3 \beta}\big| \ge R_\Delta \ge \frac{2}{3 \beta} \sqrt{\abs{\alpha^2 + 3 \beta \delta}}
$
so 
\eqs{
\sqrt{3 \beta - \frac{\alpha^2 + 3 \beta \delta}{3 \beta \big(r - \frac{\alpha}{3 \beta}\big)^2}} &\le \sqrt{3 \beta + \frac{\abs{\alpha^2 + 3 \beta \delta}}{\frac{4}{3 \beta} \abs{\alpha^2 + 3 \beta \delta}}}
=\frac{\sqrt{15 \beta}}{2}.
}
In any case, we find
\eq{\abs{\mathcal{F} \set{I_{0, \Gamma_9}^1} (r, t)} \cle \sqrt{\omega'(r)} \abs{\mathcal{F} \set{\psi_1} (\omega)} \chi_{(-\infty, r_0]}}
which can be used to directly calculate the $H^s(0,\ell)$ norm of $I_{0, \Gamma_9}^1$ as follows:
\eqs{
\norm{I_{0, \Gamma_9}^1(t)}_{H_x^s (0, \ell)}^2 &\le \norm{I_{0, \Gamma_9}^1(t)}_{H_x^s (\R)}^2 
\cle \int_{-\infty}^\infty \p{1 + r^2}^s \omega'(r) \abs{\mathcal{F} \set{\psi_1} (\omega)}^2 \chi_{(-\infty, \frac{\alpha}{3 \beta} - R_\Delta]} dr
\alignlabel{I1Gamma6pretau}}
since $\omega'(r)$ is positive and $\omega \to \pm \infty$ as $r \to \pm \infty$. 
Next, changing variables from $r$ to $\omega$ and using Lemma \ref{lemma:rtotau} below, we deduce
\eqs{
\norm{I_{0, \Gamma_9}^1 (t)}_{H_x^s (0, \ell)} \cle 
\norm{\psi_1}_{H^{\frac{s}{3}} (\R)}, \quad t\in [0, T'].
\alignlabel{I1Gamma9-est}}

\begin{lemma} \label{lemma:rtotau}
There exists some $c > 0$ such that $\p{1 + r^2}^3 \le c \b{1 + \omega (r)^2}$ for any $r \in \R$.
\end{lemma}

\begin{proof}[Proof of Lemma \ref{lemma:rtotau}]
Since $\omega = \beta r^3 - \alpha r^2 - \delta r$, we can write
$1 + \omega^2 = \beta^2 \p{1 + r^2}^3 + P(r)$ 
where $P$ is some polynomial of degree at most five. Since $1 + \omega^2$ and $\p{1 + r^2}^3$ are positive, $1 + \omega^2 = \abs{1 + \omega^2} \ge \beta^2 \p{1 + r^2}^3 - \abs{P(r)}$.
Since $P$ has degree less than six, there exists some $r_1 > 0$ such that, for all $\abs{r} > r_1$, $\abs{P(r)} < \frac{1}{2} \beta^2 \p{1 + r^2}^3$ i.e. $1 + \omega^2 \ge \frac{1}{2} \beta^2 \p{1 + r^2}^3$.
Moreover, if $\abs{r} \le r_1$, then $1 + \omega^2 \ge \p{1 + r^2}^3 \min_{\abs{r} \le r_1} \frac{1 + \omega^2}{\p{1 + r^2}^3} \ge  \frac{\p{1 + r^2}^3}{\p{1 + r_1^2}^3}$.
Overall, the claimed result holds with $c = \max \big\{\frac{2}{\beta^2}, \p{1 + r_1^2}^3\big\}$.
\end{proof}

Finally, concerning $I_{0, \Gamma_9}^0$, the only change needed in comparison with the estimation of $I_{0, \Gamma_9}^1$ is the additional factors of $\nu_+, \nu_-$. Tracing these factors throughout the estimation, we see that they appear in the form of $\p{\abs{\nu_+} + \abs{\nu_-}}^2$ in the analogue of \eqref{I1Gamma6pretau}. In this regard, we note that
$$
\abs{\nu_\pm} 
\le \frac{1}{2} \Big|r - \frac{\alpha}{\beta}\Big| + \frac{\sqrt 3}{2} \sqrt{\Big(r - \frac{\alpha}{3 \beta}\Big)^2 + \frac{4}{9 \beta^2} \abs{\alpha^2 + 3 \beta \delta}}$$ 
so that, by the triangle inequality and the fact that $r-\frac{\alpha}{3\beta} \leq -R_\delta < 0$ on $\Gamma_9$,
\eqs{
\abs{\nu_+} + \abs{\nu_-} 
&\le \big|r - \frac{\alpha}{3 \beta}\big| + \frac{2 \abs{\alpha}}{3 \beta} + \sqrt 3 \sqrt{\Big(r - \frac{\alpha}{3 \beta}\Big)^2 + \frac{4}{9 \beta^2} \abs{\alpha^2 + 3 \beta \delta}}\\
&= 
\sqrt{\p{\frac{\alpha}{3 \beta} - r + \frac{2 \abs{\alpha}}{3 \beta}}^2} + \sqrt{3 \Big(r - \frac{\alpha}{3 \beta}\Big)^2 + \frac{4}{3 \beta^2} \abs{\alpha^2 + 3 \beta \delta}}
\\
&\leq \sqrt 2 \, \sqrt{\p{\frac{\alpha}{3 \beta} - r + \frac{2 \abs{\alpha}}{3 \beta}}^2 + 3 \Big(r - \frac{\alpha}{3 \beta}\Big)^2 + \frac{4}{3 \beta^2} \abs{\alpha^2 + 3 \beta \delta}}
}
with the last step due to the inequality $\sqrt a + \sqrt b \le \sqrt 2 \sqrt{a + b}$, $a, b \ge 0$. 
From here, we can see that $\p{\abs{\nu_+} + \abs{\nu_-}}^2$ is bounded by a polynomial in $r$ which has a leading term of $8 r^2$. Using the same techniques as in the proof of Lemma \ref{lemma:rtotau}, it follows that
$\p{\abs{\nu_+} + \abs{\nu_-}}^2 \cle 1 + r^2$. 
This has the effect of increasing the exponent of $s$ in \eqref{I1Gamma6pretau} to $s + 1$, thereby resulting in the estimate
\eq{\norm{I_{0, \Gamma_9}^0 (t)}_{H_x^s (0, \ell)} \cle \norm{\psi_0}_{H_x^{\frac{s + 1}{3}} (\R)}.}
Overall, we have established that
\eq{\norm{v_9 (t)}_{H_x^s (0, \ell)} \cle \norm{\psi_0}_{H_t^{\frac{s + 1}{3}} (\R)} + \norm{\psi_1}_{H_t^{\frac{s}{3}} (\R)}, \quad s\in \N_0.\label{v9-hs-est}}

Finally, we note that estimates \eqref{v1-hs-est}, \eqref{v2-hs-est} and \eqref{v9-hs-est} can be extended to all $s\geq 0$ via interpolation. Furthermore, as noted earlier, similar estimates can be derived for the remaining terms in \eqref{v-dec}. Therefore, Theorem \ref{thm:fi-se} has been established. 
\end{proof}

We conclude this section with a smoothing effect which is characteristic of KdV-type equations like HNLS when these are considered on a bounded domain like the finite interval $(0, \ell)$ in this work. First, we prove the following time regularity result:
\begin{theorem}[Time estimates for the reduced interval problem]\label{thm:int-te}
Suppose $\psi_0 \in H^{\frac{s+1}{3}}(\R)$, $\psi_1 \in H^{\frac{s}{3}}(\R)$ and $s\geq -1$. For any $\sigma \in \N_0$ with $\sigma \leq s+1$, the  $\sigma$th spatial derivative of the solution $v(x, t)$ to the reduced interval problem \eqref{HNLSreduced} belongs to $H^{\frac{s+1-\sigma}{3}}(0, T')$ as a function of~$t$. In particular, we have the estimate
\eq{
\norm{\partial_x^\sigma v}_{L_x^\infty((0, \ell); H_t^{\frac{s + 1 - \sigma}{3}} (0, T'))} \le c \max \big\{T' e^{M_\Delta T'}, 1\big\} \Big(\norm{\psi_0}_{H^{\frac{s+1}{3}}(\R)} + \norm{\psi_1}_{H^{\frac{s}{3}}(\R)}\Big), 
\label{v-te}
}
where $c>0$ is a constant that depends only on $s, \alpha, \beta, \delta$, and the constant $M_\Delta$ is given by \eqref{md-def}. 
\end{theorem}

\begin{proof}
As before, it is sufficient to estimate $I_{\sigma, \Gamma_1}^1$, $I_{\sigma, \Gamma_2}^1$, and $I_{\sigma, \Gamma_9}^1$ of \eqref{I1def} because of the symmetry of the contours along with the fact that, for $I_{\sigma, \Gamma_1}^0$, $I_{\sigma, \Gamma_2}^0$, and $I_{\sigma, \Gamma_9}^0$, the extra factors of $\nu_\pm$ behave like factors of $k$ in magnitude, essentially increasing the value of $\sigma$ by 1. Actually, due to the fact that $\Im(\omega)=0$ along $\Gamma_1$ and $\Gamma_9$, and since we are estimating in $t$ rather than in $x$, the terms $I_{\sigma, \Gamma_1}^1$ and $I_{\sigma, \Gamma_9}^1$ can be handled in an identical way. Thus, we only provide the details for $I_{\sigma, \Gamma_1}^1$ and  $I_{\sigma, \Gamma_2}^1$.

Let $m = \frac{s+1-\sigma}{3}$. In view of the parametrization \eqref{gamma1}, the expression \eqref{I1def} for $I_{\sigma, \Gamma_1}^1$ becomes
\eqs{
I_{\sigma, \Gamma_1}^1(x, t)
= \frac{1}{2\pi} \int_{-\infty}^\infty e^{i \omega t}\frac{-(i \gamma_1)^\sigma e^{i \gamma_1 x}}{e^{i \gamma_1 \ell} \Delta(\gamma_1)}
\p{e^{-i \nu_+ \ell} - e^{-i \nu_- \ell}} \mathcal{F} \set{\psi_1} (\omega) \chi_{[\omega_0, \infty)} d\omega.
}
Hence, by the Fourier inversion theorem, 
$$
\mathcal F \set{I_{\sigma, \Gamma_1}^1}(x, \omega) = \frac{-(i \gamma_1)^\sigma e^{i \gamma_1 x}}{e^{i \gamma_1 \ell} \Delta(\gamma_1)}
\p{e^{-i \nu_+ \ell} - e^{-i \nu_- \ell}} \mathcal{F} \set{\psi_1} (\omega) \chi_{[\omega_0, \infty)}
$$
and we infer
\eqs{
\norm{I_{\sigma, \Gamma_1}^1(x)}_{H_t^m (0, T')}
= \norm{(1 + \omega^2)^{\frac{m}{2}} \frac{-(i \gamma_1)^\sigma e^{i \gamma_1 x}}{e^{i \gamma_1 \ell} \Delta(\gamma_1)}
\p{e^{-i \nu_+ \ell} - e^{-i \nu_- \ell}} \mathcal{F} \set{\psi_1} (\omega) \chi_{[\omega_0, \infty)}}_{L_\omega^2 (\R)}.
}
To handle $\gamma_1^\sigma$ in the above norm, we note that, for $r$ sufficiently large, 
$
\abs{\omega(\gamma_1(r))} \ge \beta \abs{\gamma_1}^3 - \abs \alpha \abs{\gamma_1}^2 - \abs \delta \abs{\gamma_1} \ge \frac{1}{2} \beta \abs{\gamma_1}^3$
or, equivalently, $|\gamma_1| \leq \big(\frac 2\beta |\omega|\big)^{\frac 13}$. Thus, since $|\gamma_1|$ is continuous in $r$, we have 
\eq{\label{gamma1-omega}
|\gamma_1| \leq c \, |\omega|^{\frac 13}, 
\quad r \geq r_0,
} 
for some constant $c>0$. 
Thus, using also Lemma \ref{lemma:Deltabound} for $e^{i \gamma_1 \ell} \Delta(\gamma_1)$ and the fact that $\abs{e^{i \gamma_1 x}} = e^{-rx} \leq 1$, we find
\eq{
\norm{I_{\sigma, \Gamma_1}^1(x)}_{H_t^m (0, T')}
\lesssim \norm{(1 + \omega^2)^{\frac{m}{2}} \omega^{\frac{\sigma - 1}{3}}
\mathcal{F} \set{\psi_1} (\omega)}_{L_\omega^2 (\omega_0, \infty)}
\leq 
 \norm{\psi_1}_{H^{\frac{s}{3}} (\R)}
}

Next, for $I_{\sigma, \Gamma_2}^1$, we take $j$ derivatives in time for $j \in \N_0$, parametrize according to \eqref{gamma2}, take magnitudes, and use Lemmas \ref{lemma:charDbar} and \ref{lemma:Deltabound} to get
\eqs{
\abs{\partial_t^j I_{\sigma, \Gamma_2}^1(x, t)}
\cle \int_{\frac{\pi}{2} - \phi_0}^{\frac{\pi}{2} + \phi_0} \frac{|\gamma_2|^\sigma |\omega|^j e^{-\Im(\omega) t}}{\abs{\gamma_2 - \frac{\alpha}{3 \beta}}}
\abs{\frac{d\omega}{d\gamma_2}} \abs{\mathcal{F} \set{\psi_1} (\omega)} d\theta.
}
Noting that $\abs{\frac{d\omega}{d\gamma_2}} \simeq |\gamma_2|^2$, recalling that $\text{supp}(\psi_1)\subset [0, T']$, and using the analogue of  the bound \eqref{gamma1-omega}, we have
\eqs{
\abs{\partial_t^j I_{\sigma, \Gamma_2}^1(x, t)}
&\cle 
\int_{\frac{\pi}{2} - \phi_0}^{\frac{\pi}{2} + \phi_0} |\gamma_2|^{\sigma + 1} |\omega|^j e^{-\Im(\omega) t}
\abs{\int_0^{T'} e^{-i \omega t'} \psi_1 (t') dt'} d\theta\\
&\cle \int_{\frac{\pi}{2} - \phi_0}^{\frac{\pi}{2} + \phi_0} |\omega|^{\frac{\sigma + 1}{3} + j}
\int_0^{T'} e^{\Im(\omega) (t' - t)} \abs{\psi_1 (t')} dt' d\theta\\
&\cle 
\sqrt{T'} M_\Delta^{\frac{\sigma + 1}{3} + j} e^{M_\Delta T'} \norm{\psi_1}_{L^2 (0, T')},
}
where we have also used \eqref{md-def} and the Cauchy-Schwartz inequality.
Based on this bound, we obtain
\eqs{
\norm{I_{\sigma, \Gamma_2}^1(x)}_{H_t^m (0, T')}^2
\le \norm{I_{\sigma, \Gamma_2}^1(x)}_{H_t^{\ceil m} (0, T')}^2
&\lesssim  \sum_{j = 0}^{\ceil m} T'^2 M_\Delta^{\frac{2(\sigma + 1)}{3} + 2j} e^{2M_\Delta T'} \norm{\psi_1}_{L^2 (0, T')}^2
\\
&\le {T'}^2 M_\Delta^{\frac{2(\sigma + 1)}{3}} \frac{M_\Delta^{2 (\ceil m + 1)} - 1}{M_\Delta^2 - 1}
e^{2 M_\Delta T'} \norm{\psi_1}_{H^m (0, T')}^2,
}
concluding the proof.
\end{proof}

As an immediate corollary of Theorem \ref{thm:int-te}, we have the following smoothing estimate for the finite interval problem \eqref{HNLSreduced}:
\begin{corollary}[Smoothing effect]\label{smoothing-fi-c}
The solution $v(x, t)$ to the reduced finite interval problem \eqref{HNLSreduced} satisfies the estimate
\eq{\label{smoothing-fi}
\norm{v}_{L_t^2((0, T'); H_x^{s+1}(0, \ell))} \leq c \max \big\{T' e^{M_\Delta T'}, 1\big\} \sqrt{\ell (s + 2)} \Big(\norm{\psi_0}_{H^{\frac{s + 1}{3}}(\R)} +  \norm{\psi_1}_{H^{\frac{s}{3}}(\R)}\Big), \quad s\geq -1,
}
where $c$ is the constant of estimate \eqref{v-te}.
\end{corollary}

\begin{proof}
For $s\in \Z$ with $s\geq -1$, we have
\eqs{
\norm{v}_{L_t^2((0, T'); H_x^{s+1}(0, \ell))}^2
= \sum_{\sigma=0}^{s+1} \int_0^\ell \norm{\partial_x^\sigma v(x, t)}_{L_t^2 (0, T')}^2 dx
\le \sum_{\sigma=0}^{s+1} \int_0^\ell \norm{\partial_x^\sigma v(x, t)}_{H_t^{\frac{s + 1 - \sigma}{3}} (0, T')}^2 dx
}
so \eqref{smoothing-fi} follows in light of the time estimate \eqref{v-te}. 
The case of general $s\geq -1$ is deduced via  interpolation.
\end{proof}

\subsection{Strichartz estimate}

In the low-regularity setting, in addition to the Sobolev estimate of Theorem~\ref{thm:fi-se} it is necessary to also estimate the linear initial-boundary value problem in the Strichartz-type space $L_t^q \p{(0, T'); H_x^{s, p} (0, \ell)}$, where $H^{s, p} (0, \ell)$ is the restriction on the interval $(0, \ell)$ of the Bessel potential space $H^{s, p} (\R)$ defined via the norm
\eq{\norm{\phi}_{H^{s, p} (\R)} := \norm{\mathcal{F}^{-1} \set{\p{1 + k^2}^{\frac{s}{2}} \mathcal{F}\{\phi\} (k)}}_{L^p (\R)}. \label{Bessel}}
Using once again the unified transform solution formula \eqref{reducedsol}, we will establish the following result:
\begin{theorem}[Strichartz estimate]\label{thm:strichartz}
Suppose $s \ge 0$ and $\psi_0 \in H^{\frac{s+1}{3}}(\mathbb R)$, $\psi_1 \in H^{\frac{s}{3}}(\mathbb R)$ satisfy  the support condition~\eqref{supp-cond}. Then, for any admissible pair $(q, p)$ in the sense of \eqref{adm-pair}, the solution to the reduced finite interval problem \eqref{HNLSreduced}, as given by formula \eqref{reducedsol},  satisfies 
\eq{
\norm{v}_{L_t^q \p{(0, T'); H_x^{s, p} (0, \ell)}} 
\cle 
\b{1 + (T')^{\frac{1}{q} + \frac{1}{2}}} \Big(\norm{\psi_0}_{H^{\frac{s + 1}{3}} (\R)} + \norm{\psi_1}_{H^{\frac{s}{3}} (\R)}\Big).
}
\end{theorem}

\begin{proof}
As before, we express our solution in the form \eqref{v-dec} and focus on the terms $v_1$, $v_2$ and $v_9$ since the rest of the terms can be handled similarly to these three. Instead of the parametrizations \eqref{gamma-param}, it turns out convenient to parametrize $\Gamma_1$, $\Gamma_2$ and $\Gamma_9$ as
\begin{subequations}\label{gamma-param2}
\begin{align}
\gamma_1 (r) &= r + i \sqrt{3 \Big(r - \frac{\alpha}{3\beta}\Big)^2 - \frac{\alpha^2 + 3 \beta \delta}{3 \beta^2}}, \quad r \leq \frac{\alpha}{3 \beta} - R_\Delta \sin \phi_0 =: r_0,\\
\gamma_2 (\theta) &= \frac{\alpha}{3 \beta} + R_\Delta e^{i \theta}, \quad \frac{\pi}{2} - \phi_0 \le \theta \le \frac{\pi}{2} + \phi_0,\\
\gamma_9 (r) &= r, \quad r \le \frac{\alpha}{3 \beta} - R_\Delta,
\end{align}
\end{subequations}
with $R_\Delta$ given by \eqref{rd-def} and the angle $\phi_0$ as shown in Figure \ref{fig:Dbranches}. 
\\[2mm]
\textbf{Estimation of $v_1$.} 
Using our new parametrization for $\Gamma_1$, we have
\eqs{
v_1 (x, t) 
= \frac{1}{2 \pi} \int_{-\infty}^{r_0} e^{i \gamma_1 x + i \omega t} \int_{-\infty}^\infty e^{-iry} \Psi_1(y) dy dr \alignlabel{v1withPsi1}}
where the function $\Psi_1$ is defined through its Fourier transform by
\eq{\mathcal{F} \set{\Psi_1} (r) :=
\begin{cases}
\begin{aligned}
\displaystyle \frac{1}{e^{i \gamma_1 \ell} \Delta(\gamma_1)} 
\Big[&-i \p{e^{-i \nu_+ \ell} - e^{-i \nu_- \ell}} \mathcal{F} \set{\psi_1} (\omega) 
\\
&- \p{\nu_- e^{-i \nu_+ \ell} - \nu_+ e^{-i \nu_- \ell}} \mathcal{F} \set{\psi_0} (\omega)\Big] \omega'(\gamma_1) \gamma_1',
\end{aligned} & r < r_0,\\
0, & r > r_0.
\end{cases}}
Along the lines of \cite{amo2024}, we introduce the kernel
\eq{
\mathcal K(y; x, t) = \int_{-\infty}^{r_0} e^{i\vartheta(r; x, y, t)} \mathcal k(r; x) dr
}
where 
\eq{
\mathcal k(r; x) = e^{-\sqrt{3 \p{r - \frac{\alpha}{3\beta}}^2 - \frac{\alpha^2 + 3 \beta \delta}{3 \beta^2}} x},
\quad 
\vartheta(r; x, y, t) =  r (x-y)+\omega t,
}
 so that \eqref{v1withPsi1} can be expressed in the form  
\eq{
v_1 (x, t) 
= \frac{1}{2 \pi} \int_{-\infty}^\infty \mathcal K(y; x, t)  \Psi_1(y) dy
=:
\b{K_1(t)\Psi_1}(x).
}
Via a duality argument, for any $\eta \in C_c([0, T']; C_x^\infty(0, \ell))$, 
\eq{\label{dual}
\abs{\int_0^{T'} \left\langle K_1(t) \Psi_1, \eta(\cdot, t)\right\rangle_{L_x^2(0, \ell)} dt}
\lesssim
\norm{\Psi_1}_{L^2(\R)} \norm{K_2}_{L_y^2(\R)}
}
where $K_2(y) := \displaystyle \int_0^{T'} \int_0^{\ell} \overline{\mathcal K(y; x, t)} \eta(x, t) dx dt$. 
In fact, we can write the second norm on the right-hand side as
$$
	\norm{K_2}_{L^2(\mathbb{R})}^2
	=
	\int_0^{T'}\int_0^{\ell}\eta(x,t)\bigg(\int_0^{T'}\int_0^{\ell}\overline{\eta(x',t')}K_3(x,x';t,t')dx'dt'\bigg)dxdt
$$ 
with $K_3(x,x';t,t'):=\displaystyle \int_{-\infty}^{\infty} \overline{\mathcal{K}(y;x,t)}\mathcal{K}(y;x',t')dy$, so that by Hölder's inequality in $(x,t)$ and then Minkowski's integral inequality between $x$ and $t'$, 
\begin{equation}\label{k2-l2-0}
\norm{K_2}_{L^2(\mathbb{R})}^2
\leq
\norm{\eta}_{L_t^{q'}((0,T');L_x^{p'}(0, \ell))}
\norm{\int_0^{T'} \norm{
\int_0^{\infty}\overline{\eta(x',t')}K_3(x,x';t,t')dx'
}_{L_x^{p}(0, \ell)} dt'}_{L_t^{q}(0,T')}.
\end{equation}
By the Fourier inversion theorem, 
\begin{align*}
K_3(x,x';t,t')
&=
\int_{-\infty}^{r_0}\mathcal k(r;x) \int_{-\infty}^{\infty}e^{-i\vartheta(r;x,y,t)}\int_{-\infty}^{r_0}e^{i\vartheta(r';x',y,t')}\mathcal k(r';x')dr' dy dr
\nn\\
&=
2\pi \int_{-\infty}^{r_0}\mathcal   e^{-i\vartheta(r; x, x', t-t')}  \mathcal k(r;x+x') dr.
\end{align*}
Hence, by Lemma 3 of \cite{amo2024} (whose proof relies on the classical van der Corput lemma), for $t\neq t'$ we infer that $|K_3(x,x',t,t')|\lesssim |t-t'|^{-\frac 13}$ with inequality constant independent of $x$, $x'$, $t$ and $t'$.
In turn, 
\begin{equation}\label{linf-est}
	\left\|\int_0^{\infty}\overline{\eta(x',t')}K_3(x,x';t,t')dx'\right\|_{L_x^{\infty}(0, \ell)}\lesssim |t-t'|^{-\frac 13}\|\eta(t')\|_{L_x^1(0, \ell)}.
\end{equation}
Furthermore, by the $L^2$ boundedness of the Laplace transform and, more precisely, by Lemma 4 in \cite{amo2024},
\begin{equation}\label{l2-est}
	\left\|\int_0^{\infty}\overline{\eta(x',t')}K_3(x,x';t,t')dx'\right\|_{L_x^{2}(0, \ell)}\lesssim \|\eta(t')\|_{L_x^2(0, \ell)}.
\end{equation}
Estimates \eqref{linf-est} and \eqref{l2-est} combined with the Riesz-Thorin interpolation theorem  imply 
\begin{equation}\label{stric3}
	\left\|\int_0^{\infty}\overline{\eta(x',t')}K_3(x,x';t,t')dx'\right\|_{L_x^{p}(0, \ell)}\lesssim |t-t'|^{-\frac 2q}\|\eta(t')\|_{L_x^{p'}(0, \ell)} 
\end{equation}
for any $p \geq 2$ where, importantly, the interpolation forces $q$ to be given by the admissibility condition \eqref{adm-pair}. Hence,  for any $\eta\in L_t^{q'}((0,T');L_x^{p'}(0, \ell))$, 
\begin{equation*}
	\int_0^{T'} \left\|\int_0^{\infty}\overline{\eta(x',t')}K_3(x,x';t,t')dx'\right\|_{L_x^p(0, \ell)} dt'
	\lesssim 
	\int_0^{T'}|t-t'|^{-\frac{2}{q}}\|\eta(t')\|_{L_x^{p'}(0, \ell)}dt'.
\end{equation*}
Applying the Hardy-Littlewood-Sobolev fractional integration inequality (see Theorem 1 in \cite{s1970}) on the right-hand side and combining the resulting inequality with \eqref{k2-l2-0}, we deduce $\norm{K_2}_{L^2(\R)}\lesssim \norm{\eta}_{L_t^{q'}((0,T');L_x^{p'}(0, \ell))}$ and so in view of \eqref{dual} we conclude that
\eq{\norm{v_1}_{L_t^{q} \p{(0,T'); L_x^p (0, \ell)}} \cle \norm{\Psi_1}_{L^2 (\R)}.}

Differentiating  \eqref{v1withPsi1} $j$ times with respect to $x$ for any $j \le s$ and repeating the above arguments, we find
\eq{\norm{\partial_x^j v_1}_{L_t^{q} \p{(0,T'); L_x^p (0, \ell)}} \cle \norm{\partial_x^j \Psi_1}_{L^2 (\R)} \simeq \norm{\mathcal{F} \set{\partial_x^j \Psi_1}}_{L^2 (\R)} \label{v1est1}}
after also using Plancherel's theorem. 
However, as can be seen by from the arguments presented in Section \ref{reduced-s}, we have already dealt with $\mathcal{F} \set{\partial_x^j \Psi_1}$. More specifically, $\mathcal{F} \set{\partial_x^j \Psi_1}$ is essentially the sum of $I_{j, \Gamma_1}^1$ and $I_{j, \Gamma_1}^0$ defined by~\eqref{I1def} and \eqref{I0def} albeit with a slightly altered parametrization and without the factor of $e^{i \gamma_1 x}$, which was removed by Lemma \ref{lemma:hardy} anyway. Combining this observation with estimates \eqref{I1G1est2}, \eqref{I0G1est2} and \eqref{v1est1}, we obtain
\eq{\norm{\partial_x^j v_1}_{L_t^{q} \p{(0,T'); L_x^p (0, \ell)}} \cle \norm{\psi_0}_{H^{\frac{s + 1}{3}} (\R)} + \norm{\psi_1}_{H^{\frac{s}{3}} (\R)}.}
This estimate combined with the fact that, for $s\in \N_0$, the Bessel potential space $H^{s, p}$ coincides with the Sobolev space $W^{s, p}$ (see \cite{c1961} and also discussion on page 22 of \cite{g2014m}), 
\eq{\norm{v_1}_{L_t^{q} \p{(0,T'); H_x^{s, p} (0, \ell)}} \cle \norm{\psi_0}_{H^{\frac{s + 1}{3}} (\R)} + \norm{\psi_1}_{H^{\frac{s}{3}} (\R)}, \quad s \in \N_0,}
which can be extended to all $s\geq 0$ via interpolation (Theorem 5.1 in \cite{lm1972}).
\\[2mm]
\noindent
\textbf{Estimation of $v_2$.}
A straightforward adaptation of the argument that yields the bounds \eqref{I1G2est1} and \eqref{I0G2est1} gives
\eqs{
\norm{I_{j, \Gamma_2}^1 (t)}_{L_x^p (0, \ell)} \cle \sqrt{T'} \norm{\psi_1}_{H^{\frac{s}{3}} (\R)},
\quad
\norm{I_{j, \Gamma_2}^0 (t)}_{L_x^p (0, \ell)} \cle \sqrt{T'} \norm{\psi_0}_{H^{\frac{s + 1}{3}} (\R)}, \quad s \in \N_0.}
Taking the $L^q (0, T')$ norm in $t$ of these inequalities produces the desired estimate for $s\in \N_0$ while generating a factor of $(T')^{\frac{1}{q}}$. Finally, using interpolation we can cover the entire range $s\geq 0$ and hence conclude that
\eq{\norm{v_2}_{L^{q} \p{(0,T'); H^{s, p} (0, \ell)}} \cle (T')^{\frac 1q + \frac{1}{2}} \Big(\norm{\psi_0}_{H^{\frac{s + 1}{3}} (\R)} + \norm{\psi_1}_{H^{\frac{s}{3}} (\R)}\Big), \quad s \ge 0.}

\noindent
\textbf{Estimation of $v_9$.}
Since $\Gamma_9$ lies on the real axis, $I_{0, \Gamma_9}^1 (x, t)$ makes sense for all $x\in \R$, allowing us to employ the Fourier transform  formulation of the $H^{s, p}(\R)$ norm (which controls the $H^{s, p}(0, \ell)$ norm)  instead of looking at individual derivatives:
\eq{
\norm{I_{0, \Gamma_9}^1 (t)}_{H_x^{s, p} (0, \ell)} &\le \norm{I_{0, \Gamma_9}^1 (t)}_{H_x^{s, p} (\R)} = \norm{\mathcal{F}^{-1} \set{\p{1 + r^2}^{\frac{s}{2}} \mathcal F_x\{I_{0, \Gamma_9}^1\}(r, t)}}_{L_x^p (\R)}.}
Substituting for the Fourier transform via \eqref{I0G91hat} yields 
\eqs{
\norm{I_{0, \Gamma_9}^1 (t)}_{H_x^{s, p} (0, \ell)} &\le \norm{\mathcal{F}^{-1} \set{\p{1 + r^2}^{\frac{s}{2}} \frac{e^{-i r \ell + i \omega t}}{e^{i \nu_- \ell} \Delta(r)} \p{e^{i \p{\nu_- - \nu_+} \ell} - 1} \omega'(r) \mathcal{F} \set{\psi_1} (\omega) \chi_{(-\infty, \frac{\alpha}{3 \beta} - R_\Delta]}(r)}}_{L_x^p (\R)}
\\
&= \norm{\frac{1}{2 \pi} \int_{-\infty}^\infty e^{i r x + i \omega t} \mathcal{F} \set{\phi} (r) dr}_{L_x^p (\R)},}
where the function $\phi(x)$ is defined through its Fourier transform 
\eq{\mathcal{F} \set{\phi} (r) := \p{1 + r^2}^{\frac{s}{2}} \frac{e^{-i r \ell}}{e^{i \nu_- \ell} \Delta(r)}\p{e^{i \p{\nu_- - \nu_+} \ell} - 1} \omega'(r) \mathcal{F} \set{\psi_1} (\omega) \chi_{(-\infty, \frac{\alpha}{3 \beta} - R_\Delta]}(r).}
With this notation, we can write
\eqs{
\norm{I_{0, \Gamma_9}^1 (t)}_{H_x^{s, p} (0, \ell)} &\le \norm{\frac{1}{2 \pi} \int_{-\infty}^\infty e^{i r x + i \omega t} \int_{-\infty}^\infty e^{- i r y} \phi(y) dy dr}_{L_x^p (\R)}\\
&= \norm{\frac{1}{2 \pi} \int_{-\infty}^\infty \phi(y) \int_{-\infty}^\infty e^{i r \p{x - y} + i \omega t} dr dy}_{L_x^p (\R)}
= \norm{\frac{1}{2 \pi} \int_{-\infty}^\infty I(x, y, t) \phi(y) dy}_{L_x^p (\R)},}
where $I(x, y, t) := \int_{-\infty}^\infty e^{i r \p{x - y} + i \omega t} dr$. 
For $t \ne 0$,  from the proof of Lemma 4.2 in \cite{cl2003} we have the dispersive estimate $\abs{I(x, y, t)} \cle \abs{\beta t}^{-\frac{1}{3}}$,  where  the inequality constant is independent of $x$, $y$ and $t$. With this estimate at hand, proceeding along the lines of the proof of Theorem 4.1 in \cite{cl2003} and, importantly, imposing the admissibility condition \eqref{adm-pair} for the pair $(q, p)$, we obtain the bound
$$
\norm{\frac{1}{2 \pi} \int_{-\infty}^\infty I(x, y, t) \phi(y) dy}_{L_t^q \p{(0, T'); L_x^p (\R)}} \cle \norm{\phi}_{L_x^2 (\R)} \simeq \norm{\mathcal F\{\phi\}}_{L_r^2 (\R)}.
$$
As we have already seen from \eqref{I0G91def} onward how to handle $\norm{\phi}_{L_x^2 (\R)}$, we are able to deduce the desired estimate for $I_{0, \Gamma_9}^1$. In addition, the same proof can be adapted in a straightforward way for  $I_{0, \Gamma_1}^0$, thereby completing the proof of Theorem \ref{thm:strichartz}.
\end{proof}

\section{Linear decomposition and new estimates on the half-line}\label{dec-s}

We will now use the Sobolev and Strichartz estimates derived in Section \ref{reduced-s} for the reduced interval problem~\eqref{HNLSreduced} in order to establish analogous estimates for the full forced linear interval problem \eqref{hls-ibvp}. More specifically, we will decompose problem \eqref{hls-ibvp} in several components, which can be handled either via our results from Section \ref{reduced-s} or through novel estimates obtained in the present section for linear half-line problem. It should be emphasized that these new half-line estimates were not necessary for the well-posedness of HNLS on the half-line proved in~\cite{amo2024}, since they arise through the decomposition of the interval problem \eqref{hls-ibvp} below via the  traces of the half-line solution that interact with the boundary data of problem \eqref{hls-ibvp}.

\subsection{Proof of Theorem \ref{lin-est-t} in the high-regularity setting}\label{lin-dec-ss}

Let $s>\frac 12$. By linearity, the solution of the forced linear problem \eqref{hls-ibvp} can be written as
\eq{u = U|_{x \in (0, \ell)} + \breve u \label{U-decomp1},}
where $U$ denotes the solution to the half-line problem
\eq{\label{halfline}
\begin{aligned}
&i U_t + i \beta U_{xxx} + \alpha U_{xx} + i \delta U_x = F(x, t), \quad 0 < x < \infty, \  0 < t < T,
\\
&U(x, 0) = U_0 (x),  
\\
&U(0, t) = g_0(t),  
\end{aligned}
}
and $\breve u$ satisfies the finite interval problem
\eq{\label{reduced}
\begin{aligned}
&i \breve u_t + i \beta \breve u_{xxx} + \alpha \breve u_{xx} + i \delta \breve u_x = 0, \quad 0 < x < \ell, \  0 < t < T,\\
&\breve u(x, 0) = 0,\\
&\breve u(0, t) = 0, 
\quad
\breve u(\ell, t) = h_0(t) - U(\ell, t), 
\quad
\breve u_x (\ell, t) = h_1(t) - U_x(\ell, t), 
\end{aligned}
}
where the initial datum $U_0$ and the forcing $F$ of the half-line problem are given by 
$$
U_0 = \mathcal U_0 \big|_{(0, \infty)}, \quad
F = \mathfrak F \big|_{(0, \infty) \times [0, T]}
$$ 
with $\mathcal U_0 \in H^s (\R)$ and $\mathfrak F \in C([0, T];  H_x^s (\R))$ being Sobolev extensions from $(0,\ell)$ to $\mathbb R$ of the original initial datum $u_0 \in H^s (0, \ell)$ and forcing $f \in C([0, T];  H_x^s (0, \ell))$, respectively, such that
\begin{align}
&\norm{U_0}_{H^s(0, \infty)} \leq \norm{\mathcal U_0}_{H^s (\mathbb R)} \lesssim \norm{u_0}_{H^s (0, \ell)}, \label{u0-party}\\
&\norm{F(t)}_{H_x^s(0, \infty)} \leq \norm{\mathfrak F(t)}_{H_x^s (\R)} \lesssim \norm{f(t)}_{H_x^s(0, \ell)}, \quad t\in (0, T). \label{f-party}
\end{align}
We note that the constant in inequalities \eqref{u0-party} and \eqref{f-party} is the same, as it only depends on $s$ due to the existence of a fixed bounded extension operator from $H^s(0, \ell)$ to $H^s(\R)$. 

Introducing the notation $u  = S[u_0, g_0, h_0, h_1; f]$, $U = S[U_0, g_0; F]$ and $\mathcal U = S[\mathcal U_0; \mathfrak F]$ 
for the solution operators of the interval problem \eqref{hls-ibvp}, the half-line problem \eqref{halfline} and the Cauchy problem
\eq{\label{wholeline}
\begin{aligned}
&i \mathcal U_t + i \beta \mathcal U_{xxx} + \alpha \mathcal U_{xx} + i \delta \mathcal U_x = \mathfrak F(x, t), \quad x \in \R, \ 0 < t < T,
\\
&\mathcal U(x, 0) = \mathcal U_0 (x), 
\end{aligned}
}
we further decompose \eqref{U-decomp1} by expressing $U$ and $\breve u$ as
\eq{\label{U-decomp2}
U = \mathcal W|_{x \in (0, \infty)} + \mathcal Z|_{x \in (0, \infty)} + \breve U_1 - \breve U_2, 
\quad
\breve u = \breve u_1 - \breve u_2 + \breve u_3,
}
where, using the solution operators notation above,
\begin{equation}\label{comps-def}
\begin{aligned}
&\mathcal W = S\big[\mathcal U_0; 0\big], \quad 
\mathcal Z = S\big[0; \mathfrak F\big],
\quad
\breve U_1 = S[0, g_0 - \mathcal W|_{x=0}; 0],
\quad
\breve U_2 = S\big[0, \mathcal Z|_{x=0}; 0\big],
\\
&\breve u_1 = S\big[0, 0, h_0 - \mathcal W|_{x=\ell} - \breve U_1|_{x=\ell}, h_1 - \partial_x \mathcal W|_{x=\ell} - \partial_x \breve U_1|_{x=\ell}; 0\big],
\\
&\breve u_2 = S\big[0, 0, \mathcal Z|_{x=\ell}, \partial_x \mathcal Z|_{x=\ell}; 0\big],
\quad
\breve u_3 = S\big[0, 0, \breve U_2|_{x=\ell}, \partial_x \breve U_2|_{x=\ell}; 0\big].
\end{aligned}
\end{equation}
Hence, for each $t\in (0, T)$, the decomposition \eqref{U-decomp1} yields
\begin{equation}\label{Udecomp-bound}
\begin{aligned}
\norm{u(t)}_{H_x^s(0, \ell)} &\leq \norm{\mathcal W(t)}_{H_x^s(\R)} + \norm{\mathcal Z(t)}_{H_x^s(\R)} + \big\|\breve U_1 (t)\big\|_{H_x^s(0, \infty)} + \big\|\breve U_2 (t)\big\|_{H_x^s(0, \infty)}
\\
&\quad
+ \norm{\breve u_1 (t)}_{H_x^s(0, \ell)}+ \norm{\breve u_2 (t)}_{H_x^s(0, \ell)}+ \norm{\breve u_3 (t)}_{H_x^s(0, \ell)}. 
\end{aligned}
\end{equation}
and so estimating $u$ in $C([0, T]; H_x^s (0, \ell))$ amounts to estimating each of the norms on the right-hand side of \eqref{Udecomp-bound} and then taking the supremum over $t\in [0, T]$.

At this point, it is useful to state the following results from the analysis of the half-line problem in \cite{amo2024}:
\begin{theorem}[\cite{amo2024}]\label{hl-t1}
The following estimates are satisfied by the Cauchy problems in \eqref{comps-def}:
\begin{align}
&\norm{\mathcal W(t)}_{H_x^s (\R)} = \norm{\mathcal U_0}_{H^s (\R)}, \quad s, t \in \R,\label{amo8}
\\
&\norm{\mathcal W(x)}_{H_t^{\frac{s+1}{3}} (0, T)} \leq c_s \big(1+\sqrt T\big) \norm{\mathcal U_0}_{H^s (\R)}, \quad s, x \in \R,\label{amo910}
\\
&\norm{\mathcal Z(t)}_{H_x^s (\R)} \le \norm{\mathfrak F}_{L_t^1 ((0, T); H_x^s (\R))}, \quad s \in \R, \ t \in [0, T] \label{amo25}
\\
&\norm{\mathcal Z(x)}_{H_t^{\frac{s+1}{3}} (0, T)} \le c(s, T) \norm{\mathfrak F}_{L_t^2 ((0, T); H_x^s (\R))}, \quad -1 \leq s \leq 2, \ s \neq \frac 12, \ x \in \R, \label{amo26}
\end{align}
where $c_s$ is a constant that depends only upon $\alpha, \beta, \delta, s$ and $c(s, T)$ remains bounded as $T\to 0^+$. 
Furthermore, the solution $V$ to the reduced half-line problem
\eq{\label{HNLSreducedHL}
\begin{aligned}
&i V_t + i \beta V_{xxx} + \alpha V_{xx} + i \delta V_x = 0, \quad x > 0, \ 0 < t < T',
\\
&V(x, 0) = 0, 
\\
&V(0, t) = V_0(t) \in H^{\frac{s + 1}{3}} (\R), \quad \textnormal{supp}(V_0) \subset [0, T'],
\end{aligned}
}
satisfies the estimate
\eq{\label{amo-thm7}
\norm{V}_{L_t^\infty((0, T'); H_x^s (0, \infty))} \le c_s \big(1 + \sqrt{T'} e^{c T'}\big) \norm{V_0}_{H^{\frac{s + 1}{3}}(\R)}, \quad s \ge 0.
}
\end{theorem}

\begin{corollary}[Smoothing effect]
For any $s\geq -1$, the Cauchy problem estimates \eqref{amo910} and \eqref{amo26} readily imply
\eq{
&\norm{\mathcal W}_{L_t^2((0, T); H_x^{s+1}(0, \ell))} \leq c_s \big(1+\sqrt T\big) \sqrt{\ell(s+2)} \norm{\mathcal U_0}_{H^s (\R)},  
\label{smoothing-hc}
\\
&\norm{\mathcal Z}_{L_t^2((0, T); H_x^{s+1}(0, \ell))} \leq c_s \big(1+\sqrt T\big) \sqrt{\ell(s+2)} \norm{\mathfrak F}_{L_t^1 ((0, T); H_x^s (\R))}. 
\label{smoothing-fc}
}
\end{corollary}

\begin{proof}
For $s\in \Z$ with $s\geq -1$, using the time estimate \eqref{amo910} we have
\eqs{
\norm{\mathcal W}_{L_t^2((0, T); H_x^{s+1}(0, \ell))}^2
&=
\sum_{j=0}^{s+1} \int_0^\ell \norm{S\big[\partial_x^j \mathcal U_0; 0\big](x)}_{L_t^2(0, T)}^2 dx
\\
&\leq
\sum_{j=0}^{s+1} \int_0^\ell c_s^2 (1+\sqrt T)^2 \norm{\partial_x^j \mathcal U_0}_{H^{-1}(\R)}^2 dx
\leq
c_s^2 (1+\sqrt T)^2 \ell (s+2)  \norm{\mathcal U_0}_{H^s(\R)}^2. 
}
By interpolation, we can extend the estimate to cover all $s\geq -1$. 
Similarly, for $s\in \Z$ with $s\geq -1$, we have
\eqs{
\norm{\mathcal Z}_{L_t^2((0, T); H_x^{s+1}(0, \ell))}^2
&=
\sum_{j=0}^{s+1} \int_0^\ell \norm{S\big[0; \partial_x^j \mathfrak F\big](x, t)}_{L_t^2(0, T)}^2 dx
\\
&=
\sum_{j=0}^{s+1} \int_0^\ell \norm{\int_0^t S\big[\partial_x^j \mathfrak F(t');0\big](x, t-t')dt'}_{L_t^2(0, T)}^2 dx
\\
&\leq
\sum_{j=0}^{s+1} \int_0^\ell \left(\int_0^T \norm{S\big[\partial_x^j \mathfrak F(t');0\big](x, \cdot-t')}_{L_t^2(0, T)} dt'\right)^2 dx
}
so employing the time estimate \eqref{amo910} we obtain
\eqs{
\norm{\mathcal Z}_{L_t^2((0, T); H_x^{s+1}(0, \ell))}^2
&\leq
\sum_{j=0}^{s+1} \int_0^\ell \left(\int_0^T c_s(1+\sqrt T) \norm{\partial_x^j \mathfrak F(t')}_{H_x^{-1}(\R)} dt'\right)^2 dx
\\
&=
c_s^2 (1+\sqrt T)^2 \ell \sum_{j=0}^{s+1} \norm{\mathfrak F(t')}_{L_t^1((0, T); H_x^{j-1}(\R))}^2
\\
&\leq
c_s^2 (1+\sqrt T)^2 \ell (s+2) \norm{\mathfrak F(t')}_{L_t^1((0, T); H_x^s(\R))}^2.
}
As before, for general $s\geq -1$ we proceed via interpolation.
\end{proof}

Note that the problems satisfied by $\breve U_1$ and $\breve U_2$ are essentially the same with the one satisfied by $V$, except for the important difference that the support condition of $V_0$ is not satisfied by the boundary data $g_0 - \mathcal W|_{x=0}$ and $\mathcal Z|_{x=0}$ of $\breve U_1$ and $\breve U_2$. For this reason, in order to be able to employ estimate \eqref{amo-thm7} for $\breve U_1$ and $\breve U_2$, an additional step is first necessary, namely the construction of appropriate extensions for the boundary data of $\breve U_1$ and $\breve U_2$.
\\[2mm]
\noindent
\textbf{Boundary data extension.} We begin with $\breve U_1$. We are in the high-regularity setting of $s > \frac 12$ and, in particular, we will work with $\frac{1}{2} < s < \frac{7}{2}$. Since $u_0 \in H^s(0, \ell)$ and $g_0 \in H^{\frac{s+1}{3}}(0, T)$, in our range of $s$ we have continuity of both the initial and the boundary data, which is the reason why the first of the compatibility conditions \eqref{comp-cond} is required, namely $g_0 (0) = u_0 (0)$. This implies $g_0(0) = \mathcal W(0, 0)$ so the function
\eq{\label{g-g0-w0}
g(t) := g_0(t) - \mathcal W(0, t), \quad 0<t<T,
}
satisfies $g(0) = 0$. Furthermore, thanks to the time estimate \eqref{amo910}, $g \in H^{\frac{s+1}{3}}(0, T)$ and, using also the extension inequality \eqref{u0-party}, 
\eq{\label{g-g0-u0}
\norm{g}_{H^{\frac{s+1}{3}}(0, T)} 
\leq
\norm{g_0}_{H^{\frac{s+1}{3}}(0, T)}
+
c_s \big(1+\sqrt T\big) \norm{u_0}_{H^s (0, \ell)},
\quad
s\in\mathbb R.
}
Next, let $G$ be an extension of $g$ such that
\eq{\norm{G}_{H^{\frac{s + 1}{3}} (\R)} \le 2 \norm{g}_{H^{\frac{s + 1}{3}} (0, T)} \label{high-s-extension}}
and, for $\theta \in C_c^\infty (\R)$ a smooth cutoff function such that $0\leq \theta(t) \leq 1$ for all $t\in\R$, $\theta(t) = 1$ on $[0, T]$ and $\supp (\theta) \subset [-(T + 1), T + 1]$, define $G_\theta (t) = \theta(t) G(t)$, $ t\in\R$.
Note that $\supp (G_\theta) \subset [-(T + 1), T + 1]$ and, in particular, $G_\theta(T+1) = 0$. Furthermore, since $g(0) = 0$, it follows that $G_\theta(0) = 0$. Hence, the function 
\eq{\label{G0-def}
G_0(t) := \chi_{[0, T+1]}(t) \, G_\theta(t), \quad t\in\R,
}
has $\supp (G_0) \subset [0, T + 1]$ and is continuous at $t=0$ and $t=T+1$ with $G_0(0) = G_0(T+1) = 0$.  Therefore, by Theorem 11.4 of~\cite{lm1972}, $G_0 \in H^{\frac{s+1}{3}}(\R)$ with 
\eq{
\norm{G_0}_{H^{\frac{s+1}{3}}(\R)} \leq C(s, T) \norm{G_\theta}_{H^{\frac{s+1}{3}}(0, T+1)}.
}
In addition, by the algebra property (which is valid since $s>\frac 12$) and the inequality \eqref{high-s-extension}, 
\eq{
\norm{G_\theta}_{H^{\frac{s+1}{3}}(0, T+1)}
\leq 
\norm{\theta}_{H^{\frac{s + 1}{3}} (0, T+1)} \norm{G}_{H^{\frac{s + 1}{3}} (\R)}
\leq
\tilde C(s, T) \norm{g}_{H^{\frac{s + 1}{3}} (0, T)},
}
where $\tilde C(s, T) := 2 \norm{\theta}_{H^{\frac{s + 1}{3}} (0, T+1)}$ depends on $s$ and $T$ and remains bounded as $T\to 0^+$. Therefore,
\eq{\norm{G_0}_{H^{\frac{s + 1}{3}} (\R)} \le c(s, T) \norm{g}_{H^{\frac{s + 1}{3}} (0,T)} \label{G0bound}}
where the constant $c(s, T)$ remains bounded as $T\to 0^+$. 

\vspace*{2mm}

With the above construction at hand, combining \eqref{g-g0-u0} and \eqref{G0bound} with \eqref{amo-thm7} in the case of $V_0=G_0$ and $T' = T + 1$, for $\frac{1}{2} < s < \frac{7}{2}$ we  obtain
\eq{\label{U1b-se}
\big\|\breve U_1\big\|_{L_t^\infty((0, T); H_x^s (0, \infty))} 
\leq
c(s, T) \Big(\norm{u_0}_{H^s(0,\ell)} + \norm{g_0}_{H^{\frac{s + 1}{3}} (0, T)}\Big),}
where $c(s, T)$ is a constant (different from the one in \eqref{G0bound}) that  remains bounded as $T\to 0^+$.
Estimate \eqref{U1b-se} will be used to bound the corresponding term in \eqref{Udecomp-bound}. 

The norm of $\breve U_2$  in \eqref{Udecomp-bound} can be handled via the analogue of \eqref{U1b-se} obtained for $g(t) = \mathcal Z(0, t)$ (instead of \eqref{g-g0-w0}). Indeed, for $-1\leq s \leq 2$ and $s\neq \frac 12$, thanks to the time estimate \eqref{amo26} we have $\mathcal Z|_{x=0} \in H^{\frac{s+1}{3}}(0, T)$. Moreover,  $\mathcal Z(0, 0) = 0$. Thus, $\mathcal Z(0, t)$ has the same properties with the function in \eqref{g-g0-w0} and so, proceeding as before and using successively \eqref{amo-thm7}, \eqref{amo26} and \eqref{f-party}, for $0 \leq s \le 2$ with $s \neq \frac{1}{2}$ we find
\eq{\label{U2b-se}
\big\|\breve U_2\big\|_{L_t^\infty((0, T); H_x^s (0, \infty))} 
\leq c(s, T)  \norm{f}_{L_t^2((0, T); H_x^s(0, \ell))}.
}

The three remaining terms in \eqref{Udecomp-bound}, namely the norms of the interval solutions $\breve u_1$, $\breve u_2$, $\breve u_3$, can be handled similarly to each other. Starting with $\breve u_1$, we extend our Dirichlet boundary datum at $x=\ell$ from $H^{\frac{s+1}{3}}(0, T)$ to $H^{\frac{s+1}{3}}(\R)$ using the same methodology described above for the Dirichlet datum of $\breve U_1$. This step is possible thanks to the second of the compatibility conditions \eqref{comp-cond}, namely $h_0(0) = u_0 (\ell)$. In addition, for $s < \frac{3}{2}$, the Neumann datum of $\breve u_1$ can simply be extended by zero from $H^{\frac s3}(0, T)$ to $H^{\frac s3}(\R)$. On the other hand, for $s>\frac 32$ continuity is required, leading to the third of the compatibility conditions \eqref{comp-cond}, i.e. $h_1 (0) =  u_0' (\ell)$.\footnote{Of course, additional compatibility conditions are needed as the value of $s$ increases, the next one being on the Dirichlet datum when $s > \frac{7}{2}$. However, these conditions are not relevant in the range of $0\leq s \leq 2$ considered in this work.}
Then, we can employ Theorem~\ref{thm:fi-se} with $T' = T + 1$ and with $\psi_0$ and $\psi_1$ being the extensions of $h_0 - \mathcal W|_{x=\ell} - \breve U_1|_{x=\ell}$ and  $h_1 - \partial_x \mathcal W|_{x=\ell} - \partial_x \breve U_1|_{x=\ell}$ constructed similarly to \eqref{G0-def} and hence satisfying 
\eq{\label{psi0-psi1-bound}
\begin{split}
&\norm{\psi_0}_{H^{\frac{s + 1}{3}} (\R)} 
\le 
c(s, T) 
\norm{h_0 - \mathcal W|_{x=\ell} - \breve U_1|_{x=\ell}}_{H^{\frac{s+1}{3}} (0,T)}, 
\\
&\norm{\psi_1}_{H^{\frac{s}{3}} (\R)} 
\le 
c(s, T) 
\norm{h_1 - \partial_x \mathcal W|_{x=\ell} - \partial_x \breve U_1|_{x=\ell}}_{H^{\frac{s}{3}} (0,T)},
\end{split}
\quad
\frac 12 < s < \frac 72, \ s \neq \frac 32,
}
to deduce, for the above specified range of $s$, 
\eq{\label{ub1-bound0}
\begin{split}
\norm{\breve u_1}_{L_t^\infty((0, T); H_x^s (0, \ell)} \leq
c_s \sqrt{T+1} e^{M_\Delta (T+1)}  
\Big(
&\norm{h_0}_{H^{\frac{s + 1}{3}} (0, T)} + \norm{\mathcal W(\ell)}_{H_t^{\frac{s + 1}{3}} (0, T)} + \big\|\breve U_1 (\ell)\big\|_{H_t^{\frac{s + 1}{3}} (0, T)}\\
&+ \norm{h_1}_{H^{\frac{s}{3}} (0, T)} + \norm{\partial_x \mathcal W(\ell)}_{H_t^{\frac{s}{3}} (0, T)} + \big\|\partial_x \breve U_1 (\ell)\big\|_{H_t^{\frac{s}{3}} (0, T)}
\Big).
\end{split}
}
The norms of $\mathcal W(\ell, t)$ and $\partial_x \mathcal W(\ell, t)$ on the right-hand side of \eqref{ub1-bound0} can be handled via the time estimate \eqref{amo910}, which as noted earlier was derived in \cite{amo2024} for the well-posedness of HNLS on the half-line. On the other hand, the corresponding norms of $\breve U_1(\ell, t)$ and $\partial_x \breve U_1(\ell, t)$ were not estimated in \cite{amo2024}, as that was not necessary there. In our case, however, in order to use the bound \eqref{ub1-bound0} we must study the time regularity of the half-line solution $\breve U_1$ and its spatial derivative $\partial_x \breve U_1$. This is done in Theorem \ref{thm:hl-te} below.  Using the corresponding estimate \eqref{V-te} for $\sigma = 0, 1$, as well as inequality \eqref{G0bound}, we obtain 
\eq{\label{ub1-bound1}
\begin{split}
\norm{\breve u_1}_{L_t^\infty((0, T); H_x^s (0, \ell))} 
\leq
c_s \sqrt{T+1} e^{M_\Delta (T+1)}  
\Big[
&c_s \big(1+\sqrt T\big) \norm{\mathcal U_0}_{H^s(\R)} 
+ \norm{h_0}_{H^{\frac{s + 1}{3}} (0, T)} + \norm{h_1}_{H^{\frac{s}{3}} (0, T)} 
\\
&
+ c_s \max \set{(T+1) e^{M_\Delta(T+1)}, 1} c(s, T) \norm{g}_{H^{\frac{s + 1}{3}} (0,T)}
\Big].
\end{split}
}
Then, in view of \eqref{u0-party} and \eqref{g-g0-u0}, for $\frac{1}{2} < s < \frac{7}{2}$ with $s \ne \frac{3}{2}$ we deduce
\eq{\label{ub1-bound2}
\norm{\breve u_1}_{L_t^\infty((0, T); H_x^s (0, \ell))} 
 \leq
c(s, T) 
\Big(\norm{u_0}_{H_x^s (0, \ell)} + \norm{g_0}_{H_t^{\frac{s + 1}{3}} (0, T)} + \norm{h_0}_{H_t^{\frac{s + 1}{3}} (0, T)} + \norm{h_1}_{H_t^{\frac{s}{3}} (0, T)}\Big), 
}
where the constant $c(s, T)$ remains bounded as $T\to 0^+$. 

For $\breve u_2$ and $\breve u_3$, the procedure is identical  except that  we now use \eqref{amo26}  and \eqref{f-party} to obtain
\eq{\label{ub2-bound2}
\norm{\breve u_2}_{L_t^\infty((0, T); H_x^s (0, \ell))} 
+
\norm{\breve u_3}_{L_t^\infty((0, T); H_x^s (0, \ell))} 
 \leq 
 c(s, T)  \norm{f}_{L_t^2 ( (0, T); H_x^s (\R))}, 
}
where now $\frac{1}{2} < s \leq 2$, $s \ne \frac{3}{2}$ (this restriction is the intersection of the restrictions $-1\leq s \leq 2$, $s \ne \frac{1}{2}$ and $0\leq s \leq 3$, $s \ne \frac{3}{2}$ that are needed for using \eqref{amo26} for $\mathcal Z$ and $\partial_x \mathcal Z$ in the high-regularity setting $s>\frac 12$) and once again the constant $c(s, T)$ is bounded  as $T\to 0^+$.

Overall, returning to the decomposition inequality \eqref{Udecomp-bound} and combining the estimates \eqref{amo8}, \eqref{amo25}, \eqref{U1b-se}, \eqref{U2b-se}, \eqref{ub1-bound2} and \eqref{ub2-bound2} with the inequalities \eqref{u0-party} and \eqref{f-party}, for all $\frac 12 < s \leq 2$ with $s\neq \frac 32$ we obtain the Sobolev linear estimate \eqref{sob-est}.
In addition, starting from the analogue of the decomposition \eqref{Udecomp-bound} for the space $L_t^2((0, T); H_x^{s+1}(0, \ell))$ and  employing the smoothing estimates \eqref{smoothing-hc}, \eqref{smoothing-fc}, \eqref{smoothing-hl}, \eqref{smoothing-fi}, the extension inequalities~\eqref{u0-party} and \eqref{f-party} for the initial data and the forcing, inequality \eqref{g-g0-u0}  and the extension inequality~\eqref{G0bound} for the boundary datum $g_0(t)$, the analogue of the latter inequality for $\mathcal Z(0, t)$ instead of $g_0(t)$,  and the extension inequalities \eqref{psi0-psi1-bound} together with the time estimates \eqref{amo910}, \eqref{amo26} and \eqref{V-te}, we obtain the smoothing estimate~\eqref{smooth-est} in high-regularity setting of $\frac 12 < s \leq 2$ with $s\neq \frac 32$.
As the Strichartz estimate \eqref{strich-est} is only valid in the low-regularity setting of $0\leq s<\frac 12$, the proof of Theorem \ref{lin-est-t} in the high-regularity setting is complete.

\subsection{Proof of Theorem \ref{lin-est-t} in the low-regularity setting}

For $s<\frac 12$, in addition to the Sobolev estimate \eqref{sob-est} we must establish the linear estimate \eqref{strich-est} in the family of  Strichartz spaces $L_t^q([0, T]; H_x^{s, p} (0, \ell))$ with $(q, p)$ satisfying the admissibility condition \eqref{adm-pair}. The decompositions~\eqref{U-decomp1} and \eqref{U-decomp2} yield the following low-regularity counterpart of \eqref{Udecomp-bound}:
\begin{align}
\norm{u}_{L_t^q((0, T); H_x^{s, p}(0, \ell))} &\leq \norm{\mathcal W}_{L_t^q((0, T); H_x^{s, p}(\R))} + \norm{\mathcal Z}_{L_t^q((0, T); H_x^{s, p}(\R))} + \big\|\breve U_1\big\|_{L_t^q((0, T); H_x^{s, p}(0, \infty))} + \big\|\breve U_2\big\|_{L_t^q((0, T); H_x^{s, p}(0, \infty))} \nonumber\\
&\quad 
+ \norm{\breve u_1 }_{L_t^q((0, T); H_x^{s, p}(0, \ell))}+ \norm{\breve u_2 }_{L_t^q((0, T); H_x^{s, p}(0, \ell))}+ \norm{\breve u_3 }_{L_t^q((0, T); H_x^{s, p}(0, \ell))}. \label{Udecomp-bound-Strich}
\end{align}
As in the case of high regularity, it is useful to recall the following results from the analysis of the HNLS half-line problem in \cite{amo2024}. 
\begin{theorem}[\cite{amo2024}]
For any admissible pair $(q, p)$ in the sense of \eqref{adm-pair}, the Cauchy problems in \eqref{comps-def} admit the estimates
\begin{align}
&\norm{\mathcal W}_{L_t^q ((0, T); H_x^{s, p} (\R))} \cle \norm{\mathcal U_0}_{H^s (\R)}, \quad s \in \R, \label{amo2.17}
\\
&\norm{\mathcal Z}_{L_t^q ((0, T); H_x^{s, p} (\R))} \cle \norm{\mathfrak F}_{L_t^1 ((0, T); H_x^s (\R))}, \quad s \in \R. \label{amo2.38}
\end{align}
Moreover, the solution to the reduced half-line problem \eqref{HNLSreducedHL} satisfies the estimate
\begin{equation}
\norm{V}_{L_t^q ((0, T'); H_x^{s, p} (0, \infty))} \cle \big[1 + \p{T'}^{\frac 1q + \frac 12}\big] \norm{V_0}_{H^{\frac{s + 1}{3}} (\R)}, \quad s \ge 0. \label{amo2.87}
\end{equation}
\end{theorem}

Estimates \eqref{amo2.17} and \eqref{amo2.38} can be directly employed for handling the first two terms in \eqref{Udecomp-bound-Strich}. The remaining five terms are associated with the half-line solutions $\breve U_1, \breve U_2$ and the finite interval solutions $\breve u_1, \breve u_2, \breve u_3$.  We begin with the estimation of the half-line solutions and, in particular,  with $\breve U_1 = S[0, g; 0]$ where
$$
g(t) = g_0(t) - \mathcal W(0, t)
$$
in line with \eqref{comps-def}. Recall that $g \in H^{\frac{s+1}{3}}(0, T)$ with estimate \eqref{g-g0-u0} in place. Hence, for $-1\leq s < \frac 12$ or, equivalently, $0 \le \frac{s + 1}{3} < \frac{1}{2}$,  Theorem 11.4 of~\cite{lm1972} implies that the extension $G_0$ of $g$ by zero outside the interval $[0, T]$ belongs to $H^{\frac{s + 1}{3}} (\R)$ with the estimate
\eq{\norm{G_0}_{H^{\frac{s + 1}{3}} (\R)} \leq c(s, T) \norm{g}_{H^{\frac{s + 1}{3}} (0, T)}.\label{low-s-extension}}
This fact allows us to identify $\breve U_1$  as the solution $V$ of the reduced half-line problem \eqref{HNLSreducedHL} with $T'=T$ and $V_0 = G_0$. Hence, we can employ estimate \eqref{amo2.87} and then use successively inequalities \eqref{low-s-extension} and \eqref{g-g0-u0} to infer
\eq{\label{U1breve-est}
\big\|\breve U_1\big\|_{L_t^q ((0, T); H_x^{s, p} (0, \infty))} \leq
c(s, T) \Big( \norm{u_0}_{H^s(0,\ell)} + \norm{g_0}_{H^{\frac{s + 1}{3}} (0, T)}\Big), \quad 0\le s < \frac{1}{2},}
where $(q, p)$ is any pair satisfying \eqref{adm-pair} and the constant $c(s, T)$ remains bounded as $T\to 0^+$. 

We proceed to $\breve U_2=S[0, \mathcal Z(0, \cdot); 0]$.  For this term,  simply repeating the steps used for $\breve U_1$ above would yield the  forcing in the norm of the space $L_t^2((0, T); H_x^s(0, \ell))$. However, via Lemma \ref{lemma:alt-breveU2} below (see also expression (6.11) in~\cite{hm2020}), it is possible to make the forcing appear in the norm of the space $L_t^1((0, T); H_x^s(0, \ell))$ instead, thereby allowing us to establish the low-regularity well-posedness of Theorem \ref{low-wp-t} for a broader range of nonlinearities.
\begin{lemma} \label{lemma:alt-breveU2}
The solution $\breve U_2 = S\big[0, \mathcal Z|_{x=0}; 0\big]$ can  be expressed in the form
\eq{\label{alt-breveU2}
\breve U_2 (x, t) = -i \int_{t' = 0}^T S \b{0, S[\mathfrak F(\cdot, t'); 0] (0, \cdot); 0} (x, t - t') dt'.}
\end{lemma}

\begin{proof}
The expression (65) in \cite{amo2024} provides the unified transform solution to the reduced half-line problem~\eqref{HNLSreducedHL}. Combining the specific form of that solution with the Duhamel's principle and the homogeneous Cauchy problem solution, we arrive at the formula
\eq{\breve U_2 (x, t) = \frac{i}{2 \pi} \int_{\partial \tilde D_0} e^{i k x + i \omega t} \omega' \int_{t'' = 0}^T e^{-i \omega t''} \int_{t' = 0}^{t''} S[\mathfrak F(\cdot, t'); 0] (0, t'' - t') dt' dt'' dk,}
where $\tilde D_0$ and  $\omega$ are defined by \eqref{dntil-def} and \eqref{omega}, respectively. 
Interchanging the integrals with respect to $t'$ and $t''$ and then making the change of variable $t'' \mapsto t'' + t'$ yields
\eqs{
\breve U_2 (x, t) 
= -\frac{1}{2 \pi}\int_{t' = 0}^T  \int_{\partial \tilde D_0} e^{i k x + i \omega (t - t')} \left(-i \omega'\right) \int_{t'' = 0}^{T - t'} e^{-i \omega t''} S[\mathfrak F(\cdot, t'); 0] (0, t'') dt'' dk dt'.}
Recalling that $\Im(k) > 0$ and $\Im(\omega) \le 0$ for $k \in \overline{\tilde D_0}$, we observe that for $t'' \ge T - t' \ge t - t'$ the exponential $e^{i k x + i \omega (t - t' - t'')}$ decays as $|k| \to \infty$ within $\overline{\tilde D_0}$. Therefore, along the lines of the argument used in Appendix \ref{app-s},  we can use analyticity and Cauchy's theorem to augment the integration range of the  inner integral  from $[0, T-t']$ to $[0, T]$, resulting in the desired representation \eqref{alt-breveU2}.
\end{proof}

Starting from formula \eqref{alt-breveU2} and applying Minkowski's integral inequality twice, we find
\eq{
\big\|\breve U_2\big\|_{L_t^q ((0, T); H_x^{s, p} (0, \infty))} 
\leq \int_{t' = 0}^T \norm{S \b{0, S[\mathfrak F(\cdot, t'); 0] (0, \cdot); 0} (\cdot, \cdot - t')}_{L_t^q ((0, T); H_x^{s, p} (0, \infty))} dt'. \label{breveU2-bound1}}
Now, observe that the reduced half-line Strichartz estimate \eqref{amo2.87} with $T' = T$ allows us to readily estimate this Strichartz norm on the right-hand side. This is because, due to the fact that $0 \le s < \frac{1}{2}$, the relevant boundary data $S[\mathfrak F(\cdot, t'); 0] (0, \cdot)$ can readily be extended by zero ourside $(0, T)$, thus fulfilling the compact support requirement of~\eqref{HNLSreducedHL}. Hence, using also the time estimate \eqref{amo910}, we have
\eq{
\norm{S \b{0, S[\mathfrak F(\cdot, t'); 0] (0, \cdot); 0} (\cdot, \cdot - t')}_{L_t^q ((0, T); H_x^{s, p} (0, \infty))} 
\leq 
c_s \big(1+\sqrt T e^{cT}\big) \big(1+T^{\frac 1q+\frac 12}\big) \norm{\mathfrak F(t')}_{H_x^s (\R)}.}
In turn, for any $0 \le s < \frac{1}{2}$ and any admissible pair $(q, p)$ in the sense of \eqref{adm-pair}, inequality \eqref{breveU2-bound1} yields
\eq{\big\|\breve U_2\big\|_{L_t^q ((0, T); H_x^{s, p} (0, \infty))} 
\leq
c_s \big(1+\sqrt T e^{cT}\big) \big(1+T^{\frac 1q+\frac 12}\big) \norm{\mathfrak F}_{L_t^1((0, T); H_x^s (\R))}. \label{breveU2-bound2}}

Back to \eqref{Udecomp-bound-Strich}, having completed the estimation of the half-line solutions $\breve U_1, \breve U_2$, we turn out attention to the finite interval solutions $\breve u_1, \breve u_2, \breve u_3$.  The procedure for $\breve u_1$ is identical to the one used for $\breve U_1$. Specifically, we let $\psi_0 = h_0 - \mathcal W|_{x=\ell} - \breve U_1|_{x=\ell}$, $\psi_1 = h_1 - \partial_x \mathcal W|_{x=\ell} - \partial_x \breve U_1|_{x=\ell}$ and then employ Theorem 11.4 of \cite{lm1972} to extend these functions by zero outside $(0, T)$ to produce the extensions $\Psi_0, \Psi_1$ supported in $[0, T]$ and satisfying the bounds
\eq{
\norm{\Psi_0}_{H^{\frac{s + 1}{3}} (\R)} \leq c(s, T) \norm{\psi_0}_{H^{\frac{s + 1}{3}} (0, T)},
\quad
\norm{\Psi_1}_{H^{\frac{s + 1}{3}} (\R)} \leq c(s, T) \norm{\psi_1}_{H^{\frac{s + 1}{3}} (0, T)}.
\label{psi-extensions}}
Next, we let $v(x, t)$ be the solution to the reduced finite interval problem \eqref{HNLSreduced} with $T' = T$ and the extensions $\Psi_0$, $\Psi_1$ as boundary conditions. Then, combining Theorem \ref{thm:strichartz} with inequalities \eqref{psi-extensions}, the time estimate \eqref{amo910} (used twice, once for $\mathcal W$ and once for $\partial_x \mathcal W$ after shifting $s$ to $s-1$ in the latter case) and inequality \eqref{u0-party}, we obtain
\eqs{
\norm{\breve u_1}_{L_t^q ((0, T); H_x^{s, p} (0, \ell))} 
\leq
\big(1+T^{\frac 1q+\frac 12}\big) c(s, T) \Big(&\norm{h_0}_{H^{\frac{s + 1}{3}} (0, T)} + \norm{h_1}_{H^{\frac{s}{3}} (0, T)} + \norm{u_0}_{H^s (0,\ell)}  
\\
&+\big\|\breve U_1(\ell)\big\|_{H^{\frac{s + 1}{3}} (0, T)} + \big\|\partial_x \breve U_1(\ell)\big\|_{H^{\frac{s}{3}} (0, T)}\Big).
\alignlabel{breveu1-bound1}
}
At this point, as in the high-regularity case, it becomes evident that novel time estimates are needed for the half-line solution $\breve U_1 = S[0, g_0 - \mathcal W|_{x=0}; 0]$. These estimates are established in Theorem \ref{thm:hl-te} below and can be readily employed after simply extending the boundary datum $g_0 - \mathcal W|_{x=0}$ by zero outside $(0, T)$ to yield
\eqs{\big\|\breve U_1(\ell)\big\|_{H_t^{\frac{s + 1}{3}} (0, T)} + \big\|\partial_x \breve U_1(\ell)\big\|_{H_t^{\frac{s}{3}} (0, T)} 
\leq
c_s \max\set{T e^{M_\Delta T}, 1} \Big(\norm{g_0}_{H^{\frac{s + 1}{3}} (0, T)} + \norm{\mathcal W(0)}_{H^{\frac{s + 1}{3}} (0, T)}\Big),
\quad
s\geq 0.
}
In view of this estimate, the time estimate \eqref{amo910}, and inequality \eqref{u0-party}, we deduce via inequality \eqref{breveu1-bound1} that
\eq{\label{u1breve-est}
\norm{\breve u_1}_{L_t^q ((0, T); H_x^{s, p} (0, \ell))} 
\leq
\big(1+T^{\frac 1q+\frac 12}\big) c(s, T)
\Big(
\norm{u_0}_{H^s (0, \ell)} + \norm{g_0}_{H^{\frac{s + 1}{3}} (0, T)} 
+ \norm{h_0}_{H^{\frac{s + 1}{3}} (0, T)} + \norm{h_1}_{H^{\frac{s}{3}} (0, T)}\Big)}
for any $s\geq 0$ and any admissible pair \eqref{adm-pair}.

Regarding the term $\breve u_2 = S\big[0, 0, \mathcal Z|_{x=\ell}, \partial_x \mathcal Z|_{x=\ell}; 0\big]$, since its boundary data are traces of the forced Cauchy problem solution $\mathcal Z$, we face a similar issue as with $\breve U_2$ earlier. Namely, following the method used for $\breve u_1$ will eventually result in an $L_t^2(0, T)$ norm of the forcing  (through the time estimate \eqref{amo26}), while it is actually possible to instead have an $L_t^2(0, T)$ norm instead, provided we can establish the following  following finite interval analogue of Lemma \ref{lemma:alt-breveU2}.

\begin{lemma} \label{lemma:alt-breveu2}
The solution $\breve u_2 = S[0, 0, \mathcal Z(\ell, \cdot), \partial_x \mathcal Z(\ell, \cdot); 0]$ can be expressed in the form
\eq{\breve u_2 (x, t) = -i \int_{t' = 0}^T S \big[0, 0, S\b{\mathfrak F(\cdot, t'); 0] (\ell, \cdot),  S[\partial_x \mathfrak F(\cdot, t'); 0} (\ell, \cdot); 0\big] (x, t - t') dt'.\label{alt-breveu2}}
\end{lemma}

\begin{proof}
Starting from the unified transform formula \eqref{hls-sol} and recalling the definition \eqref{tilde-transform} along with the representation for $\mathcal Z(x, t)$ obtained via Duhamel's principle and the homogeneous Cauchy problem solution, we have
\eq{
\begin{split}
\breve u_2 (x, t) =
\frac{i}{2 \pi} \int_{\partial \tilde D} &\frac{e^{-i k \p{\ell - x} + i \omega t}}{\Delta(k)}  \bigg[ \p{e^{-i \nu_+ \ell} - e^{-i \nu_- \ell}} \omega' \int_{t'' = 0}^T e^{-i \omega t''} \int_{t' = 0}^{t''} \partial_x S[\mathfrak F(\cdot, t'); 0] (x, t'' - t') dt' dt''\\
&- i \p{\nu_- e^{-i \nu_+ \ell} - \nu_+ e^{-i \nu_- \ell}} \omega' \int_{t'' = 0}^T e^{-i \omega t''} \int_{t' = 0}^{t''} S[\mathfrak F(\cdot, t'); 0] (\ell, t'' - t') dt' dt'' \bigg] dk.
\end{split}
}
Interchanging the integrals with respect to $t'$ and $t''$  and then making the substitution $t'' \mapsto t'' + t'$ yields
\eq{\label{u2b-nn}
\begin{split}
\breve u_2 (x, t) 
= \frac{i}{2 \pi} \int_{t' = 0}^{T} \int_{\partial \tilde D} &\frac{e^{-i k \p{\ell - x} + i \omega (t - t')}}{\Delta(k)}  \bigg[ \p{e^{-i \nu_+ \ell} - e^{-i \nu_- \ell}} \omega' \int_{t'' = 0}^{T - t'} e^{-i \omega t''} \partial_x S[\mathfrak F(\cdot, t'); 0] (x, t'') dt''\\
&- i \p{\nu_- e^{-i \nu_+ \ell} - \nu_+ e^{-i \nu_- \ell}} \omega' \int_{t'' = 0}^{T - t'} e^{-i \omega t''} S[\mathfrak F(\cdot, t'); 0] (\ell, t'') dt'' \bigg] dk dt'.
\end{split}
}
The exponentials involved in the first half of the above integrand combine to
\eq{\label{u2b-nn-exp}
\begin{aligned}
\frac{1}{\Delta (k)} \, e^{-i k (\ell - x) + -i \omega (t'' + t' - t)} \p{e^{-i \nu_+ \ell} - e^{-i \nu_- \ell}} 
&= \frac{1}{e^{i \nu_0 \ell} \Delta (k)} \, e^{i k x + -i \omega (t'' + t' - t)} \p{e^{-i \nu_+ \ell} - e^{-i \nu_- \ell}}\\
&= \frac{1}{e^{i \nu_+ \ell} \Delta (k)} \, e^{-i \nu_0 (\ell - x) + -i \omega (t'' + t' - t)} \p{1 - e^{i \mu_0 \ell}}\\
&= \frac{1}{e^{i \nu_- \ell} \Delta (k)} \, e^{-i \nu_0 (\ell - x) + -i \omega (t'' + t' - t)} \p{e^{-i \mu_0 \ell} - 1},
\end{aligned}
}
where we recall that $\nu_0 = k$ and $\mu_0 = \nu_+ - \nu_-$. Each of the three forms on the right-hand side is suitable for working in each of the three regions $\tilde D_n \subset \tilde D$, $n\in\set{0, +, -}$. Importantly, by Lemma \ref{lemma:Deltabound},
$\frac{1}{e^{i \nu_n \ell} \Delta (k)} \longrightarrow 0$ as $|k| \to \infty$ with $k \in \tilde D_n$.
For $k \in \tilde D_0$ we have $\Im(\nu_0) > 0$, while for $k \in \tilde D_{\pm}$ we have $\Im(\nu_0) < 0$, so the factor containing $\nu_0$ also decays. Also, $\Im(\omega) < 0$ inside $\tilde D$, so for $t'' \ge T - t'$ the exponential $e^{-i \omega (t'' + t' - t)}$ decays as well. Finally, by the characterization of each $\tilde D_n$ provided by Lemma \ref{lemma:charDbar}, the differences of exponentials are bounded. Therefore, the entire exponential factor given by \eqref{u2b-nn-exp} decays as $|k| \to \infty$ in $\tilde D$. A similar argument can be made for the exponential emerging from the second half of  \eqref{u2b-nn}. Thus,  by analyticity and Cauchy's theorem, the  range of the integral in $t''$ can be augmented up to $T$,  resulting in the desired representation \eqref{alt-breveu2}.
\end{proof}

Proceeding similarly to the derivation of the half-line estimate \eqref{breveU2-bound1}, we use the representation \eqref{alt-breveu2} to obtain
\eq{
\norm{\breve u_2}_{L_t^q ((0, T); H_x^{s, p} (0, \ell))} 
\leq \int_{t' = 0}^T \norm{S \big[0, 0, S[\mathfrak F(\cdot, t'); 0] (\ell, \cdot),  S[\partial_x\mathfrak F(\cdot, t'); 0] (\ell, \cdot); 0\big] (\cdot, \cdot - t')}_{L_t^q ((0, T); H_x^{s, p} (0, \ell))} dt'. \label{breveu2-bound1}}
Following the same reasoning as below \eqref{breveU2-bound1}, we can estimate the Strichartz norm on the right-hand side of \eqref{breveu2-bound1} by combining Theorem \ref{thm:strichartz} for the reduced interval problem (with $T'=T$, $\psi_0 = S[\mathfrak F(\cdot, t'); 0] (\ell, \cdot)$ and $\psi_1 = S[\partial_x \mathfrak F(\cdot, t'); 0] (\ell, \cdot)$) with the time estimate \eqref{amo910} (for $\mathcal W = S[\mathfrak F(\cdot, t'); 0] (\ell, \cdot)$ and $s$, as well as for $\mathcal W = S[\partial_x  \mathfrak F(\cdot, t'); 0] (\ell, \cdot)$ and $s-1$) to conclude that
\eq{
\norm{\breve u_2}_{L_t^q ((0, T); H_x^{s, p} (0, \ell))}
&\cle
 \int_{t' = 0}^T \Big(
\norm{S[\mathfrak F(\cdot, t'); 0] (\ell, \cdot)}_{H_t^{\frac{s+1}{3}} (0, T)}
+
\norm{S[\partial_x \mathfrak F(\cdot, t'); 0] (\ell, \cdot)}_{H_t^{\frac{s}{3}} (0, T)}\Big) dt'
\nn\\
&\lesssim 
 \norm{\mathfrak F}_{L_t^1((0, T); H_x^s(0, \ell))}
\label{breveu2-bound2}
}
for any $0 \leq s < \frac 12$ and any admissible pair \eqref{adm-pair}.
We remark that Theorem \ref{thm:strichartz} does apply to \eqref{breveu2-bound1} despite the time shift that the theorem does not consider. Indeed, the factor induced by the shift is purely oscillatory and so it disappears once norms are taken, leaving the proof of the theorem unaffected.

Our final task in regard to the decomposition inequality \eqref{Udecomp-bound-Strich} is to estimate the Strichartz norm of the term $\breve u_3 = S\big[0, 0, \breve U_2|_{x=\ell}, \partial_x \breve U_2|_{x=\ell}; 0\big]$. Recalling that $\breve U_2 = S\big[0, \mathcal Z|_{x=0}; 0\big]$ and $\mathcal Z = S\big[0; \mathfrak F\big]$, we realize that we once again face the same issue as with $\breve U_2$ and $\breve u_2$, except that now the forced Cauchy solution $\mathcal Z$ is nested one level deeper within the finite interval solution (through the traces of the half-line solution $\breve U_2$, whose boundary datum is a trace of $\mathcal Z$). The remedy here is a simple application of Lemma \ref{lemma:alt-breveU2}. 
Specifically, extending the relevant boundary data by zero outside $(0, T)$ via the reasoning given below \eqref{breveU2-bound1}, we use Theorem \ref{thm:strichartz} to infer
\eq{\norm{\breve u_3}_{L_t^q ((0, T); H_x^{s, p} (0, \ell))} \leq c_s \big(1+T^{\frac 1q+\frac 12}\big) \Big(\big\|\breve U_2|_{x=\ell}\big\|_{H_t^{\frac{s + 1}{3}} (0, T)} + \big\|\partial_x \breve U_2|_{x=\ell}\big\|_{H_t^{\frac{s}{3}} (0, T)}\Big).}
We then apply Lemma \ref{lemma:alt-breveU2} and Minkowski's inequality to obtain
\eq{
\begin{split}
\norm{\breve u_3}_{L_t^q ((0, T); H_x^{s, p} (0, \ell))} 
\leq
c_s \big(1+T^{\frac 1q+\frac 12}\big)
\bigg(&\int_{t' = 0}^T \norm{S \b{0, S[\mathfrak F(\cdot, t'); 0] (0, \cdot); 0} (\ell, \cdot - t')}_{H_t^{\frac{s + 1}{3}} (0, T)} dt'\\
&+ \int_{t' = 0}^T \norm{\partial_x S \b{0, S[\mathfrak F(\cdot, t'); 0] (0, \cdot); 0} (\ell, \cdot - t')}_{H_t^{\frac{s}{3}} (0, T)} dt'
\bigg).
\end{split}}
At this point, Theorem \ref{thm:hl-te} implies
\eq{\norm{\breve u_3}_{L_t^q ((0, T); H_x^{s, p} (0, \ell))} 
\leq
c_s \big(1+T^{\frac 1q+\frac 12}\big)  \max\set{T e^{M_\Delta T}, 1}
\int_{t' = 0}^T \norm{S[\mathfrak F(\cdot, t'); 0] (0, \cdot)}_{H_t^{\frac{s + 1}{3}} (0, T)} dt', \quad s\geq 0,}
again considering the fact that translations in time have no effect on the proof of the theorem. After this step, we invoke the time estimate \eqref{amo910} to deduce
\eq{\label{u3breve-est}
\norm{\breve u_3}_{L_t^q ((0, T); H_x^{s, p} (0, \ell))} 
\leq
c_s \big(1+T^{\frac 1q+\frac 12}\big)  \max\set{T e^{M_\Delta T}, 1} \norm{\mathfrak F}_{L_t^1((0, T); H_x^s (\R))}}
for any $s\geq 0$ and and admissible pair \eqref{adm-pair}.

\vspace*{2mm}

We are now ready to finalize the estimation of the Strichartz norm of the forced linear interval problem \eqref{hls-ibvp}. Indeed, the decomposition inequality \eqref{Udecomp-bound-Strich} combined with the component estimates \eqref{amo2.17}, \eqref{amo2.38},  \eqref{U1breve-est}, \eqref{breveU2-bound2}, \eqref{u1breve-est}, \eqref{breveu2-bound2}, \eqref{u3breve-est} and the extension inequalities \eqref{u0-party}, \eqref{f-party}, readily yields the Strichartz estimate \eqref{strich-est}. 
Moreover, combining the analogue of the decomposition \eqref{Udecomp-bound} for the space $L_t^2((0, T); H_x^{s+1}(0, \ell))$  with the smoothing estimates \eqref{smoothing-hc}, \eqref{smoothing-fc}, \eqref{smoothing-hl}, \eqref{smoothing-fi}, the extension inequalities \eqref{u0-party}, \eqref{f-party} for the initial data and the forcing, inequalities \eqref{g-g0-u0} and \eqref{low-s-extension} for $g_0(t)$, the analogue of the latter inequality for $\mathcal Z(0, t)$ instead of $g_0(t)$,  and the extension inequalities~\eqref{psi-extensions} together with the time estimates \eqref{amo910} and \eqref{V-te}, we obtain the smoothing estimate~\eqref{smooth-est} in the low-regularity range $0\leq s< \frac 12$. 
This step completes the proof of Theorem \ref{lin-est-t} for the forced linear interval problem \eqref{hls-ibvp} in the low-regularity setting. 

\subsection{Time estimates for the half-line problem}\label{hl-te-ss}

As noted earlier, the decomposition inequalities \eqref{Udecomp-bound} and \eqref{Udecomp-bound-Strich} introduce the need for studying the time regularity of the solution to the reduced half-line problem \eqref{HNLSreducedHL}. In particular,  it is necessary to estimate the norms $\norm{V(x)}_{H_t^{\frac{s+1}{3}} (0, T')}$ and  $\norm{\partial_x V(x)}_{H_t^{\frac s3} (0, T')}$ for each $x\in (0, \infty)$. Furthermore, analogously to the finite interval, we wish to establish the analogue of the smoothing effect given in Corollary \ref{smoothing-fi-c}.
This latter task will require us to estimate, for all $\sigma \le s + 1$, the $\sigma$th spatial derivative of $V$ in $H_t^{\frac{s + 1}{3}} (0, T')$. Our results can be combined into the following theorem:
\begin{theorem}[Time estimates for the reduced half-line problem]\label{thm:hl-te}
Suppose $V_0 \in H^{\frac{s+1}{3}}(\R)$ and $s\geq -1$. For any $\sigma \in \N_0$ such that $\sigma \leq s + 1$, the $\sigma$th spatial derivative of the solution $V(x, t)$ to the reduced half-line problem \eqref{HNLSreducedHL} belongs to $H^{\frac{s+1-\sigma}{3}}(0, T')$ as a function of~$t$. In particular, we have the estimate
\eq{
\norm{\partial_x^\sigma V}_{L_x^\infty((0, \infty); H_t^{\frac{s + 1 - \sigma}{3}} (0, T'))} \le c \max \big\{T' e^{M_\Delta T'}, 1\big\} \norm{V_0}_{H^{\frac{s+1}{3}}(\R)}, 
\label{V-te}
}
where $c>0$ is a constant that depends only on $s, \alpha, \beta, \delta$, and the constant $M_\Delta$ is given by \eqref{md-def}. 
\end{theorem}

\begin{proof}
Let $m=\frac{s+1-\sigma}{3}$, and start by assuming $\sigma = 0$. By the expression (65) in \cite{amo2024} and the support condition for $V_0$, the solution to problem \eqref{HNLSreducedHL} is given by the unified transform formula
\eq{V(x, t) = -\frac{1}{2 \pi} \int_{\partial \tilde D_0} e^{i k x + i \omega t} \omega' \mathcal{F} \{V_0\} (\omega) dk,}
where the contour $\partial \tilde D_0$ is defined by \eqref{dntil-def}. 
Recalling that $\partial \tilde D_0 = \Gamma_1 \cup \Gamma_2 \cup \Gamma_3$ (see Figure \ref{fig:Dbranches}), we consider each contour separately and write
\eq{\label{HLsoldecomp}
V(x, t) = -\frac{1}{2 \pi} \sum_{n = 1}^3 I_{\Gamma_n}(x, t),
\quad
I_{\Gamma_n}(x, t) := \int_{\Gamma_n} e^{i k x + i \omega t} \omega' \mathcal{F} \{V_0\} (\omega) dk.
}

Since $\Gamma_1$ and $\Gamma_3$ are similar, we will only examine $I_{\Gamma_1}$, which when parametrized by \eqref{gamma1} becomes
\eq{
I_{\Gamma_1}(x, t)
= \int_\infty^{r_0} e^{i \gamma_1 x + i \omega t} \mathcal{F} \{V_0\} (\omega) \frac{d\omega}{dr} dr.}
Making the change of variables from $r$ to $\omega$, we have
\eq{
I_{\Gamma_1}(x, t) 
=  -\int_{\omega_0}^\infty e^{i \omega t} \cdot e^{i \gamma_1 x} \mathcal{F} \{V_0\} (\omega)  d\omega
}
which implies
$\mathcal{F}_t \set{I_{\Gamma_1}} (x, \omega) = -2\pi e^{i \gamma_1 x} \mathcal{F} \{V_0\} (\omega) \chi_{[\omega_0, \infty)}$.
In turn, substituting for $\gamma_1$ via \eqref{gamma1} and using the fact that $r, x>0$, we find
\eq{
\norm{I_{\Gamma_1} (x)}_{H_t^m (0, T')}^2 
\le \norm{I_{\Gamma_1} (x)}_{H_t^m (\R)}^2
\cle
\int_{\omega_0}^\infty \p{1 + \omega^2}^m \abs{\mathcal{F} \{V_0\} (\omega)}^2  d\omega 
\leq \norm{V_0}_{H^m (\R)}^2.
 \label{G1-est}
}

For $I_{\Gamma_2}$, the parametrization \eqref{gamma2} yields
\eq{I_{\Gamma_2}(x, t) = \int_{\frac{\pi}{2} + \phi_0}^{\frac{\pi}{2} - \phi_0} e^{i \gamma_2 x + i \omega t} \frac{d\omega}{d\gamma_2} \mathcal{F} \{V_0\} (\omega) i R_\Delta e^{i \theta} d\theta.}
Differentiating in $t$, taking the magnitude and noting that $\sin(\theta) \ge 0$ and $\Im(\omega) \le 0$ on $\Gamma_2$, we find
\eqs{
\big|\partial_t^j I_{\Gamma_2}(x, t)\big| &\le \int_{\frac{\pi}{2} - \phi_0}^{\frac{\pi}{2} + \phi_0} e^{-R_\Delta \sin(\theta) x - \Im(\omega) t} \abs{\omega}^j \abs{\frac{d\omega}{d\gamma_2}} \abs{\mathcal{F} \{V_0\} (\omega)} R_\Delta d\theta\\
&\le \int_{\frac{\pi}{2} - \phi_0}^{\frac{\pi}{2} + \phi_0} e^{-\Im(\omega) t} \abs{\omega}^j \abs{\frac{d\omega}{d\gamma_2}} \abs{\mathcal{F} \{V_0\} (\omega)} R_\Delta d\theta}
for each $j\in\N_0$. Using the bound \eqref{md-def} for $\omega$ and the bound \eqref{domegadkbound} for the derivative of $\omega$ (which allows us to absorb the derivative as a constant), we further obtain
\eq{\big|\partial_t^j I_{\Gamma_2}(x, t)\big| \cle M_\Delta^j \int_{\frac{\pi}{2} - \phi_0}^{\frac{\pi}{2} + \phi_0} e^{-\Im(\omega) t} \abs{\mathcal{F} \{V_0\} (\omega)} d\theta.
}
Next, thanks to the support condition for $V_0$, we switch from $\mathcal{F} \{V_0\} (\omega)$ to $\tilde V_0 (\omega, T')$ defined by \eqref{tilde-transform}, so that 
\eqs{
\big|\partial_t^j I_{\Gamma_2}(x, t)\big| &\cle M_\Delta^j  \int_{\frac{\pi}{2} - \phi_0}^{\frac{\pi}{2} + \phi_0} e^{-\Im(\omega) t} \ \bigg|\int_0^{T'} e^{-i \omega t'} V_0 (t') dt'\bigg| d\theta
\le M_\Delta^j \int_{\frac{\pi}{2} - \phi_0}^{\frac{\pi}{2} + \phi_0}   \int_0^{T'} e^{\Im(\omega)(t'-t)} \abs{V_0 (t')} dt' d\theta.
}
Hence, using the bound $e^{\Im(\omega) (t'-t)} \leq e^{M_\Delta T'}$ with  $M_\Delta$ given by \eqref{md-def}, we have
\eqs{
\big|\partial_t^j I_{\Gamma_2}(x, t)\big| \cle M_\Delta^j e^{M_\Delta T'} \int_{\frac{\pi}{2} - \phi_0}^{\frac{\pi}{2} + \phi_0} \int_0^{T'} \abs{V_0 (t')} dt' d\theta
\lesssim
\sqrt {T'} M_\Delta^j  e^{M_\Delta T'} \norm{V_0}_{L^2 (0, T')}
}
with the last step due to the Cauchy-Schwarz inequality. Since the right-hand side does not depend on $t$, for each $x\in (0, \infty)$ and any $m\geq 0$ (or, equivalently, $s\geq -1$) we have
\eq{\label{G2-est}
\norm{I_{\Gamma_2} (x)}_{H_t^m (0, T')}
\leq
\norm{I_{\Gamma_2} (x)}_{H_t^{\left\lceil m\right\rceil} (0, T')}
=
\sum_{j=0}^{\left\lceil m\right\rceil} \big\|\partial_t^j I_{\Gamma_2} (x)\big\|_{L_t^2(0, T')}
\lesssim 
T'   e^{M_\Delta T'} 
\frac{M_\Delta^{\left\lceil m\right\rceil+1}-1}{M_\Delta-1} \norm{V_0}_{L^2 (0, T')}.
}

Also, as noted earlier, the estimation of $I_{\Gamma_3}$ is similar to $I_{\Gamma_1}$ and leads to the analogue of \eqref{G1-est}. Thus,  \eqref{G1-est} and \eqref{G2-est} combined with the decomposition \eqref{HLsoldecomp} yield the time estimate \eqref{V-te} for $V$ when $\sigma = 0$.

The case where $\sigma \ne 0$ can be established in a similar fashion, since differentiating \eqref{HLsoldecomp} $\sigma$ times in $x$ only results in additional factors of $k$ inside the integrand. 
Along $\Gamma_1$, these factors can be handled via \eqref{gamma1-omega}. 
A similar argument can be employed along $\Gamma_3$. Finally, on $\Gamma_2$, the relevant integrand gains a factor of $\big(\frac{\alpha}{3 \beta} + R_\Delta e^{i \theta}\big)^\sigma$, which has a magnitude bounded by the constant $\big(\frac{\abs \alpha}{3 \beta} + R_\Delta\big)^\sigma$. Hence, the argument for $I_{\Gamma_2}$ also follows through for $\partial_x^\sigma I_{\Gamma_2}$ provided that $m \ge 0$ i.e. $s \ge \sigma - 1$.
\end{proof}

An immediate consequence of estimate \eqref{V-te} is the following corollary, which provides the half-line analogue of the smoothing effect of Corollary \ref{smoothing-fi-c} and can be established in the same way.
\begin{corollary}[Smoothing effect]
The solution $V(x, t)$ to the reduced half-line problem \eqref{HNLSreducedHL} satisfies the estimate
\eq{\label{smoothing-hl}
\norm{V}_{L_t^2((0, T'); H_x^{s+1}(0, \ell))} \leq c \max \big\{T' e^{M_\Delta T'}, 1\big\} \sqrt{\ell (s + 2)} \norm{V_0}_{H^{\frac{s + 1}{3}}(\R)}, \quad s\geq -1,
}
where $c$ is the constant of estimate \eqref{V-te}.
\end{corollary}

Finally, for completeness of presentation and since these were not derived in the half-line work \cite{amo2024}, below we provide time estimates for the full forced linear half-line problem \eqref{halfline}. We emphasize that these are not needed for the proof of Theorem \ref{lin-est-t} given earlier. Indeed, the only half-line time estimates used in that proof are those of Theorem \ref{thm:hl-te} for the reduced half-line problem \eqref{HNLSreducedHL}.

\begin{theorem}[Time estimates for the full half-line problem]\label{thm:hl-te2}
Suppose $0 \le s \le 2$ with $s \ne \frac{1}{2}, \frac{3}{2}$. Then,  the solution $U(x, t)$ to the forced linear half-line problem \eqref{halfline} belongs to $H^{\frac{s + 1}{3}}(0, T)$ as a function of $t$, and the derivative $\partial_x U(x, t)$ belongs to $H^{\frac{s}{3}}(0, T)$ as a function of $t$. In particular, we have the estimate
\eq{
\begin{aligned}
&\quad
\norm{U}_{L_x^\infty((0, \infty); H_t^{\frac{s + 1}{3}} (0, T))}  + \norm{\partial_x U}_{L_x^\infty((0, \infty); H_t^{\frac{s}{3}} (0, T))} 
\\
&\cle
\max \set{T e^{2 M_\Delta T}, 1} \max \set{T^{\frac{\p{1 - 2 s}}{6}}, \sqrt T, 1} \Big[ \big(1+\sqrt T\,\big) \norm{U_0}_{H^{s}(0, \infty)}
+ \norm{g_0 (x)}_{H^{\frac{s + 1}{3}}(0, T)}
\\
&\quad
+ A(s, T) \norm{F}_{L_t^2 \p{(0, T); H_x^{s} (0, \infty)}}
 \Big]
\end{aligned}}
where
\eq{\label{ast-def}
A(s, T) = \begin{cases} \max \set{T^{\frac{1}{2}} \big(1 + T^{\frac{1}{2}}\big), T^{a(s)}, T^{a(s - 1)}}, &0\leq s < 2, \\ 1 + T^{\frac{1}{2}} \big(1 + T^{\frac{1}{2}}\big), &s=2, 
\end{cases}
\quad
a(s) =
\begin{cases}
\frac{1 - 2 s}{6}, & -1 \le s < \frac{1}{2}\\
\frac{2 - s}{3}, & \frac{1}{2} \le s < 2.
\end{cases}}
\end{theorem}

\begin{proof}
Theorem \ref{thm:hl-te2} follows by combining the linear decompositions \eqref{U-decomp1} and \eqref{U-decomp2} with Theorem \ref{hl-t1} on the Cauchy problem estimates proved in \cite{amo2024} and Theorem \ref{thm:hl-te} on the new time estimates for the reduced half-line problem derived above. 
\end{proof}

\section{Nonlinear theory: proofs of Theorems \ref{high-wp-t} and \ref{low-wp-t}}\label{wp-s}

\subsection{High-regularity well-posedness}
Suppose $\frac 12 < s \leq 2$ with $s\neq \frac 32$. For lifespan $T>0$ yet to be determined, let 
$$
X_T^s := C_t([0, T]; H_x^s(0, \ell)) \cap L_t^2((0, T); H_x^{s+1}(0, \ell))
$$ 
with norm equal to the maximum of the norms of the two spaces involved. 
Consider the iteration map $u \mapsto \Phi(u)$ on $X_T^s$, where 
\eq{
\Phi(u) := S\big[u_0, g_0, h_0, h_1; \kappa |u|^{\lambda-1}u\big]
\label{Phi}
}
and $S\big[u_0, g_0, h_0, h_1; f\big]$ is the solution to the forced linear problem \eqref{hls-ibvp}. Then, the Sobolev estimates \eqref{sob-est} and~\eqref{smooth-est} imply
\eq{\label{fest}
\begin{split}
\norm{\Phi(u)}_{X_T^s} 
\leq
c(s, T) \Big(
\norm{u_0}_{H^s(0, \ell)}  
+
\norm{g_0}_{H^{\frac{s+1}{3}}(0, T)}
&+
\norm{h_0}_{H^{\frac{s+1}{3}}(0, T)}
+
\norm{h_1}_{H^{\frac{s}{3}}(0, T)}
\\
&
+ 
\sqrt T \norm{\kappa |u|^{\lambda-1}u}_{L_t^\infty((0, T); H_x^s(0, \ell))} 
\Big),
\end{split}
}
where $c(s, T) := \max\big\{c_1(s, T), c_2(s, T), c_2(s, T) \sqrt T\big\}$. 
Since $s>\frac 12$, by the $H_x^s(0, \ell)$ algebra property for $\lambda$ odd or by Lemma 3.1 in \cite{bo2016} otherwise (which generalizes the algebra property to any $\lambda$ satisfying \eqref{l-cond}), estimate \eqref{fest} becomes
\eq{
\norm{\Phi(u)}_{X_T^s}
\leq
c(s, T) \Big(
\norm{u_0}_{H^s(0, \ell)}  
+
\norm{g_0}_{H^{\frac{s+1}{3}}(0, T)}
+
\norm{h_0}_{H^{\frac{s+1}{3}}(0, T)}
+
\norm{h_1}_{H^{\frac{s}{3}}(0, T)}
+ 
|\kappa| c_s \sqrt T \norm{u}_{X_T^s}^\lambda\Big).
}
Defining the radius
\eq{
\varrho = \varrho(T) := 2 c(s, T) \Big(
\norm{u_0}_{H^s(0, \ell)}  
+
\norm{g_0}_{H^{\frac{s+1}{3}}(0, T)}
+
\norm{h_0}_{H^{\frac{s+1}{3}}(0, T)}
+
\norm{h_1}_{H^{\frac{s}{3}}(0, T)}\Big)
}
and assuming $\varrho>0$ (if not, the zero solution certainly exists and is unique as a result of the separate uniqueness argument below), we let $B_\varrho(0) \subset X_T^s$ be the ball of radius $\varrho$ centered at the origin. Then, for any $u \in \overline{B_\varrho(0)}$, 
\eq{
\norm{\Phi(u)}_{X_T^s}
\leq
\frac{\varrho}{2}
+ 
 |\kappa| c_s c(s, T) \sqrt T \varrho^\lambda,
}
therefore, if $T>0$ satisfies the condition
\eq{\label{T-into}
|\kappa| c_s c(s, T) \sqrt T \varrho^{\lambda-1} \leq \frac12
}
it follows that $\Phi(u) \in \overline{B_\varrho(0)}$, i.e. $u\mapsto \Phi(u)$ maps $\overline{B_\varrho(0)}$ into $\overline{B_\varrho(0)}$. 

Similarly, for any two $u_1, u_2 \in \overline{B_\varrho(0)}$, the Sobolev estimates \eqref{sob-est} and \eqref{smooth-est} imply
\eq{\label{diffs}
\norm{\Phi(u_1)-\Phi(u_2)}_{X_T^s} 
\leq
|\kappa| c(s, T) \sqrt T
\norm{|u_1|^{\lambda-1}u_1-|u_2|^{\lambda-1}u_2}_{L_t^\infty((0, T); H_x^s(0, \ell))}
}
with $c(s, T)$ as above. The norm of the difference of nonlinearities on the right hand side will be handled via the following lemma.

\begin{lemma} \label{lemma:H-pow-diff}
Suppose that $s > \frac{1}{2}$, $\lambda > 1$ satisfies the conditions \eqref{l-cond} of Theorem \ref{high-wp-t}, and $u_1, u_2 \in H^s (0, \ell)$. Then,
\eqs{ 
\norm{|u_1|^{\lambda - 1} u_1 - |u_2|^{\lambda - 1} u_2}_{H^s (0, \ell)}
\leq c(s, \lambda) \p{\norm{u_1}^{\lambda - 1} + \norm{u_2}^{\lambda - 1}} \norm{u_1 - u_2}_{H^s (0, \ell)}.
}
\end{lemma}

\begin{proof}
If $\lambda$ is odd, then the result follows via straightforward manipulations and the algebra property in $H_x^s(0, \ell)$. Otherwise, let $\mathcal U_1, \mathcal U_2$ be the usual global extensions of $u_1, u_2$ satisfying the analogue of \eqref{u0-party} and note that, due to the existence of a fixed bounded extension operator from $H^s(0, \ell)$ to $H^s(\R)$, their difference also satisfies the same inequality, namely
\eqs{
\norm{\mathcal U_j}_{H^s (\R)}
\lesssim \norm{u_j}_{H^s (0, \ell)},
\ j = 1, 2,
\quad
\norm{\mathcal U_1 - \mathcal U_2}_{H^s (\R)}
\cle \norm{u_1 - u_2}_{H^s (0, \ell)}.
}
Then, by Lemma 3.1 in \cite{bo2016}, 
\eqs{
\norm{|u_1|^{\lambda - 1} u_1 - |u_2|^{\lambda - 1} u_2}_{H^s (0, \ell)}
&\le \norm{|\mathcal U_1|^{\lambda - 1} \mathcal U_1 - |\mathcal U_2|^{\lambda - 1} \mathcal U_2}_{H^s (\R)}\\
&\cle \p{\norm{\mathcal U_1}_{H^s (\R)}^{\lambda - 1} + \norm{\mathcal U_2}_{H^s (\R)}^{\lambda - 1}} \norm{\mathcal U_1 - \mathcal U_2}_{H^s (\R)}
}
which in view of the above extension inequalities yields the desired inequality.
\end{proof}

Combining \eqref{diffs} with Lemma \ref{lemma:H-pow-diff} and the fact that $u_1, u_2 \in \overline{B_\varrho(0)}$ yields
\eq{
\norm{\Phi(u_1)-\Phi(u_2)}_{X_T^s} 
\leq
|\kappa| c(s, \lambda, T) \sqrt T
\p{\varrho^{\lambda-1}+\varrho^{\lambda-1}}
\norm{u_1-u_2}_{X_T^s}
}
where $c(s, \lambda, T) := c(s, \lambda) c(s, T)$. Hence, for $T>0$ satisfying the condition \eqref{T-into} as well as the condition (note that the two constants involved are not the same)
\eq{\label{T-contr}
|\kappa| c(s, \lambda, T) \sqrt T \varrho^{\lambda-1} < \frac 12
}
we conclude that $u\mapsto \Phi(u)$ is a contraction on $\overline{B_\varrho(0)}$. Therefore, by the Banach fixed point theorem, the map $u\mapsto \Phi(u)$ has a unique fixed point in $\overline{B_\varrho(0)}$, implying the existence of a unique solution  to the HNLS finite interval problem \eqref{hnls-ibvp} in $\overline{B_\varrho(0)}\subset X_T^s$.
\\[2mm]
\noindent
\textbf{Extending uniqueness to the whole of $C_t([0, T]; H_x^s(0, \ell))$.}
Suppose $u_1, u_2 \in C_t([0, T]; H_x^s(0, \ell))$ are two solutions of \eqref{hnls-ibvp} corresponding to the same set of initial and boundary data. Then, their difference $w = u_1-u_2$ satisfies the problem
\begin{equation}
\begin{aligned}\label{w-ibvp}
&i w_t + i \beta w_{xxx} + \alpha w_{xx} + i \delta w_x = \kappa \p{|u_1|^{\lambda - 1} u_1-|u_2|^{\lambda - 1} u_2}, \quad 0 < x < \ell, \ 0 < t < T,
\\
&w(x, 0) = 0,\\
&w(0, t) = w(\ell, t) = w_x (\ell, t) = 0. 
\end{aligned}
\end{equation}
Multiplying the equation for $w$ by $\overline w$ and integrating in $x$, we have
\eqs{
i \int_0^\ell \overline w w_t dx + i \beta \int_0^\ell \overline w w_{xxx} dx + \alpha \int_0^\ell \overline w w_{xx} dx + i \delta \int_0^\ell \overline w w_x dx = \kappa \int_0^\ell \overline w \p{|u_1|^{\lambda - 1} u_1-|u_2|^{\lambda - 1} u_2} dx
}
so integrating by parts on the left-hand side while employing the boundary conditions for $w$ yields
\eqs{
i \int_0^\ell \overline w w_t dx - i \beta \int_0^\ell \overline w_x w_{xx} dx - \alpha \int_0^\ell |w_{x}|^2 dx + i \delta \int_0^\ell \overline w w_x dx = \kappa \int_0^\ell \overline w \p{|u_1|^{\lambda - 1} u_1-|u_2|^{\lambda - 1} u_2} dx.
}
Taking imaginary parts, we obtain
\eq{
\frac 12 \frac{d}{dt}  \int_0^\ell |w|^2 dx
+  \frac{\beta}{2} |w_x(0, t)|^2
=
\Im\Big[\kappa \int_0^\ell \overline w \p{|u_1|^{\lambda - 1} u_1-|u_2|^{\lambda - 1} u_2} dx\Big].
}
Therefore, since $\beta > 0$, 
\begin{align}
\frac 12 \frac{d}{dt}  \norm{w}_{L_x^2(0, \ell)}^2
&\leq
\Im\Big[\kappa \int_0^\ell \overline w \p{|u_1|^{\lambda - 1} u_1-|u_2|^{\lambda - 1} u_2} dx\Big]
\nn\\
&\leq
|\kappa| \int_0^\ell |w| \abs{|u_1|^{\lambda - 1} u_1-|u_2|^{\lambda - 1} u_2} dx.
\label{en-est-0}
\end{align}

At this point, we employ the following standard result (for a proof, see e.g. Lemma 3.3 in \cite{hkmms2024}):
\begin{lemma}\label{mvt-l-w}
For any pair of complex numbers $u_1, u_2$,
\begin{equation}
|u_1|^{\lambda-1} u_1 - |u_2|^{\lambda-1} u_2 
=
\frac{\lambda+1}{2} \left(\int_0^1  |Z_\tau|^{\lambda-1} d\tau\right) \left(u_1-u_2\right)
+
\frac{\lambda-1}{2} \left(\int_0^1 |Z_\tau|^{\lambda-3}
Z_\tau^2 \, d\tau\right)  \overline{\left(u_1-u_2\right)},
\end{equation}
where $Z_\tau := \tau u_1 + \left(1-\tau\right) u_2$, $\tau\in [0, 1]$.
\end{lemma}
Using Lemma \ref{mvt-l-w} and the fact that $\lambda \ge 2$ and $w=u_1-u_2$, we find
\eqs{
\abs{|u_1|^{\lambda - 1} u_1 - |u_2|^{\lambda - 1} u_2}
&\le c \int_0^1 \p{\abs{\tau u_1}^{\lambda - 1} + \abs{(1 - \tau) u_2}^{\lambda - 1}} d\tau \abs{w}
\\
&\le c \p{|u_1|^{\lambda - 1} + |u_2|^{\lambda - 1}} \abs{w}, 
}
where $c > 0$ depends on $\lambda$. Hence, \eqref{en-est-0} becomes
\eq{
\frac{1}{2} \frac{d}{dt} \norm{w(t)}_{L_x^2(0, \ell)}^2
&\le c \abs \kappa \p{\norm{u_1(t)}_{L_x^\infty (0, \ell)}^{\lambda - 1} + \norm{u_2(t)}_{L_x^\infty (0, \ell)}^{\lambda - 1}} \norm{w(t)}_{L^2_x (0, \ell)}^2
\nn\\
&\le c \abs \kappa \p{\norm{u_1}_{X_T^s}^{\lambda - 1} + \norm{u_2}_{X_T^s}^{\lambda - 1}} \norm{w(t)}_{L^2_x (0, \ell)}^2
\label{en-est}
}
with the last inequality thanks to the embedding $H^s(0, \ell) \hookrightarrow L^\infty(0, \ell)$ (which applies since $s>\frac 12$).
Since the constant on the right-hand side is positive, the differential inequality \eqref{en-est} can be solved  via a simple application of Gr\"onwall's inequality to yield
\eq{
\norm{w(t)}_{L_x^2(0, \ell)}^2 \leq \norm{w(0)}_{L_x^2(0, \ell)}^2 e^{2c |\kappa| \p{\norm{u_1}_{X_T^s} + \norm{u_2}_{X_T^s}}^{\lambda-1} t}.
}
As the right-hand side is zero in view of the initial condition $w(x, 0) = 0$, we conclude that $u_1 \equiv u_2$ in all of $C_t([0, T]; H_x^s(0, \ell))$.
\\[2mm]
\noindent
\textbf{Continuous dependence on the data.} The final task for showing Hadamard well-posedness is to prove that the solution depends continuously on the data. More precisely, we will show that the data-to-solution map is locally Lipschitz. For some $T > 0$, let $d \in \mathcal D_T := H^s (0, \ell) \times  H^{\frac{s + 1}{3}} (0, T) \times  H^{\frac{s + 1}{3}} (0, T) \times H^{\frac{s}{3}} (0, T)$ be an arbitrary data point  with corresponding norm
\eq{\label{d-norm-def}
\norm{d}_{\mathcal D_T} := \norm{u_0}_{H^s(0, \ell)} + \norm{g_0}_{H^{\frac{s+1}{3}}(0, T)} + \norm{h_0}_{H^{\frac{s+1}{3}}(0, T)} +\norm{h_1}_{H^{\frac{s}{3}}(0, T)}.
}
Select $R > 0$ and define $B_R (d)$ to be the open ball of radius $R$ centered at $d$. Given two data points $d_1 = \big(u_0^{(1)}, g_0^{(1)}, h_0^{(1)}, h_1^{(1)}\big)$ and $d_2 = \big(u_0^{(2)}, g_0^{(2)}, h_0^{(2)}, h_1^{(2)}\big)$ in $B_R(d)$, let $u_1$ and $u_2$ be the solutions obtained as fixed points of the maps $u\mapsto \Phi_1(u)$ and $u \mapsto \Phi_2(u)$ respectively, which are defined by \eqref{Phi} but with data $d_1$ and $d_2$ respectively. Note that these solutions exist somewhere within the balls $\overline{B_{\varrho_1}(0)}$ and $\overline{B_{\varrho_2}(0)}$, respectively, where
\eq{
\varrho_j = 2 c(s, T_j) \norm{d_j}_{\mathcal D_{T_j}},
\quad
j=1, 2,
} 
and the lifespans $T_j \in (0, T]$ satisfy the conditions \eqref{T-into} and \eqref{T-contr}. Since $d_j \in B_R (d)$, 
\eqs{\norm{d_j}_{\mathcal D_{T_j}} = \big\|u_0^{(j)}\big\|_{H^s(0, \ell)} + \big\|g_0^{(j)}\big\|_{H^{\frac{s+1}{3}}(0, T_j)} + \big\|h_0^{(j)}\big\|_{H^{\frac{s+1}{3}}(0, T_j)} + \big\|h_1^{(j)}\big\|_{H^{\frac{s}{3}}(0, T_j)} < \norm{d}_{\mathcal D_T} + R,}
or, in other words, using the boundedness of $c(s, T_j)$ as $T_j \to 0^+$,
\eq{\varrho_j < 2 \sup_{t \in (0, T]} \b{c(s, t)} \p{\norm{d}_{\mathcal D_T} + R} = 2 c(s) \p{\norm{d}_{\mathcal D_T} + R} =: \varrho_c.}
Now, let $T_c \in (0, T]$ satisfy (referencing the constants in \eqref{T-into} and \eqref{T-contr})
\eq{
T_c < \frac{1}{4|\kappa|^2 C(s, \lambda)^2 \varrho_c^{2(\lambda-1)}}, \label{Tc-1}
\quad
C(s, \lambda) := \sup_{t \in (0, T]} \b{c_s c(s, t) + c(s, \lambda, t)},
}
and possibly other restrictions determined below. Since $\varrho_c > \max\{\varrho_1, \varrho_2\}$ and $\lambda > 1$, it follows that $T_c \le \min\{T_1, T_2\}$ so that $u_1, u_2$ exist for all $t \in [0, T_c]$ and, crucially, $T_c$ is independent of $d_1, d_2 \in B_R (d)$. Consider the associated solution space $X_{T_c}^s := C_t([0, T_c]; H_x^s(0, \ell)) \cap L_t^2((0, T_c); H_x^{s+1}(0, \ell))$ and note that $X_{T_1}^s, X_{T_2}^s \subset X_{T_c}^s$. 
Then, since the solutions $u_1$ and $u_2$ have been obtained as fixed points of $u\mapsto \Phi_1(u)$ and $u\mapsto \Phi_2(u)$ in the balls $\overline{B_{\varrho_1}(0)} \subset X_{T_1}^s$ and $\overline{B_{\varrho_2}(0)} \subset X_{T_2}^s$, respectively, we have
\eqs{
\norm{u_1-u_2}_{X_{T_c}^s}
&=
\norm{\Phi_1(u_1)-\Phi_2(u_2)}_{X_{T_c}^s}
\nn\\
&=
\norm{S\big[u_0^{(1)}-u_0^{(2)}, g_0^{(1)}-g_0^{(2)}, h_0^{(1)}-h_0^{(2)}, h_1^{(1)}-h_1^{(2)}; \kappa \p{|u_1|^{\lambda-1}u_1-|u_2|^{\lambda-1}u_2}\big]}_{X_{T_c}^s}.
}
At this point, we employ the Sobolev estimates \eqref{sob-est} and \eqref{smooth-est} together with Lemma \ref{lemma:H-pow-diff} to infer
\eq{
\begin{aligned}
\norm{u_1-u_2}_{X_{T_c}^s}
\leq
c(s, \lambda, T_c) \Big[
&\big\|u_0^{(1)}-u_0^{(2)}\big\|_{H^s(0, \ell)}  
+ \big\|g_0^{(1)}-g_0^{(2)}\big\|_{H^{\frac{s+1}{3}}(0, T_c)}
+ \big\|h_0^{(1)}-h_0^{(2)}\big\|_{H^{\frac{s+1}{3}}(0, T_c)}
\\
&+ \big\|h_1^{(1)}-h_1^{(2)}\big\|_{H^{\frac{s}{3}}(0, T_c)}
+ \abs \kappa \sqrt{T_c} \p{\norm{u_1}_{X_{T_c}^s}^{\lambda-1} + \norm{u_2}_{X_{T_c}^s}^{\lambda-1}} \norm{u_1 - u_2}_{X_{T_c}^s}
\Big].
\end{aligned}
}
By rearranging and considering the fact that $u_1, u_2 \in B_{\varrho_c}(0) \subset X_{T_c}^s$, we deduce (after recalling the definition \eqref{d-norm-def})
\eq{\label{u1-u2-ineq}
\b{1 - 2 \abs \kappa c(s, \lambda, T_c) \sqrt{T_c} \varrho_c^{\lambda - 1}} \norm{u_1 - u_2}_{X_{T_c}^s}
\leq
c(s, \lambda, T_c) \norm{d_1-d_2}_{\mathcal D_{T_c}}.
}
Thus, if $T_c$ is small enough such that it not only satisfies \eqref{Tc-1} but also
\eq{
\abs \kappa c(s, \lambda, T_c) \sqrt{T_c} \varrho_c^{\lambda - 1}
&< \frac{1}{2},
\label{Tc-cts}
}
then we can rearrange inequality \eqref{u1-u2-ineq} to
\eq{
\norm{u_1 - u_2}_{X_{T_c}^s}
\leq
\frac{c(s, \lambda, T_c)}{1 - 2 \abs \kappa c(s, \lambda, T_c) \sqrt{T_c} \varrho_c^{\lambda - 1}} \norm{d_1-d_2}_{\mathcal D_{T_c}},
}
which shows that the unique solution $u$ to \eqref{hnls-ibvp} is locally Lipschitz and so, in particular, it depends continuously on the data. Hence, we conclude that the HNLS finite interval problem \eqref{hnls-ibvp} is well-posed in the Hadamard sense, completing the proof of the high-regularity  Theorem \ref{high-wp-t}.

\subsection{Low-regularity well-posedness}

For $0\leq s < \frac 12$ and lifespan $T>0$ yet to be determined, define the solution space
\eq{\label{Y-def}
Y_T^{s, \lambda} := C_t([0, T]; H_x^s(0, \ell)) \cap L_t^2((0, T); H_x^{s+1}(0, \ell)) \cap L_t^{\frac{6\lambda}{(1-2s)(\lambda-1)}}((0, T); H_x^{s, \frac{2\lambda}{1+2(\lambda-1)s}}(0, \ell))
}
with norm equal to the maximum of the norms of the three spaces involved. 
Note that the particular choice of $(q, p) = \big(\frac{6\lambda}{(1-2s)(\lambda-1)}, \frac{2\lambda}{1+2(\lambda-1)s}\big)$, which is indeed admissible in the sense of~\eqref{adm-pair}, is dictated by the nonlinear estimates of Lemma \ref{qnls-non-l} below.
Using the Sobolev estimate \eqref{smooth-est}, as well as the Strichartz estimate \eqref{strich-est} twice, once for $(q, p) = (\infty, 2)$ (this choice gives the analogue of the Sobolev estimate \eqref{sob-est} for $0\leq s < \frac 12$)  and once more for $(q, p) = \big(\frac{6\lambda}{(1-2s)(\lambda-1)}, \frac{2\lambda}{1+2(\lambda-1)s}\big)$, we infer that the iteration map \eqref{Phi} satisfies
\eq{
\begin{split}
\norm{\Phi(u)}_{Y_T^{s, \lambda}}
\leq c(s, \lambda, T) \Big(&\norm{u_0}_{H^s (0, \ell)} + \norm{g_0}_{H^{\frac{s + 1}{3}} (0, T)} + \norm{h_0}_{H^{\frac{s + 1}{3}} (0, T)} + \norm{h_1}_{H^{\frac{s}{3}} (0, T)}
\\
&+ \abs \kappa \norm{|u|^{\lambda - 1} u}_{L_t^1 ((0, T); H_x^s (0,\ell))}
\Big),
\end{split}
\label{lin-est-low3}
}
where $c(s, \lambda, T) := \max\big\{c_2(s, T), c_3(s, 2, T), c_3(s, \frac{2\lambda}{1+2(\lambda-1)s}, T)\big\}$.
At this point, we remark that the algebra property in $H_x^s(0, \ell)$, which was employed for simplifying the nonlinear term in the high-regularity case, is no longer available (since $s<\frac 12$). Instead, we resort to the following fundamental nonlinear estimates:
\begin{lemma}\label{qnls-non-l}
Suppose $0\leq s < \frac{1}{2}$, $2 \leq \lambda \leq \frac{7-2s}{1-2s}$ and $(q, p) = \big(\frac{6 \lambda}{(1 - 2 s) (\lambda - 1)}, \frac{2 \lambda}{1 + 2 (\lambda - 1) s}\big)$. Then, 
\begin{align}\label{non-est}
\left\| |\varphi|^{\lambda-1} \varphi \right\|_{L_t^1((0, T); H_x^s(\R))}
&\leq 
c(s, \lambda)
T^{\frac{q-\lambda}{q}} \left\| \varphi \right\|_{L_t^q((0, T); H_x^{s, p}(\R))}^\lambda,
\\
\label{non-diff-est}
\left\| |\varphi|^{\lambda-1} \varphi - |\psi|^{\lambda-1} \psi \right\|_{L_t^1((0, T); H_x^s(\R))}
&\leq
c(s, \lambda)
T^{\frac{q-\lambda}{q}}
\left(
\left\| \varphi \right\|_{L_t^q((0, T); H_x^{s, p}(\R))}^{\lambda-1}
+
\left\| \psi \right\|_{L_t^q((0, T); H_x^{s, p}(\R))}^{\lambda-1}
\right) 
\nonumber\\
&\quad
\cdot
\left\| \varphi - \psi \right\|_{L_t^q((0, T); H_x^{s, p}(\R))}.
\end{align}
\end{lemma}

The techniques leading to the proof of Lemma \ref{qnls-non-l} are standard in the literature related to the NLS Cauchy problem (see, for example, the proof of Proposition 3.1 in \cite{mo2024} for the two-dimensional analogue of this result). As noted above, the particular choice of  exponents in the Strichartz space component of \eqref{Y-def} is dictated precisely by Lemma \ref{qnls-non-l}. 

Thanks to the nonlinear estimate \eqref{non-est}, estimate \eqref{lin-est-low3} yields
\eq{\label{424f}
\begin{aligned}
\norm{\Phi(u)}_{Y_T^{s, \lambda}}
\leq c(s, \lambda, T) \Big(&\norm{u_0}_{H^s (0, \ell)} + \norm{g_0}_{H^{\frac{s + 1}{3}} (0, T)} + \norm{h_0}_{H^{\frac{s + 1}{3}} (0, T)} + \norm{h_1}_{H^{\frac{s}{3}} (0, T)}
\\
&+ \abs \kappa c(s, \lambda) T^{\frac{7 - \lambda + 2 s (\lambda - 1)}{6}} \norm{u}_{Y_T^{s, \lambda}}^\lambda
\Big).
\end{aligned}
}
Defining the radius
\eq{\varrho = \varrho(T) := 2 c(s, \lambda, T) \Big(\norm{u_0}_{H^s (0, \ell)} + \norm{g_0}_{H^{\frac{s + 1}{3}} (0, T)} + \norm{h_0}_{H^{\frac{s + 1}{3}} (0, T)} + \norm{h_1}_{H^{\frac{s}{3}} (0, T)}\Big), \label{low-rho}}
and assuming that $\varrho>0$ (if $\varrho=0$, then we have the zero solution), we consider the ball $B_\varrho (0) \subset Y_T^{s, \lambda}$ of radius $\varrho$ centered at the origin. For any $u \in \overline{B_\varrho (0)}$, estimate \eqref{424f} implies
\eq{
\norm{\Phi(u)}_{Y_T^{s, \lambda}}
\le \frac{\varrho}{2} + \abs \kappa c(s, \lambda) c(s, \lambda, T) T^{\frac{7 - \lambda + 2 s (\lambda - 1)}{6}} \varrho^\lambda.
}
Thus, as long as $\lambda < \frac{7 - 2 s}{1 - 2 s}$, we can always choose $T > 0$ small enough such that
\eq{
\abs \kappa c(s, \lambda) c(s, \lambda, T) T^{\frac{7 - \lambda + 2 s (\lambda - 1)}{6}} \varrho^{\lambda - 1} \le \frac{1}{2},
\label{low-T-into}
}
a condition which guarantees that $\Phi$ maps $\overline{B_\varrho (0)}$ into itself.
On the other hand, in the critical case $\lambda = \frac{7 - 2 s}{1 - 2 s}$ the power of $T$ disappears from~\eqref{low-T-into}. In that case, in order to ensure that $\Phi(u)$ remains in $\overline{B_\varrho (0)}$, we make $\varrho$ small enough by taking sufficiently small initial data norm $\norm{u_0}_{H^s (0, \ell)}$ and also by choosing $T > 0$ small enough to control the norms of the boundary data in \eqref{low-rho} so that~\eqref{low-T-into} is satisfied.

Next, suppose that $u_1, u_2 \in \overline{B_\varrho (0)}$. Then, proceeding similarly to the path that resulted in \eqref{lin-est-low3}, we have
\eq{
\norm{\Phi(u_1) - \Phi(u_2)}_{Y_T^{s, \lambda}}
&\leq \abs \kappa c(s, \lambda, T) \norm{|u_1|^{\lambda - 1} u_1 - |u_2|^{\lambda - 1} u_2}_{L_t^1 ((0, T); H_x^s (0,\ell))}.
}
Thus, in view of the nonlinear estimate \eqref{non-diff-est} and the fact that $u_1, u_2 \in \overline{B_\varrho (0)}$,  
\eq{
\norm{\Phi(u_1) - \Phi(u_2)}_{Y_T^{s, \lambda}}
\leq 2 \abs \kappa c(s, \lambda) c(s, \lambda, T) \, T^{\frac{7 - \lambda + 2 s (\lambda - 1)}{6}} \varrho^{\lambda - 1} \norm{u_1 - u_2}_{Y_T^{s, \lambda}}
}
so $\Phi$ is a contraction on $\overline{B_\varrho (0)}$ provided that
\eq{
\abs{\kappa} c(s, \lambda) c(s, \lambda, T) T^{\frac{7 - \lambda + 2 s (\lambda - 1)}{6}} \varrho^{\lambda - 1} < \frac{1}{2}.
\label{low-T-contr}
}
As before, for $2 \le \lambda \le \frac{7 - 2 s}{1 - 2 s}$ we can choose $T > 0$ sufficiently small so that \eqref{low-T-contr} is satisfied, taking $\norm{u_0}_{H^s(0, \ell)}$ to be small enough in the critical case $\lambda = \frac{7 - 2 s}{1 - 2 s}$. Note that the contraction condition \eqref{low-T-contr} is stronger than the into condition \eqref{low-T-into} due to the fact that the constant involved is the same in both cases (in contrast to the high-regularity case where the constants are different).

By Banach's fixed point theorem, if $T > 0$ (and, in the critical case $\lambda = \frac{7 - 2 s}{1 - 2 s}$, the norm $\norm{u_0}_{H^s (0, \ell)}$) is small enough such that  \eqref{low-T-contr} is satisfied, then there exists a unique fixed point of the map $u \mapsto \Phi(u)$ in $\overline{B_\varrho (0)}$, amounting to a unique solution of the HNLS finite interval problem \eqref{hnls-ibvp} in $\overline{B_\varrho (0)}$.

The argument used in the high-regularity case to extend uniqueness to the entire solution space  $C_t([0, T]; H_x^s(0, \ell))$ can be adapted to obtain the analogous result in the low-regularity setting and the space $C_t([0, T]; H_x^s(0, \ell)) \cap L_t^{\frac{6\lambda}{(1-2s)(\lambda-1)}}((0, T); H_x^{s, \frac{2\lambda}{1+2(\lambda-1)s}}(0, \ell))$ after appropriate mollification along the lines of Section~8 in \cite{h2005}. Moreover, the local Lipschitz continuity in the low-regularity setting can be proved in the same way as in the high-regularity case, with Lemma \ref{qnls-non-l} replaced by Lemma \ref{lemma:H-pow-diff}. Hence, Theorem~\ref{low-wp-t} on the Hadamard well-posedness of the HNLS finite interval problem \eqref{hnls-ibvp} in the low-regularity setting has been established.
\\[2mm]
\noindent
\noindent
\textit{Funding and/or Conflicts of interests/Competing interests.} The authors have no relevant financial or non-financial interests to disclose.  
\textit{Data availability.} Data sharing is not applicable to this article as no datasets were generated or analyzed during the current study.

\bibliographystyle{myamsalpha}
\bibliography{references.bib}

@article {ccg2020,
    AUTHOR = {Capistrano-Filho, Roberto de A. and Cavalcante, M\'arcio and
              Gallego, Fernando A.},
     TITLE = {Lower regularity solutions of the biharmonic {S}chr\"odinger
              equation in a quarter plane},
   JOURNAL = {Pacific J. Math.},
  FJOURNAL = {Pacific Journal of Mathematics},
    VOLUME = {309},
      YEAR = {2020},
    NUMBER = {1},
     PAGES = {35--70},
      ISSN = {0030-8730,1945-5844},
   MRCLASS = {35C15 (35G15 35G30 35Q55)},
  MRNUMBER = {4202004},
       DOI = {10.2140/pjm.2020.309.35},
       URL = {https://doi.org/10.2140/pjm.2020.309.35},
}

@article {c2017,
    AUTHOR = {Cavalcante, M\'arcio},
     TITLE = {The initial boundary value problem for some quadratic
              nonlinear {S}chr\"odinger equations on the half-line},
   JOURNAL = {Differential Integral Equations},
  FJOURNAL = {Differential and Integral Equations. An International Journal
              for Theory \& Applications},
    VOLUME = {30},
      YEAR = {2017},
    NUMBER = {7-8},
     PAGES = {521--554},
      ISSN = {0893-4983},
   MRCLASS = {35Q55 (35B30)},
  MRNUMBER = {3646462},
       URL = {http://projecteuclid.org/euclid.die/1493863393},
}

@article {cc2020,
    AUTHOR = {Cavalcante, M\'arcio and Corcho, Ad\'an J.},
     TITLE = {Well-posedness and lower bounds of the growth of weighted
              norms for the {S}chr\"odinger--{K}orteweg--de~{V}ries
              interactions on the half-line},
   JOURNAL = {J. Evol. Equ.},
  FJOURNAL = {Journal of Evolution Equations},
    VOLUME = {20},
      YEAR = {2020},
    NUMBER = {4},
     PAGES = {1563--1596},
      ISSN = {1424-3199,1424-3202},
   MRCLASS = {35Q53 (35B65 35Q55)},
  MRNUMBER = {4181959},
MRREVIEWER = {Zhaohui\ Huo},
       DOI = {10.1007/s00028-020-00566-1},
       URL = {https://doi.org/10.1007/s00028-020-00566-1},
}

@article {ck2020,
    AUTHOR = {Cavalcante, M\'arcio and Kwak, Chulkwang},
     TITLE = {The initial-boundary value problem for the {K}awahara equation
              on the half-line},
   JOURNAL = {NoDEA Nonlinear Differential Equations Appl.},
  FJOURNAL = {NoDEA. Nonlinear Differential Equations and Applications},
    VOLUME = {27},
      YEAR = {2020},
    NUMBER = {5},
     PAGES = {Paper No. 45, 50},
      ISSN = {1021-9722,1420-9004},
   MRCLASS = {35Q53 (35G31)},
  MRNUMBER = {4130241},
       DOI = {10.1007/s00030-020-00648-6},
       URL = {https://doi.org/10.1007/s00030-020-00648-6},
}

@article {hkmms2024,
    AUTHOR = {Hennig, Dirk and Karachalios, Nikos I. and Mantzavinos,
              Dionyssios and Cuevas-Maraver, Jes\'{u}s and Stratis, Ioannis
              G.},
     TITLE = {On the proximity between the wave dynamics of the integrable focusing nonlinear {S}chr\"{o}dinger equation and its non-integrable generalizations},
   JOURNAL = {J. Differential Equations},
  FJOURNAL = {Journal of Differential Equations},
    VOLUME = {397},
      YEAR = {2024},
     PAGES = {106--165},
      ISSN = {0022-0396,1090-2732},
   MRCLASS = {35Q55 (35B35 37K40)},
  MRNUMBER = {4720515},
       DOI = {10.1016/j.jde.2024.03.005},
       URL = {https://doi.org/10.1016/j.jde.2024.03.005},
}

@article {hmy2019-kdv,
    AUTHOR = {Himonas, A. Alexandrou and Mantzavinos, Dionyssios and Yan,
              Fangchi},
     TITLE = {The {K}orteweg--de {V}ries equation on an interval},
   JOURNAL = {J. Math. Phys.},
  FJOURNAL = {Journal of Mathematical Physics},
    VOLUME = {60},
      YEAR = {2019},
    NUMBER = {5},
     PAGES = {051507, 26},
      ISSN = {0022-2488,1089-7658},
   MRCLASS = {35Q53},
  MRNUMBER = {3947621},
MRREVIEWER = {Anthony\ D.\ Osborne},
       DOI = {10.1063/1.5080366},
       URL = {https://doi.org/10.1063/1.5080366},
}

@article{amo2024,
	AUTHOR = {Alkin, Aykut and Mantzavinos, Dionyssios and Ozsari, Turker},
	TITLE = {Local well-posedness of the higher-order nonlinear {S}chrödinger equation on the half-line: Single-boundary condition case},
	JOURNAL = {Studies in Applied Mathematics},
	VOLUME = {152},
	YEAR = {2024},
	NUMBER = {1},
	PAGES = {203-248},
	URL = {https://onlinelibrary.wiley.com/doi/abs/10.1111/sapm.12642},
}

@article {boy2021,
	AUTHOR = {Batal, Ahmet and Ozsari , Turker and Yilmaz, Kemal Cem},
	TITLE = {Stabilization of higher order {S}chr\"{o}dinger equations on a
	finite interval: {P}art {I}},
	JOURNAL = {Evol. Equ. Control Theory},
	FJOURNAL = {Evolution Equations and Control Theory},
	VOLUME = {10},
	YEAR = {2021},
	NUMBER = {4},
	PAGES = {861--919},
	ISSN = {2163-2472},
	MRCLASS = {93D15 (35Q55 35Q93 93C20)},
	MRNUMBER = {4338836},
	MRREVIEWER = {Francesco Ferrante},
	DOI = {10.3934/eect.2020095},
	URL = {https://doi.org/10.3934/eect.2020095},
}

@article {bfo2020,
	AUTHOR = {Batal, A. and Fokas, A. S. and Ozsari , T.},
	TITLE = {Fokas method for linear boundary value problems involving
	mixed spatial derivatives},
	JOURNAL = {Proc. A.},
	FJOURNAL = {Proceedings A},
	VOLUME = {476},
	YEAR = {2020},
	NUMBER = {2239},
	PAGES = {20200076, 15},
	ISSN = {1364-5021},
	MRCLASS = {35K20 (35B30 35Q41)},
	MRNUMBER = {4133772},
	DOI = {10.1098/rspa.2020.0076},
	URL = {https://doi.org/10.1098/rspa.2020.0076},
}

@article {bbv2007,
	AUTHOR = {Bisognin, Eleni and Bisognin, Vanilde and Vera Villagr\'{a}n,
	Octavio Paulo},
	TITLE = {Stabilization of solutions to higher-order nonlinear
	{S}chr\"{o}dinger equation with localized damping},
	JOURNAL = {Electron. J. Differential Equations},
	FJOURNAL = {Electronic Journal of Differential Equations},
	YEAR = {2007},
	PAGES = {No. 06, 18},
	ISSN = {1072-6691},
	MRCLASS = {35Q53 (35B35 93D15)},
	MRNUMBER = {2278420},
}

@article {ccs2019,
    AUTHOR = {Cavalcanti, Marcelo M. and Corr\^ea, Wellington J. and
              Sep\'ulveda, Mauricio A. and V\'ejar-Asem, Rodrigo},
     TITLE = {Finite difference scheme for a higher order nonlinear
              {S}chr\"odinger equation},
   JOURNAL = {Calcolo},
  FJOURNAL = {Calcolo. A Quarterly on Numerical Analysis and Theory of
              Computation},
    VOLUME = {56},
      YEAR = {2019},
    NUMBER = {4},
     PAGES = {Paper No. 40, 32},
      ISSN = {0008-0624,1126-5434},
   MRCLASS = {65M06 (35C08 35Q55)},
  MRNUMBER = {4019729},
       DOI = {10.1007/s10092-019-0336-1},
       URL = {https://doi.org/10.1007/s10092-019-0336-1},
}

@article {cpv2005,
	AUTHOR = {Ceballos, Juan Carlos and Pavez, Ricardo and Vera
	Villagr\'{a}n, Octavio Paulo},
	TITLE = {Exact boundary controllability for higher order nonlinear
	{S}chr\"{o}dinger equations with constant coefficients},
	JOURNAL = {Electron. J. Differential Equations},
	FJOURNAL = {Electronic Journal of Differential Equations},
	YEAR = {2005},
	PAGES = {No. 122, 31},
	ISSN = {1072-6691},
	MRCLASS = {93B05 (35Q55 93C20)},
	MRNUMBER = {2174554},
	MRREVIEWER = {Bing-Yu Zhang},
}

@article {f2023,
    AUTHOR = {Faminskii, Andrei V.},
     TITLE = {The higher order nonlinear {S}chr\"{o}dinger equation with
              quadratic nonlinearity on the real axis},
   JOURNAL = {Adv. Differential Equations},
  FJOURNAL = {Advances in Differential Equations},
    VOLUME = {28},
      YEAR = {2023},
    NUMBER = {5-6},
     PAGES = {413--466},
      ISSN = {1079-9389},
   MRCLASS = {35Q53 (35Q55)},
  MRNUMBER = {4555001},
}

@article {f1997,
	AUTHOR = {Fokas, A. S.},
	TITLE = {A unified transform method for solving linear and certain
	nonlinear {PDE}s},
	JOURNAL = {Proc. Roy. Soc. London Ser. A},
	FJOURNAL = {Proceedings of the Royal Society. London. Series A.
	Mathematical, Physical and Engineering Sciences},
	VOLUME = {453},
	YEAR = {1997},
	NUMBER = {1962},
	PAGES = {1411--1443},
	ISSN = {0962-8444},
	MRCLASS = {35A22 (34A55 34L25 35Q53 35Q55)},
	MRNUMBER = {1469927},
	DOI = {10.1098/rspa.1997.0077},
	URL = {https://doi.org/10.1098/rspa.1997.0077},
}

@book {f2008,
	AUTHOR = {Fokas, Athanassios S.},
	TITLE = {A unified approach to boundary value problems},
	SERIES = {CBMS-NSF Regional Conference Series in Applied Mathematics},
	VOLUME = {78},
	PUBLISHER = {Society for Industrial and Applied Mathematics (SIAM),
	Philadelphia, PA},
	YEAR = {2008},
	PAGES = {xvi+336},
	ISBN = {978-0-898716-51-1},
	MRCLASS = {35C15 (35Q15 35Q53 35Q55 37K15 42A38 44A12)},
	MRNUMBER = {2451953},
	MRREVIEWER = {R. G. Airapetyan},
	DOI = {10.1137/1.9780898717068},
	URL = {https://doi.org/10.1137/1.9780898717068},
}

@article {kai2013,
    AUTHOR = {Kaikina, Elena I.},
     TITLE = {Inhomogeneous {N}eumann initial-boundary value problem for the
              nonlinear {S}chr\"odinger equation},
   JOURNAL = {J. Differential Equations},
  FJOURNAL = {Journal of Differential Equations},
    VOLUME = {255},
      YEAR = {2013},
    NUMBER = {10},
     PAGES = {3338--3356},
      ISSN = {0022-0396,1090-2732},
   MRCLASS = {35Q55 (35A01 35B40)},
  MRNUMBER = {3093366},
MRREVIEWER = {Mahendra\ Panthee},
       DOI = {10.1016/j.jde.2013.07.036},
       URL = {https://doi.org/10.1016/j.jde.2013.07.036},
}

@article {fhm2017,
    AUTHOR = {Fokas, Athanassios S. and Himonas, A. Alexandrou and Mantzavinos, Dionyssios},
     TITLE = {The nonlinear {S}chr\"{o}dinger equation on the half-line},
   JOURNAL = {Trans. Amer. Math. Soc.},
  FJOURNAL = {Transactions of the American Mathematical Society},
    VOLUME = {369},
      YEAR = {2017},
    NUMBER = {1},
     PAGES = {681--709},
      ISSN = {0002-9947},
   MRCLASS = {35Q55 (35G16 35G31)},
  MRNUMBER = {3557790},
MRREVIEWER = {Peter E. Zhidkov},
       DOI = {10.1090/tran/6734},
       URL = {https://doi-org.www2.lib.ku.edu/10.1090/tran/6734},
}

@article {cb1991,
    AUTHOR = {Carroll, Robert and Bu, Qiyue},
     TITLE = {Solution of the forced nonlinear {S}chr\"odinger ({NLS})
              equation using {PDE} techniques},
   JOURNAL = {Appl. Anal.},
  FJOURNAL = {Applicable Analysis. An International Journal},
    VOLUME = {41},
      YEAR = {1991},
    NUMBER = {1-4},
     PAGES = {33--51},
      ISSN = {0003-6811,1563-504X},
   MRCLASS = {35Q55 (58D07 81Q05)},
  MRNUMBER = {1103847},
MRREVIEWER = {Jagannathan\ Gomatam},
       DOI = {10.1080/00036819108840015},
       URL = {https://doi.org/10.1080/00036819108840015},
}

@article {sb2001,
    AUTHOR = {Strauss, Walter and Bu, Charles},
     TITLE = {An inhomogeneous boundary value problem for nonlinear
              {S}chr\"odinger equations},
   JOURNAL = {J. Differential Equations},
  FJOURNAL = {Journal of Differential Equations},
    VOLUME = {173},
      YEAR = {2001},
    NUMBER = {1},
     PAGES = {79--91},
      ISSN = {0022-0396,1090-2732},
   MRCLASS = {35Q55 (35A05)},
  MRNUMBER = {1836245},
MRREVIEWER = {Yvan\ Martel},
       DOI = {10.1006/jdeq.2000.3871},
       URL = {https://doi.org/10.1006/jdeq.2000.3871},
}

@article {bo2016,
	AUTHOR = {Batal, Ahmet and Ozsari , Turker},
	TITLE = {Nonlinear {S}chr\"{o}dinger equations on the half-line with
	nonlinear boundary conditions},
	JOURNAL = {Electron. J. Differential Equations},
	FJOURNAL = {Electronic Journal of Differential Equations},
	YEAR = {2016},
	PAGES = {Paper No. 222, 20},
	MRCLASS = {35Q55 (35B30)},
	MRNUMBER = {3547411},
}

@article {bsz2003,
    AUTHOR = {Bona, Jerry L. and Sun, Shu Ming and Zhang, Bing-Yu},
     TITLE = {A nonhomogeneous boundary-value problem for the {K}orteweg-de
              {V}ries equation posed on a finite domain},
   JOURNAL = {Comm. Partial Differential Equations},
  FJOURNAL = {Communications in Partial Differential Equations},
    VOLUME = {28},
      YEAR = {2003},
    NUMBER = {7-8},
     PAGES = {1391--1436},
      ISSN = {0360-5302,1532-4133},
   MRCLASS = {35Q53 (35B30)},
  MRNUMBER = {1998942},
MRREVIEWER = {Woodford\ W.\ Zachary},
       DOI = {10.1081/PDE-120024373},
       URL = {https://doi.org/10.1081/PDE-120024373},
}

@article {bsz2018,
	AUTHOR = {Bona, Jerry L. and Sun, Shu-Ming and Zhang, Bing-Yu},
	TITLE = {Nonhomogeneous boundary-value problems for one-dimensional
	nonlinear {S}chr\"{o}dinger equations},
	JOURNAL = {J. Math. Pures Appl. (9)},
	FJOURNAL = {Journal de Math\'{e}matiques Pures et Appliqu\'{e}es. Neuvi\`eme S\'{e}rie},
	VOLUME = {109},
	YEAR = {2018},
	PAGES = {1--66},
	ISSN = {0021-7824},
	MRCLASS = {35Q55 (35Q35)},
	MRNUMBER = {3734975},
	MRREVIEWER = {Noel Frederick Smyth},
	DOI = {10.1016/j.matpur.2017.11.001},
	URL = {https://doi.org/10.1016/j.matpur.2017.11.001},
}

@article {cl2003,
	AUTHOR = {Carvajal, X. and Linares, F.},
	TITLE = {A higher-order nonlinear {S}chr\"{o}dinger equation with variable
	coefficients},
	JOURNAL = {Differential Integral Equations},
	FJOURNAL = {Differential and Integral Equations. An International Journal
	for Theory \& Applications},
	VOLUME = {16},
	YEAR = {2003},
	NUMBER = {9},
	PAGES = {1111--1130},
	ISSN = {0893-4983},
	MRCLASS = {35Q55 (35B30)},
	MRNUMBER = {1989544},
	MRREVIEWER = {Takayoshi Ogawa},
}

@article {c2004,
	AUTHOR = {Carvajal, Xavier},
	TITLE = {Local well-posedness for a higher order nonlinear
	{S}chr\"{o}dinger equation in {S}obolev spaces of negative
	indices},
	JOURNAL = {Electron. J. Differential Equations},
	FJOURNAL = {Electronic Journal of Differential Equations},
	YEAR = {2004},
	PAGES = {No. 13, 10},
	ISSN = {1072-6691},
	MRCLASS = {35Q55 (35B30)},
	MRNUMBER = {2036197},
}

@article {c2006,
	AUTHOR = {Carvajal, Xavier},
	TITLE = {Sharp global well-posedness for a higher order {S}chr\"{o}dinger
	equation},
	JOURNAL = {J. Fourier Anal. Appl.},
	FJOURNAL = {The Journal of Fourier Analysis and Applications},
	VOLUME = {12},
	YEAR = {2006},
	NUMBER = {1},
	PAGES = {53--70},
	ISSN = {1069-5869},
	MRCLASS = {35Q55 (35B30)},
	MRNUMBER = {2215677},
	MRREVIEWER = {W. Watzlawek},
	DOI = {10.1007/s00041-005-5028-3},
	URL = {https://doi-org.libezproxy.iyte.edu.tr/10.1007/s00041-005-5028-3},
}

@article {chen2018,
	AUTHOR = {Chen, Mo},
	TITLE = {Stabilization of the higher order nonlinear {S}chr\"{o}dinger
	equation with constant coefficients},
	JOURNAL = {Proc. Indian Acad. Sci. Math. Sci.},
	FJOURNAL = {Indian Academy of Sciences. Proceedings. Mathematical
	Sciences},
	VOLUME = {128},
	YEAR = {2018},
	NUMBER = {3},
	PAGES = {Art. 39, 15},
	ISSN = {0253-4142},
	MRCLASS = {93D15 (35Q53 35Q55 93C20)},
	MRNUMBER = {3814780},
	DOI = {10.1007/s12044-018-0410-7},
	URL = {https://doi-org.libezproxy.iyte.edu.tr/10.1007/s12044-018-0410-7},
}

@article {kod2024,
    AUTHOR = {Kalimeris, Konstantinos and Ozsari, Turker and Dikaios,
              Nikolaos},
     TITLE = {Numerical computation of {N}eumann controls for the heat
              equation on a finite interval},
   JOURNAL = {IEEE Trans. Automat. Control},
  FJOURNAL = {Institute of Electrical and Electronics Engineers.
              Transactions on Automatic Control},
    VOLUME = {69},
      YEAR = {2024},
    NUMBER = {1},
     PAGES = {161--173},
      ISSN = {0018-9286,1558-2523},
   MRCLASS = {93C20 (93-08)},
  MRNUMBER = {4703061},
       DOI = {10.1109/tac.2023.3263753},
       URL = {https://doi.org/10.1109/tac.2023.3263753},
}

@article {ozs2015,
    AUTHOR = {Ozsari, Turker},
     TITLE = {Well-posedness for nonlinear {S}chr\"odinger equations with
              boundary forces in low dimensions by {S}trichartz estimates},
   JOURNAL = {J. Math. Anal. Appl.},
  FJOURNAL = {Journal of Mathematical Analysis and Applications},
    VOLUME = {424},
      YEAR = {2015},
    NUMBER = {1},
     PAGES = {487--508},
      ISSN = {0022-247X,1096-0813},
   MRCLASS = {35Q55 (35B30)},
  MRNUMBER = {3286575},
       DOI = {10.1016/j.jmaa.2014.11.034},
       URL = {https://doi.org/10.1016/j.jmaa.2014.11.034},
}

@article {ok2023,
    AUTHOR = {\"Ozsar{\i}, Turker and Kalimeris, Konstantinos},
     TITLE = {Existence of unattainable states for {S}chr\"odinger type
              flows on the half-line},
   JOURNAL = {IMA J. Math. Control Inform.},
  FJOURNAL = {IMA Journal of Mathematical Control and Information},
    VOLUME = {40},
      YEAR = {2023},
    NUMBER = {4},
     PAGES = {789--803},
      ISSN = {0265-0754,1471-6887},
   MRCLASS = {93B03 (93B05 93C20)},
  MRNUMBER = {4681262},
       DOI = {10.1093/imamci/dnad032},
       URL = {https://doi.org/10.1093/imamci/dnad032},
}

@article {mz2001,
    AUTHOR = {Micu, Sorin and Zuazua, Enrique},
     TITLE = {On the lack of null-controllability of the heat equation on
              the half-line},
   JOURNAL = {Trans. Amer. Math. Soc.},
  FJOURNAL = {Transactions of the American Mathematical Society},
    VOLUME = {353},
      YEAR = {2001},
    NUMBER = {4},
     PAGES = {1635--1659},
      ISSN = {0002-9947,1088-6850},
   MRCLASS = {93B05 (35B37 35K05 93C20)},
  MRNUMBER = {1806726},
MRREVIEWER = {Patrick\ Martinez},
       DOI = {10.1090/S0002-9947-00-02665-9},
       URL = {https://doi.org/10.1090/S0002-9947-00-02665-9},
}

@article {ko2020,
	AUTHOR = {Kalimeris, Konstantinos and Ozsari , Turker},
	TITLE = {An elementary proof of the lack of null controllability for
	the heat equation on the half line},
	JOURNAL = {Appl. Math. Lett.},
	FJOURNAL = {Applied Mathematics Letters. An International Journal of Rapid
	Publication},
	VOLUME = {104},
	YEAR = {2020},
	PAGES = {106241, 6},
	ISSN = {0893-9659},
	MRCLASS = {93B15 (35K05 35K20 93C20)},
	MRNUMBER = {4056215},
	MRREVIEWER = {Alain Brillard},
	DOI = {10.1016/j.aml.2020.106241},
	URL = {https://doi.org/10.1016/j.aml.2020.106241},
}

@article {h2005,
	AUTHOR = {Holmer, Justin},
	TITLE = {The initial-boundary-value problem for the 1{D} nonlinear
	{S}chr\"{o}dinger equation on the half-line},
	JOURNAL = {Differential Integral Equations},
	FJOURNAL = {Differential and Integral Equations. An International Journal
	for Theory \& Applications},
	VOLUME = {18},
	YEAR = {2005},
	NUMBER = {6},
	PAGES = {647--668},
	ISSN = {0893-4983},
	MRCLASS = {35Q55 (35B30)},
	MRNUMBER = {2136703},
	MRREVIEWER = {Markus Kunze},
}

@article {l1997,
	AUTHOR = {Laurey, Corinne},
	TITLE = {The {C}auchy problem for a third order nonlinear {S}chr\"{o}dinger
	equation},
	JOURNAL = {Nonlinear Anal.},
	FJOURNAL = {Nonlinear Analysis. Theory, Methods \& Applications. An
	International Multidisciplinary Journal},
	VOLUME = {29},
	YEAR = {1997},
	NUMBER = {2},
	PAGES = {121--158},
	ISSN = {0362-546X},
	MRCLASS = {35Q55},
	MRNUMBER = {1446222},
	MRREVIEWER = {Nakao Hayashi},
	DOI = {10.1016/S0362-546X(96)00081-8},
	URL = {https://doi-org.libezproxy.iyte.edu.tr/10.1016/S0362-546X(96)00081-8},
}

@article {oy2022,
	AUTHOR = {\"{O}zsar{\i} , Turker and Yilmaz, Kemal Cem},
	TITLE = {Stabilization of higher order {S}chr\"{o}dinger equations on a
	finite interval: {P}art {II}},
	JOURNAL = {Evol. Equ. Control Theory},
	FJOURNAL = {Evolution Equations and Control Theory},
	VOLUME = {11},
	YEAR = {2022},
	NUMBER = {4},
	PAGES = {1087--1148},
	ISSN = {2163-2472},
	MRCLASS = {93D15 (35Q55 35Q93 93C20)},
	MRNUMBER = {4455286},
	DOI = {10.3934/eect.2021037},
	URL = {https://doi.org/10.3934/eect.2021037},
}

@article {s1997,
	AUTHOR = {Staffilani, Gigliola},
	TITLE = {On the generalized {K}orteweg-de {V}ries-type equations},
	JOURNAL = {Differential Integral Equations},
	FJOURNAL = {Differential and Integral Equations. An International Journal
	for Theory \& Applications},
	VOLUME = {10},
	YEAR = {1997},
	NUMBER = {4},
	PAGES = {777--796},
	ISSN = {0893-4983},
	MRCLASS = {35A07 (35B30 35Q53 35Q55)},
	MRNUMBER = {1741772},
	MRREVIEWER = {Cornelia Schiebold},
}

@article {tak2000,
	AUTHOR = {Takaoka, Hideo},
	TITLE = {Well-posedness for the higher order nonlinear {S}chr\"{o}dinger
	equation},
	JOURNAL = {Adv. Math. Sci. Appl.},
	FJOURNAL = {Advances in Mathematical Sciences and Applications},
	VOLUME = {10},
	YEAR = {2000},
	NUMBER = {1},
	PAGES = {149--171},
	ISSN = {1343-4373},
	MRCLASS = {35Q55},
	MRNUMBER = {1769176},
	MRREVIEWER = {Tohru Ozawa},
}

@article {oy2019,
	AUTHOR = {\"{O}zsar{\i} , Turker and Yolcu, Nermin},
	TITLE = {The initial-boundary value problem for the biharmonic
	{S}chr\"{o}dinger equation on the half-line},
	JOURNAL = {Commun. Pure Appl. Anal.},
	FJOURNAL = {Communications on Pure and Applied Analysis},
	VOLUME = {18},
	YEAR = {2019},
	NUMBER = {6},
	PAGES = {3285--3316},
	ISSN = {1534-0392},
	MRCLASS = {35Q55 (35B30 35B45 35C15)},
	MRNUMBER = {3985385},
	DOI = {10.3934/cpaa.2019148},
	URL = {https://doi.org/10.3934/cpaa.2019148},
}

@article {ko2022,
	AUTHOR = {Köksal, Bilge and \"{O}zsar{\i} , Turker},
	TITLE = { The interior-boundary {S}trichartz estimate for the {S}chr\"odinger equation on the half line revisited},
	JOURNAL = {Turkish J. Math.},
	FJOURNAL = {Turkish Journal of Mathematics},
	VOLUME = {46},
	YEAR = {2022},
	NUMBER = {8},
	PAGES = {3323–3351},
	ISSN = {1300-0098},
	MRCLASS = {35A22, 35Q55, 35C15, 35B65, 35B4},
	MRNUMBER = {},
}

@book {s1994,
    AUTHOR = {Strichartz, Robert S.},
     TITLE = {A guide to distribution theory and {F}ourier transforms},
    SERIES = {Studies in Advanced Mathematics},
 PUBLISHER = {CRC Press, Boca Raton, FL},
      YEAR = {1994},
     PAGES = {x+213},
      ISBN = {0-8493-8273-4},
   MRCLASS = {42-01 (35-01 46-01 46Fxx)},
  MRNUMBER = {1276724},
MRREVIEWER = {J.\ Horv\'{a}th},
}

@article {hm2020,
    AUTHOR = {Himonas, A. Alexandrou and Mantzavinos, Dionyssios},
     TITLE = {Well-posedness of the nonlinear {S}chr\"{o}dinger equation on the
              half-plane},
   JOURNAL = {Nonlinearity},
  FJOURNAL = {Nonlinearity},
    VOLUME = {33},
      YEAR = {2020},
    NUMBER = {10},
     PAGES = {5567--5609},
      ISSN = {0951-7715},
   MRCLASS = {35Q55 (35B30)},
  MRNUMBER = {4151418},
MRREVIEWER = {Zhaohui Huo},
       DOI = {10.1088/1361-6544/ab9499},
       URL = {https://doi.org/10.1088/1361-6544/ab9499},
}

@book {lm1972,
    AUTHOR = {Lions, J.-L. and Magenes, E.},
     TITLE = {Non-homogeneous boundary value problems and applications.
              {V}ol. {I}},
    SERIES = {Die Grundlehren der mathematischen Wissenschaften, Band 181},
      NOTE = {Translated from the French by P. Kenneth},
 PUBLISHER = {Springer-Verlag, New York-Heidelberg},
      YEAR = {1972},
     PAGES = {xvi+357},
   MRCLASS = {35JXX (35KXX 35LXX 46E35)},
  MRNUMBER = {0350177},
}

@book {s1970,
    AUTHOR = {Stein, Elias M.},
     TITLE = {Singular integrals and differentiability properties of
              functions},
    SERIES = {Princeton Mathematical Series, No. 30},
 PUBLISHER = {Princeton University Press, Princeton, N.J.},
      YEAR = {1970},
     PAGES = {xiv+290},
   MRCLASS = {46.38 (26.00)},
  MRNUMBER = {0290095},
MRREVIEWER = {R. E. Edwards},
}

@book {g2014m,
    AUTHOR = {Grafakos, Loukas},
     TITLE = {Modern {F}ourier analysis},
    SERIES = {Graduate Texts in Mathematics},
    VOLUME = {250},
   EDITION = {Third},
 PUBLISHER = {Springer, New York},
      YEAR = {2014},
     PAGES = {xvi+624},
      ISBN = {978-1-4939-1229-2; 978-1-4939-1230-8},
   MRCLASS = {42-01 (42Bxx)},
  MRNUMBER = {3243741},
MRREVIEWER = {Atanas G. Stefanov},
       DOI = {10.1007/978-1-4939-1230-8},
       URL = {https://doi.org/10.1007/978-1-4939-1230-8},
}

@inproceedings {c1961,
    AUTHOR = {Calder\'{o}n, A.-P.},
     TITLE = {Lebesgue spaces of differentiable functions and distributions},
 BOOKTITLE = {Proc. {S}ympos. {P}ure {M}ath., {V}ol. {IV}},
     PAGES = {33--49},
 PUBLISHER = {American Mathematical Society, Providence, R.I.},
      YEAR = {1961},
   MRCLASS = {46.40 (46.38)},
  MRNUMBER = {0143037},
MRREVIEWER = {G. Marinescu},
}

@article {ck2002,
    AUTHOR = {Colliander, J. E. and Kenig, C. E.},
     TITLE = {The generalized {K}orteweg-de {V}ries equation on the half
              line},
   JOURNAL = {Comm. Partial Differential Equations},
  FJOURNAL = {Communications in Partial Differential Equations},
    VOLUME = {27},
      YEAR = {2002},
    NUMBER = {11-12},
     PAGES = {2187--2266},
      ISSN = {0360-5302},
   MRCLASS = {35Q53},
  MRNUMBER = {1944029},
MRREVIEWER = {Woodford W. Zachary},
       DOI = {10.1081/PDE-120016157},
       URL = {https://doi-org.www2.lib.ku.edu/10.1081/PDE-120016157},
}

@article {h2006,
    AUTHOR = {Holmer, Justin},
     TITLE = {The initial-boundary value problem for the {K}orteweg-de
              {V}ries equation},
   JOURNAL = {Comm. Partial Differential Equations},
  FJOURNAL = {Communications in Partial Differential Equations},
    VOLUME = {31},
      YEAR = {2006},
    NUMBER = {7-9},
     PAGES = {1151--1190},
      ISSN = {0360-5302},
   MRCLASS = {35Q53 (35B30)},
  MRNUMBER = {2254610},
MRREVIEWER = {Henrik Kalisch},
       DOI = {10.1080/03605300600718503},
       URL = {https://doi-org.www2.lib.ku.edu/10.1080/03605300600718503},
}

@article {hy2022b,
    AUTHOR = {Himonas, A. Alexanddrou and Yan, Fangchi},
     TITLE = {A higher dispersion {K}d{V} equation on the half-line},
   JOURNAL = {J. Differential Equations},
  FJOURNAL = {Journal of Differential Equations},
    VOLUME = {333},
      YEAR = {2022},
     PAGES = {55--102},
      ISSN = {0022-0396},
   MRCLASS = {35Q55 (35G16 35G31 35Q53 37K10)},
  MRNUMBER = {4441362},
       DOI = {10.1016/j.jde.2022.06.003},
       URL = {https://doi-org.www2.lib.ku.edu/10.1016/j.jde.2022.06.003},
}

@article {hm2022,
    AUTHOR = {Himonas, A. Alexandrou and Mantzavinos, Dionyssios},
     TITLE = {The {R}obin and {N}eumann problems for the nonlinear
              {S}chr\"{o}dinger equation on the half-plane},
   JOURNAL = {Proc. A.},
  FJOURNAL = {Proceedings A},
    VOLUME = {478},
      YEAR = {2022},
    NUMBER = {2265},
     PAGES = {Paper No. 279, 20},
      ISSN = {1364-5021},
   MRCLASS = {35Q55},
  MRNUMBER = {4492208},
}

@article {fhm2016,
    AUTHOR = {Fokas, Athanassios S. and Himonas, A. Alexandrou and
              Mantzavinos, Dionyssios},
     TITLE = {The {K}orteweg--de {V}ries equation on the half-line},
   JOURNAL = {Nonlinearity},
  FJOURNAL = {Nonlinearity},
    VOLUME = {29},
      YEAR = {2016},
    NUMBER = {2},
     PAGES = {489--527},
      ISSN = {0951-7715},
   MRCLASS = {35Q53 (35B30 35G16 35G31 37K10 47H10)},
  MRNUMBER = {3461607},
MRREVIEWER = {V. Oproiu},
       DOI = {10.1088/0951-7715/29/2/489},
       URL = {https://doi-org.www2.lib.ku.edu/10.1088/0951-7715/29/2/489},
}

@article {bsz2002,
    AUTHOR = {Bona, Jerry L. and Sun, S. M. and Zhang, Bing-Yu},
     TITLE = {A non-homogeneous boundary-value problem for the {K}orteweg-de
              {V}ries equation in a quarter plane},
   JOURNAL = {Trans. Amer. Math. Soc.},
  FJOURNAL = {Transactions of the American Mathematical Society},
    VOLUME = {354},
      YEAR = {2002},
    NUMBER = {2},
     PAGES = {427--490},
      ISSN = {0002-9947},
   MRCLASS = {35Q53 (35B30)},
  MRNUMBER = {1862556},
MRREVIEWER = {Vladislav G. Dubrovsky},
       DOI = {10.1090/S0002-9947-01-02885-9},
       URL = {https://doi-org.www2.lib.ku.edu/10.1090/S0002-9947-01-02885-9},
}

@article {bsz2006,
    AUTHOR = {Bona, Jerry L. and Sun, S. M. and Zhang, Bing-Yu},
     TITLE = {Boundary smoothing properties of the {K}orteweg-de {V}ries
              equation in a quarter plane and applications},
   JOURNAL = {Dyn. Partial Differ. Equ.},
  FJOURNAL = {Dynamics of Partial Differential Equations},
    VOLUME = {3},
      YEAR = {2006},
    NUMBER = {1},
     PAGES = {1--69},
      ISSN = {1548-159X,2163-7873},
   MRCLASS = {35Q53 (35B65 76B03 76B15)},
  MRNUMBER = {2221746},
MRREVIEWER = {Justin\ A.\ Holmer},
       DOI = {10.4310/DPDE.2006.v3.n1.a1},
       URL = {https://doi.org/10.4310/DPDE.2006.v3.n1.a1},
}

@article {bsz2008,
    AUTHOR = {Bona, Jerry L. and Sun, S. M. and Zhang, Bing-Yu},
     TITLE = {Non-homogeneous boundary value problems for the {K}orteweg-de
              {V}ries and the {K}orteweg-de {V}ries-{B}urgers equations in a
              quarter plane},
   JOURNAL = {Ann. Inst. H. Poincar\'{e} C Anal. Non Lin\'{e}aire},
  FJOURNAL = {Annales de l'Institut Henri Poincar\'{e} C. Analyse Non
              Lin\'{e}aire},
    VOLUME = {25},
      YEAR = {2008},
    NUMBER = {6},
     PAGES = {1145--1185},
      ISSN = {0294-1449,1873-1430},
   MRCLASS = {35Q53 (35B30 76B15)},
  MRNUMBER = {2466325},
MRREVIEWER = {Enzo\ Vitillaro},
       DOI = {10.1016/j.anihpc.2007.07.006},
       URL = {https://doi.org/10.1016/j.anihpc.2007.07.006},
}

@article {et2016,
    AUTHOR = {Erdogan, M. B. and Tzirakis, N.},
     TITLE = {Regularity properties of the cubic nonlinear {S}chr\"{o}dinger
              equation on the half line},
   JOURNAL = {J. Funct. Anal.},
  FJOURNAL = {Journal of Functional Analysis},
    VOLUME = {271},
      YEAR = {2016},
    NUMBER = {9},
     PAGES = {2539--2568},
      ISSN = {0022-1236},
   MRCLASS = {35Q55 (35B65)},
  MRNUMBER = {3545224},
MRREVIEWER = {Amin Esfahani},
       DOI = {10.1016/j.jfa.2016.08.012},
       URL = {https://doi-org.www2.lib.ku.edu/10.1016/j.jfa.2016.08.012},
}

@article {h1929,
    AUTHOR = {Hardy, G. H.},
     TITLE = {Remarks in {A}ddition to {D}r. {W}idder's {N}ote on
              {I}nequalities},
   JOURNAL = {J. London Math. Soc.},
  FJOURNAL = {The Journal of the London Mathematical Society},
    VOLUME = {4},
      YEAR = {1929},
    NUMBER = {3},
     PAGES = {199--202},
      ISSN = {0024-6107,1469-7750},
   MRCLASS = {99-04},
  MRNUMBER = {1575046},
       DOI = {10.1112/jlms/s1-4.3.199},
       URL = {https://doi.org/10.1112/jlms/s1-4.3.199},
}

@incollection {k1985,
	AUTHOR = {Kodama, Yuji},
	TITLE = {Optical solitons in a monomode fiber},
	NOTE = {Transport and propagation in nonlinear systems (Los Alamos,
	N.M., 1984)},
	JOURNAL = {J. Statist. Phys.},
	FJOURNAL = {Journal of Statistical Physics},
	VOLUME = {39},
	YEAR = {1985},
	NUMBER = {5-6},
	PAGES = {597--614},
	ISSN = {0022-4715},
	MRCLASS = {78A55},
	MRNUMBER = {807002},
	DOI = {10.1007/BF01008354},
	URL = {https://doi.org/10.1007/BF01008354},
}

@article{kh1987,
	title={Nonlinear pulse propagation in a monomode dielectric guide},
	author={Kodama, Yuji and Hasegawa, Akira},
	journal={IEEE Journal of Quantum Electronics},
	volume={23},
	number={5},
	pages={510--524},
	year={1987},
	publisher={IEEE}
}

@article {f2024,
    AUTHOR = {Faminskii, Andrei V.},
     TITLE = {Global weak solutions of an initial-boundary value problem on
              a half-line for the higher order nonlinear {S}chr\"{o}dinger
              equation},
   JOURNAL = {J. Math. Anal. Appl.},
  FJOURNAL = {Journal of Mathematical Analysis and Applications},
    VOLUME = {533},
      YEAR = {2024},
    NUMBER = {2},
     PAGES = {Paper No. 128003, 22},
      ISSN = {0022-247X,1096-0813},
   MRCLASS = {35Q55 (35D30 35G31)},
  MRNUMBER = {4674971},
       DOI = {10.1016/j.jmaa.2023.128003},
       URL = {https://doi.org/10.1016/j.jmaa.2023.128003},
}

@article {mo2024,
    AUTHOR = {Mantzavinos, Dionyssios and Ozsari, Turker},
     TITLE = {Low-regularity solutions of the nonlinear {S}chr\"{o}dinger equation on the spatial quarter-plane},
   JOURNAL = {arXiv:2403.15350v1},
  FJOURNAL = {},
    VOLUME = {},
      YEAR = {2024},
    NUMBER = {},
     PAGES = {(to appear in Siam Journal on Mathematical Analysis)},
      ISSN = {},
   MRCLASS = {},
  MRNUMBER = {},
MRREVIEWER = {},
       URL = {https://arxiv.org/abs/2403.15350v1},
}

@article {fm2023,
    AUTHOR = {Faminskii, A. V. and Martynov, E. V.},
     TITLE = {Inverse problems for the higher order nonlinear
              {S}chr\"odinger equation},
   JOURNAL = {J. Math. Sci. (N.Y.)},
  FJOURNAL = {Journal of Mathematical Sciences (New York)},
    VOLUME = {274},
      YEAR = {2023},
    NUMBER = {4},
     PAGES = {475--492},
      ISSN = {1072-3374,1573-8795},
   MRCLASS = {35Q55},
  MRNUMBER = {4644504},
}

@article {mok2024,
    AUTHOR = {Mantzavinos, Dionyssios and Ozsari, Turker and {Y}ilmaz, {K}emal {C}em},
     TITLE = {Ginzburg-Landau equation on a finite interval and rapid chaos suppression via a finite dimensional backstepping controller},
   JOURNAL = {(preprint)},
  FJOURNAL = {},
    VOLUME = {},
      YEAR = {},
    NUMBER = {},
     PAGES = {},
      ISSN = {},
   MRCLASS = {},
  MRNUMBER = {},
MRREVIEWER = {},
       URL = {},
}

@article{amo2025,
	AUTHOR = {Alkin, Aykut and Mantzavinos, Dionyssios and Ozsari, Turker},
	TITLE = {Local well-posedness of the higher-order nonlinear {S}chrödinger equation on the half-line: the case of two boundary conditions},
	JOURNAL = {(preprint)},
	VOLUME = {},
	YEAR = {2025},
	NUMBER = {},
	PAGES = {},
	URL = {},
}

\appendix

\section{Linear solution formula via the unified transform}\label{app-s}

We employ the unified transform of Fokas to derive the solution formula \eqref{hls-sol} for the forced linear problem on a finite interval given by \eqref{hls-ibvp}. Throughout our derivation, we work under the assumption of sufficient smoothness and decay. 
We begin by defining the finite interval Fourier transform pair
\eq{\label{fi-ft-def}
\hat \phi (k) := \int_0^\ell e^{-i k x} \phi (x) dx, \quad k\in\mathbb C,
\qquad
\phi (x) =
\frac{1}{2 \pi} \int_{-\infty}^\infty e^{i k x} \hat \phi (k) dk, \quad  0 < x < \ell,
}
where we note that the inverse formula follows by applying the standard Fourier inversion theorem to a function defined on $\R$ but with support on $(0, \ell)$ and equal to $\phi(x)$ there. We emphasize that, contrary to the usual Fourier transform on the whole line, which is only defined for $k\in\R$, the finite interval Fourier transform is entire with respect to $k\in\C$ thanks to the fact that the spatial variable $x$ is bounded (see Section 7.2 in \cite{s1994}).

Taking the finite interval Fourier transform \eqref{fi-ft-def} of the forced linear higher-order Schr\"odinger equation in \eqref{hls-ibvp} and then integrating in $t$, we obtain what is known in the unified transform terminology as the \textit{global relation}:
\eqs{
\begin{split}
e^{-i \omega t} \hat u(k, t) &= \hat u_0 (k) - i \int_0^t e^{-i\omega t'} \hat f(k, t') dt' + \beta \tilde g_2 (\omega, t) + i \p{\beta k - \alpha} \tilde g_1 (\omega, t) - \p{\beta k^2 - \alpha k - \delta} \tilde g_0 (\omega, t)\\
&\quad- e^{-i k \ell} \b{\beta \tilde h_2 (\omega, t) + i \p{\beta k - \alpha} \tilde h_1 (\omega, t) - \p{\beta k^2 - \alpha k - \delta} \tilde h_0 (\omega, t)}, \quad k\in\C,
\end{split}
\alignlabel{preinv}}
where $\omega$ is given by \eqref{omega} and we have introduced the notation 
\eq{
g_2 (t) = u_{xx} (0, t), \quad g_1 (t) = u_{x} (0, t), \quad h_2 (t) = u_{xx} (\ell, t),
\quad 
\tilde \phi (\omega, t) := \int_0^t e^{-i \omega t'} \phi (t') dt'. \label{tilde-transform}}
Inverting the global relation for $k\in\R$ via \eqref{fi-ft-def}, we obtain the following integral representation for the solution:
\eqs{
\begin{split}
u(x, t) &= \frac{1}{2 \pi} \int_{-\infty}^{\infty} e^{i k x + i \omega t} \b{\hat u_0 (k) - i \int_0^t e^{-i\omega t'} \hat f(k, t') dt'} dk\\
&\quad + \frac{1}{2 \pi} \int_{-\infty}^{\infty} e^{i k x + i \omega t} \Big[\beta \tilde g_2 (\omega, t) + i \p{\beta k - \alpha} \tilde g_1 (\omega, t) - \p{\beta k^2 - \alpha k - \delta} \tilde g_0 (\omega, t)\Big] dk\\
&\quad - \frac{1}{2 \pi} \int_{-\infty}^{\infty} e^{-i k (\ell - x) + i \omega t} \b{\beta \tilde h_2 (\omega, t) + i \p{\beta k - \alpha} \tilde h_1 (\omega, t) - \p{\beta k^2 - \alpha k - \delta} \tilde h_0 (\omega, t)} dk.
\end{split} \alignlabel{intrep}}
This expression is not an explicit formula for the solution because it contains the unknown boundary values $g_1, g_2, h_2$ through their respective time transforms defined by \eqref{tilde-transform}. The key idea behind the unified transform is to combine the fact that these transforms depend on $k$ only through $\omega$ with appropriate symmetries of $\omega$ in the complex $k$-plane which, therefore, leave the transforms invariant.

The first step in the implementation of the above plan is to perform appropriate contour deformations by exploiting analyticity and exponential decay. Indeed, motivated by the decay of the exponentials involved in \eqref{intrep}, and thanks to the fact that the relevant integrands are entire in $k$, we use Cauchy's theorem to deform the contours of integration of the second and third integrals in \eqref{intrep} from $\R$ to the positively oriented boundaries $\partial \tilde D_0$ and $\partial \tilde D_+\cup\partial \tilde D_-$, respectively, where the regions $\tilde D_0, \tilde D_+, \tilde D_-$ are defined by \eqref{dntil-def} and shown in Figure \ref{fig:Dbranches}. We note that these regions are obtained by puncturing the regions $D_0, D_+, D_-$ (given by \eqref{dn-def} and shown in Figure \ref{fig:Dpm}) by a disk of radius $R_\Delta$ (given by \eqref{rd-def}) centered at $\frac{\alpha}{3\beta}$. This is done proactively, now that the integrands are still entire in $k$, in order to ensure that neither the branch cut introduced later for the symmetries of $\omega$ (see discussion below~\eqref{csr}) nor the zeros of the quantity $\Delta$ given by \eqref{Delta}, which will eventually appear in the denominator of our integrands, are enclosed by the deformed contours of integration.
Eventually, by Cauchy's theorem and use of the exponential decay, we arrive at the modified integral representation
\eqs{
\begin{split}
u(x, t) &= \frac{1}{2 \pi} \int_{-\infty}^{\infty} e^{i k x + i \omega t} \b{\hat u_0 (k) - i \int_0^t e^{-i\omega t'} \hat f(k, t') dt'} dk\\
&\quad + \frac{1}{2 \pi} \int_{\partial \tilde D_0} e^{i k x + i \omega t} \Big[\beta \tilde g_2 (\omega, t) + i \p{\beta k - \alpha} \tilde g_1 (\omega, t) - \p{\beta k^2 - \alpha k - \delta} \tilde g_0 (\omega, t)\Big] dk\\
&\quad + \frac{1}{2 \pi} \int_{\partial \tilde D_+ \cup \partial \tilde D_-} e^{-i k (\ell - x) + i \omega t} \b{\beta \tilde h_2 (\omega, t) + i \p{\beta k - \alpha} \tilde h_1 (\omega, t) - \p{\beta k^2 - \alpha k - \delta} \tilde h_0 (\omega, t)} dk.
\end{split} \alignlabel{intrepdef}}

\begin{figure}[ht!]
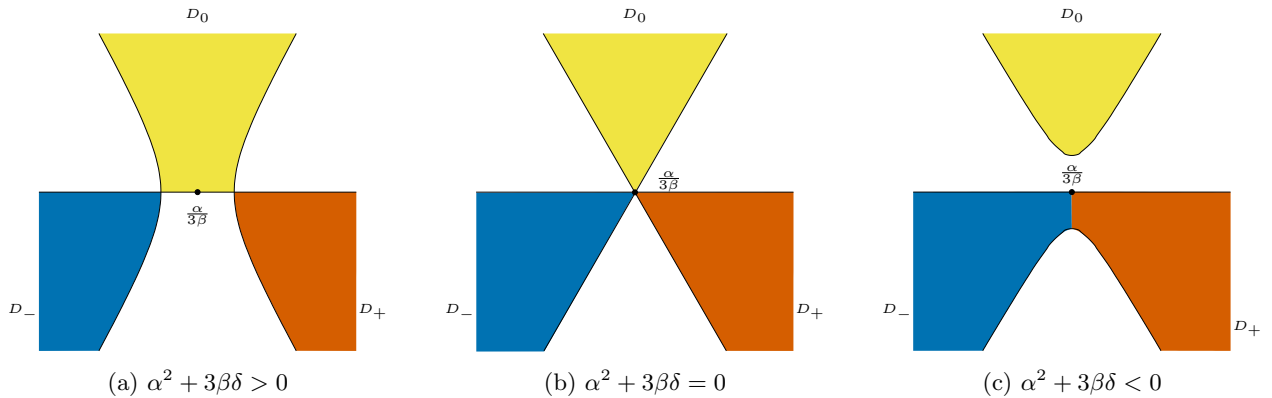

\centering
\begin{subfigure}{.33\textwidth}
\centering
%
%
{\pgfkeys{/pgf/fpu/.try=false}%
\ifx\XFigwidth\undefined\dimen1=0pt\else\dimen1\XFigwidth\fi
\divide\dimen1 by 5691
\ifx\XFigheight\undefined\dimen3=0pt\else\dimen3\XFigheight\fi
\divide\dimen3 by 5151
\ifdim\dimen1=0pt\ifdim\dimen3=0pt\dimen1=1657sp\dimen3\dimen1
  \else\dimen1\dimen3\fi\else\ifdim\dimen3=0pt\dimen3\dimen1\fi\fi
\tikzpicture[x=+\dimen1, y=+\dimen3]
{\ifx\XFigu\undefined\catcode`\@11
\def\temp{\alloc@1\dimen\dimendef\insc@unt}\temp\XFigu\catcode`\@12\fi}
\XFigu1657sp
\ifdim\XFigu<0pt\XFigu-\XFigu\fi
\pgfdeclarearrow{
  name = xfiga1,
  parameters = {
    \the\pgfarrowlinewidth \the\pgfarrowlength \the\pgfarrowwidth\ifpgfarrowopen o\fi},
  defaults = {
	  line width=+7.5\XFigu, length=+120\XFigu, width=+60\XFigu},
  setup code = {
    \dimen7 2.1\pgfarrowlength\pgfmathveclen{\the\dimen7}{\the\pgfarrowwidth}
    \dimen7 2\pgfarrowwidth\pgfmathdivide{\pgfmathresult}{\the\dimen7}
    \dimen7 \pgfmathresult\pgfarrowlinewidth
    \pgfarrowssettipend{+\dimen7}
    \pgfarrowssetbackend{+-\pgfarrowlength}
    \dimen9 -\pgfarrowlength\advance\dimen9 by-0.45\pgfarrowlinewidth
    \pgfarrowssetlineend{+\dimen9}
    \dimen9 -\pgfarrowlength\advance\dimen9 by-0.5\pgfarrowlinewidth
    \pgfarrowssetvisualbackend{+\dimen9}
    \pgfarrowshullpoint{+\dimen7}{+0pt}
    \pgfarrowsupperhullpoint{+-\pgfarrowlength}{+0.5\pgfarrowwidth}
    \pgfarrowssavethe\pgfarrowlinewidth
    \pgfarrowssavethe\pgfarrowlength
    \pgfarrowssavethe\pgfarrowwidth
  },
  drawing code = {\pgfsetdash{}{+0pt}
    \ifdim\pgfarrowlinewidth=\pgflinewidth\else\pgfsetlinewidth{+\pgfarrowlinewidth}\fi
    \pgfpathmoveto{\pgfqpoint{-\pgfarrowlength}{-0.5\pgfarrowwidth}}
    \pgfpathlineto{\pgfqpoint{0pt}{0pt}}
    \pgfpathlineto{\pgfqpoint{-\pgfarrowlength}{0.5\pgfarrowwidth}}
    \pgfpathclose
    \ifpgfarrowopen\pgfusepathqstroke\else\pgfsetfillcolor{.}
	\ifdim\pgfarrowlinewidth>0pt\pgfusepathqfillstroke\else\pgfusepathqfill\fi\fi
  }
}
\definecolor{xfigc32}{rgb}{0.941,0.894,0.259}
\definecolor{xfigc33}{rgb}{0.000,0.447,0.698}
\definecolor{xfigc34}{rgb}{0.835,0.369,0.000}
\clip(1882,-7104) rectangle (7573,-1953);
\tikzset{inner sep=+0pt, outer sep=+0pt}
\pgfsetfillcolor{black}
\pgfsetlinewidth{+7.5\XFigu}
\pgfsetdash{}{+0pt}
\pgfsetstrokecolor{black}
\pgfsetbeveljoin
\pgfsetfillcolor{xfigc34}
\fill (7086,-7088)--(7086,-7087)--(7086,-7084)--(7086,-7078)--(7086,-7069)--(7086,-7056)
  --(7086,-7039)--(7086,-7017)--(7086,-6989)--(7086,-6955)--(7086,-6915)--(7086,-6869)
  --(7086,-6817)--(7086,-6759)--(7086,-6694)--(7086,-6624)--(7086,-6548)--(7086,-6467)
  --(7086,-6381)--(7086,-6291)--(7086,-6198)--(7086,-6102)--(7086,-6004)--(7087,-5906)
  --(7087,-5808)--(7087,-5710)--(7087,-5614)--(7087,-5521)--(7087,-5431)--(7087,-5345)
  --(7087,-5264)--(7087,-5188)--(7087,-5118)--(7087,-5053)--(7087,-4995)--(7087,-4943)
  --(7087,-4897)--(7087,-4857)--(7087,-4823)--(7087,-4795)--(7087,-4773)--(7087,-4756)
  --(7087,-4743)--(7087,-4734)--(7087,-4728)--(7087,-4725)--(7087,-4724)--(7086,-4724)
  --(7082,-4724)--(7076,-4724)--(7066,-4724)--(7051,-4724)--(7032,-4724)--(7007,-4724)
  --(6977,-4724)--(6940,-4724)--(6897,-4724)--(6848,-4724)--(6792,-4724)--(6731,-4724)
  --(6664,-4724)--(6591,-4724)--(6515,-4724)--(6434,-4724)--(6351,-4724)--(6266,-4724)
  --(6179,-4724)--(6093,-4724)--(6008,-4724)--(5925,-4724)--(5844,-4724)--(5768,-4724)
  --(5695,-4724)--(5628,-4724)--(5567,-4724)--(5511,-4724)--(5462,-4724)--(5419,-4724)
  --(5382,-4724)--(5352,-4724)--(5327,-4724)--(5308,-4724)--(5293,-4724)--(5283,-4724)
  --(5277,-4724)--(5273,-4724)--(5272,-4724)--(5272,-4726)--(5272,-4731)--(5272,-4740)
  --(5272,-4751)--(5273,-4766)--(5273,-4783)--(5274,-4802)--(5275,-4822)--(5277,-4844)
  --(5279,-4869)--(5281,-4896)--(5285,-4927)--(5289,-4960)--(5294,-4993)--(5298,-5021)
  --(5302,-5043)--(5304,-5058)--(5306,-5068)--(5307,-5074)--(5308,-5078)--(5309,-5082)
  --(5310,-5088)--(5313,-5098)--(5316,-5113)--(5321,-5135)--(5328,-5163)--(5336,-5196)
  --(5344,-5229)--(5352,-5257)--(5358,-5279)--(5362,-5294)--(5365,-5304)--(5367,-5310)
  --(5368,-5314)--(5370,-5318)--(5372,-5324)--(5375,-5334)--(5380,-5349)--(5387,-5371)
  --(5397,-5399)--(5408,-5432)--(5419,-5465)--(5429,-5493)--(5437,-5515)--(5443,-5530)
  --(5447,-5540)--(5449,-5546)--(5451,-5550)--(5452,-5554)--(5455,-5560)--(5459,-5570)
  --(5465,-5585)--(5474,-5607)--(5485,-5635)--(5498,-5668)--(5510,-5697)--(5520,-5723)
  --(5529,-5744)--(5536,-5760)--(5540,-5771)--(5543,-5778)--(5545,-5783)--(5547,-5787)
  --(5548,-5790)--(5550,-5795)--(5554,-5802)--(5558,-5813)--(5565,-5829)--(5575,-5850)
  --(5586,-5876)--(5599,-5905)--(5612,-5934)--(5624,-5960)--(5633,-5981)--(5641,-5996)
  --(5646,-6007)--(5649,-6015)--(5652,-6020)--(5653,-6023)--(5655,-6027)--(5657,-6031)
  --(5661,-6039)--(5666,-6050)--(5674,-6065)--(5684,-6086)--(5696,-6112)--(5710,-6141)
  --(5724,-6170)--(5736,-6196)--(5747,-6217)--(5754,-6232)--(5760,-6243)--(5763,-6251)
  --(5766,-6256)--(5767,-6259)--(5769,-6262)--(5771,-6267)--(5775,-6275)--(5781,-6286)
  --(5788,-6301)--(5799,-6322)--(5812,-6348)--(5826,-6377)--(5840,-6406)--(5853,-6432)
  --(5864,-6453)--(5872,-6468)--(5877,-6479)--(5881,-6487)--(5884,-6491)--(5885,-6495)
  --(5887,-6498)--(5890,-6503)--(5893,-6511)--(5899,-6522)--(5907,-6537)--(5918,-6558)
  --(5931,-6584)--(5946,-6613)--(5961,-6642)--(5974,-6668)--(5985,-6689)--(5993,-6705)
  --(5999,-6716)--(6003,-6723)--(6006,-6728)--(6008,-6732)--(6009,-6735)--(6012,-6740)
  --(6016,-6747)--(6022,-6758)--(6030,-6774)--(6041,-6795)--(6055,-6821)--(6070,-6850)
  --(6085,-6879)--(6100,-6907)--(6113,-6931)--(6125,-6954)--(6136,-6974)--(6146,-6993)
  --(6155,-7011)--(6164,-7027)--(6172,-7042)--(6179,-7055)--(6185,-7066)--(6189,-7075)
  --(6192,-7081)--(6194,-7084)--(6195,-7086);
\pgfsetfillcolor{xfigc32}
\fill (3255,-2362)--(3256,-2364)--(3258,-2367)--(3261,-2373)--(3265,-2382)--(3271,-2393)
  --(3278,-2406)--(3286,-2421)--(3295,-2437)--(3304,-2455)--(3314,-2474)--(3325,-2494)
  --(3337,-2517)--(3350,-2541)--(3365,-2569)--(3380,-2598)--(3395,-2627)--(3409,-2653)
  --(3420,-2674)--(3428,-2690)--(3434,-2701)--(3438,-2708)--(3441,-2713)--(3442,-2717)
  --(3444,-2720)--(3447,-2725)--(3451,-2732)--(3457,-2743)--(3465,-2759)--(3476,-2780)
  --(3489,-2806)--(3504,-2835)--(3519,-2864)--(3532,-2890)--(3543,-2911)--(3551,-2926)
  --(3557,-2937)--(3560,-2945)--(3563,-2950)--(3565,-2953)--(3566,-2957)--(3569,-2961)
  --(3573,-2969)--(3578,-2980)--(3586,-2995)--(3597,-3016)--(3610,-3042)--(3624,-3071)
  --(3638,-3100)--(3651,-3126)--(3662,-3147)--(3669,-3162)--(3675,-3173)--(3679,-3181)
  --(3681,-3186)--(3683,-3189)--(3684,-3192)--(3687,-3197)--(3690,-3205)--(3696,-3216)
  --(3703,-3231)--(3714,-3252)--(3726,-3278)--(3740,-3307)--(3754,-3336)--(3766,-3362)
  --(3776,-3383)--(3784,-3398)--(3789,-3409)--(3793,-3417)--(3795,-3421)--(3797,-3425)
  --(3798,-3428)--(3801,-3433)--(3804,-3441)--(3809,-3452)--(3817,-3467)--(3826,-3488)
  --(3838,-3514)--(3851,-3543)--(3864,-3572)--(3875,-3598)--(3885,-3619)--(3892,-3635)
  --(3896,-3646)--(3900,-3653)--(3902,-3658)--(3903,-3662)--(3905,-3665)--(3907,-3670)
  --(3910,-3677)--(3914,-3688)--(3921,-3704)--(3930,-3725)--(3940,-3751)--(3952,-3780)
  --(3965,-3813)--(3976,-3841)--(3985,-3863)--(3991,-3878)--(3995,-3888)--(3998,-3894)
  --(3999,-3898)--(4001,-3902)--(4003,-3908)--(4007,-3918)--(4013,-3933)--(4021,-3955)
  --(4031,-3983)--(4042,-4016)--(4053,-4049)--(4063,-4077)--(4070,-4099)--(4075,-4114)
  --(4078,-4124)--(4080,-4130)--(4082,-4134)--(4083,-4138)--(4085,-4144)--(4088,-4154)
  --(4092,-4169)--(4098,-4191)--(4106,-4219)--(4114,-4252)--(4122,-4285)--(4129,-4313)
  --(4134,-4335)--(4137,-4350)--(4140,-4360)--(4141,-4366)--(4142,-4370)--(4143,-4374)
  --(4144,-4380)--(4146,-4390)--(4148,-4405)--(4152,-4427)--(4156,-4455)--(4161,-4488)
  --(4165,-4521)--(4169,-4552)--(4171,-4579)--(4173,-4604)--(4175,-4626)--(4176,-4646)
  --(4177,-4665)--(4177,-4682)--(4178,-4697)--(4178,-4708)--(4178,-4717)--(4178,-4722)
  --(4178,-4724)--(4180,-4724)--(4184,-4724)--(4193,-4724)--(4205,-4724)--(4224,-4724)
  --(4248,-4724)--(4279,-4724)--(4316,-4724)--(4359,-4724)--(4409,-4724)--(4464,-4724)
  --(4524,-4724)--(4588,-4724)--(4655,-4724)--(4723,-4724)--(4790,-4724)--(4857,-4724)
  --(4921,-4724)--(4981,-4724)--(5036,-4724)--(5086,-4724)--(5129,-4724)--(5166,-4724)
  --(5197,-4724)--(5221,-4724)--(5240,-4724)--(5252,-4724)--(5261,-4724)--(5265,-4724)
  --(5267,-4724)--(5267,-4722)--(5267,-4717)--(5267,-4708)--(5267,-4697)--(5268,-4682)
  --(5268,-4665)--(5269,-4646)--(5270,-4626)--(5272,-4604)--(5274,-4579)--(5276,-4552)
  --(5280,-4521)--(5284,-4488)--(5289,-4455)--(5293,-4427)--(5297,-4405)--(5299,-4390)
  --(5301,-4380)--(5302,-4374)--(5303,-4370)--(5304,-4366)--(5305,-4360)--(5308,-4350)
  --(5311,-4335)--(5316,-4313)--(5323,-4285)--(5331,-4252)--(5339,-4219)--(5347,-4191)
  --(5353,-4169)--(5357,-4154)--(5360,-4144)--(5362,-4138)--(5363,-4134)--(5365,-4130)
  --(5367,-4124)--(5370,-4114)--(5375,-4099)--(5382,-4077)--(5392,-4049)--(5403,-4016)
  --(5414,-3983)--(5424,-3955)--(5432,-3933)--(5438,-3918)--(5442,-3908)--(5444,-3902)
  --(5446,-3898)--(5447,-3894)--(5450,-3888)--(5454,-3878)--(5460,-3863)--(5469,-3841)
  --(5480,-3813)--(5493,-3780)--(5505,-3751)--(5515,-3725)--(5524,-3704)--(5531,-3688)
  --(5535,-3677)--(5538,-3670)--(5540,-3665)--(5542,-3662)--(5543,-3658)--(5545,-3653)
  --(5549,-3646)--(5553,-3635)--(5560,-3619)--(5570,-3598)--(5581,-3572)--(5594,-3543)
  --(5607,-3514)--(5619,-3488)--(5628,-3467)--(5636,-3452)--(5641,-3441)--(5644,-3433)
  --(5647,-3428)--(5648,-3425)--(5650,-3421)--(5652,-3417)--(5656,-3409)--(5661,-3398)
  --(5669,-3383)--(5679,-3362)--(5691,-3336)--(5705,-3307)--(5719,-3278)--(5731,-3252)
  --(5742,-3231)--(5749,-3216)--(5755,-3205)--(5758,-3197)--(5761,-3192)--(5762,-3189)
  --(5764,-3186)--(5766,-3181)--(5770,-3173)--(5776,-3162)--(5783,-3147)--(5794,-3126)
  --(5807,-3100)--(5821,-3071)--(5835,-3042)--(5848,-3016)--(5859,-2995)--(5867,-2980)
  --(5872,-2969)--(5876,-2961)--(5879,-2957)--(5880,-2953)--(5882,-2950)--(5885,-2945)
  --(5888,-2937)--(5894,-2926)--(5902,-2911)--(5913,-2890)--(5926,-2864)--(5941,-2835)
  --(5956,-2806)--(5969,-2780)--(5980,-2759)--(5988,-2743)--(5994,-2732)--(5998,-2725)
  --(6001,-2720)--(6003,-2717)--(6004,-2713)--(6007,-2708)--(6011,-2701)--(6017,-2690)
  --(6025,-2674)--(6036,-2653)--(6050,-2627)--(6065,-2598)--(6080,-2569)--(6095,-2541)
  --(6108,-2517)--(6120,-2494)--(6131,-2474)--(6141,-2455)--(6150,-2437)--(6159,-2421)
  --(6167,-2406)--(6174,-2393)--(6180,-2382)--(6184,-2373)--(6187,-2367)--(6189,-2364)
  --(6190,-2362);
\pgfsetfillcolor{xfigc33}
\fill (2362,-7088)--(2362,-7087)--(2362,-7084)--(2362,-7078)--(2362,-7069)--(2362,-7056)
  --(2362,-7039)--(2362,-7017)--(2362,-6989)--(2362,-6955)--(2362,-6915)--(2362,-6869)
  --(2362,-6817)--(2362,-6759)--(2362,-6694)--(2362,-6624)--(2362,-6548)--(2362,-6467)
  --(2362,-6381)--(2362,-6291)--(2362,-6198)--(2362,-6102)--(2362,-6004)--(2362,-5906)
  --(2362,-5808)--(2362,-5710)--(2362,-5614)--(2362,-5521)--(2362,-5431)--(2362,-5345)
  --(2362,-5264)--(2362,-5188)--(2362,-5118)--(2362,-5053)--(2362,-4995)--(2362,-4943)
  --(2362,-4897)--(2362,-4857)--(2362,-4823)--(2362,-4795)--(2362,-4773)--(2362,-4756)
  --(2362,-4743)--(2362,-4734)--(2362,-4728)--(2362,-4725)--(2362,-4724)--(2363,-4724)
  --(2367,-4724)--(2373,-4724)--(2383,-4724)--(2398,-4724)--(2417,-4724)--(2442,-4724)
  --(2472,-4724)--(2509,-4724)--(2552,-4724)--(2601,-4724)--(2657,-4724)--(2718,-4724)
  --(2785,-4724)--(2858,-4724)--(2934,-4724)--(3015,-4724)--(3098,-4724)--(3183,-4724)
  --(3270,-4724)--(3356,-4724)--(3441,-4724)--(3524,-4724)--(3605,-4724)--(3681,-4724)
  --(3754,-4724)--(3821,-4724)--(3882,-4724)--(3938,-4724)--(3987,-4724)--(4030,-4724)
  --(4067,-4724)--(4097,-4724)--(4122,-4724)--(4141,-4724)--(4156,-4724)--(4166,-4724)
  --(4172,-4724)--(4176,-4724)--(4177,-4724)--(4177,-4726)--(4177,-4731)--(4177,-4740)
  --(4177,-4751)--(4176,-4766)--(4176,-4783)--(4175,-4802)--(4174,-4822)--(4172,-4844)
  --(4170,-4869)--(4168,-4896)--(4164,-4927)--(4160,-4960)--(4155,-4993)--(4151,-5021)
  --(4147,-5043)--(4145,-5058)--(4143,-5068)--(4142,-5074)--(4141,-5078)--(4140,-5082)
  --(4139,-5088)--(4136,-5098)--(4133,-5113)--(4128,-5135)--(4121,-5163)--(4113,-5196)
  --(4105,-5229)--(4097,-5257)--(4091,-5279)--(4087,-5294)--(4084,-5304)--(4082,-5310)
  --(4081,-5314)--(4079,-5318)--(4077,-5324)--(4074,-5334)--(4069,-5349)--(4062,-5371)
  --(4052,-5399)--(4041,-5432)--(4030,-5465)--(4020,-5493)--(4012,-5515)--(4006,-5530)
  --(4002,-5540)--(4000,-5546)--(3998,-5550)--(3997,-5554)--(3994,-5560)--(3990,-5570)
  --(3984,-5585)--(3975,-5607)--(3964,-5635)--(3951,-5668)--(3939,-5697)--(3929,-5723)
  --(3920,-5744)--(3913,-5760)--(3909,-5771)--(3906,-5778)--(3904,-5783)--(3902,-5787)
  --(3901,-5790)--(3899,-5795)--(3895,-5802)--(3891,-5813)--(3884,-5829)--(3874,-5850)
  --(3863,-5876)--(3850,-5905)--(3837,-5934)--(3825,-5960)--(3816,-5981)--(3808,-5996)
  --(3803,-6007)--(3800,-6015)--(3797,-6020)--(3796,-6023)--(3794,-6027)--(3792,-6031)
  --(3788,-6039)--(3783,-6050)--(3775,-6065)--(3765,-6086)--(3753,-6112)--(3739,-6141)
  --(3725,-6170)--(3713,-6196)--(3702,-6217)--(3695,-6232)--(3689,-6243)--(3686,-6251)
  --(3683,-6256)--(3682,-6259)--(3680,-6262)--(3678,-6267)--(3674,-6275)--(3668,-6286)
  --(3661,-6301)--(3650,-6322)--(3637,-6348)--(3623,-6377)--(3609,-6406)--(3596,-6432)
  --(3585,-6453)--(3577,-6468)--(3572,-6479)--(3568,-6487)--(3565,-6491)--(3564,-6495)
  --(3562,-6498)--(3559,-6503)--(3556,-6511)--(3550,-6522)--(3542,-6537)--(3531,-6558)
  --(3518,-6584)--(3503,-6613)--(3488,-6642)--(3475,-6668)--(3464,-6689)--(3456,-6705)
  --(3450,-6716)--(3446,-6723)--(3443,-6728)--(3441,-6732)--(3440,-6735)--(3437,-6740)
  --(3433,-6747)--(3427,-6758)--(3419,-6774)--(3408,-6795)--(3394,-6821)--(3379,-6850)
  --(3364,-6879)--(3349,-6907)--(3336,-6931)--(3324,-6954)--(3313,-6974)--(3303,-6993)
  --(3294,-7011)--(3285,-7027)--(3277,-7042)--(3270,-7055)--(3264,-7066)--(3260,-7075)
  --(3257,-7081)--(3255,-7084)--(3254,-7086);
\draw (2362,-4724)--(7087,-4724);
\pgfsetarrowsend{}
\draw (6193,-7087)--(6192,-7085)--(6190,-7082)--(6187,-7076)--(6183,-7067)--(6177,-7056)
  --(6170,-7043)--(6162,-7028)--(6153,-7011)--(6144,-6994)--(6134,-6975)--(6123,-6954)
  --(6111,-6931)--(6098,-6907)--(6083,-6879)--(6068,-6850)--(6053,-6821)--(6039,-6795)
  --(6028,-6774)--(6020,-6759)--(6014,-6748)--(6010,-6740)--(6007,-6735)--(6006,-6732)
  --(6004,-6728)--(6001,-6724)--(5997,-6716)--(5991,-6705)--(5983,-6690)--(5972,-6669)
  --(5959,-6643)--(5944,-6614)--(5929,-6585)--(5916,-6559)--(5905,-6538)--(5897,-6523)
  --(5891,-6512)--(5888,-6504)--(5885,-6499)--(5883,-6496)--(5882,-6493)--(5879,-6488)
  --(5875,-6480)--(5870,-6469)--(5862,-6454)--(5851,-6433)--(5838,-6407)--(5824,-6378)
  --(5810,-6349)--(5797,-6323)--(5786,-6302)--(5779,-6287)--(5773,-6276)--(5769,-6268)
  --(5767,-6263)--(5765,-6260)--(5764,-6257)--(5761,-6252)--(5758,-6244)--(5752,-6233)
  --(5745,-6218)--(5734,-6197)--(5722,-6171)--(5708,-6142)--(5694,-6113)--(5682,-6087)
  --(5672,-6066)--(5664,-6051)--(5659,-6040)--(5655,-6032)--(5653,-6028)--(5651,-6024)
  --(5650,-6021)--(5647,-6016)--(5644,-6008)--(5639,-5997)--(5631,-5982)--(5622,-5961)
  --(5610,-5935)--(5597,-5906)--(5584,-5877)--(5573,-5851)--(5563,-5830)--(5556,-5814)
  --(5552,-5803)--(5548,-5796)--(5546,-5791)--(5545,-5788)--(5543,-5784)--(5541,-5779)
  --(5538,-5772)--(5534,-5761)--(5527,-5745)--(5518,-5724)--(5508,-5698)--(5496,-5669)
  --(5483,-5636)--(5472,-5608)--(5463,-5586)--(5457,-5571)--(5453,-5561)--(5450,-5555)
  --(5449,-5551)--(5447,-5547)--(5445,-5541)--(5441,-5531)--(5435,-5516)--(5427,-5494)
  --(5417,-5466)--(5406,-5433)--(5395,-5400)--(5385,-5372)--(5378,-5350)--(5373,-5335)
  --(5370,-5325)--(5368,-5319)--(5366,-5315)--(5365,-5311)--(5363,-5305)--(5360,-5295)
  --(5356,-5280)--(5350,-5258)--(5342,-5230)--(5334,-5197)--(5326,-5164)--(5319,-5136)
  --(5314,-5114)--(5311,-5099)--(5308,-5089)--(5307,-5083)--(5306,-5079)--(5305,-5075)
  --(5304,-5069)--(5302,-5059)--(5300,-5044)--(5296,-5022)--(5292,-4994)--(5287,-4961)
  --(5283,-4928)--(5279,-4900)--(5277,-4878)--(5275,-4863)--(5274,-4853)--(5274,-4847)
  --(5273,-4842)--(5273,-4838)--(5273,-4832)--(5272,-4822)--(5272,-4807)--(5271,-4785)
  --(5270,-4757)--(5270,-4724)--(5270,-4691)--(5271,-4663)--(5272,-4641)--(5272,-4626)
  --(5273,-4616)--(5273,-4610)--(5273,-4606)--(5274,-4602)--(5274,-4596)--(5275,-4586)
  --(5277,-4571)--(5279,-4549)--(5283,-4521)--(5287,-4488)--(5292,-4455)--(5296,-4427)
  --(5300,-4405)--(5302,-4390)--(5304,-4380)--(5305,-4374)--(5306,-4370)--(5307,-4366)
  --(5308,-4360)--(5311,-4350)--(5314,-4335)--(5319,-4313)--(5326,-4285)--(5334,-4252)
  --(5342,-4219)--(5350,-4191)--(5356,-4169)--(5360,-4154)--(5363,-4144)--(5365,-4138)
  --(5366,-4134)--(5368,-4130)--(5370,-4124)--(5373,-4114)--(5378,-4099)--(5385,-4077)
  --(5395,-4049)--(5406,-4016)--(5417,-3983)--(5427,-3955)--(5435,-3933)--(5441,-3918)
  --(5445,-3908)--(5447,-3902)--(5449,-3898)--(5450,-3894)--(5453,-3888)--(5457,-3878)
  --(5463,-3863)--(5472,-3841)--(5483,-3813)--(5496,-3780)--(5508,-3751)--(5518,-3725)
  --(5527,-3704)--(5534,-3688)--(5538,-3677)--(5541,-3670)--(5543,-3665)--(5545,-3662)
  --(5546,-3658)--(5548,-3653)--(5552,-3646)--(5556,-3635)--(5563,-3619)--(5573,-3598)
  --(5584,-3572)--(5597,-3543)--(5610,-3514)--(5622,-3488)--(5631,-3467)--(5639,-3452)
  --(5644,-3441)--(5647,-3433)--(5650,-3428)--(5651,-3425)--(5653,-3421)--(5655,-3417)
  --(5659,-3409)--(5664,-3398)--(5672,-3383)--(5682,-3362)--(5694,-3336)--(5708,-3307)
  --(5722,-3278)--(5734,-3252)--(5745,-3231)--(5752,-3216)--(5758,-3205)--(5761,-3197)
  --(5764,-3192)--(5765,-3189)--(5767,-3186)--(5769,-3181)--(5773,-3173)--(5779,-3162)
  --(5786,-3147)--(5797,-3126)--(5810,-3100)--(5824,-3071)--(5838,-3042)--(5851,-3016)
  --(5862,-2995)--(5870,-2980)--(5875,-2969)--(5879,-2961)--(5882,-2957)--(5883,-2953)
  --(5885,-2950)--(5888,-2945)--(5891,-2937)--(5897,-2926)--(5905,-2911)--(5916,-2890)
  --(5929,-2864)--(5944,-2835)--(5959,-2806)--(5972,-2780)--(5983,-2759)--(5991,-2743)
  --(5997,-2732)--(6001,-2725)--(6004,-2720)--(6006,-2717)--(6007,-2713)--(6010,-2708)
  --(6014,-2701)--(6020,-2690)--(6028,-2674)--(6039,-2653)--(6053,-2627)--(6068,-2598)
  --(6083,-2569)--(6098,-2541)--(6111,-2517)--(6123,-2494)--(6134,-2474)--(6144,-2455)
  --(6153,-2437)--(6162,-2421)--(6170,-2406)--(6177,-2393)--(6183,-2382)--(6187,-2373)
  --(6190,-2367)--(6192,-2364)--(6193,-2362);
\draw (3255,-7087)--(3256,-7085)--(3258,-7082)--(3261,-7076)--(3265,-7067)--(3271,-7056)
  --(3278,-7043)--(3286,-7028)--(3295,-7011)--(3304,-6994)--(3314,-6975)--(3325,-6954)
  --(3337,-6931)--(3350,-6907)--(3365,-6879)--(3380,-6850)--(3395,-6821)--(3409,-6795)
  --(3420,-6774)--(3428,-6759)--(3434,-6748)--(3438,-6740)--(3441,-6735)--(3442,-6732)
  --(3444,-6728)--(3447,-6724)--(3451,-6716)--(3457,-6705)--(3465,-6690)--(3476,-6669)
  --(3489,-6643)--(3504,-6614)--(3519,-6585)--(3532,-6559)--(3543,-6538)--(3551,-6523)
  --(3557,-6512)--(3560,-6504)--(3563,-6499)--(3565,-6496)--(3566,-6493)--(3569,-6488)
  --(3573,-6480)--(3578,-6469)--(3586,-6454)--(3597,-6433)--(3610,-6407)--(3624,-6378)
  --(3638,-6349)--(3651,-6323)--(3662,-6302)--(3669,-6287)--(3675,-6276)--(3679,-6268)
  --(3681,-6263)--(3683,-6260)--(3684,-6257)--(3687,-6252)--(3690,-6244)--(3696,-6233)
  --(3703,-6218)--(3714,-6197)--(3726,-6171)--(3740,-6142)--(3754,-6113)--(3766,-6087)
  --(3776,-6066)--(3784,-6051)--(3789,-6040)--(3793,-6032)--(3795,-6028)--(3797,-6024)
  --(3798,-6021)--(3801,-6016)--(3804,-6008)--(3809,-5997)--(3817,-5982)--(3826,-5961)
  --(3838,-5935)--(3851,-5906)--(3864,-5877)--(3875,-5851)--(3885,-5830)--(3892,-5814)
  --(3896,-5803)--(3900,-5796)--(3902,-5791)--(3903,-5788)--(3905,-5784)--(3907,-5779)
  --(3910,-5772)--(3914,-5761)--(3921,-5745)--(3930,-5724)--(3940,-5698)--(3952,-5669)
  --(3965,-5636)--(3976,-5608)--(3985,-5586)--(3991,-5571)--(3995,-5561)--(3998,-5555)
  --(3999,-5551)--(4001,-5547)--(4003,-5541)--(4007,-5531)--(4013,-5516)--(4021,-5494)
  --(4031,-5466)--(4042,-5433)--(4053,-5400)--(4063,-5372)--(4070,-5350)--(4075,-5335)
  --(4078,-5325)--(4080,-5319)--(4082,-5315)--(4083,-5311)--(4085,-5305)--(4088,-5295)
  --(4092,-5280)--(4098,-5258)--(4106,-5230)--(4114,-5197)--(4122,-5164)--(4129,-5136)
  --(4134,-5114)--(4137,-5099)--(4140,-5089)--(4141,-5083)--(4142,-5079)--(4143,-5075)
  --(4144,-5069)--(4146,-5059)--(4148,-5044)--(4152,-5022)--(4156,-4994)--(4161,-4961)
  --(4165,-4928)--(4169,-4900)--(4171,-4878)--(4173,-4863)--(4174,-4853)--(4174,-4847)
  --(4175,-4842)--(4175,-4838)--(4175,-4832)--(4176,-4822)--(4176,-4807)--(4177,-4785)
  --(4178,-4757)--(4178,-4724)--(4178,-4691)--(4177,-4663)--(4176,-4641)--(4176,-4626)
  --(4175,-4616)--(4175,-4610)--(4175,-4606)--(4174,-4602)--(4174,-4596)--(4173,-4586)
  --(4171,-4571)--(4169,-4549)--(4165,-4521)--(4161,-4488)--(4156,-4455)--(4152,-4427)
  --(4148,-4405)--(4146,-4390)--(4144,-4380)--(4143,-4374)--(4142,-4370)--(4141,-4366)
  --(4140,-4360)--(4137,-4350)--(4134,-4335)--(4129,-4313)--(4122,-4285)--(4114,-4252)
  --(4106,-4219)--(4098,-4191)--(4092,-4169)--(4088,-4154)--(4085,-4144)--(4083,-4138)
  --(4082,-4134)--(4080,-4130)--(4078,-4124)--(4075,-4114)--(4070,-4099)--(4063,-4077)
  --(4053,-4049)--(4042,-4016)--(4031,-3983)--(4021,-3955)--(4013,-3933)--(4007,-3918)
  --(4003,-3908)--(4001,-3902)--(3999,-3898)--(3998,-3894)--(3995,-3888)--(3991,-3878)
  --(3985,-3863)--(3976,-3841)--(3965,-3813)--(3952,-3780)--(3940,-3751)--(3930,-3725)
  --(3921,-3704)--(3914,-3688)--(3910,-3677)--(3907,-3670)--(3905,-3665)--(3903,-3662)
  --(3902,-3658)--(3900,-3653)--(3896,-3646)--(3892,-3635)--(3885,-3619)--(3875,-3598)
  --(3864,-3572)--(3851,-3543)--(3838,-3514)--(3826,-3488)--(3817,-3467)--(3809,-3452)
  --(3804,-3441)--(3801,-3433)--(3798,-3428)--(3797,-3425)--(3795,-3421)--(3793,-3417)
  --(3789,-3409)--(3784,-3398)--(3776,-3383)--(3766,-3362)--(3754,-3336)--(3740,-3307)
  --(3726,-3278)--(3714,-3252)--(3703,-3231)--(3696,-3216)--(3690,-3205)--(3687,-3197)
  --(3684,-3192)--(3683,-3189)--(3681,-3186)--(3679,-3181)--(3675,-3173)--(3669,-3162)
  --(3662,-3147)--(3651,-3126)--(3638,-3100)--(3624,-3071)--(3610,-3042)--(3597,-3016)
  --(3586,-2995)--(3578,-2980)--(3573,-2969)--(3569,-2961)--(3566,-2957)--(3565,-2953)
  --(3563,-2950)--(3560,-2945)--(3557,-2937)--(3551,-2926)--(3543,-2911)--(3532,-2890)
  --(3519,-2864)--(3504,-2835)--(3489,-2806)--(3476,-2780)--(3465,-2759)--(3457,-2743)
  --(3451,-2732)--(3447,-2725)--(3444,-2720)--(3442,-2717)--(3441,-2713)--(3438,-2708)
  --(3434,-2701)--(3428,-2690)--(3420,-2674)--(3409,-2653)--(3395,-2627)--(3380,-2598)
  --(3365,-2569)--(3350,-2541)--(3337,-2517)--(3325,-2494)--(3314,-2474)--(3304,-2455)
  --(3295,-2437)--(3286,-2421)--(3278,-2406)--(3271,-2393)--(3265,-2382)--(3261,-2373)
  --(3258,-2367)--(3256,-2364)--(3255,-2362);
\pgfsetfillcolor{black}
\pgftext[base,right,at=\pgfqpointxy{7558}{-6518}] {\fontsize{5}{6}\usefont{T1}{ptm}{m}{n}$D_+$}
\pgftext[base,left,at=\pgfqpointxy{1897}{-6529}] {\fontsize{5}{6}\usefont{T1}{ptm}{m}{n}$D_-$}
\pgftext[base,at=\pgfqpointxy{4725}{-2127}] {\fontsize{5}{6}\usefont{T1}{ptm}{m}{n}$D_0$}
\pgftext[base,at=\pgfqpointxy{4722}{-5098}] {\fontsize{5}{6}\usefont{T1}{ptm}{m}{n}$\frac{\alpha}{3 \beta}$}
\filldraw  (4724,-4724) circle [radius=+37];
\endtikzpicture}%
\caption{$\alpha^2 + 3 \beta \delta > 0$}
\end{subfigure}%
\begin{subfigure}{.33\textwidth}
\centering
%
%
{\pgfkeys{/pgf/fpu/.try=false}%
\ifx\XFigwidth\undefined\dimen1=0pt\else\dimen1\XFigwidth\fi
\divide\dimen1 by 5691
\ifx\XFigheight\undefined\dimen3=0pt\else\dimen3\XFigheight\fi
\divide\dimen3 by 5151
\ifdim\dimen1=0pt\ifdim\dimen3=0pt\dimen1=1657sp\dimen3\dimen1
  \else\dimen1\dimen3\fi\else\ifdim\dimen3=0pt\dimen3\dimen1\fi\fi
\tikzpicture[x=+\dimen1, y=+\dimen3]
{\ifx\XFigu\undefined\catcode`\@11
\def\temp{\alloc@1\dimen\dimendef\insc@unt}\temp\XFigu\catcode`\@12\fi}
\XFigu1657sp
\ifdim\XFigu<0pt\XFigu-\XFigu\fi
\pgfdeclarearrow{
  name = xfiga1,
  parameters = {
    \the\pgfarrowlinewidth \the\pgfarrowlength \the\pgfarrowwidth\ifpgfarrowopen o\fi},
  defaults = {
	  line width=+7.5\XFigu, length=+120\XFigu, width=+60\XFigu},
  setup code = {
    \dimen7 2.1\pgfarrowlength\pgfmathveclen{\the\dimen7}{\the\pgfarrowwidth}
    \dimen7 2\pgfarrowwidth\pgfmathdivide{\pgfmathresult}{\the\dimen7}
    \dimen7 \pgfmathresult\pgfarrowlinewidth
    \pgfarrowssettipend{+\dimen7}
    \pgfarrowssetbackend{+-\pgfarrowlength}
    \dimen9 -\pgfarrowlength\advance\dimen9 by-0.45\pgfarrowlinewidth
    \pgfarrowssetlineend{+\dimen9}
    \dimen9 -\pgfarrowlength\advance\dimen9 by-0.5\pgfarrowlinewidth
    \pgfarrowssetvisualbackend{+\dimen9}
    \pgfarrowshullpoint{+\dimen7}{+0pt}
    \pgfarrowsupperhullpoint{+-\pgfarrowlength}{+0.5\pgfarrowwidth}
    \pgfarrowssavethe\pgfarrowlinewidth
    \pgfarrowssavethe\pgfarrowlength
    \pgfarrowssavethe\pgfarrowwidth
  },
  drawing code = {\pgfsetdash{}{+0pt}
    \ifdim\pgfarrowlinewidth=\pgflinewidth\else\pgfsetlinewidth{+\pgfarrowlinewidth}\fi
    \pgfpathmoveto{\pgfqpoint{-\pgfarrowlength}{-0.5\pgfarrowwidth}}
    \pgfpathlineto{\pgfqpoint{0pt}{0pt}}
    \pgfpathlineto{\pgfqpoint{-\pgfarrowlength}{0.5\pgfarrowwidth}}
    \pgfpathclose
    \ifpgfarrowopen\pgfusepathqstroke\else\pgfsetfillcolor{.}
	\ifdim\pgfarrowlinewidth>0pt\pgfusepathqfillstroke\else\pgfusepathqfill\fi\fi
  }
}
\definecolor{xfigc32}{rgb}{0.941,0.894,0.259}
\definecolor{xfigc33}{rgb}{0.000,0.447,0.698}
\definecolor{xfigc34}{rgb}{0.835,0.369,0.000}
\clip(1882,-7104) rectangle (7573,-1953);
\tikzset{inner sep=+0pt, outer sep=+0pt}
\pgfsetfillcolor{black}
\pgfsetlinewidth{+7.5\XFigu}
\pgfsetdash{}{+0pt}
\pgfsetstrokecolor{black}
\pgfsetbeveljoin
\pgfsetfillcolor{xfigc33}
\fill (2365,-7097)--(2365,-7096)--(2365,-7093)--(2365,-7087)--(2365,-7078)--(2365,-7065)
  --(2365,-7048)--(2365,-7026)--(2365,-6998)--(2365,-6964)--(2365,-6924)--(2365,-6878)
  --(2365,-6826)--(2365,-6768)--(2365,-6703)--(2365,-6633)--(2365,-6557)--(2365,-6476)
  --(2365,-6390)--(2365,-6300)--(2365,-6207)--(2365,-6111)--(2365,-6013)--(2365,-5915)
  --(2365,-5817)--(2365,-5719)--(2365,-5623)--(2365,-5530)--(2365,-5440)--(2365,-5354)
  --(2365,-5273)--(2365,-5197)--(2365,-5127)--(2365,-5062)--(2365,-5004)--(2365,-4952)
  --(2365,-4906)--(2365,-4866)--(2365,-4832)--(2365,-4804)--(2365,-4782)--(2365,-4765)
  --(2365,-4752)--(2365,-4743)--(2365,-4737)--(2365,-4734)--(2365,-4733)--(2366,-4733)
  --(2369,-4733)--(2375,-4733)--(2384,-4733)--(2397,-4733)--(2414,-4733)--(2436,-4733)
  --(2464,-4733)--(2498,-4733)--(2537,-4733)--(2583,-4732)--(2635,-4732)--(2694,-4732)
  --(2758,-4732)--(2829,-4732)--(2905,-4732)--(2986,-4731)--(3072,-4731)--(3162,-4731)
  --(3255,-4731)--(3350,-4730)--(3448,-4730)--(3546,-4730)--(3644,-4730)--(3742,-4730)
  --(3837,-4729)--(3930,-4729)--(4020,-4729)--(4106,-4729)--(4187,-4728)--(4263,-4728)
  --(4334,-4728)--(4398,-4728)--(4457,-4728)--(4509,-4728)--(4555,-4727)--(4594,-4727)
  --(4628,-4727)--(4656,-4727)--(4678,-4727)--(4695,-4727)--(4708,-4727)--(4717,-4727)
  --(4723,-4727)--(4726,-4727)--(4727,-4727)--(4726,-4728)--(4725,-4730)--(4723,-4735)
  --(4718,-4742)--(4713,-4752)--(4705,-4766)--(4694,-4784)--(4682,-4807)--(4666,-4834)
  --(4648,-4866)--(4626,-4903)--(4602,-4946)--(4575,-4994)--(4544,-5047)--(4511,-5106)
  --(4474,-5170)--(4435,-5238)--(4393,-5312)--(4349,-5389)--(4303,-5470)--(4255,-5554)
  --(4205,-5641)--(4154,-5730)--(4103,-5821)--(4051,-5912)--(3999,-6003)--(3948,-6093)
  --(3897,-6182)--(3847,-6269)--(3799,-6353)--(3753,-6434)--(3709,-6511)--(3667,-6585)
  --(3628,-6653)--(3591,-6717)--(3558,-6776)--(3527,-6829)--(3500,-6877)--(3476,-6920)
  --(3454,-6957)--(3436,-6989)--(3420,-7016)--(3408,-7039)--(3397,-7057)--(3389,-7071)
  --(3384,-7081)--(3379,-7088)--(3377,-7093)--(3376,-7095)--(3375,-7096);
\pgfsetfillcolor{xfigc34}
\fill (7086,-7088)--(7086,-7087)--(7086,-7084)--(7086,-7078)--(7086,-7069)--(7086,-7056)
  --(7086,-7039)--(7086,-7017)--(7086,-6989)--(7086,-6955)--(7086,-6915)--(7086,-6869)
  --(7086,-6817)--(7086,-6759)--(7086,-6694)--(7086,-6624)--(7086,-6548)--(7086,-6467)
  --(7086,-6381)--(7086,-6291)--(7086,-6198)--(7086,-6102)--(7086,-6004)--(7087,-5906)
  --(7087,-5808)--(7087,-5710)--(7087,-5614)--(7087,-5521)--(7087,-5431)--(7087,-5345)
  --(7087,-5264)--(7087,-5188)--(7087,-5118)--(7087,-5053)--(7087,-4995)--(7087,-4943)
  --(7087,-4897)--(7087,-4857)--(7087,-4823)--(7087,-4795)--(7087,-4773)--(7087,-4756)
  --(7087,-4743)--(7087,-4734)--(7087,-4728)--(7087,-4725)--(7087,-4724)--(7086,-4724)
  --(7083,-4724)--(7077,-4724)--(7068,-4724)--(7055,-4724)--(7038,-4724)--(7016,-4724)
  --(6988,-4724)--(6954,-4724)--(6915,-4724)--(6869,-4724)--(6817,-4724)--(6758,-4724)
  --(6694,-4724)--(6624,-4725)--(6548,-4725)--(6467,-4725)--(6381,-4725)--(6291,-4725)
  --(6198,-4725)--(6103,-4725)--(6005,-4725)--(5907,-4726)--(5809,-4726)--(5711,-4726)
  --(5616,-4726)--(5523,-4726)--(5433,-4726)--(5347,-4726)--(5266,-4726)--(5190,-4726)
  --(5120,-4727)--(5056,-4727)--(4997,-4727)--(4945,-4727)--(4899,-4727)--(4860,-4727)
  --(4826,-4727)--(4798,-4727)--(4776,-4727)--(4759,-4727)--(4746,-4727)--(4737,-4727)
  --(4731,-4727)--(4728,-4727)--(4727,-4727)--(4728,-4728)--(4729,-4730)--(4731,-4735)
  --(4736,-4742)--(4741,-4752)--(4749,-4766)--(4760,-4784)--(4773,-4806)--(4788,-4833)
  --(4807,-4865)--(4828,-4903)--(4853,-4945)--(4880,-4993)--(4911,-5046)--(4945,-5105)
  --(4981,-5168)--(5021,-5236)--(5063,-5309)--(5107,-5386)--(5154,-5467)--(5202,-5551)
  --(5252,-5638)--(5303,-5726)--(5355,-5816)--(5408,-5907)--(5460,-5998)--(5512,-6088)
  --(5563,-6176)--(5613,-6263)--(5661,-6347)--(5708,-6428)--(5752,-6505)--(5794,-6578)
  --(5834,-6646)--(5870,-6709)--(5904,-6768)--(5935,-6821)--(5962,-6869)--(5987,-6911)
  --(6008,-6949)--(6027,-6981)--(6042,-7008)--(6055,-7030)--(6066,-7048)--(6074,-7062)
  --(6079,-7072)--(6084,-7079)--(6086,-7084)--(6087,-7086)--(6088,-7087);
\pgfsetfillcolor{xfigc32}
\fill (3361,-2362)--(3362,-2363)--(3363,-2365)--(3366,-2370)--(3370,-2377)--(3376,-2387)
  --(3384,-2401)--(3394,-2419)--(3407,-2441)--(3423,-2469)--(3441,-2501)--(3463,-2538)
  --(3487,-2581)--(3515,-2629)--(3546,-2682)--(3580,-2740)--(3616,-2804)--(3656,-2873)
  --(3698,-2946)--(3743,-3023)--(3789,-3104)--(3838,-3188)--(3888,-3275)--(3939,-3363)
  --(3992,-3454)--(4044,-3545)--(4096,-3635)--(4149,-3726)--(4200,-3814)--(4250,-3901)
  --(4299,-3985)--(4345,-4066)--(4390,-4143)--(4432,-4216)--(4472,-4285)--(4508,-4349)
  --(4542,-4407)--(4573,-4460)--(4601,-4508)--(4625,-4551)--(4647,-4588)--(4665,-4620)
  --(4681,-4648)--(4694,-4670)--(4704,-4688)--(4712,-4702)--(4718,-4712)--(4722,-4719)
  --(4725,-4724)--(4726,-4726)--(4727,-4727)--(4728,-4726)--(4729,-4724)--(4732,-4719)
  --(4736,-4712)--(4742,-4702)--(4750,-4688)--(4760,-4670)--(4773,-4648)--(4789,-4620)
  --(4807,-4588)--(4829,-4551)--(4854,-4508)--(4882,-4460)--(4913,-4407)--(4947,-4349)
  --(4983,-4285)--(5023,-4216)--(5066,-4143)--(5110,-4066)--(5157,-3985)--(5206,-3901)
  --(5256,-3814)--(5308,-3726)--(5360,-3635)--(5413,-3544)--(5466,-3454)--(5518,-3363)
  --(5570,-3275)--(5620,-3188)--(5669,-3104)--(5716,-3023)--(5760,-2946)--(5803,-2873)
  --(5843,-2804)--(5879,-2740)--(5913,-2682)--(5944,-2629)--(5972,-2581)--(5997,-2538)
  --(6019,-2501)--(6037,-2469)--(6053,-2441)--(6066,-2419)--(6076,-2401)--(6084,-2387)
  --(6090,-2377)--(6094,-2370)--(6097,-2365)--(6098,-2363)--(6099,-2362);
\draw (2362,-4724)--(7087,-4724);
\pgfsetarrowsend{}
\draw (3372,-7087)--(6099,-2362);
\draw (6088,-7087)--(3361,-2362);
\pgfsetfillcolor{black}
\pgftext[base,left,at=\pgfqpointxy{1897}{-6529}] {\fontsize{5}{6}\usefont{T1}{ptm}{m}{n}$D_-$}
\pgftext[base,at=\pgfqpointxy{4725}{-2127}] {\fontsize{5}{6}\usefont{T1}{ptm}{m}{n}$D_0$}
\pgftext[base,right,at=\pgfqpointxy{7558}{-6518}] {\fontsize{5}{6}\usefont{T1}{ptm}{m}{n}$D_+$}
\filldraw  (4727,-4727) circle [radius=+37];
\pgftext[base,left,at=\pgfqpointxy{5034}{-4570}] {\fontsize{5}{6}\usefont{T1}{ptm}{m}{n}$\frac{\alpha}{3 \beta}$}
\endtikzpicture}%
\caption{$\alpha^2 + 3 \beta \delta = 0$}
\end{subfigure}%
\begin{subfigure}{.33\textwidth}
\centering
%
%
{\pgfkeys{/pgf/fpu/.try=false}%
\ifx\XFigwidth\undefined\dimen1=0pt\else\dimen1\XFigwidth\fi
\divide\dimen1 by 5694
\ifx\XFigheight\undefined\dimen3=0pt\else\dimen3\XFigheight\fi
\divide\dimen3 by 5153
\ifdim\dimen1=0pt\ifdim\dimen3=0pt\dimen1=1657sp\dimen3\dimen1
  \else\dimen1\dimen3\fi\else\ifdim\dimen3=0pt\dimen3\dimen1\fi\fi
\tikzpicture[x=+\dimen1, y=+\dimen3]
{\ifx\XFigu\undefined\catcode`\@11
\def\temp{\alloc@1\dimen\dimendef\insc@unt}\temp\XFigu\catcode`\@12\fi}
\XFigu1657sp
\ifdim\XFigu<0pt\XFigu-\XFigu\fi
\pgfdeclarearrow{
  name = xfiga1,
  parameters = {
    \the\pgfarrowlinewidth \the\pgfarrowlength \the\pgfarrowwidth\ifpgfarrowopen o\fi},
  defaults = {
	  line width=+7.5\XFigu, length=+120\XFigu, width=+60\XFigu},
  setup code = {
    \dimen7 2.1\pgfarrowlength\pgfmathveclen{\the\dimen7}{\the\pgfarrowwidth}
    \dimen7 2\pgfarrowwidth\pgfmathdivide{\pgfmathresult}{\the\dimen7}
    \dimen7 \pgfmathresult\pgfarrowlinewidth
    \pgfarrowssettipend{+\dimen7}
    \pgfarrowssetbackend{+-\pgfarrowlength}
    \dimen9 -\pgfarrowlength\advance\dimen9 by-0.45\pgfarrowlinewidth
    \pgfarrowssetlineend{+\dimen9}
    \dimen9 -\pgfarrowlength\advance\dimen9 by-0.5\pgfarrowlinewidth
    \pgfarrowssetvisualbackend{+\dimen9}
    \pgfarrowshullpoint{+\dimen7}{+0pt}
    \pgfarrowsupperhullpoint{+-\pgfarrowlength}{+0.5\pgfarrowwidth}
    \pgfarrowssavethe\pgfarrowlinewidth
    \pgfarrowssavethe\pgfarrowlength
    \pgfarrowssavethe\pgfarrowwidth
  },
  drawing code = {\pgfsetdash{}{+0pt}
    \ifdim\pgfarrowlinewidth=\pgflinewidth\else\pgfsetlinewidth{+\pgfarrowlinewidth}\fi
    \pgfpathmoveto{\pgfqpoint{-\pgfarrowlength}{-0.5\pgfarrowwidth}}
    \pgfpathlineto{\pgfqpoint{0pt}{0pt}}
    \pgfpathlineto{\pgfqpoint{-\pgfarrowlength}{0.5\pgfarrowwidth}}
    \pgfpathclose
    \ifpgfarrowopen\pgfusepathqstroke\else\pgfsetfillcolor{.}
	\ifdim\pgfarrowlinewidth>0pt\pgfusepathqfillstroke\else\pgfusepathqfill\fi\fi
  }
}
\definecolor{xfigc32}{rgb}{0.941,0.894,0.259}
\definecolor{xfigc33}{rgb}{0.000,0.447,0.698}
\definecolor{xfigc34}{rgb}{0.835,0.369,0.000}
\clip(1882,-7105) rectangle (7576,-1952);
\tikzset{inner sep=+0pt, outer sep=+0pt}
\pgfsetlinewidth{+7.5\XFigu}
\pgfsetdash{}{+0pt}
\pgfsetstrokecolor{black}
\pgfsetfillcolor{black}
\pgfsetbeveljoin
\pgfsetfillcolor{xfigc33}
\fill (2362,-7087)--(2362,-7086)--(2362,-7083)--(2362,-7077)--(2362,-7068)--(2362,-7055)
  --(2362,-7038)--(2362,-7016)--(2362,-6988)--(2362,-6954)--(2362,-6914)--(2362,-6869)
  --(2362,-6816)--(2362,-6758)--(2362,-6693)--(2362,-6623)--(2362,-6547)--(2362,-6466)
  --(2362,-6380)--(2362,-6290)--(2362,-6197)--(2362,-6101)--(2362,-6004)--(2362,-5905)
  --(2362,-5807)--(2362,-5710)--(2362,-5614)--(2362,-5521)--(2362,-5431)--(2362,-5345)
  --(2362,-5264)--(2362,-5188)--(2362,-5118)--(2362,-5053)--(2362,-4995)--(2362,-4942)
  --(2362,-4897)--(2362,-4857)--(2362,-4823)--(2362,-4795)--(2362,-4773)--(2362,-4756)
  --(2362,-4743)--(2362,-4734)--(2362,-4728)--(2362,-4725)--(2362,-4724)--(2363,-4724)
  --(2366,-4724)--(2372,-4724)--(2381,-4724)--(2394,-4724)--(2411,-4724)--(2433,-4724)
  --(2461,-4724)--(2495,-4724)--(2534,-4724)--(2580,-4724)--(2632,-4724)--(2690,-4723)
  --(2755,-4723)--(2825,-4723)--(2901,-4723)--(2982,-4723)--(3068,-4723)--(3157,-4723)
  --(3250,-4722)--(3346,-4722)--(3443,-4722)--(3541,-4722)--(3639,-4722)--(3736,-4722)
  --(3832,-4722)--(3925,-4721)--(4014,-4721)--(4100,-4721)--(4181,-4721)--(4257,-4721)
  --(4327,-4721)--(4392,-4721)--(4450,-4720)--(4502,-4720)--(4548,-4720)--(4587,-4720)
  --(4621,-4720)--(4649,-4720)--(4671,-4720)--(4688,-4720)--(4701,-4720)--(4710,-4720)
  --(4716,-4720)--(4719,-4720)--(4720,-4720)--(4720,-4723)--(4720,-4729)--(4720,-4741)
  --(4720,-4759)--(4720,-4784)--(4721,-4816)--(4721,-4854)--(4721,-4897)--(4722,-4945)
  --(4722,-4995)--(4722,-5044)--(4723,-5092)--(4723,-5135)--(4723,-5173)--(4724,-5205)
  --(4724,-5230)--(4724,-5248)--(4724,-5260)--(4724,-5266)--(4724,-5269)--(4722,-5269)
  --(4719,-5270)--(4713,-5272)--(4704,-5274)--(4693,-5278)--(4680,-5282)--(4665,-5288)
  --(4649,-5295)--(4631,-5303)--(4612,-5313)--(4592,-5325)--(4569,-5340)--(4544,-5358)
  --(4517,-5380)--(4488,-5406)--(4462,-5431)--(4439,-5455)--(4419,-5476)--(4403,-5494)
  --(4391,-5508)--(4383,-5518)--(4377,-5526)--(4373,-5533)--(4370,-5538)--(4367,-5544)
  --(4363,-5550)--(4357,-5560)--(4349,-5572)--(4337,-5589)--(4321,-5611)--(4301,-5639)
  --(4278,-5671)--(4252,-5707)--(4229,-5740)--(4207,-5771)--(4188,-5798)--(4173,-5821)
  --(4161,-5839)--(4152,-5852)--(4145,-5863)--(4140,-5870)--(4137,-5876)--(4134,-5880)
  --(4131,-5885)--(4128,-5891)--(4123,-5898)--(4116,-5909)--(4107,-5924)--(4095,-5943)
  --(4080,-5967)--(4061,-5997)--(4039,-6030)--(4016,-6067)--(3993,-6104)--(3971,-6138)
  --(3952,-6168)--(3937,-6192)--(3925,-6212)--(3916,-6227)--(3909,-6237)--(3904,-6245)
  --(3901,-6251)--(3898,-6255)--(3895,-6260)--(3892,-6266)--(3887,-6273)--(3881,-6284)
  --(3871,-6300)--(3859,-6320)--(3844,-6345)--(3825,-6376)--(3803,-6411)--(3780,-6449)
  --(3757,-6487)--(3735,-6523)--(3716,-6554)--(3700,-6581)--(3687,-6602)--(3677,-6619)
  --(3669,-6632)--(3663,-6642)--(3658,-6650)--(3654,-6656)--(3650,-6663)--(3646,-6671)
  --(3640,-6680)--(3633,-6692)--(3623,-6707)--(3612,-6727)--(3598,-6750)--(3581,-6778)
  --(3562,-6809)--(3543,-6841)--(3520,-6879)--(3500,-6912)--(3485,-6938)--(3473,-6959)
  --(3464,-6973)--(3458,-6983)--(3453,-6991)--(3450,-6996)--(3447,-7001)--(3444,-7007)
  --(3440,-7014)--(3434,-7024)--(3426,-7037)--(3417,-7052)--(3407,-7069)--(3397,-7086);
\pgfsetfillcolor{xfigc32}
\fill (3397,-2362)--(3398,-2363)--(3399,-2366)--(3402,-2370)--(3406,-2377)--(3411,-2386)
  --(3417,-2397)--(3425,-2409)--(3434,-2424)--(3443,-2440)--(3454,-2459)--(3467,-2481)
  --(3482,-2506)--(3500,-2535)--(3520,-2569)--(3543,-2607)--(3562,-2639)--(3581,-2670)
  --(3598,-2698)--(3612,-2721)--(3623,-2741)--(3633,-2756)--(3640,-2768)--(3646,-2777)
  --(3650,-2785)--(3654,-2792)--(3658,-2798)--(3663,-2806)--(3669,-2816)--(3677,-2829)
  --(3687,-2846)--(3700,-2867)--(3716,-2894)--(3735,-2925)--(3757,-2961)--(3780,-2999)
  --(3803,-3037)--(3825,-3072)--(3844,-3103)--(3859,-3128)--(3871,-3148)--(3881,-3164)
  --(3887,-3175)--(3892,-3182)--(3895,-3188)--(3898,-3193)--(3901,-3197)--(3904,-3203)
  --(3909,-3211)--(3916,-3221)--(3925,-3236)--(3937,-3256)--(3952,-3280)--(3971,-3310)
  --(3993,-3344)--(4016,-3381)--(4039,-3418)--(4061,-3451)--(4080,-3481)--(4095,-3505)
  --(4107,-3524)--(4116,-3539)--(4123,-3550)--(4128,-3557)--(4131,-3563)--(4134,-3568)
  --(4137,-3572)--(4140,-3578)--(4145,-3585)--(4152,-3596)--(4161,-3609)--(4173,-3627)
  --(4188,-3650)--(4207,-3677)--(4229,-3708)--(4252,-3741)--(4278,-3777)--(4301,-3809)
  --(4321,-3837)--(4337,-3859)--(4349,-3876)--(4357,-3888)--(4363,-3898)--(4367,-3904)
  --(4370,-3910)--(4373,-3915)--(4377,-3922)--(4383,-3930)--(4391,-3940)--(4403,-3954)
  --(4419,-3972)--(4439,-3993)--(4462,-4017)--(4488,-4042)--(4514,-4065)--(4537,-4085)
  --(4557,-4102)--(4573,-4115)--(4585,-4125)--(4593,-4133)--(4599,-4139)--(4603,-4143)
  --(4606,-4147)--(4609,-4150)--(4613,-4153)--(4619,-4157)--(4627,-4161)--(4639,-4165)
  --(4655,-4170)--(4675,-4174)--(4698,-4178)--(4724,-4179)--(4750,-4178)--(4773,-4174)
  --(4793,-4170)--(4809,-4165)--(4821,-4161)--(4830,-4157)--(4835,-4153)--(4839,-4150)
  --(4843,-4147)--(4846,-4143)--(4850,-4139)--(4855,-4133)--(4864,-4125)--(4876,-4115)
  --(4892,-4102)--(4912,-4085)--(4935,-4065)--(4961,-4042)--(4987,-4017)--(5010,-3993)
  --(5030,-3972)--(5046,-3954)--(5058,-3940)--(5066,-3930)--(5072,-3922)--(5076,-3915)
  --(5079,-3910)--(5082,-3904)--(5086,-3898)--(5092,-3888)--(5100,-3876)--(5112,-3859)
  --(5128,-3837)--(5148,-3809)--(5171,-3777)--(5197,-3741)--(5220,-3708)--(5242,-3677)
  --(5261,-3650)--(5276,-3627)--(5288,-3609)--(5297,-3596)--(5304,-3585)--(5309,-3578)
  --(5312,-3572)--(5315,-3568)--(5318,-3563)--(5321,-3557)--(5326,-3550)--(5333,-3539)
  --(5342,-3524)--(5354,-3505)--(5369,-3481)--(5388,-3451)--(5410,-3418)--(5433,-3381)
  --(5456,-3344)--(5479,-3309)--(5500,-3275)--(5519,-3245)--(5537,-3216)--(5553,-3189)
  --(5569,-3164)--(5583,-3140)--(5597,-3118)--(5610,-3097)--(5622,-3077)--(5633,-3058)
  --(5643,-3042)--(5651,-3028)--(5658,-3017)--(5663,-3009)--(5666,-3003)--(5668,-3000)
  --(5669,-2999)--(5671,-2996)--(5674,-2990)--(5681,-2979)--(5691,-2963)--(5705,-2940)
  --(5722,-2912)--(5742,-2879)--(5764,-2842)--(5788,-2803)--(5811,-2764)--(5833,-2727)
  --(5853,-2694)--(5870,-2666)--(5884,-2643)--(5894,-2627)--(5901,-2616)--(5904,-2610)
  --(5906,-2607)--(5908,-2604)--(5912,-2597)--(5919,-2584)--(5930,-2566)--(5944,-2543)
  --(5961,-2515)--(5979,-2485)--(5996,-2454)--(6013,-2426)--(6027,-2403)--(6038,-2385)
  --(6045,-2372)--(6049,-2365)--(6051,-2362);
\pgfsetfillcolor{xfigc34}
\fill (4720,-4720)--(4721,-4720)--(4724,-4720)--(4730,-4720)--(4739,-4720)--(4752,-4720)
  --(4769,-4720)--(4792,-4720)--(4819,-4720)--(4853,-4720)--(4893,-4720)--(4939,-4720)
  --(4991,-4721)--(5050,-4721)--(5114,-4721)--(5185,-4721)--(5261,-4721)--(5342,-4721)
  --(5428,-4721)--(5518,-4722)--(5612,-4722)--(5707,-4722)--(5805,-4722)--(5904,-4723)
  --(6002,-4723)--(6100,-4723)--(6195,-4723)--(6289,-4723)--(6379,-4724)--(6465,-4724)
  --(6546,-4724)--(6622,-4724)--(6693,-4724)--(6757,-4724)--(6816,-4724)--(6868,-4725)
  --(6914,-4725)--(6954,-4725)--(6988,-4725)--(7015,-4725)--(7038,-4725)--(7055,-4725)
  --(7068,-4725)--(7077,-4725)--(7083,-4725)--(7086,-4725)--(7087,-4725)--(7087,-4726)
  --(7087,-4729)--(7087,-4735)--(7087,-4744)--(7087,-4757)--(7087,-4774)--(7087,-4796)
  --(7087,-4824)--(7087,-4858)--(7087,-4898)--(7087,-4943)--(7087,-4996)--(7087,-5054)
  --(7087,-5119)--(7087,-5189)--(7087,-5265)--(7087,-5346)--(7087,-5432)--(7087,-5522)
  --(7087,-5615)--(7087,-5711)--(7087,-5808)--(7087,-5907)--(7087,-6005)--(7087,-6102)
  --(7087,-6198)--(7087,-6291)--(7087,-6381)--(7087,-6467)--(7087,-6548)--(7087,-6624)
  --(7087,-6694)--(7087,-6759)--(7087,-6817)--(7087,-6870)--(7087,-6915)--(7087,-6955)
  --(7087,-6989)--(7087,-7017)--(7087,-7039)--(7087,-7056)--(7087,-7069)--(7087,-7078)
  --(7087,-7084)--(7087,-7087)--(7087,-7088)--(7085,-7088)--(7081,-7088)--(7073,-7088)
  --(7061,-7088)--(7043,-7088)--(7020,-7088)--(6991,-7088)--(6956,-7088)--(6914,-7088)
  --(6867,-7088)--(6815,-7087)--(6758,-7087)--(6697,-7087)--(6634,-7087)--(6569,-7087)
  --(6504,-7087)--(6441,-7087)--(6380,-7087)--(6323,-7087)--(6271,-7086)--(6224,-7086)
  --(6182,-7086)--(6147,-7086)--(6118,-7086)--(6095,-7086)--(6077,-7086)--(6065,-7086)
  --(6057,-7086)--(6053,-7086)--(6051,-7086)--(6050,-7085)--(6049,-7082)--(6046,-7078)
  --(6042,-7071)--(6037,-7062)--(6031,-7051)--(6023,-7039)--(6015,-7024)--(6005,-7008)
  --(5994,-6989)--(5981,-6967)--(5967,-6942)--(5949,-6913)--(5929,-6879)--(5906,-6841)
  --(5887,-6809)--(5868,-6778)--(5852,-6750)--(5837,-6727)--(5826,-6707)--(5816,-6692)
  --(5809,-6680)--(5804,-6671)--(5799,-6663)--(5795,-6656)--(5791,-6650)--(5786,-6642)
  --(5780,-6632)--(5773,-6619)--(5762,-6602)--(5749,-6581)--(5733,-6554)--(5714,-6523)
  --(5692,-6487)--(5669,-6449)--(5646,-6411)--(5624,-6376)--(5605,-6345)--(5590,-6320)
  --(5578,-6300)--(5568,-6284)--(5562,-6273)--(5557,-6266)--(5554,-6260)--(5551,-6255)
  --(5548,-6251)--(5545,-6245)--(5540,-6237)--(5533,-6227)--(5524,-6212)--(5512,-6192)
  --(5497,-6168)--(5478,-6138)--(5456,-6104)--(5433,-6067)--(5410,-6030)--(5388,-5997)
  --(5369,-5967)--(5354,-5943)--(5342,-5924)--(5333,-5909)--(5326,-5898)--(5321,-5891)
  --(5318,-5885)--(5315,-5880)--(5312,-5876)--(5309,-5870)--(5304,-5863)--(5297,-5852)
  --(5288,-5839)--(5276,-5821)--(5261,-5798)--(5242,-5771)--(5220,-5740)--(5197,-5707)
  --(5171,-5671)--(5148,-5639)--(5128,-5611)--(5112,-5589)--(5100,-5572)--(5092,-5560)
  --(5086,-5550)--(5082,-5544)--(5079,-5538)--(5076,-5533)--(5072,-5526)--(5066,-5518)
  --(5058,-5508)--(5046,-5494)--(5030,-5476)--(5010,-5455)--(4987,-5431)--(4961,-5406)
  --(4932,-5380)--(4904,-5358)--(4880,-5340)--(4857,-5325)--(4836,-5313)--(4817,-5303)
  --(4800,-5295)--(4783,-5288)--(4768,-5282)--(4755,-5278)--(4744,-5274)--(4735,-5272)
  --(4729,-5270)--(4726,-5269)--(4724,-5269);
\draw (2362,-4724)--(7087,-4724);
\pgfsetarrowsend{}
\draw (3397,-7086)--(3398,-7085)--(3399,-7082)--(3402,-7078)--(3406,-7071)--(3411,-7062)
  --(3417,-7051)--(3425,-7039)--(3434,-7024)--(3443,-7008)--(3454,-6989)--(3467,-6967)
  --(3482,-6942)--(3500,-6913)--(3520,-6879)--(3543,-6841)--(3562,-6809)--(3581,-6778)
  --(3598,-6750)--(3612,-6727)--(3623,-6707)--(3633,-6692)--(3640,-6680)--(3646,-6671)
  --(3650,-6663)--(3654,-6656)--(3658,-6650)--(3663,-6642)--(3669,-6632)--(3677,-6619)
  --(3687,-6602)--(3700,-6581)--(3716,-6554)--(3735,-6523)--(3757,-6487)--(3780,-6449)
  --(3803,-6411)--(3825,-6376)--(3844,-6345)--(3859,-6320)--(3871,-6300)--(3881,-6284)
  --(3887,-6273)--(3892,-6266)--(3895,-6260)--(3898,-6255)--(3901,-6251)--(3904,-6245)
  --(3909,-6237)--(3916,-6227)--(3925,-6212)--(3937,-6192)--(3952,-6168)--(3971,-6138)
  --(3993,-6104)--(4016,-6067)--(4039,-6030)--(4061,-5997)--(4080,-5967)--(4095,-5943)
  --(4107,-5924)--(4116,-5909)--(4123,-5898)--(4128,-5891)--(4131,-5885)--(4134,-5880)
  --(4137,-5876)--(4140,-5870)--(4145,-5863)--(4152,-5852)--(4161,-5839)--(4173,-5821)
  --(4188,-5798)--(4207,-5771)--(4229,-5740)--(4252,-5707)--(4278,-5671)--(4301,-5639)
  --(4321,-5611)--(4337,-5589)--(4349,-5572)--(4357,-5560)--(4363,-5550)--(4367,-5544)
  --(4370,-5538)--(4373,-5533)--(4377,-5526)--(4383,-5518)--(4391,-5508)--(4403,-5494)
  --(4419,-5476)--(4439,-5455)--(4462,-5431)--(4488,-5406)--(4514,-5383)--(4537,-5363)
  --(4557,-5346)--(4573,-5333)--(4585,-5323)--(4593,-5315)--(4599,-5309)--(4603,-5305)
  --(4606,-5301)--(4609,-5298)--(4613,-5295)--(4619,-5291)--(4627,-5287)--(4639,-5283)
  --(4655,-5278)--(4675,-5274)--(4698,-5270)--(4724,-5269)--(4750,-5270)--(4773,-5274)
  --(4793,-5278)--(4809,-5283)--(4821,-5287)--(4830,-5291)--(4835,-5295)--(4839,-5298)
  --(4843,-5301)--(4846,-5305)--(4850,-5309)--(4855,-5315)--(4864,-5323)--(4876,-5333)
  --(4892,-5346)--(4912,-5363)--(4935,-5383)--(4961,-5406)--(4987,-5431)--(5010,-5455)
  --(5030,-5476)--(5046,-5494)--(5058,-5508)--(5066,-5518)--(5072,-5526)--(5076,-5533)
  --(5079,-5538)--(5082,-5544)--(5086,-5550)--(5092,-5560)--(5100,-5572)--(5112,-5589)
  --(5128,-5611)--(5148,-5639)--(5171,-5671)--(5197,-5707)--(5220,-5740)--(5242,-5771)
  --(5261,-5798)--(5276,-5821)--(5288,-5839)--(5297,-5852)--(5304,-5863)--(5309,-5870)
  --(5312,-5876)--(5315,-5880)--(5318,-5885)--(5321,-5891)--(5326,-5898)--(5333,-5909)
  --(5342,-5924)--(5354,-5943)--(5369,-5967)--(5388,-5997)--(5410,-6030)--(5433,-6067)
  --(5456,-6104)--(5478,-6138)--(5497,-6168)--(5512,-6192)--(5524,-6212)--(5533,-6227)
  --(5540,-6237)--(5545,-6245)--(5548,-6251)--(5551,-6255)--(5554,-6260)--(5557,-6266)
  --(5562,-6273)--(5568,-6284)--(5578,-6300)--(5590,-6320)--(5605,-6345)--(5624,-6376)
  --(5646,-6411)--(5669,-6449)--(5692,-6487)--(5714,-6523)--(5733,-6554)--(5749,-6581)
  --(5762,-6602)--(5773,-6619)--(5780,-6632)--(5786,-6642)--(5791,-6650)--(5795,-6656)
  --(5799,-6663)--(5804,-6671)--(5809,-6680)--(5816,-6692)--(5826,-6707)--(5837,-6727)
  --(5852,-6750)--(5868,-6778)--(5887,-6809)--(5906,-6841)--(5929,-6879)--(5949,-6913)
  --(5967,-6942)--(5981,-6967)--(5994,-6989)--(6005,-7008)--(6015,-7024)--(6023,-7039)
  --(6031,-7051)--(6037,-7062)--(6042,-7071)--(6046,-7078)--(6049,-7082)--(6050,-7085)
  --(6051,-7086);
\draw (3397,-2362)--(3398,-2363)--(3399,-2366)--(3402,-2370)--(3406,-2377)--(3411,-2386)
  --(3417,-2397)--(3425,-2409)--(3434,-2424)--(3443,-2440)--(3454,-2459)--(3467,-2481)
  --(3482,-2506)--(3500,-2535)--(3520,-2569)--(3543,-2607)--(3562,-2639)--(3581,-2670)
  --(3598,-2698)--(3612,-2721)--(3623,-2741)--(3633,-2756)--(3640,-2768)--(3646,-2777)
  --(3650,-2785)--(3654,-2792)--(3658,-2798)--(3663,-2806)--(3669,-2816)--(3677,-2829)
  --(3687,-2846)--(3700,-2867)--(3716,-2894)--(3735,-2925)--(3757,-2961)--(3780,-2999)
  --(3803,-3037)--(3825,-3072)--(3844,-3103)--(3859,-3128)--(3871,-3148)--(3881,-3164)
  --(3887,-3175)--(3892,-3182)--(3895,-3188)--(3898,-3193)--(3901,-3197)--(3904,-3203)
  --(3909,-3211)--(3916,-3221)--(3925,-3236)--(3937,-3256)--(3952,-3280)--(3971,-3310)
  --(3993,-3344)--(4016,-3381)--(4039,-3418)--(4061,-3451)--(4080,-3481)--(4095,-3505)
  --(4107,-3524)--(4116,-3539)--(4123,-3550)--(4128,-3557)--(4131,-3563)--(4134,-3568)
  --(4137,-3572)--(4140,-3578)--(4145,-3585)--(4152,-3596)--(4161,-3609)--(4173,-3627)
  --(4188,-3650)--(4207,-3677)--(4229,-3708)--(4252,-3741)--(4278,-3777)--(4301,-3809)
  --(4321,-3837)--(4337,-3859)--(4349,-3876)--(4357,-3888)--(4363,-3898)--(4367,-3904)
  --(4370,-3910)--(4373,-3915)--(4377,-3922)--(4383,-3930)--(4391,-3940)--(4403,-3954)
  --(4419,-3972)--(4439,-3993)--(4462,-4017)--(4488,-4042)--(4514,-4065)--(4537,-4085)
  --(4557,-4102)--(4573,-4115)--(4585,-4125)--(4593,-4133)--(4599,-4139)--(4603,-4143)
  --(4606,-4147)--(4609,-4150)--(4613,-4153)--(4619,-4157)--(4627,-4161)--(4639,-4165)
  --(4655,-4170)--(4675,-4174)--(4698,-4178)--(4724,-4179)--(4750,-4178)--(4773,-4174)
  --(4793,-4170)--(4809,-4165)--(4821,-4161)--(4830,-4157)--(4835,-4153)--(4839,-4150)
  --(4843,-4147)--(4846,-4143)--(4850,-4139)--(4855,-4133)--(4864,-4125)--(4876,-4115)
  --(4892,-4102)--(4912,-4085)--(4935,-4065)--(4961,-4042)--(4987,-4017)--(5010,-3993)
  --(5030,-3972)--(5046,-3954)--(5058,-3940)--(5066,-3930)--(5072,-3922)--(5076,-3915)
  --(5079,-3910)--(5082,-3904)--(5086,-3898)--(5092,-3888)--(5100,-3876)--(5112,-3859)
  --(5128,-3837)--(5148,-3809)--(5171,-3777)--(5197,-3741)--(5220,-3708)--(5242,-3677)
  --(5261,-3650)--(5276,-3627)--(5288,-3609)--(5297,-3596)--(5304,-3585)--(5309,-3578)
  --(5312,-3572)--(5315,-3568)--(5318,-3563)--(5321,-3557)--(5326,-3550)--(5333,-3539)
  --(5342,-3524)--(5354,-3505)--(5369,-3481)--(5388,-3451)--(5410,-3418)--(5433,-3381)
  --(5456,-3344)--(5478,-3310)--(5497,-3280)--(5512,-3256)--(5524,-3236)--(5533,-3221)
  --(5540,-3211)--(5545,-3203)--(5548,-3197)--(5551,-3193)--(5554,-3188)--(5557,-3182)
  --(5562,-3175)--(5568,-3164)--(5578,-3148)--(5590,-3128)--(5605,-3103)--(5624,-3072)
  --(5646,-3037)--(5669,-2999)--(5692,-2961)--(5714,-2925)--(5733,-2894)--(5749,-2867)
  --(5762,-2846)--(5773,-2829)--(5780,-2816)--(5786,-2806)--(5791,-2798)--(5795,-2792)
  --(5799,-2785)--(5804,-2777)--(5809,-2768)--(5816,-2756)--(5826,-2741)--(5837,-2721)
  --(5852,-2698)--(5868,-2670)--(5887,-2639)--(5906,-2607)--(5929,-2569)--(5949,-2535)
  --(5967,-2506)--(5981,-2481)--(5994,-2459)--(6005,-2440)--(6015,-2424)--(6023,-2409)
  --(6031,-2397)--(6037,-2386)--(6042,-2377)--(6046,-2370)--(6049,-2366)--(6050,-2363)
  --(6051,-2362);
\pgfsetfillcolor{black}
\pgftext[base,at=\pgfqpointxy{4725}{-2127}] {\fontsize{5}{6}\usefont{T1}{ptm}{m}{n}$D_0$}
\pgftext[base,right,at=\pgfqpointxy{7561}{-6770}] {\fontsize{5}{6}\usefont{T1}{ptm}{m}{n}$D_+$}
\pgftext[base,left,at=\pgfqpointxy{1897}{-6529}] {\fontsize{5}{6}\usefont{T1}{ptm}{m}{n}$D_-$}
\pgftext[base,at=\pgfqpointxy{4729}{-4485}] {\fontsize{5}{6}\usefont{T1}{ptm}{m}{n}$\frac{\alpha}{3 \beta}$}
\filldraw  (4724,-4724) circle [radius=+37];
\endtikzpicture}%
\caption{$\alpha^2 + 3 \beta \delta < 0$}
\end{subfigure}
\caption{The regions $D_0$, $D_+$, $D_-$ depending on the sign of the quantity $\alpha^2 + 3 \beta \delta$.}
\label{fig:Dpm}
\end{figure}

The symmetries of $\omega$ can be identified by solving the equation $\omega(\nu) = \omega(k)$ for $\nu = \nu(k)$. There are three solutions to this polynomial equation: the trivial solution $\nu_0 = k$, and two nontrivial solutions given by $\nu_\pm$ in \eqref{nupm}. The latter solutions involve the complex square root
\eq{\label{csr}
\Big[\Big(k - \frac{\alpha}{3 \beta}\Big)^2 - \frac{4}{9 \beta^2} \p{\alpha^2 + 3 \beta \delta}\Big]^{\frac 12},
}
which is made single-valued by introducing an appropriate branch cut in the complex $k$-plane as follows. 
When $\alpha^2 + 3 \beta \delta > 0$,  \eqref{csr} has two real branch points given by
\eq{b_\pm = \frac{\alpha}{3 \beta} \pm \frac{2}{3 \beta} \sqrt{\alpha^2 + 3 \beta \delta}.} 
Then, taking a branch cut from $b_-$ to $b_+$ along the real line, we define
\eq{
\Big[\Big(k - \frac{\alpha}{3 \beta}\Big)^2 - \frac{4}{9 \beta^2} \p{\alpha^2 + 3 \beta \delta}\Big]^{\frac 12} &= \b{\p{k - b_-} \p{k - b_+}}^{\frac{1}{2}}
= \sqrt{r_- r_+} e^{i \frac{\theta_- + \theta_+}{2}}, \label{rootdefp}}
where the radii $r_\pm = \left|k-b_\pm\right| \geq 0$ and the counterclockwise angles $\theta_\pm\in [0, 2\pi)$ are shown in Figure \ref{fig:Dbranches}.
When $\alpha^2 + 3 \beta \delta < 0$, the branch points of \eqref{csr} are complex and given by 
\eq{b_\pm = \frac{\alpha}{3 \beta} \pm \frac{2}{3 \beta} i \sqrt{\abs{\alpha^2 + 3 \beta \delta}}.} 
Thus, we define $r_\pm$ as before but measure the counterclockwise angles $\theta_\pm \in [0, 2\pi)$ from the rays emanating vertically from $b_\pm$, resulting in the following definition for \eqref{csr}:
\eq{
\Big[\Big(k - \frac{\alpha}{3 \beta}\Big)^2 - \frac{4}{9 \beta^2} \p{\alpha^2 + 3 \beta \delta}\Big]^{\frac 12} 
&= \b{\p{k - b_-} \p{k - b_+}}^{\frac{1}{2}}
= \sqrt{r_- r_+} e^{i \frac{\theta_- + \theta_+ + \pi}{2}}. \label{rootdefm}}
Finally, when $\alpha^2 + 3 \beta \delta = 0$, the zeros of the radicand in \eqref{csr} coincide at $b_\pm = \frac{\alpha}{3 \beta}$. Therefore, since both signs are already present in \eqref{nupm}, we simply define \eqref{csr} to be
\eq{\Big[\Big(k - \frac{\alpha}{3 \beta}\Big)^2 - \frac{4}{9 \beta^2} \p{\alpha^2 + 3 \beta \delta}\Big]^{\frac 12} = k - \frac{\alpha}{3 \beta}. \label{rootdefz}}

At this point, we provide an alternative characterization of the regions $\tilde D_0$, $\tilde D_+$, $\tilde D_-$ that will be useful when employing the symmetries of $\omega$ later. We begin with the following observation:
\begin{lemma} \label{lemma:Dchar}
Let $k\in \C$ off the branch cut of $\nu_\pm$. Then, $\Im(\omega) < 0$ if and only if $\Im(\nu_0) \Im(\nu_+) \Im(\nu_-) > 0$. Moreover, the zeros of $\Im(\nu_0) \Im(\nu_+) \Im(\nu_-)$ and of $\Im(\omega)$, excluding those that may lie on the branch cut, are identical. In other words, off the branch cut, the sign of $\Im(\nu_0) \Im(\nu_+) \Im(\nu_-)$ is the opposite of the sign of $\Im(\omega)$.
\end{lemma}

\begin{proof}
Letting $c = \frac{\alpha^2 + 3 \beta \delta}{3 \beta^2}$ and $z = k - \frac{\alpha}{3 \beta}$, we compute
\eq{\Im(\nu_\pm) &= -\frac{1}{2} \Im(z) \pm \frac{\sqrt 3}{2} \Re \Big[\Big(z^2 - \frac{4}{3} c\Big)^\frac{1}{2}\Big], \label{imnupm}}
Applying the identities $\big[\Re(\zeta^{\frac{1}{2}})\big]^2 = \frac{1}{2} \p{\abs \zeta + \Re(\zeta)}$ and $\big[\Im(\zeta^{\frac{1}{2}})\big]^2 = \frac{1}{2} \p{\abs \zeta - \Re(\zeta)}$, which are valid for any $\zeta \in \C$, and multiplying by conjugates, we eventually obtain
\eq{
\Im(\nu_+) \Im(\nu_-) 
&= \frac{1}{8} \Big(4 c + 5 \b{\Im(z)}^2 - 3 \b{\Re(z)}^2 - 3 \Big|z^2 - \frac{4}{3} c\Big|\Big)
\nn\\
&= \p{c + \b{\Im(z)}^2 - 3 \b{\Re(z)}^2} \frac{\b{\Im(z)}^2}{\b{\Im(z)}^2 + \frac{3}{2} \b{\abs{z^2 - \frac{4}{3} c} - \Re \p{z^2 - \frac{4}{3} c}}}
= H(k) Z(k), \label{imsnupmHZ}
\\
H(k) &:= \frac{\alpha^2 + 3 \beta \delta}{3 \beta^2} + \b{\Im(k)}^2 - 3 \Big[\Re\Big(k - \frac{\alpha}{3 \beta}\Big)\Big]^2, \label{hyperbola}
\\
Z(k) &:= \frac{\b{\Im(k)}^2}{\b{\Im(k)}^2 + 3 \Big[\Im \Big(\Big[\Big(k - \frac{\alpha}{3 \beta}\Big)^2 - \frac{4}{9 \beta^2} \p{\alpha^2 + 3 \beta \delta}\Big]^{\frac 12}\Big)\Big]^2}.
}
Recall that the branch cut is excluded from the lemma. Furthermore, $Z$ is undefined when $\Im(k) = 0$ (which implies that the radicand must be real) and, at the same time, $\big(k - \frac{\alpha}{3 \beta}\big)^2 - \frac{4}{9 \beta^2} \p{\alpha^2 + 3 \beta \delta} \ge 0$. When $k$ is real and not on the branch cut, the second condition will always hold. Hence, $Z$ is undefined precisely on the real axis. However, it can be seen directly that these values of $k$ behave consistently with the lemma, as $\Im(\omega) = \Im(\nu_0) \Im(\nu_+) \Im(\nu_-) = 0$ for $k\in \R$.
Otherwise, for $k\notin \R$ and off the branch cut, $Z$ is well-defined and  positive. Since $\Im(\omega) = -\beta \Im(k) H(k)$ in view of \eqref{reomega}, and $\Im(\nu_0) \Im(\nu_+) \Im(\nu_-) = \Im(k) H(k) Z(k)$, the lemma follows.
\end{proof}

Lemma \ref{lemma:Dchar} identifies $D$ as the region where $\Im(\nu_0) \Im(\nu_+) \Im(\nu_-) > 0$ (up to a portion of the branch cut that is missing in the case $\alpha^2 + 3 \beta \delta < 0$). However, it would be helpful if we could characterize the three regions $D_0$, $D_+$, $D_-$ individually. This is accomplished by the next lemma.

\begin{lemma} \label{lemma:Dschar}
For each $n \in \set{0, +, -}$, the set $D_n$ without the branch cut consists of precisely those values of $k$ that are off the branch cut and such that $\Im(\nu_n) > 0$ and the remaining two symmetries have negative imaginary parts.
\end{lemma}

\begin{proof}
Let $k$ be a point off the branch cut. Suppose first that $k \notin D = D_0 \cup D_+ \cup D_-$ so that $\Im(\omega) \geq 0$. Then, by Lemma \ref{lemma:Dchar}, it must be that $\Im(\nu_0) \Im(\nu_+) \Im(\nu_-) \le 0$ and hence it is impossible to have one symmetry with a positive imaginary part and the other two with negative imaginary parts. This implies that the conditions on the imaginary parts of the symmetries described by the lemma will not be satisfied.

Now suppose that $k \in D$ and not on the branch cut so that, again by Lemma \ref{lemma:Dchar}, we have $\Im(\nu_0) \Im(\nu_+) \Im(\nu_-) > 0$. For this inequality to hold, either exactly two factors must be negative or all three must be positive. It turns out that the latter case can never happen. To see this, recall that the symmetries satisfy the equation $\beta \nu^3 - \alpha \nu^2 - \delta \nu - \beta k^3 + \alpha k^2 - \delta k = 0$. Thus, by Vieta's formula, $\nu_0 + \nu_+ + \nu_- = \frac{\alpha}{\beta}$, so
\eq{\Im(\nu_0) + \Im(\nu_+) + \Im(\nu_-) = 0. \label{vietaim}}
As all three imaginary parts cannot be positive simultaneously, exactly one of $\Im(\nu_0)$, $\Im(\nu_+)$, and $\Im(\nu_-)$ must be positive. By continuity of each symmetry (which, in $D$, is only lost when crossing the branch cut in the case that $\alpha^2 + 3 \beta \delta < 0$), we can check one point in each of $D_0$, $D_+$, and $D_-$ to determine the signs of the symmetries in the entirety of the region. This process results in the characterization given by the lemma.
\end{proof}

The following lemma expands the results of Lemma \ref{lemma:Dchar} to include the boundary of $D$.

\begin{lemma}[Characterization of $\overline D_n$ beyond the branch points] \label{lemma:charDbar}
Suppose $k \in \C$ is such that $\big|k - \frac{\alpha}{3 \beta}\big| \geq \sqrt 6 \big|b_\pm - \frac{\alpha}{3 \beta}\big| = \frac{2\sqrt 2}{\sqrt 3 \beta} \sqrt{\abs{\alpha^2 + 3 \beta \delta}}$ and let $n \in \set{0, +, -}$.
\begin{enumerate}[label=\textnormal{(\roman*)}, leftmargin=7mm, itemsep=1mm, topsep=2mm]
\item $k \in \overline{D_n}$ if and only if $\Im(\nu_n) > 0$ and the other two symmetries have nonpositive imaginary parts. Moreover, if $k \in \overline D_n$, then these two remaining symmetries will simultaneously have zero imaginary parts if and only if $\alpha^2 + 3 \beta \delta = 0$ and $k = \frac{\alpha}{3 \beta}$.
\item If $k \in \overline D_n$, then $\Im(\nu_n)$ is bounded below by a constant multiple of $\big|k - \frac{\alpha}{3 \beta}\big|$. Specifically,
\eq{\label{imnu0bound}
\Im(\nu_0) \ge \frac{\sqrt{23}}{4 \sqrt 2} \, \Big|k - \frac{\alpha}{3 \beta}\Big|, \ k \in \overline D_0, 
\quad
\Im(\nu_\pm) \ge \frac{1}{4} \, \Big|k - \frac{\alpha}{3 \beta}\Big|, \ k\in \overline D_\pm.
}
\end{enumerate}
\end{lemma}

\begin{remark}
In the special case of $\beta = 1$, $\alpha = \delta = 0$, in which the linear higher-order Schr\"odinger equation in~\eqref{hls-ibvp} reduces to the Airy equation, the results of Lemma \ref{lemma:charDbar} are trivial. Indeed, the symmetries simplify to $\nu_0 = k$, $\nu_+ = e^{i \frac{2 \pi}{3}} k$, $\nu_- = e^{i \frac{4 \pi}{3}} k$ and the sets $\set{k \in \C : \Im(\nu_n) > 0}$ for $n \in \set{0, +, -}$ are simply rotations of the upper half-plane. For each choice of $n$, asserting that $\Im(\nu_n) > 0$ and the other two symmetries have nonpositive real part clearly corresponds to the region $\overline D_n$ (refer to the case $\alpha^2 + 3 \beta \delta = 0$ in Figure \ref{fig:Dpm}). Lemma \ref{lemma:charDbar} establishes that the same basic behavior holds for all valid choices of parameters $\beta, \alpha, \delta$, even though the regions $D_n$ are more complicated and the symmetries are no longer simple rotations of $k$.
\end{remark}

\noindent
\textit{Proof.}
(i) Suppose $k \in \overline D_n$ lies beyond the branch points in the sense implied by the statement of the lemma. If $k \in D_n$, then Lemma \ref{lemma:Dschar} proves (i). Therefore, we will focus on the boundaries of $D_n$ that are beyond the branch points. Suppose $k$ is on the real axis. Obviously $\Im(\nu_0) = 0$, and by \eqref{imnupm} together with \eqref{rootdefp}, \eqref{rootdefm}, \eqref{rootdefz}, and the fact that we are beyond the branch points, $\Im(\nu_\pm)$ are both nonzero with opposite signs. To maintain their continuity, their signs should be the same as if they were in $D_n$, so (i) is consistent on the real axis.

If $k$ lies on the hyperbola (or, in the case of $\alpha^2 + 3 \beta \delta = 0$, the diagonal lines) defining part of $\partial D_n$, then by~\eqref{imsnupmHZ} at least one of $\Im(\nu_\pm)$ is zero. However, since we are not on the real axis, \eqref{vietaim} requires that exactly one must be zero and the other has the opposite sign of $\Im(\nu_0)$. This is sufficient to prove the case of $k \in \overline D_0$. For $k \in \overline D_+$ on the hyperbola, because the imaginary parts $\nu_0, \nu_-$ must have matching signs within $D_+$, we can use continuity to claim that $\Im(\nu_-)$ cannot have the opposite sign of $\Im(\nu_0)$ on the hyperbola. Hence, it must be the case that $\Im(\nu_+) > 0$ to balance the negativity of $\Im(\nu_0)$, while $\Im(\nu_-)$ becomes zero. The argument is similar for $\overline D_-$.

Now suppose that, for some $k$ beyond the branch points, $\Im(\nu_n) > 0$ and the other two symmetries have nonpositive imaginary parts. If both of the latter two symmetries have negative imaginary parts, then $k \in D_n \subset \overline D_n$ by Lemma \ref{lemma:Dschar}. Note that the  imaginary parts of these two symmetries cannot vanish simultaneously, because then~\eqref{vietaim} would be violated. Thus, assume that exactly one of them has zero imaginary part. Given an arbitrarily small neighborhood of $k$, we can always find a point in that neighborhood such that the imaginary part of this symmetry becomes negative, which implies that $k \in \overline D_n$. For example, if we have that $\Im(\nu_0) > 0$, $\Im(\nu_+) = 0$, and $\Im(\nu_-) < 0$, then in light of \eqref{imnupm}, we can adjust $k$ by an arbitrarily small amount in such a way that $\Re \big(\big[\big(k - \frac{\alpha}{3 \beta}\big)^2 - \frac{4}{9 \beta^2} \p{\alpha^2 + 3 \beta \delta}\big]^{\frac 12}\big)$ remains constant, $\Im(k)$ increases, and $\Im(\nu_0)$, $\Im(\nu_-)$ maintain their signs. Since this new point satisfies $\Im(\nu_0) > 0$ and the other symmetries have negative imaginary parts, this point belongs to $D_0$, so $k \in \overline D_0$.
\\[2mm]
(ii) Starting from $\Im(\nu_0) = \Im(k)$ in $\overline D_0$, we write
$\Im(k) = \big|k - \frac{\alpha}{3 \beta}\big| \sin \theta$, where $\theta = \text{arg}\big(k - \frac{\alpha}{3 \beta}\big)$. Taking into consideration the geometries shown in Figure \ref{fig:Dpm}, the smallest possible value of $\sin \theta$ will occur in the case $\alpha^2 + 3 \beta \delta > 0$ at the point where the hyperbola is closest to the real axis. This happens where the hyperbola is a distance of $\frac{2 \sqrt 2}{\sqrt 3 \beta} \sqrt{\alpha^2 + 3 \beta \delta}$ away from $\frac{\alpha}{3 \beta}$, as this is where we stop considering $\overline D_0$. Here, \eqref{hyperbola} implies that if $k$ is on the hyperbola in $\partial D_0$ and $\big|k - \frac{\alpha}{3 \beta}\big| = \frac{2 \sqrt 2}{\sqrt 3 \beta} \sqrt{\alpha^2 + 3 \beta \delta}$, then $\sin^2 \theta \ge \frac{23}{32}$ so $\Im(\nu_0) \ge \frac{\sqrt{23}}{4 \sqrt 2} \, \big|k - \frac{\alpha}{3 \beta}\big|$.

As for $k \in \overline D_\pm$, \eqref{imnupm}, \eqref{rootdefp}, \eqref{rootdefm} and \eqref{rootdefz} imply
\eqs{
\Im(\nu_\pm) &\ge \pm \frac{\sqrt 3}{2} \Re \Big(\Big[\Big(k - \frac{\alpha}{3 \beta}\Big)^2 - \frac{4}{9 \beta^2} \p{\alpha^2 + 3 \beta \delta}\Big]^{\frac 12}\Big) = \frac{\sqrt 3}{2} \, \Big|\Re \Big(\Big[\Big(k - \frac{\alpha}{3 \beta}\Big)^2 - \frac{4}{9 \beta^2} \p{\alpha^2 + 3 \beta \delta}\Big]^{\frac 12}\Big)\Big|.
}
Using the identity $\big[\Re(\zeta^{\frac{1}{2}})\big]^2 = \frac{1}{2} \p{\abs \zeta + \Re(\zeta)}$, $\zeta \in \C$, along with the half-angle formula, we eventually find
\eqs{
\Im(\nu_\pm) &\ge \frac{\sqrt 3}{2} \, \Big|k - \frac{\alpha}{3 \beta}\Big| \sqrt{\cos^2 (\theta) - \frac{2 \p{\abs{\alpha^2 + 3 \beta \delta} + \alpha^2 + 3 \beta \delta}}{9 \beta^2 \big|k - \frac{\alpha}{3 \beta}\big|^2}},
}
where again $\theta = \text{arg} (k - \frac{\alpha}{3 \beta})$. If $\alpha^2 + 3 \beta \delta \le 0$, then the above inequality becomes 
$\Im(\nu_\pm) 
\geq 
\frac{\sqrt 3}{2} \big|k - \frac{\alpha}{3 \beta}\big| \abs{\cos\theta}$
and we need to find a lower bound for $\abs{\cos \theta}$. According to \eqref{hyperbola}, if $\alpha^2 + 3 \beta \delta \le 0$, then the hyperbola is at a distance $\frac{2\sqrt 2}{\sqrt 3 \beta} \sqrt{\abs{\alpha^2 + 3 \beta \delta}}$ from $\frac{\alpha}{3 \beta}$ when $\big(k_R - \frac{\alpha}{3 \beta}\big)^2 = \frac{7 \abs{\alpha^2 + 3 \beta \delta}}{12 \beta^2}$. Therefore, $\cos^2 (\theta) \ge \frac{7}{32}$, which implies $\Im(\nu_\pm) \ge \frac{\sqrt{7}}{4 \sqrt 2} \big|k - \frac{\alpha}{3 \beta}\big|$. If instead $\alpha^2 + 3 \beta \delta > 0$, then using the fact that $\big|k - \frac{\alpha}{3 \beta}\big|\geq \frac{2\sqrt 2}{\sqrt 3\beta}\sqrt{\abs{\alpha^2 + 3 \beta \delta}}$ we have
$$
\frac{2\b{\abs{\alpha^2 + 3 \beta \delta}+\p{\alpha^2 + 3 \beta \delta}}}{9 \beta^2 \big|k - \frac{\alpha}{3 \beta}\big|^2}
=
\frac{4\abs{\alpha^2 + 3 \beta \delta}}{9 \beta^2 \big|k - \frac{\alpha}{3 \beta}\big|^2}
\leq
\frac 16
$$
so, since now $\cos(\theta) \geq \frac 12$ due to the angle of the asymptotes of the hyperbola, we obtain the desired inequality
\eqs{
\Im(\nu_\pm) 
\geq 
\frac{\sqrt 3}{2}\, \Big|k - \frac{\alpha}{3 \beta}\Big| \sqrt{\cos^2(\theta) - \frac 16}
\geq 
\frac{1}{4} \, \Big|k - \frac{\alpha}{3 \beta}\Big|.
}

\vspace*{-.75cm} \qed

\vspace*{5mm}

With a more thorough understanding of the region $D$, we are ready to employ the symmetries $\nu_\pm$ in order to remove the unknown terms from  the integral representation \eqref{intrepdef}. Returning to the global relation \eqref{preinv} and replacing $k$ by $\nu_+$ and $\nu_-$ gives rise to the two additional relations
\eq{
e^{-i \omega t} \hat u(\nu_+, t) &= \hat u_0 (\nu_+) - i \int_0^t e^{-i\omega t'} \hat f(\nu_+, t') dt' + \beta \tilde g_2 (\omega, t) + i \p{\beta \nu_+ - \alpha} \tilde g_1 (\omega, t) - \p{\beta \nu_+^2 - \alpha \nu_+ - \delta} \tilde g_0 (\omega, t)
\nn\\
&\quad - e^{-i \nu_+ \ell} \b{\beta \tilde h_2 (\omega, t) + i \p{\beta \nu_+ - \alpha} \tilde h_1 (\omega, t) - \p{\beta \nu_+^2 - \alpha \nu_+ - \delta} \tilde h_0 (\omega, t)}, \quad k\in \C, \label{GR+}
\\
e^{-i \omega t} \hat u(\nu_-, t) &= \hat u_0 (\nu_-) - i \int_0^t e^{-i\omega t'} \hat f(\nu_-, t') dt'+ \beta \tilde g_2 (\omega, t) + i \p{\beta \nu_- - \alpha} \tilde g_1 (\omega, t) - \p{\beta \nu_-^2 - \alpha \nu_- - \delta} \tilde g_0 (\omega, t) \nn\\
&\quad - e^{-i \nu_- \ell} \b{\beta \tilde h_2 (\omega, t) + i \p{\beta \nu_- - \alpha} \tilde h_1 (\omega, t) - \p{\beta \nu_-^2 - \alpha \nu_- - \delta} \tilde h_0 (\omega, t)}, \quad k\in \C. \label{GR-}
}

The two identities  \eqref{GR+} and \eqref{GR-} together with the original global relation \eqref{preinv} are valid for all values of $k \in \tilde D$ and can be treated as a $3 \times 3$ linear system for the unknown transforms $\beta \tilde g_2 (\omega, t)$, $\tilde g_1 (\omega, t)$, $\beta \tilde h_2 (\omega, t)$. Solving this system for these unknowns and substituting the resulting expressions into the integral representation  \eqref{intrepdef}  while simplifying using Vieta's formulas and the fact that $\omega' = -\beta \mu_+ \mu_-$, we eventually obtain 
\eq{
u(x, t) &=\frac{1}{2 \pi} \int_{-\infty}^{\infty} e^{i k x + i \omega t} \Big[\hat u_0 (k) - i \int_0^t e^{-i\omega t'} \hat f(k, t') dt'\Big] dk
\nn\\
&\quad - \frac{1}{2 \pi} \int_{\partial \tilde D_0} \frac{e^{i k x + i \omega t}}{\Delta(k)}  \p{\mu_+ e^{-i \nu_+ \ell} + \mu_- e^{-i \nu_- \ell}} \Big[\hat u_0 (k) - i \int_0^t e^{-i\omega t'} \hat f(k, t') dt'\Big] dk
\nn\\
&\quad + \frac{1}{2 \pi} \int_{\partial \tilde D_+ \cup \partial \tilde D_-} \frac{e^{-i k (\ell - x) + i \omega t}}{\Delta(k)} \mu_0 \Big[\hat u_0 (k) - i \int_0^t e^{-i\omega t'} \hat f(k, t') dt'\Big] dk
\nn\\
&\quad + \frac{1}{2 \pi} \int_{\partial \tilde D_0 \cup \partial \tilde D_+ \cup \partial \tilde D_-} \frac{e^{-i k \p{\ell - x} + i \omega t}}{\Delta(k)}  \bigg\{ \mu_+ \Big[\hat u_0 (\nu_+) - i \int_0^t e^{-i\omega t'} \hat f(\nu_+, t') dt'\Big]
\nn\\
&\qquad
 + \mu_- \Big[\hat u_0 (\nu_-) - i \int_0^t e^{-i\omega t'} \hat f(\nu_-, t') dt'\Big]
 - \mu_0 \omega' \, \tilde g_0 (\omega, t) 
\nn\\
&\qquad
 - \p{\nu_- e^{-i \nu_+ \ell} - \nu_+ e^{-i \nu_- \ell}} \omega' \, \tilde h_0 (\omega, t)
- i \p{e^{-i \nu_+ \ell} - e^{-i \nu_- \ell}} \omega' \, \tilde h_1 (\omega, t) 
 \bigg\} \, dk
 \nn\\
&\quad - \frac{1}{2 \pi} \int_{\partial \tilde D_0} \frac{e^{i k x}}{\Delta(k)} \b{-\p{\mu_+ e^{-i \nu_+ \ell} + \mu_- e^{-i \nu_- \ell}} \hat u(\nu_0, t) + \mu_+  e^{-i \nu_0 \ell}\hat u(\nu_+, t) + \mu_- e^{-i \nu_0 \ell} \hat u(\nu_-, t)} dk
\nn\\
&\quad - \frac{1}{2 \pi} \int_{\partial \tilde D_+ \cup \partial \tilde D_-} \frac{e^{-i k \p{\ell - x}}}{\Delta(k)} \b{\mu_0 \hat u(\nu_0, t) + \mu_+ \hat u(\nu_+, t) + \mu_- \hat u(\nu_-, t)} dk,
\label{intrepsub}
}
where $\Delta(k)$ is defined by \eqref{Delta} and we recall that $\nu_0=k$. 

The expression \eqref{intrepsub} still contains unknowns in the form of the transforms $\hat u(\nu_n, t)$, $n\in\set{0, +, -}$. In addition, one needs to ensure that $\Delta(k)$ does not vanish on any of the contours of integration and, furthermore, in the neighborhoods surrounding the zeros of $\Delta$ our integrands may lose analyticity, preventing us from using Cauchy's theorem to eliminate the aforementioned unknowns (see relevant argument later). 
The following lemma guarantees that, if $\Delta(k)$ has any zeros, these do not lie in the closure of $\tilde D_0 \cup \tilde D_+ \cup \tilde D_-$ and hence $\Delta(k)$ is bounded away from zero in that region.

\begin{lemma} \label{lemma:Deltabound}
Let $R_\Delta>0$ be given by \eqref{rd-def} so that, in particular, $R_\Delta$ is greater than the distance $\frac{2}{3 \beta} \sqrt{\abs{\alpha^2 + 3 \beta \delta}}$ of $\frac{\alpha}{3 \beta}$ from the branch points, should these exist. Then, there is some number $c > 0$ such that, for any $n \in \set{0, +, -}$, if $\big|k - \frac{\alpha}{3 \beta}\big| \ge R_\Delta$ and $k \in \overline{D_n}$ then $\abs{e^{i \nu_n \ell} \Delta(k)} \ge c \big|k - \frac{\alpha}{3 \beta}\big|$. 
\end{lemma}

\begin{proof}
We begin by noting that the value of $R_\Delta$ provided by \eqref{rd-def} is certainly not optimal as such a task is of no consequence for our purposes.
Alongside the fact that $\big|k - \frac{\alpha}{3 \beta}\big| > \frac{2}{3 \beta} \sqrt{\abs{\alpha^2 + 3 \beta \delta}}$, repeated application of the triangle inequality yields the bounds
\eq{
\sqrt 3\, \Big|k - \frac{\alpha}{3 \beta}\Big| \sqrt{1 - \frac{4 \abs{\alpha^2 + 3 \beta \delta}}{9 \beta^2 \big|k - \frac{\alpha}{3 \beta}\big|^2}} \le \abs{\mu_0} &\le \sqrt 3 \, \Big|k - \frac{\alpha}{3 \beta}\Big| \sqrt{1 + \frac{4 \abs{\alpha^2 + 3 \beta \delta}}{9 \beta^2 \big|k - \frac{\alpha}{3 \beta}\big|^2}}, \label{mu0bound}\\
\Big|k - \frac{\alpha}{3 \beta}\Big| \Bigg(\frac{3}{2} - \frac{\sqrt 3}{2} \sqrt{1 + \frac{4 \abs{\alpha^2 + 3 \beta \delta}}{9 \beta^2 \big|k - \frac{\alpha}{3 \beta}\big|^2}}\,\Bigg) \le \abs{\mu_\pm} &\le \Big|k - \frac{\alpha}{3 \beta}\Big| \Bigg(\frac{3}{2} + \frac{\sqrt 3}{2} \sqrt{1 + \frac{4 \abs{\alpha^2 + 3 \beta \delta}}{9 \beta^2 \big|k - \frac{\alpha}{3 \beta}\big|^2}}\,\Bigg). \label{mupmbound}
}
If $n = 0$, then by the characterization of $\overline{D_0}$ beyond the branch points given in Lemma \ref{lemma:charDbar} we have
\eqs{
\abs{e^{i \nu_0 \ell} \Delta(k)} &\ge \abs{\mu_0} - \p{\abs{\mu_+} + \abs{\mu_-}} e^{-\Im \p{\nu_0} \ell}\\
&\ge \Big|k - \frac{\alpha}{3 \beta}\Big| \Bigg[\sqrt 3 \sqrt{1 - \frac{4 \abs{\alpha^2 + 3 \beta \delta}}{9 \beta^2 \big|k - \frac{\alpha}{3 \beta}\big|^2}} - \Bigg(3 + \sqrt 3 \sqrt{1 + \frac{4 \abs{\alpha^2 + 3 \beta \delta}}{9 \beta^2 \big|k - \frac{\alpha}{3 \beta}\big|^2}}\, \Bigg) e^{-\Im \p{\nu_0} \ell}\,\Bigg]. 
}
Since $\big|k - \frac{\alpha}{3 \beta}\big| \ge \frac{2 \sqrt 2}{\sqrt 3 \beta} \sqrt{\abs{\alpha^2 + 3 \beta \delta}}$, the two square roots containing $k$ on the right-hand side are bounded below by $\frac{\sqrt 5}{\sqrt 6}$ and  $\frac{\sqrt 7}{\sqrt 2}$ respectively. Furthermore, by \eqref{imnu0bound} and the fact that $\big|k - \frac{\alpha}{3 \beta}\big| \ge \frac{9}{\ell}$, $\Im(\nu_0) \ell \ge \frac{3 \sqrt{23}}{\sqrt 2}$. Altogether, these bounds combine to produce 
$
\abs{e^{i \nu_0 \ell} \Delta(k)} \ge \big|k - \frac{\alpha}{3 \beta}\big| \big[\frac{\sqrt{5}}{\sqrt 2} - \big(3 + \frac{\sqrt 7}{\sqrt 2}\big) e^{-\frac{9 \sqrt{23}}{4 \sqrt 2}}\big]
$
as claimed.
For $n \in \set{+,-}$ similar steps result in the inequality 
$
\abs{e^{i \nu_\pm \ell} \Delta(k)}
\ge \big|k - \frac{\alpha}{3 \beta}\big|
\big[\big(\frac{3 \sqrt 2- \sqrt 7}{2 \sqrt 2}\big)
- \big(\frac{3 \sqrt 2 + 3 \sqrt 7}{2 \sqrt 2}\big) e^{-\frac{9}{4}}\big]$, 
concluding the proof.
\end{proof}

We will now argue that the terms involving the unknowns $\hat u(\nu_n, t)$, $n\in\set{0, +, -}$, on the right-hand side of \eqref{intrepsub} have zero contribution in that expression. This process involves the consideration of the integrals of each term taken over arcs of radius $R$ centered at $\frac{\alpha}{3 \beta}$ that subtend the respective $\tilde D_n$. By Cauchy's theorem and analyticity of the relevant integrands, it suffices to show that the values of these integrals decay uniformly to zero as $R \to \infty$.
For each term containing $\hat u(\nu_n, t)$ within the final two integrals of \eqref{intrepsub} and for any fixed $R = \big|k - \frac{\alpha}{3 \beta}\big|$, we parametrize the arc by $k = \frac{\alpha}{3 \beta} + R e^{i \theta}$ for appropriate bounds on $\theta$ so that the arc subtends the associated region $\tilde D_n$. 

The significant factors appearing alongside $\hat u(\nu_n, t)$ are the following. (i) Each term has a factor of $\frac{1}{\Delta(k)} = \frac{e^{i \nu_n \ell}}{e^{i \nu_n \ell} \Delta(k)}$. By Lemma \ref{lemma:Deltabound}, the behavior of $\frac{1}{\Delta(k)}$ is like $\frac 1R e^{i \nu_n \ell}$.
(ii) There is always one factor of $\mu_0$, $\mu_+$, or $\mu_-$, whose magnitudes according to \eqref{mu0bound} and \eqref{mupmbound} behave as factors of $R$.
(iii) In view of the definition \eqref{fi-ft-def}, the transform $\hat u(\nu_n, t)$ contains the exponential $e^{-i \nu_n y}$ (here, $n$ does not necessarily correspond with the region $\tilde D_n$).
(iv) A factor of $i R e^{i \theta}$ appears due to the parametrization. 
(v) Finally, any other exponentials appearing explicitly within the term must also be considered.
Overall, thanks to Lemma \ref{lemma:charDbar}, the aforementioned factors are guaranteed to combine to produce uniform exponential decay. Therefore, the final two integrals in \eqref{intrepsub} vanish, resulting in the explicit solution formula
\eq{
\begin{split}
u(x, t) &=\frac{1}{2 \pi} \int_{-\infty}^{\infty} e^{i k x + i \omega t} \Big[\hat u_0 (k) - i \int_0^t e^{-i\omega t'} \hat f(k, t') dt'\Big] dk
\\
&\quad - \frac{1}{2 \pi} \int_{\partial \tilde D_0} \frac{e^{i k x + i \omega t}}{\Delta(k)}  \p{\mu_+ e^{-i \nu_+ \ell} + \mu_- e^{-i \nu_- \ell}} \Big[\hat u_0 (k) - i \int_0^t e^{-i\omega t'} \hat f(k, t') dt'\Big] dk
\\
&\quad + \frac{1}{2 \pi} \int_{\partial \tilde D_+ \cup \partial \tilde D_-} \frac{e^{-i k (\ell - x) + i \omega t}}{\Delta(k)} \mu_0 \Big[\hat u_0 (k) - i \int_0^t e^{-i\omega t'} \hat f(k, t') dt'\Big] dk
\\
&\quad + \frac{1}{2 \pi} \int_{\partial \tilde D_0 \cup \partial \tilde D_+ \cup \partial \tilde D_-} \frac{e^{-i k \p{\ell - x} + i \omega t}}{\Delta(k)}  \bigg\{ \mu_+ \Big[\hat u_0 (\nu_+) - i \int_0^t e^{-i\omega t'} \hat f(\nu_+, t') dt'\Big]
\\
&\qquad
 + \mu_- \Big[\hat u_0 (\nu_-) - i \int_0^t e^{-i\omega t'} \hat f(\nu_-, t') dt'\Big]
 - \mu_0 \omega' \, \tilde g_0 (\omega, t) 
\\
&\qquad
 - \p{\nu_- e^{-i \nu_+ \ell} - \nu_+ e^{-i \nu_- \ell}} \omega' \, \tilde h_0 (\omega, t)
- i \p{e^{-i \nu_+ \ell} - e^{-i \nu_- \ell}} \omega' \, \tilde h_1 (\omega, t) 
 \bigg\} \, dk.
 \end{split}
\label{sol-t}
}

The final step in the derivation of the unified transform solution \eqref{hls-sol} consists in replacing the time transforms of the boundary data and the forcing in formula \eqref{sol-t} by 
$$
\tilde g_0(\omega, T), \quad \tilde h_0(\omega, T), \quad \tilde h_1(\omega, T), 
\quad
\int_0^T e^{-i \omega t'} \hat f(k, t') dt'
$$  
for any fixed $T>t$. We emphasize that this task is not necessary for obtaining the solution to problem \eqref{hls-ibvp}, since~\eqref{sol-t} is indeed an explicit solution formula that contains only known terms. However, replacing the variable $t$ by a fixed $T$ in the above transforms facilitates significantly the analysis of Section \ref{reduced-s} that leads to the proofs of Theorems \ref{thm:fi-se} and \ref{thm:strichartz} for the Sobolev and Strichartz estimates of the reduced initial-boundary value problem \eqref{HNLSreduced}, which as already noted play the key role in establishing  the general linear estimates of Theorem \ref{lin-est-t}. 

The replacement described above amounts to showing that the quantities 
\eq{
\int_t^T e^{-i \omega t'} g_0(t') dt', \quad
\int_t^T e^{-i \omega t'} h_0(t') dt', \quad
\int_t^T e^{-i \omega t'} h_1(t') dt', \quad
\int_t^T e^{-i \omega t'} \hat f(k, t') dt'
\label{tilde-aug}
}
have zero contributions when replacing the relevant transforms in the last  three $k$-integrals of \eqref{sol-t}. Therefore, replacing $t$ with $T$ inside these transforms does not alter \eqref{sol-t} and results in the desired solution formula \eqref{hls-sol}. We refer to \eqref{tilde-aug} as the tilde transform augmentations of the associated functions. 

We proceed as before with an argument relying on Cauchy's theorem similar to the one used for eliminating $\hat u(\nu_n,t)$ from \eqref{intrepsub}. The main difference is that the exponential $e^{-i \omega t'}$ appears from the tilde transform augmentation instead of the exponential $e^{-i \nu_n y}$, which is no longer present. We also have the occasional factor of $\omega'$, which behaves like $R^2$. As before, all terms in consideration can be shown to display uniform exponential decay, except for those terms corresponding to the boundary data $h_0$ and $h_1$ in the regions $\tilde D_\pm$, which become oscillatory as we approach the real axis. For these terms, we first integrate by parts within the tilde transform augmentation. This produces an extra factor of $\frac{1}{\omega}$, which behaves like $R^{-3}$. Hence, while we lose uniform exponential decay, we have uniform algebraic decay of order $R^{-1}$ in the case of $h_0$ and $R^{-2}$ in the case of $h_1$. Consequently, the integrals~\eqref{tilde-aug} vanish when substituted in place of their associated time transforms in \eqref{sol-t}, as claimed.

\end{document}